\documentclass[11pt,reqno,tbtags]{amsart}
\usepackage{amssymb}
\usepackage{mathabx}
\usepackage{ifdraft}
\usepackage{url}
\usepackage[square,numbers]{natbib}

\numberwithin{equation}{section}

\renewcommand\le{\leqslant}
\renewcommand\ge{\geqslant}

\allowdisplaybreaks

\newenvironment{comment}{\setbox0=\vbox\bgroup}{\egroup} 

\newtheorem{theorem}{Theorem}[section]
\newtheorem{lemma}[theorem]{Lemma}

\newtheorem{corollary}[theorem]{Corollary}

\theoremstyle{definition}
\newtheorem{example}[theorem]{Example}
\newtheorem{definition}[theorem]{Definition}

\newtheorem{remark}[theorem]{Remark}

\newtheorem*{ack}{Acknowledgement}

\theoremstyle{remark}
\newtheorem*{claim}{Claim}

\newenvironment{romenumerate}[1][0pt]{
\addtolength{\leftmargini}{#1}\begin{enumerate}
 }{\end{enumerate}}

\newcounter{oldenumi}
\newenvironment{romenumerateq}
{\setcounter{oldenumi}{\value{enumi}}
\begin{romenumerate} \setcounter{enumi}{\value{oldenumi}}}
{\end{romenumerate}}

\newenvironment{bienumerate}
{\begin{enumerate}
 }{\end{enumerate}}

\newenvironment{alphenumerate}
{\begin{enumerate}
 }{\end{enumerate}}

\newcounter{thmenumerate}

\newcounter{romxenumerate}   

\newcounter{xenumerate}   

\newcommand\pfitemx[1]{\par#1:}
\newcommand\pfitemref[1]{\pfitemx{\ref{#1}}}

\newcommand\step[2]{\smallskip\noindent\emph{Step #1: #2} \noindent}

\newcommand{\refT}[1]{Theorem~\ref{#1}}
\newcommand{\refC}[1]{Corollary~\ref{#1}}
\newcommand{\refL}[1]{Lemma~\ref{#1}}
\newcommand{\refR}[1]{Remark~\ref{#1}}
\newcommand{\refS}[1]{Section~\ref{#1}}
\newcommand{\refSS}[1]{Subsection~\ref{#1}}

\newcommand{\refD}[1]{Definition~\ref{#1}}
\newcommand{\refE}[1]{Example~\ref{#1}}

\newcommand{\refApp}[1]{Appendix~\ref{#1}}

\newcommand{\refand}[2]{\ref{#1} and~\ref{#2}}

\begingroup
  \divide\count255 by 60
  \multiply\count255 by -60
  \advance\count255 by \time
  \ifnum \count255 < 10 \xdef\klockan{\the\count1.0\the\count255}
  \else\xdef\klockan{\the\count1.\the\count255}\fi
\endgroup

\newcommand\nopf{\qed}   

\newcommand{\sumi}{\sum_{i=1}^\infty}
\newcommand{\sumj}{\sum_{j=1}^\infty}

\newcommand{\summ}{\sum_{m=1}^\infty}
\newcommand{\sumnn}{\sum_{n=-\infty}^\infty}
\newcommand{\sumn}{\sum_{n=1}^\infty}
\newcommand{\sumik}{\sum_{i=1}^k}
\newcommand{\sumin}{\sum_{i=1}^n}
\newcommand{\sumkn}{\sum_{k=1}^n}
\newcommand{\prodik}{\prod_{i=1}^k}
\newcommand{\sumnm}{\sum_{n=1}^M}

\newcommand\set[1]{\ensuremath{\{#1\}}}
\newcommand\bigset[1]{\ensuremath{\bigl\{#1\bigr\}}}

\newcommand\lrset[1]{\ensuremath{\left\{#1\right\}}}
\newcommand\xpar[1]{(#1)}
\newcommand\bigpar[1]{\bigl(#1\bigr)}
\newcommand\Bigpar[1]{\Bigl(#1\Bigr)}
\newcommand\biggpar[1]{\biggl(#1\biggr)}
\newcommand\lrpar[1]{\left(#1\right)}

\newcommand\xcpar[1]{\{#1\}}

\newcommand\abs[1]{|#1|}
\newcommand\bigabs[1]{\bigl|#1\bigr|}
\newcommand\Bigabs[1]{\Bigl|#1\Bigr|}

\newcommand\lrabs[1]{\left|#1\right|}
\def\rompar(#1){\textup(#1\textup)}    
\newcommand\xfrac[2]{#1/#2}

\newcommand\innprod[1]{\langle#1\rangle}
\newcommand\lrinnprod[1]{\left\langle#1\right\rangle}

\def\xexp(#1){e^{#1}}
\newcommand\ceil[1]{\lceil#1\rceil}
\newcommand\floor[1]{\lfloor#1\rfloor}

\newcommand\nn{^{(n)}}
\newcommand\nnn{_1^{\infty}}
\newcommand\ntoo{\ensuremath{{n\to\infty}}}

\newcommand\bmin{\wedge}
\newcommand\norm[1]{\|#1\|}
\newcommand\nnorm[1]{|||#1|||}

\newcommand\downto{\searrow}
\newcommand\upto{\nearrow}

\newcommand\punkt[1]{\if.#1\else.\spacefactor1000\fi{#1}}
\newcommand\iid{i.i.d\punkt}    
\newcommand\ie{i.e\punkt}
\newcommand\eg{e.g\punkt}
\newcommand\viz{viz\punkt}
\newcommand\cf{cf\punkt}
\newcommand{\as}{a.s\punkt}
\newcommand{\aex}{a.e\punkt}

\newcommand\ii{\mathrm{i}}

\newcommand{\tend}{\longrightarrow}
\newcommand\dto{\overset{\mathrm{d}}{\tend}}
\newcommand\pto{\overset{\mathrm{p}}{\tend}}
\newcommand\wto{\overset{\mathrm{w}}{\tend}}
\newcommand\wxto{\overset{\mathrm{w*}}{\tend}}
\newcommand\asto{\overset{\mathrm{a.s.}}{\tend}}
\newcommand\eqd{\overset{\mathrm{d}}{=}}

\newcommand\bbR{\mathbb R}
\newcommand\bbC{\mathbb C}
\newcommand\bbN{\mathbb N}
\newcommand\bbT{\mathbb T}
\newcommand\bbQ{\mathbb Q}
\newcommand\bbZ{\mathbb Z}

\newcounter{CC}
\newcounter{cc}

\renewcommand\Re{\operatorname{Re}}
\renewcommand\Im{\operatorname{Im}}

\newcommand\E{\operatorname{\mathbb E{}}}
\renewcommand\P{\operatorname{\mathbb P{}}}
\newcommand\Var{\operatorname{Var}}
\newcommand\var{\operatorname{var}}
\newcommand\Cov{\operatorname{Cov}}

\newcommand\Bin{\operatorname{Bin}}
\newcommand\Be{\operatorname{Be}}

\newcommand\supp{\operatorname{supp}}

\newcommand\Tr{\operatorname{Tr}}

\newcommand\ga{\alpha}
\newcommand\gb{\beta}
\newcommand\gd{\delta}
\newcommand\gD{\Delta}
\newcommand\gf{\varphi}
\newcommand\gam{\gamma}

\newcommand\gl{\lambda}

\newcommand\go{\omega}
\newcommand\gO{\Omega}
\newcommand\gs{\sigma}
\newcommand\gS{\Sigma}

\newcommand\gz{\zeta}
\newcommand\eps{\varepsilon}

\renewcommand\phi{\xxx}  

\newcommand\cA{\mathcal A}
\newcommand\cB{\mathcal B}
\newcommand\cC{\mathcal C}
\newcommand\cD{\mathcal D}

\newcommand\cF{\mathcal F}

\newcommand\cI{\mathcal I}
\newcommand\cJ{\mathcal J}

\newcommand\cL{{\mathcal L}}
\newcommand\cM{\mathcal M}
\newcommand\cN{\mathcal N}

\newcommand\cS{{\mathcal S}}
\newcommand\cT{{\mathcal T}}

\newcommand\cZ{{\mathcal Z}}

\newcommand\tM{{\widetilde M}}

\newcommand\hX{{\hat X}}

\newcommand\ett[1]{\boldsymbol1\xcpar{#1}}

\newcommand\etta{\boldsymbol1}

\newcommand\limn{\lim_{n\to\infty}}

\newcommand\qw{^{-1}}

\newcommand\qq{^{1/2}}
\newcommand\qqw{^{-1/2}}

\newcommand\gsf{$\sigma$-field}

\renewcommand{\=}{:=}

\newcommand\intoi{\int_0^1}

\newcommand\oi{[0,1]}
\newcommand\ooi{(0,1]}

\newcommand\setoi{\set{0,1}}

\newcommand\dd{\,\mathrm{d}}

\newcommand{\chf}{characteristic function}
\newcommand{\ui}{uniformly integrable}
\newcommand\rv{random variable}

\newcommand\doi{\ensuremath{D\oi}}
\newcommand\doix[1]{\ensuremath{D(\oi^{#1})}}
\newcommand\doik{\doix{k}}
\newcommand\coi{\ensuremath{C\oi}}
\newcommand\cooi{{c_0\ooi}}
\newcommand\coo{c_{00}}
\newcommand\coooi{{\coo\ooi}}
\newcommand\tensor{\otimes}
\newcommand\ptensor{\widehat\tensor}
\newcommand\itensor{\widecheck\tensor}
\newcommand\tpx[1]{^{\tensor #1}}
\newcommand\tpk{\tpx{k}}
\newcommand\itpx[1]{^{\itensor #1}}
\newcommand\itpk{\itpx{k}}
\newcommand\ptpx[1]{^{\ptensor #1}}
\newcommand\ptpk{\ptpx{k}}
\newcommand\xx{x^*}
\newcommand\yy{y^*}
\newcommand\zz{z^*}
\newcommand\zzz{z\qx}
\newcommand\uapprox{uniformly approximable by finite rank operators} 
\newcommand\capprox{approximable on compacts by finite rank operators} 
\newcommand\normp[1]{\norm{#1}_\pi}
\newcommand\normi[1]{\norm{#1}_\eps}
\newcommand\normlips[1]{\norm{#1}_{\Lip_s}}
\newcommand\ap{approximation property}
\newcommand\q{^*}
\newcommand\kth{$k$:th}
\newcommand\kthm{$k$:th moment}
\newcommand\kkx[1]{#1_1,\dots,#1_k}

\newcommand\tp{tensor product}
\newcommand\qx{^{**}}
\newcommand\mx[2]{#1_1 #2 \dotsm #2 #1_k}
\newcommand\mxx[2]{#1_1 #2 \dots #2 #1_k}
\newcommand\hI{\widehat I}
\newcommand\hf{\hat f}
\newcommand\tI{\tilde I}

\newcommand\nbhb{neighbourhood base}
\newcommand\cardc{\mathfrak c}
\newcommand\cardm{\mathfrak m}
\newcommand\U{\mathrm U}
\newcommand\ps{probability space}
\newcommand\pmx{probability measure}
\newcommand\assep{\as{} separably valued}
\newcommand\wassep{weakly \as{} separably valued}
\newcommand\wassepy{weak \as{} separability}
\newcommand\ofp{(\gO,\cF,\P)}
\newcommand\ettau{\etta_{[u,1]}}
\newcommand\ettaU{\etta_{[U,1]}}

\newcommand\tga{\tilde\ga}
\newcommand\meas{measurable}
\newcommand\Bmeas{Bochner measurable}
\newcommand\Bormeas{Borel measurable}
\newcommand\wmeas{weakly measurable}
\newcommand\cmeas{$\cC$-measurable}
\newcommand\dmeas{$\cD$-measurable}
\newcommand\dunx{T_X}
\newcommand\idt{\mathbf t} 
\newcommand\lxx{;}
\newcommand\seq{_{n=1}^\infty}
\newcommand\SCc{Stone--\v Cech compactification}
\newcommand\comp{^{\mathsf c}}
\newcommand\ogoi{[0,\go_1]}
\newcommand\goi{\go_1}
\newcommand\Ba{\mathsf{Ba}}
\newcommand\qwe[1]{^{(#1)}}
\newcommand\loo{\ell^\infty}
\newcommand\looq{(\ell^\infty)\q}
\newcommand\bbNx{\bbN\q}
\newcommand\gbn{\gb\bbN}
\newcommand\gbnx{\gbn\setminus\bbN}
\newcommand\cb{C_{\mathsf b}}
\newcommand\mr{M_{\mathsf r}}
\newcommand\weakx{weak$^*$}
\newcommand\mux{\mu_*}
\newcommand\Mx{M_*}
\newcommand\mxw{m_*}
\newcommand\axw{A_*}
\newcommand\im{\operatorname{im}}
\newcommand\Lip{\mathrm{Lip}}
\newcommand\Lipw{\mathrm{Lip}^w}
\newcommand\sign{\operatorname{sign}} 
\newcommand\ettae{\etta_E}
\newcommand\cco{\mathcal{C}_0}
\newcommand\llm{L_{\ell,m}}
\newcommand\bw{\ensuremath{\cB_w}}

\newcommand\EE{\bar{\E}}
\newcommand\ck{C(K)}
\newcommand\chI{C(\hI)}
\newcommand\kkg{2k_G}
\newcommand\kkkkg{8k_G}

\newcommand\txx{\tilde x\q}
\newcommand\ppar[1]{\ensuremath{(#1)}}
\newcommand\coss{{c_0(S)}}
\newcommand\cossk{{c_0(S^k)}}
\newcommand\css{{C(S\q)}}
\newcommand\bh{B(H)}
\newcommand\cx{\mathcal K}
\newcommand\ch{\cx(H)}
\newcommand\trx{\mathcal N}
\newcommand\trh{\trx(H)}
\newcommand\Tx{T_X}
\newcommand\Txx{T_X^{\phantom*}}
\newcommand\Txq{T_X\q}
\newcommand\Tyx{T_Y^{\phantom*}}
\newcommand\Tyq{T_Y\q}
\newcommand\gDo{\gD^\circ}
\newcommand\mm{^{(m)}}
\newcommand\mmi{^{(m+1)}}
\newcommand\kkk{^{(k)}}
\newcommand\gax{\ga}
\newcommand\gamx{\gam}
\newcommand\gx{{\xi}}
\newcommand\Bx{B_1}
\newcommand\symm{\operatorname{Symm}}

\newcommand{\Holder}{H\"older}

\newcommand\CS{Cauchy--Schwarz}
\newcommand\CSineq{\CS{} inequality}

\newcommand\citetq[2]{\citeauthor{#2} \cite[{\frenchspacing #1}]{#2}}

\begin{document}
\title
{Higher moments of Banach space valued random variables}

\date{22 August, 2012; corrected 25 August 2012}

\author{Svante Janson}
\thanks{SJ partly supported by the Knut and Alice Wallenberg Foundation}
\address{Department of Mathematics, Uppsala University, PO Box 480,
SE-751~06 Uppsala, Sweden}
\email{svante.janson@math.uu.se}

\author{Sten Kaijser}
\address{Department of Mathematics, Uppsala University, PO Box 480,
SE-751~06 Uppsala, Sweden}
\email{sten.kaijser@math.uu.se}

\subjclass[2010]{60B11; 46G10} 

\begin{comment}  

60 Probability theory and stochastic processes

60B11 Probability theory on linear topological spaces [See also 28C20]

46  Functional analysis 
46G   Measures, integration, derivative, holomorphy (all involving
infinite-dimensional spaces) [See also 28-XX, 46Txx] 

46G10  Vector-valued measures and integration [See also 28Bxx, 46B22] 

\end{comment}

\begin{abstract} 
We define the $k$:th moment of a Banach space valued random variable
as the expectation of its $k$:th tensor power; thus the moment (if it exists)
is an element of a tensor power of the original Banach space.

We study both the projective and injective tensor products, and their
relation. Moreover, in order to be general and flexible, we study three
different types of expectations: Bochner integrals, Pettis integrals and
Dunford integrals. 

One of the problems studied is whether two random variables with the same
injective moments (of a given order)
necessarily have the same projective moments; this is of
interest in applications. We show that this holds if the Banach space has
the approximation property, but not in general.

Several sections are devoted to results
in special Banach spaces, including Hilbert spaces,  $C(K)$ and $D[0,1]$.
The latter space is non-separable, which complicates the arguments, and 
we prove various preliminary results on \eg{} measurability in $D[0,1]$
that we need.

One of the main motivations of this paper is the application to Zolotarev
metrics and their use in the contraction method. This is sketched in an
appendix. 
\end{abstract}

\maketitle

\section{Introduction}\label{S:intro}
Let $X$ be a random variable with values in a Banach space $B$.
To avoid measurability problems, we assume for most of this section for
simplicity 
that $B$ is separable and $X$ \Bormeas; 
see \refS{Smeas} for measurability in the general case. Moreover, for
definiteness, we consider real Banach spaces only; the complex case is similar.

If $\E\norm X <\infty$, then the mean $\E X$ exists as an element of $B$
(\eg{} as a Bochner integral $\int X\dd\P$, see \refS{Sintegrals}).
Suppose now that 
we want to define the \kth{} moments of $X$
for some $k\ge2$, 
assuming for simplicity $\E\norm X ^k<\infty$.
If $B$ is finite-dimensional, then the second moment of $X$ is a matrix (the
covariance matrix, if $X$ is centred), and higher moments are described by
higher-dimensional arrays of joint moments of the components.
In general, it is natural to define the \kth{} moment of $X$ using tensor
products, see \refS{Stensor} for details: $X\tpk$ is a random element of the
projective tensor product 
$B\ptpk$, and 
we define the projective \kth{} moment of $X$ as the expectation
$\E X\tpk\in B\ptpk$
(when this expectation exists, \eg{} if $\E\norm X ^k<\infty$);
we denote this moment by $\E X\ptpk$.
In particular, the second moment is $\E X\ptpx2=\E(X\tensor X)\in B\ptensor B$.

An alternative is to consider 
the injective tensor product 
$B\itpk$ and 
the injective \kth{} moment
$\E X\itpk\in B\itpk$.

Another alternative is to consider \emph{weak moments}, \ie, joint moments of
the real-valued random variables
$\xx(X)$ for $\xx\in B^*$ (the dual space).
The weak \kth{} moment thus can be defined as the function
\begin{equation}\label{momweak}
  (\xx_1,\dots,\xx_k)\mapsto \E\bigpar{\xx_1(X)\dotsm\xx_k(X)} \in \bbR,
\end{equation}
assuming that this expectation always exists (which holds, for example, if
$\E\norm X ^k<\infty$). 
Note  that the weak \kth{} moment is a $k$-linear form on $B^*$.

The purpose of the present paper is to study these moments and their
relations in detail, thus 
providing a platform for further work using moments of 
Banach space valued \rv s.
In particular, we shall give sufficient, and sometimes necessary,
conditions for the existence of moments in various situations.

One example of our results on relations between the different moments
is that, at least in the separable case, the weak
\kth{} moment is equivalent to the injective moment.
(See \refT{TID} for a precise statement.)

We study also the problem of moment equality:
if $Y$ is a second random variable with values in $B$,
we may ask 
 whether $X$ and $Y$ have the same \kth{} moments, for a given $k$ and a
 given type of moment. 
(Assume for example 
that $\E\norm X ^k,\E\norm Y^k <\infty$ so that the moments exist.)
This problem, for the second moment, appears for example in connection with
the central limit theorem for Banach space valued \rv s, see \eg{}
\cite[Chapter 10]{LedouxT} where weak moments are used. (In particular, a
$B$-valued \rv{} $X$ is said to be \emph{pregaussian} if $\E X=0$ and there
exists a Gaussian $B$-valued \rv{} $Y$ such that $X$ and $Y$ have the same
weak second moments.)
This problem is also
important when proving convergence in distribution
of some Banach space valued \rv{s}
using a Zolotarev metric,
see \refApp{Azolo}.
(The Zolotarev metrics are, for example, often used
when applying the contraction method for sequences of random variables with a
suitable recursive structure, see \eg{} 
Neiniger and Sulzbach \cite{NeiningerS}.
This applications is one of the main motivations of the present
paper.)

As an example of results obtained in later sections,
let us consider this problem of moment equality further.
In particular, we want to compare the property that $X$ and $Y$ have the
same \kth{} moment for the different types of moments.
(For simplicity, we assume as above 
that $\E\norm X ^k,\E\norm Y
^k<\infty$, so the moments exist.)

Since the dual space of $B\ptpk$ is the space of bounded $k$-linear forms on
$B$ (see \refS{Stensor}), it follows  (by \refS{Sintegrals})
that the \kth{} projective moments $\E X\ptpk$ and $\E Y\ptpk$ are equal
if and only if
\begin{equation}
  \label{multi=}
\E \ga(X,\dots,X)=\E \ga(Y,\dots,Y),
\qquad \ga\in L(B^k;\bbR),
\end{equation}
where $L(B^k;\bbR)$ denotes the space of
bounded $k$-linear forms on $B$.
(See \refC{Cptpk=}.)

Moreover, by the definition of weak moments, $X$ and $Y$ have the same weak
\kth{} moments if and only if 
\begin{equation}
  \label{weak=}
\E \bigpar{\xx_1(X)\dotsm\xx_k(X)}=\E \bigpar{\xx_1(Y)\dotsm\xx_k(Y)},
\qquad \xx_1,\dots,\xx_k\in B^*.
\end{equation}
We shall see (\refC{Citpk=d}) that 
this 
holds if and only if the injective moments $\E X\itpk$ and $\E Y\itpk$
are equal. 

For any $\xx_1,\dots,\xx_k\in B^*$, the mapping
$(x_1,\dots,x_k)\mapsto \xx_1(x_1)\dotsm\xx_k(x_k)$ is 
a bounded $k$-linear
form on $B$, and thus \eqref{weak=} is a special case of \eqref{multi=}.
Consequently, \eqref{multi=}$\implies$\eqref{weak=}, \ie, if $X$ and $Y$
have the same projective moments, then they have the same weak (and
injective) moments.
Does the converse hold? (This question is of practical importance in
applications of the contraction method, see \refApp{Azolo}.)
We show that this problem is non-trivial, and deeply connected to the
\emph{approximation property} of Banach spaces. 
(See \refS{Sapprox} for definitions and proofs.)
In particular, we have the following results.

\begin{theorem}\label{Tapprox}
  If $B$ is a separable Banach space with the approximation property and $X$
  and\/ $Y$ are 
\rv{s} in $B$ such that\/ $\E\norm X ^k,\E\norm Y ^k<\infty$,
  then
\eqref{multi=} and \eqref{weak=} are equivalent, \ie, $X$ and\/ $Y$ have the
same projective \kth{} moments if and only if they have the same weak \kth{}
moments. 
\end{theorem}

\begin{theorem}\label{Tmotex}
  There exists a separable Banach space $B$ and bounded 
random variables $X$
  and\/ $Y$ in $B$ such that,
for $k=2$, 
\eqref{weak=} holds but not \eqref{multi=}.
\end{theorem}

In \refT{Tmotex}, we may further require $B$ to be reflexive.

Note that all classical Banach spaces have the approximation property
(for example, $\ell^p$, $c_0$, $L^p(\mu)$, $C(K)$, and any Banach space with
a basis), and that counterexamples are notoriously difficult to find.
In fact, the approximation property was formulated and studied by
\citet{Grothendieck-resume,Grothendieck-memoir}, but 
it took almost 20 years until a
Banach space without the approximation property was found
by \citet{Enflo}. Hence, it is unlikely that such spaces will appear in
applications, and \refT{Tapprox} ought to apply to any separable Banach
space $B$ that will be used in practice.
(Note, however, that the non-separable Banach space $B(H)$ of bounded
operators in a Hilbert space lacks the approximation property, see
\citet{Szankowski}.)

In applications, $B$ is often a function space, for example $C\oi$. In this
case, we can weaken the condition for equality of moments further, by
considering only point evaluations in \eqref{weak=}. This yields the
following result, stated here more generally
for $C(K)$ where $K$ is a compact metric space.

\begin{theorem}
  \label{TC}
Let $B=C(K)$ where $K$ is a compact metric space,
and let $X$
and\/ $Y$ be 
\rv{s} in $\ck$ such that $\E\norm X ^k$, $\E\norm Y ^k<\infty$.
Then \eqref{multi=} is equivalent to \eqref{weak=}, and further to
\begin{equation}\label{punkt=}
  \E \bigpar{X(t_1)\dotsm X(t_k)}
=
 \E \bigpar{Y(t_1)\dotsm Y(t_k)},
\qquad t_1,\dots,t_k\in K.
\end{equation}
\end{theorem}

In this case it is thus enough to study joint moments of $X(t)$, and the
\kth{} moment of $X$ is described by the real-valued function
$\E \bigpar{X(t_1)\dotsm X(t_k)}$ on $K^k$.
We shall further see that if $k=2$, the integrability conditions can be
weakened to 
 $\sup_{t\in K}|X(t)|^2<\infty$ 
and  $\sup_{t\in K}|Y(t)|^2<\infty$
(\refT{TCKD2}).

\begin{remark}
  A standard polarisation argument, as in the proof of \refT{Temm}, 
shows that \eqref{weak=} is equivalent to 
\begin{equation}
  \label{weak==}
\E \bigpar{\xx(X)^k}=\E \bigpar{\xx(Y)^k},
\qquad \xx\in B^*.
\end{equation}
In other words, the weak $k$:th moments of $X$ and $Y$ are equal if and only
$\xx(X)$ 
and $\xx(Y)$ have the same $k$:th moments for every $\xx\in B\q$.
Hence we can use \eqref{weak==} instead of \eqref{weak=} in the
results above. 
In contrast, in \eqref{punkt=}, it is essential to allow different
$t_1,\dots,t_k$ and joint moments.
\end{remark}

We have so far, for simplicity, assumed that $B$ is separable. The same
definitions apply in the non-separable case, but there are then technical
complications (concerning measurability) that complicate both statements and
proofs, and our results are less complete in this case.

The space \doi{} is a non-separable Banach space that is important in
applications. 
We treat \doi{} in detail in Sections \ref{SD}--\ref{SDmom};
by special arguments we obtain essentially
the same results as in the separable case.
In particular, we shall see
(\refT{TDX})
that \refT{TC} holds also for $B= \doi$, with the usual measurability
condition in \doi.

For the reader's convenience (and our own), we have 
in \refS{Snotations}
collected some notation used in this paper,
and
in Sections \ref{Smeas}--\ref{Sintegrals}
 preliminaries on 
measurability,
\tp{s} and integration in Banach spaces.
(There are no new results in these sections.)

The main definitions of the moments are given in \refS{Smom}, together
with various results giving sufficient, and sometimes necessary, conditions
for their existence. 
(Some simple examples are given in \refS{Sex}.)
We give simple sufficient conditions that are enough
for many applications, but we also give more precise 
results. We try to be
as precise as possible, and therefore we use three different types of
integrability (Bochner, Pettis and Dunford, see \refS{Sintegrals})
in the definition of the projective and injective
moments, leading to six different cases that we
treat in detail. 
(The multitude of cases may be bewildering, and contributes to the length of 
the paper. The reader is recommended to concentrate on separable Banach
spaces and Bochner integrals at the first reading; this is enough for many
applications. However, for \eg{} applications to \doi, this is not enough
which is one motivation for considering also Pettis and Dunford integrals.)

The \ap{} is defined and used in \refS{Sapprox}.

Sections \ref{SH}--\ref{SDmom}  study special Banach spaces:
Hilbert spaces in \refS{SH};
$L^p(\mu)$ in \refS{SLp};
$C(K)$ (where $K$ is a compact space) in
\refS{SCK} (with emphasis on the separable case, when 
$K$ is metrizable);
$c_0(S)$ in \refS{Sc0} (with emphasis on the non-separable case); 
and finally,
as said above, $\doi$  in 
Sections \ref{SD}--\ref{SDmom}, where
Sections \ref{SD}--\ref{SmeasD} contain various preliminary results on \eg{}
the dual space and maximal ideal space as well as measurability and
separability of \rv{s} in \doi.
In these sections we give many results on existence of moments of the different
types for \rv{s} in these spaces. 

In the final two sections 
we consider the collection of moments of all orders.
\refS{Sunique} shows that, under certain conditions, the moments determine
the distribution,
and \refS{Sconv} treats the problem whether convergence of the moments for a
sequence of \rv{s} implies convergence in distribution. We give both
positive and  negative results.

The appendices discuss two well-known constructions related to moments. 

\refApp{Srepr} describes the construction of a Hilbert space (the
\emph{reproducing Hilbert space}) connected to a $B$-valued \rv; we show
that this is closely related to the injective second moment.

\refApp{Azolo} describes the Zolotarev metrics and their connection to
projective moments; as said above this is a
major motivation for the present paper. 

Throughout we usually try to give as general results as possible.
We also give various   counterexamples showing 
limitations of the results,
especially in the non-separable case; 
these (and many technical remarks) can be skipped at the first reading.
Some open problems are stated explicitly or implicitly.
(As is often the case with integration in Banach spaces, \cf{} \eg{} 
\cite{FremlinT}, the non-separable case is much more complicated than the
separable case and several open problems remain.)

For completeness we include some known results that we need, with or 
without proof; we try to give references in both cases, but omit them for
some results or arguments that we regard as standard. (The absence of a
reference thus does not imply that the result is new.)
We believe that many other results are new.

\begin{ack} 
We thank Joe Diestel, Ralph Neininger, Oleg Reinov, 
Viggo Stoltenberg-Hansen and Henning Sulzbach
for helpful comments.
\end{ack}

\section{Notations}\label{Snotations}

We will use the following standard notations, usually without comment.

$L(B_1,\dots,B_k;B')$ is the space of bounded $k$-linear maps $\mx
B\times\to B'$. 
In particular, with $B'=\bbR$, we have the space of bounded $k$-linear forms.
When $B_1=\dots=B_k=B$ we also write $L(B^k;B')$.

$B\q$ denotes the dual space of the Banach space $B$, 
\ie, the space $L(B\lxx\bbR)$ of bounded linear
functionals $B\to\bbR$. If $x\in B$ and $\xx\in B\q$, we use
the notations $\xx(x)$ and $\innprod{\xx,x}$,
or $\innprod{x,\xx}$,  as synonyms.
We write $\xx\perp B_1$, for a subset $B_1\subseteq B$, if
$\innprod{\xx,x}=0$ for every $x\in B_1$.

We use several standard Banach spaces in our results and examples; for
convenience we recall their definitions here.

For any set $S$ and $p\in[1,\infty)$,
$\ell^p(S)$ is the Banach space of all
functions $S\to\bbR$ such that the norm
$\norm{f}_{\ell^p(S)}\=\lrpar{\sum_{s\in S}|f(s)|^p}^{1/p}$ is finite. 
(We are mainly interested in the cases $p=1,2$.)
Further,
$\ell^\infty(S)$ is the Banach space of all
bounded functions $S\to\bbR$, with the norm
$\norm{f}_{\ell^\infty(S)}\=\sup_{s\in S}|f(s)|$.
We use sometimes the notation $f_s$ for $f(s)$, in particular when $S=\bbN$.

We define the
\emph{support}  of $f$ as $\supp(f)\=\set{s:f(s)\neq0}$, and note that if
$f\in\ell^p(S)$, with $p<\infty$, 
then $\supp(f)$ is   countable even if $S$ is uncountable.

$c_0(S)$ is the space of all function $f:S\to\bbR$ such that
\set{s:|f(s)|>\eps} is finite for each $\eps>0$; this is a 
closed subspace of $\ell^\infty$, and is thus a Banach space with the 
inherited norm $\norm{f}_{c_0(S)}\=\norm{f}_{\ell^\infty(S)}\=\sup_{s\in
  S}|f(s)|$.
Note that every element of $c_0(S)$ has countable support.

$e_s$ denotes the function $e_s(t)\=\ett{t=s}$ that is 1 at $s$ and 0
everywhere else (defined for $t$ in some set $S$, which will be clear from the
context). 

Let $\coo(S)$ be the space of all functions $f:S\to\bbR$ with $\supp (f)$
finite; this is the linear span of \set{e_s:s\in S}.
Then $\coo(S)\subseteq c_0(S)\subseteq\ell^\infty(S)$, and $\coo(S)$ is
dense in $c_0(S)$.
Hence $c_0(S)$ is the closed linear span in $\ell^\infty(S)$ of the
functions $\set{e_s:s\in S}$. 
It follows that
\begin{equation*}
  \coo(S)\q= c_0(S)\q=\ell^1(S),
\end{equation*}
with the standard pairing $\innprod{f,g}=\sum_{s\in S}f(s)g(s)$.

When $S=\bbN$, we write just $\ell^p$ and $c_0$.

$L^p(S)$, where $S=(S,\cS,\mu)$ is a measure space and $p\in[1,\infty)$, 
is the space of all
measurable functions $f:S\to\bbR$ such that $\int_S|f|^p<\infty$ (as usual
identifying functions that are equal \aex).

$C(K)$, where $K$ is a compact topological space, is the space of all
continuous functions $f:K\to\bbR$, with the norm $\sup_K|f|$.
(We are particulary interested in \coi.)

For a compact set $K$,  $M(K)$ is the Banach space of all signed
Borel measures on $K$. 
By the Riesz representation theorem,
the dual $C(K)\q$ can be identified with the subspace
$\mr(K)$ of $M(K)$ 
consisting of regular measures; 
see \eg{} \cite[Theorem 7.3.5]{Cohn}, 
\cite[Theorem  IV.6.3]{Dunford-Schwartz} or  \cite[Theorem  III.5.7]{Conway}.
(If \eg{} $K$ is compact and metrizable, 
then every signed measure is regular, so
$C(K)\q=M(K)$, see \cite[Propositions 7.1.12, 7.2.3, 7.3.3]{Cohn}.)

$\doi$ denotes the linear space of functions $\oi\to\bbR$ that are
right-continuous with left limits, see \eg{} \cite[Chapter 3]{Billingsley}.
The norm is $\sup_{\oi}|f|$. See further \refS{SD}.

$\gd_s$ denotes the Dirac measure at $s$, as well as the corresponding point
evaluation $f\mapsto f(s)$ seen as a linear functional on a suitable
function space.

$\ofp$ denotes the underlying probability space where our random variables
are defined; $\go$ denotes an element of $\gO$.
We  assume that $\ofp$ is complete.

We let $\EE$ denote the upper integral of a, possibly non-measurable, 
real-valued function on $\gO$:  
\begin{equation}\label{ee}
\EE Y \=\inf\bigset{\E Z: Z\ge Y \text{ and $Z$ is \meas}}.  
\end{equation}
(If $Y$ is \meas, then $\EE Y =\E Y$.)
In particular, if $X$ is $B$-valued, then $\EE\norm X<\infty$ if and only if
there exists a positive random variable $Z$ with $\norm X\le Z$ and $\E
Z<\infty$. 

The exponent $k$ is, unless otherwise stated, an arbitrary fixed integer
$\ge1$, but the case $k=1$ is often trivial. For applications, $k=2$ is the
most important, and the reader is adviced to primarily think of that case.

\section{Measurability}\label{Smeas}
 
A $B$-valued \rv{} is a function $X:\gO\to B$ defined on some 
\ps{} $(\gO,\cF,\P)$. (As said above, we  assume that the \ps{} is complete.)
We further want $X$ to be measurable, and there are
several possibilities to consider; we will use the following definitions.

\begin{definition}
Let $X:\gO\to B$ be a function on some probability space $\ofp$
with values in a Banach space $B$.
  \begin{romenumerate}[-10pt]
  \item \label{borel}
$X$ is \emph{Borel measurable} if $X$ is measurable with respect
to the Borel $\gs$-field $\cB$ on $B$, \ie, the $\gs$-field generated by the
open sets. 
  \item \label{weakly}
$X$ is \emph{weakly measurable} if $X$ is measurable with respect
to the $\gs$-field $\bw$ on $B$ generated by the continuous linear functionals,
\ie, if 
$\innprod{\xx,X}$ is measurable for every $\xx\in B\q$.
\item \label{assep}
$X$ is \emph{\assep} if there exists a separable subspace $B_1\subseteq B$
  such that $X\in B_1$ a.s.
\item \label{wassep}
$X$ is \emph{\wassep} if there exists a separable subspace $B_1\subseteq B$
  such that if $\xx\in B\q$ and $\xx\perp B_1$, then $\xx(X)=0$ a.s.
\item \label{bmeas}
$X$ is \emph{\Bmeas} if $X$ is Borel measurable and \assep.
  \end{romenumerate}
\end{definition}

\begin{remark}\label{Rmeas}
$X$ is \Bmeas{} if and only if $X$ is Borel measurable and \emph{tight}, \ie,
  for every $\eps>0$, there exists a compact subset $K\subset B$ such that
  $\P(X\in K)>1-\eps$, see \cite[Theorem 1.3]{Billingsley}.
(This is also called \emph{Radon}.) 
Some authors use the name \emph{strongly \meas}.
Moreover, $X$ is \Bmeas{} if and only if there exists a sequence $X_n$ of
measurable simple functions $\gO\to B$ such that $X_n\to X$ \as, see
\cite[III.2.10 and III.6.10--14]{Dunford-Schwartz}. 
(This is often taken as the definition of \Bmeas.)
See further \cite[Chapter 2.1]{LedouxT}.

The name \emph{scalarly \meas} is sometimes used for \wmeas.

Some authors (\eg{} \cite{Dunford-Schwartz}) let ``measurable'' mean what we
call \Bmeas{}, 
so care should be taken to avoid misunderstandings.
\end{remark}

If $B$ is separable, there is no problem. Then every $X$ is trivially
\assep; moreover, it is easy to see that $\bw=\cB$.
Hence every weakly measurable $X$ is measurable, and the notions
in \ref{borel}, \ref{weakly} and \ref{bmeas} are equivalent.
This extends immediately to \assep{} \rv{s} in an arbitrary Banach space:

\begin{theorem}[Pettis, {\cite[III.6.10--11]{Dunford-Schwartz}}] \label{Twm}
  If $X$ is \assep, then $X$ is Borel
measurable (and thus \Bmeas) if and only if
  $X$ is weakly measurable.
\nopf
\end{theorem}

\begin{remark}\label{RUlam} 
There is a converse, assuming some set theory hypotheses.
(See further \refR{Rcardinals}.)
By \cite{MS}
(see also \cite[Appendix III]{Billingsley1}),
if the cardinality of $B$ is
not real-measurable 
then every Borel measurable \rv{} in $B$ is \assep. 
It follows, in particular,  that if 
the continuum hypothesis holds and furthermore
there is no inaccessible cardinal,
then every Borel measurable \rv{} in any Banach space is \assep. 
(These hypotheses are both consistent with the usual ZFC set theory; see
further \refR{Rcardinals}.)
Hence, it can safely be assumed that every Borel measurable \rv{} that will
appear in an application is \assep. In other words, if we want to study
\rv{s} that are not \assep, then we cannot use Borel measurability.
\end{remark}

In view of this, we prefer to use separable Banach spaces, or at least
\assep{} \rv{s}, whenever possible. Many standard Banach spaces are 
separable, for example $L^p\oi$ ($1\le p<\infty$) and $\coi$.
However, the non-separable space $\doi$ 
(see \refS{SD})
is important for applications, and it is easily seen (\refT{TDassep} below)
that a $D$-valued \rv{} $X$ is \assep{} if and only if there 
exists a (non-random) countable set $A\subset\oi$ such that $X$ \as{} has
all its points of discontinuity in $A$. 
In practical applications 
this means that a
random variable $X\in \doi$ is \assep{} if it has jumps at deterministic
places, but not if there are jumps at random places. 

\begin{example}[{See \cite[Section 15]{Billingsley}}]\label{Eetta}
Let $U\sim \U(0,1)$ be a uniformly distributed random
  variable,
and let $X$ be the random element of $\doi$ given by $X=\etta_{[U,1]}$,
\ie{} $X(t)=\ett{U\le t}$. 
(This is the empirical distribution function of $U$, seen as a sample of
size 1.) Since the functions $\etta_{[u,1]}$, $u\in\oi$, all have distance 1
to each other, they form a discrete subset of $\doi$ and it follows that $X$
is not \assep. 
In particular, $X$ is not Bochner measurable.

Moreover, if $A$ is a non-measurable subset of $\oi$,
then the set $\set{\etta_{[u,1]}: u\in A}$ is a closed subset of $\doi$, and
thus a Borel set. 
If we take the probability space $\ofp$
where $U$ and thus $X$ are 
defined to be $\oi$ with Lebesgue measure, and $U(\go)=\go$,
then $X\qw(A)=A\notin\cF$ and,
consequently,
$X$ is not Borel measurable in $\doi$.
(If we assume the continuum hypothesis, $X$ cannot be Borel measurable for
any \ps{} $\ofp$, see \refR{RUlam}.)

On the other hand, $X(t)$ is measurable for each $t$, and 
by \citet{Pestman}, see \refC{C1} and \refT{TDw} below, $X$ is weakly
measurable.

Furthermore, it follows from \cite{Pestman} (or \refC{C1}) also
that if $\xx\in
D\q$ and $\xx\perp \coi$, then  $\xx(\ettau)=0$ for all but countably many
$u$; thus $\xx(X)=0$ \as{} which shows that $X$ is \wassep.
(Cf.\ \refT{TDwassep}.)

Cf.\ also \cite[Example 3-2-2]{Talagrand:Pettis} which studies essentially
the same example but as an element of $L^\infty\oi$.
\end{example}

\begin{remark}
  If $X$ is Borel \meas, then $\norm X$ is \meas, since $x\mapsto\norm x$ is
  continuous. However, if $X$ only is \wmeas, then $\norm X$ is not always
  \meas{} without additional hypotheses. 
(For this reason, we will sometimes use the upper integral $\EE\norm X$.)
If $X$ is \wmeas{} and \assep, then $\norm X$ is \meas, \eg{} by \refT{Twm}.
(In particular, there is no problem when $B$ is separable.)
Furthermore, 
if $B$ has the property that there exists a countable norm-determining set
of linear functionals, then every \wmeas{} $X$ in $B$ has $\norm X$ \meas;
$\doi$ is an example of such a space.
\end{remark}

\begin{remark}
Several other forms of measurability may be considered, 
for example 
using the \emph{Baire} $\gs$-field (generated by the continuous functions
$B\to\bbR$) 
\cite{Talagrand:Pettis}
or the $\gs$-field generated by the closed (or open) balls in $B$
\cite{Dudley66,Dudley67}, \cite{Billingsley}.
Note further that, in general, 
$\bw$ is not the same as  the $\gs$-field generated by the weak topology.
(In fact, $\bw$ equals the Baire \gsf{} for the weak topology \cite{Edgar},
\cite{Talagrand:Pettis}.) 
When $B$ is separable, all these coincide with the Borel $\gs$-field.

See also \cite{Edgar}, \cite{EdgarII},
\cite{LedouxT}
and \cite{Talagrand:Pettis},
where further possibilities are discussed. 
\end{remark}

\section{Tensor products of Banach spaces}\label{Stensor}

We give a summary of the definitions and some properties
of the two main tensor products
of Banach spaces. We refer to \eg{} \citet{Blei} or \citet{Ryan} for further
details. We consider the general case of the tensor product of $k$ different
spaces. The tensor products we consider
(both algebraic and completed) are associative in a natural way; for
example, 
$B_1\tensor B_2\tensor B_3 = (B_1\tensor B_2)\tensor B_3
= B_1\tensor (B_2\tensor B_3)$, 
and the
general case may be reduced to tensor products of two spaces. (Many authors,
including \cite{Ryan}, thus consider only this case.)

\subsection{Algebraic tensor products}
The \emph{algebraic \tp} of a finite sequence of vector spaces $\kkx B$
(over an arbitrary field) can be defined in an abstract 
way as a vector space $B_1\tensor\dotsm\tensor B_k$ with
a $k$-linear map 
$B_1\times\dotsm\times B_k\to B_1\tensor\dotsm\tensor B_k$, written  
$(x_1,\dots x_k)\to x_1\tensor\dotsm\tensor x_k$, such that if 
$\ga:B_1\times\dotsm\times B_k\to A$ is any $k$-linear map, then there is a
unique linear map $\tga:B_1\tensor\dotsm\tensor B_k\to A$ such that
\begin{equation}
  \label{tp}
\ga(x_1,\dots,x_k)=\tga(x_1\tensor\dotsm\tensor x_k).
\end{equation} 
(All such spaces are naturally isomorphic, so the tensor product is uniquely
defined, up to trivial isomorphisms.)

Several concrete constructions can also be given. One useful construction is
to let $B_i^\sharp$ be the algebraic dual of $B_i$ and define 
$\mx B\tensor$ as a subspace of the linear space of all 
$k$-linear forms on $\mx{B^\sharp}{\times}$; more precisely 
we define
$x_1\tensor\dotsm\tensor x_k$ as the 
$k$-linear form on $\mx{B^\sharp}{\times}$ defined by
\begin{equation}
  \label{onbx}
\mx{x}{\tensor}(\xx_1,\dots,\xx_k)=\xx_1(x_1)\dotsm \xx_k(x_k),
\end{equation}
and then define $\mx B\tensor$  as the linear span of all $\mx x\tensor$.

We can modify this construction by replacing $B_i^\sharp$ by any subspace
that separates the points of $B_i$. In particular, when each $B_i$ is a
Banach space, we can regard $\mx B\tensor$ as a subspace of the space of
$k$-linear forms on $\mx{B\q}\times$.

An element of $\mx B\tensor$  of the form $\mx x\tensor$ is  called an
\emph{elementary tensor}. Note that not
every element of $\mx B\tensor$ is an elementary tensor, but every element
is a finite linear combination of elementary tensors (in a non-unique way).

A sequence of linear operators $T_i:A_i\to B_i$ defines 
a unique linear map $\mx T\tensor:\mx A\tensor\to\mx B\tensor$
such that $\mx T\tensor(\mx x\tensor)=T_1x_1\tensor\dotsm\tensor T_kx_k$
for elementary tensors.

\subsection{Completed tensor products}
When $\kkx B$ are Banach spaces, we can define several different
(non-equivalent) norms on $\mx B\tensor$. For each norm we then can take the
completion of $\mx B\tensor$, obtaining a Banach space.
We consider two cases (the two main cases) of completed tensor products. 

The \emph{projective tensor norm} is defined on $\mx B\tensor$ by
\begin{equation}\label{normproj}
  \normp{u}
\=\inf\Bigpar{\sum_j\norm{x_{1j}}_{B_1}\dotsm \norm{x_{kj}}_{B_k}:
u = \sum_j x_{1j}\tensor\dotsm \tensor x_{kj}},
\end{equation}
taking the infimum over all ways of writing $u$ as a finite sum of
elementary tensors. The corresponding completed tensor product is called the
\emph{projective tensor product} and is written
$\mx B{\ptensor}$.
It is easily seen that every $u\in \mx B\ptensor$ can be written
(non-uniquely) as an absolutely convergent infinite sum of elementary
tensors: for any $\eps>0$,
\begin{equation}\label{proj}
u = \sumj x_{1j}\tensor\dotsm \tensor x_{kj}
\quad\text{with}\quad
\sumj\norm{x_{1j}}_{B_1}\dotsm \norm{x_{kj}}_{B_k}
\le(1+\eps)\normp{u}
<\infty;
\end{equation}
equivalently,
\begin{multline}\label{proj2}
u = \sumj \gl_jx_{1j}\tensor\dotsm \tensor x_{kj}
\quad\text{with}\quad  \norm{x_{ij}}_{B_i}\le1,
\ \gl_j\ge0 \text{ and }
\\
\sumj\gl_j\le(1+\eps)\normp{u}
<\infty.
\end{multline}

The \emph{injective tensor norm} is defined on $\mx B\tensor$ by 
using \eqref{onbx} to regard
$\mx B\tensor$ as a subspace of $L(B_1\q,\dots,B_k\q;\bbR)$, the bounded
$k$-linear forms on $\mx {B\q}\times$, and taking the induced norm, \ie,
\begin{equation}\label{norminj}
  \normi{u}
\=\sup\bigpar{\abs{u(\xx_1,\dots,\xx_k)}: 
\norm{\xx_1}_{B\q_1},\dots,\norm{\xx_k}_{B\q_k}\le1}.
\end{equation}
The corresponding completed tensor product is called the
\emph{injective tensor product} and is written
$\mx B{\itensor}$; note that this is simply the closure of $\mx {B}\tensor$ in
$L(\kkx{B\q};\bbR)$ and thus can be regarded as a closed subspace of
$L(\kkx{B\q};\bbR)$.

For elementary tensors we have
\begin{equation}\label{pik}
  \normp{\mx x\tensor} =   \normi{\mx x\tensor} 
= \norm{x_1}_{B_1}\dotsm  \norm{x_k}_{B_k},
\end{equation}
so the two norms coincide. (Any such norm on $\mx B\tensor$ is called a
\emph{crossnorm}.) 
Moreover, it is easily seen that for any $u\in B_1\tensor\dotsm\tensor B_k$,
$\normi{u}\le\normp{u}$; hence the identity map $u\mapsto u$ extends
to a canonical bounded linear map (of norm 1)
\begin{equation}
  \label{iota}
\iota:B_1\ptensor\dotsm\ptensor B_k \to B_1\itensor\dotsm\itensor B_k
\end{equation}
with $\norm\iota\le1$. Unfortunately, this map is not always injective; we
return to this problem in \refS{Sapprox}, where we shall see that this is
the source of the difference between \eqref{multi=} and \eqref{weak=}.

Consider for simplicity the case $k=2$. As said above, $B_1\itensor B_2$ can
be seen (isometrically) as a subspace of $L(B_1\q,B_2\q;\bbR)$. 
Moreover, an elementary tensor $x_1\tensor x_2$ defines a bounded  linear
operator $B_1\q\to B_2$ by $\xx_1\mapsto \xx_1(x_1)x_2$, and this extends by
linearity to a mapping $B_1\tensor B_2\to L(B_1\q;B_2)$, which is an
isometry for the injective tensor norm \eqref{norminj}.
Hence, this mapping extends 
to an isometric embedding of 
$B_1\itensor B_2$ as a subspace of the space $L(B_1\q;B_2)$ of bounded
linear operators $B_1\q\to B_2$. 
Explicitly, $u\in B_1\itensor B_2$ corresponds to the operator $T_u: B_1\q\to
B_2$ given by
\begin{equation}
  \label{keiller}
\innprod{T_u \xx,\yy}=\innprod{\xx\tensor\yy,u},
\qquad \xx\in B_1\q,\, \yy\in B_2\q.
\end{equation}
(By symmetry, there is also an embedding into $L(B_2\q;B_1)$.
If $u\in B_1\itensor B_2$ corresponds to the operator $T_u: B_1\q\to B_2$, it
also corresponds to $T_u\q:B_2\q\to B_1$.)

When the mapping $\iota$ in \eqref{iota} (with $k=2$) is injective,
we may regard $B_1\ptensor B_2$ as a subspace of $B_1\itensor B_2$ (with a
generally different, larger, norm); hence $B_1\ptensor B_2$ may be regarded as a
certain space of operators $B_1\q\to B_2$ in this case too.

\begin{example}
  \label{Efinite}
Suppose that $B_1,\dots,B_k$ are finite-dimensional, 
and let $\set{e_{ij}}_j$, $j\in \cJ_i$, be a basis of $B_i$ for $i=1,\dots,k$.
It is easily seen that 
the elementary tensors
$e_J\=e_{1j_1}\tensor\dotsm\tensor e_{kj_k}$, with 
$J\=(j_i)_i\in \prodik \cJ_i$, form a basis of
$\mx B\tensor$. 

Hence $\dim(\mx B\tensor)=\prod_{i=1}^k\dim(B_i)<\infty$. In particular, the
tensor product is complete for any norm, and thus
$B_1\ptensor\dotsm\ptensor B_k = B_1\itensor\dotsm\itensor B_k
=B_1\tensor\dotsm\tensor B_k$ as sets (although the norms generally differ).

Using the basis $(e_J)_J$, any element of the \tp{} may be written as $\sum_J
a_Je_J$, where the coordinates $a_J$ are indexed by $J\in \prodik \cJ_i$, so 
it is natural to consider the coordinates as a $k$-dimensional array.
(A matrix in the case $k=2$.)
\end{example}

\begin{example}\label{EHilbert}
Let $H$ be a Hilbert space.
The tensor products $H\ptensor H$ and $H\itensor H$ 
can be identified with the spaces of trace class operators 
and compact operators  on $H$, respectively, see \refT{THtp};
moreover, $H\ptensor H\cong (H\itensor H)\q$.

In this case, another interesting choice of norm on $H\tensor H$ is the
Hilbertian tensor norm, given by the inner product 
$\innprod{x_1\tensor x_2,y_1\tensor y_2} =
\innprod{x_1,y_1}\innprod{x_2,y_2}$; the corresponding completed \tp{}
$H\tensor_2 H$ can be identified with the space of Hilbert--Schmidt operators
on $H$. (We will not use this \tp{} in the present paper.)
\end{example}

If $T_i:A_i\to B_i$ are  bounded linear operators, then $\mx T\tensor$ is
bounded for both the projective norms and the injective norms, and extends
thus to 
bounded linear maps $\mx T\ptensor:\mx A\ptensor\to\mx B\ptensor$
and $\mx T\itensor:\mx A\itensor\to\mx B\itensor$.
We note the following lemma, which we for simplicity state for the case
$k=2$, although the result extends to general $k$.

\begin{lemma}
  \label{Linj}
If\/ $T:A_1\to B_1$ and $U:A_2\to B_2$ are injective linear operators
between Banach spaces, then $T\itensor U:A_1\itensor A_2 \to B_1\itensor
B_2$ is injective.
\end{lemma}

\begin{proof} 
Consider first the case when $A_1=B_1$ and $T=I$, the identity operator.
  We can regard $A_1\itensor A_2$ and $A_1\itensor B_2$ as subspaces of
  $L(A_1\q\lxx A_2)$ and $L(A_1\q\lxx B_2)$ , and then $I\itensor U$ is the mapping 
  $L(A_1\q\lxx A_2)\to L(A_1\q\lxx B_2)$ given by $S\mapsto US$, which is injective
  when $U$ is. 

In general we factorize 
$T\itensor U =  (T\itensor I)(I\itensor U)$ and note that both factors
$T\itensor I: A_1\itensor B_2\to A_2\itensor B_2$
and 
$I\tensor U: A_1\itensor B_1\to A_1\itensor B_2$ are injective by the first
case and symmetry.
\end{proof}

\begin{remark}
  \label{Rinjinj} 
The same proof shows that if $T$ and $U$ are isometric embeddings, then so
is $T\itensor U$. (In other words, the injective tensor product is injective
\cite[Section 6.1]{Ryan}.)
\end{remark}

\begin{remark}\label{Rprojinj}
  \refL{Linj} is in general \emph{not} true for the projective tensor
  product,
see \refE{Exq} and \refR{Rxq}.

The projective tensor product has instead the dual property that if $T$ and
$U$ are quotient mappings (\ie, onto), then so is $T\ptensor U$. 
\end{remark}

When considering the dual space of a completed tensor product, note that 
$\mx B\tensor$ is a dense subspace (by definition); hence a continuous linear
functional is the same as a linear functional on $\mx B\tensor$ that is bounded
for the chosen tensor norm. Furthermore, each such linear functional can by
\eqref{tp} be identified with a $k$-linear form on $\mx B\times$.
For the projective tensor product, the definition \eqref{normproj} of the norm 
implies that all bounded $k$-linear forms yield bounded linear functionals.

\begin{theorem}[{\cite[Theorem 2.9]{Ryan}}]\label{Tproj*}
  The dual $(\mx B\ptensor)\q$ consists of all bounded $k$-linear forms
$\mx B\times\to\bbR$, with the natural pairing
$\innprod{\ga,\mx x\tensor}=\ga(x_1,\dots,x_k)$.
\nopf
\end{theorem}

For $\mx B\itensor$, the dual space consists of a subset of all bounded
$k$-linear forms; these forms are called \emph{integral}, and can be
described as follows. Let $K_i$ be the closed unit ball of $B_i\q$ with the
weak${}\q$ topology; thus $K_i$ is a compact space.

\begin{theorem} [{\cite[Proposition 3.14]{Ryan}}]\label{Tinj*}
If  $\chi\in(\mx B\itensor)\q$, then there exists a (non-unique) signed
measure $\mu\in M(\mx K\times)$, 
with $\norm{\mu}=\norm{\chi}$, 
such that
  \begin{equation}\label{tinj*}
\chi(\mx x\tensor)
=\int_{\mx K\times} \xx_1(x_1)\dotsm\xx_k(x_k)\dd\mu(\xx_1,\dots,\xx_k).	
  \end{equation}
Conversely, \eqref{tinj*} defines a bounded linear functional for every 
signed measure $\mu\in M(\mx K\times)$.
\end{theorem}
\begin{proof}
  We have defined an embedding of $\mx B\tensor$ in $L(B_1\q,\dots,B_k\q;\bbR)$
by \eqref{onbx}, and taking the restriction of the operators to $\mx
K\times$ gives a linear map into $C(\mx K\times)$, which by the definition
\eqref{norminj} of the injective tensor norm is an isometry.
Hence we can regard $\mx B\itensor$ as a subspace of $C(\mx K\times)$, and
the result follows by the Hahn--Banach theorem together with the Riesz
representation theorem.
\end{proof}

An elementary tensor $\mx{\xx}\tensor\in\mx{B\q}\tensor$ defines a 
 $k$-linear form on $\mx B\times$ by
$(x_1,\dots,x_k)\mapsto\xx_1(x_1)\dotsm\xx_k(x_k)$, and thus a linear functional
on $\mx B\tensor$ by 
\begin{equation}\label{*tp}
\innprod{\mx{\xx}\tensor,\mx x\tensor}=\xx_1(x_1)\dotsm\xx_k(x_k).
\end{equation}
By the definitions above, see \eqref{norminj} and \eqref{onbx},
this linear functional extends to a linear functional on  $\mx B\itensor$ of
norm $\norm{\xx_1}\dotsm\norm{\xx_k}$, and it 
follows by \eqref{normproj}, 
linearity and continuity that every tensor $u\q\in\mx{B\q}\ptensor$
defines a linear functional on $\mx B\itensor$ with norm at most
$\norm{u\q}_\pi$, 
\ie, an integral form on $\mx B\times$. 
In fact, by \eqref{proj}, the forms  obtained in this way are exactly the 
integral forms such that there is a representation \eqref{tinj*} with a
discrete measure $\mu$.
These forms are called
\emph{nuclear forms} on $\mx B\times$.
By \eqref{proj2}, every nuclear form can be written as 
\begin{equation}\label{proj2*}
\chi = \sumj \gl_j\xx_{1j}\tensor\dotsm \tensor \xx_{kj}
\end{equation}
with $\gl_j\ge0$, $\sumj|\gl_j|<\infty$ and $\norm{\xx_{ij}}\le1$.

\begin{remark}\label{Rnuclear}
Let $\cI(B_1,\dots,B_k)$ and $\cN(B_1,\dots,B_k)$ be the spaces of 
integral and nuclear forms, respectively, on $\mx B\times$. Both
$\cI(B_1,\dots,B_k)$ and $\cN(B_1,\dots,B_k)$ are Banach spaces
(with the natural norms suggested by their definitions, see \cite{Ryan}),
and there
is an inclusion $\cN(B_1,\dots,B_k)\subset\cI(B_1,\dots,B_k)$ with the
inclusion  mapping having norm at most 1. Furthermore, there is a quotient
mapping $\mx{B\q}\ptensor\to\cN(B_1,\dots,B_k)$. In general, none of these
maps is an isomorphism, but there are important cases when one or both are.
\end{remark}

\begin{remark}
  \label{Rnucops}
When $k=2$ (bilinear forms), there are corresponding notions for linear
operators.

We say that an operator $T:B_1\to B_2$ is \emph{integral}
if the corresponding bilinear form on $B_1\times B_2\q$ given by 
$(x,\yy)\mapsto \innprod{Tx,\yy}$ is integral.  
Conversely, a bilinear form $\ga:B_1\times B_2\to\bbR$ is integral if and
only if the corresponding operator $T:B_1\to B_2\q$ given by 
$\innprod{Tx,y}=\ga(x,y)$ is integral.
(When $B_2$ is reflexive, these are obviously equivalent. In general, see
\cite[Proposition 3.22]{Ryan}.)

Similarly, an operator $T:B_1\to B_2$ is \emph{nuclear} if
$\innprod{Tx,\yy}=\innprod{u,x\tensor\yy}$ for some $u\in B_1\q\ptensor
B_2$. 
When $B_2$ is reflexive, this says precisely that 
corresponding bilinear form on $B_1\times B_2\q$ given by 
$(x,\yy)\mapsto \innprod{Tx,\yy}$ is nuclear.
\end{remark}

\section{Vector-valued integration}\label{Sintegrals}

We summarize the definitions of the main types of vector-valued integrals,
see \eg{} \cite{Dunford-Schwartz},
\cite{DiestelUhl},
\cite{Talagrand:Pettis} and \cite{Ryan} for details.
In order to conform to the rest of this paper, we use probabilistic
language and consider the expectation of a $B$-valued random variable $X$.
(The definitions and results extend to integrals of Banach space valued
functions defined on arbitrary measure spaces with only notational changes.)

The \emph{Bochner integral} is a straightforward generalization of the
Lebesgue integral to Banach space valued functions.
We have the following characterization 
\cite[III.2.22 and III.6.9]{Dunford-Schwartz}, \cf{} \refT{Twm}. 
\begin{theorem}\label{TBochner}
A random variable $X:(\gO,\cF,\P)\to B$ is Bochner integrable if and
  only if $X$ is \Bmeas{} and 
$\E \norm X<\infty$.
\nopf
\end{theorem}

As discussed in \refS{Smeas}, \Bmeas{} variables are not enough for all
applications. 
For a more general integral, suppose only that $\xx(X)$ is integrable for
every $\xx\in B\q$. (In particular, $X$ is weakly measurable.)
Then the linear map 
\begin{equation}
  \label{tx}
\dunx:\xx\mapsto \xx(X)
\end{equation}
 maps $B\q$ into
$L^1(\P)$ and by the closed graph theorem, this map is bounded. Hence
$\xx\mapsto\E\xx(X)=\int\xx(X)\dd\P$ is a bounded linear map $B\q\to\bbR$,
\ie, an 
element of $B\qx$. This element is called the \emph{Dunford integral} of
$X$.
We can write the definition as
\begin{equation}\label{dunford}
\innprod{ \E X,\xx} = \E\innprod{\xx,X},
\end{equation}
noting that in general $\E X\in B\qx$.

In the special case that the Dunford integral $\E X\in B$, and moreover 
(see \refR{Rpettis})
$\E(X\etta_E)\in B$ for every event $E$
(\ie, every  measurable set $E\subseteq\gO$), 
we say that $\E X$ is the \emph{Pettis integral} of $X$. 
Note that the Pettis integral by definition is an element of $B$. 
If $X$ is Bochner integrable then it is Pettis integrable (and
Dunford integrable) and the
integrals coincide, but the converse does not hold.
The examples below show that a Pettis integrable function may fail to be
Bochner integrable because of either integrability or measurability problems.
(\citet[Chapter 5]{Talagrand:Pettis} and  
\cite{Stefansson} give characterisations of
Pettis integrability, 
but they are not always easy to apply; it seems that 
there is no simple necessary and sufficient condition.
See further \eg{}
\cite{EdgarII} and  \cite{Huff}.)

\begin{remark}\label{Rdual}
  If $B$ is reflexive, \ie{} $B\qx=B$, then every Dunford integrable
  function is trivially Pettis integrable (and conversely).
However, we will not find much use of this, since we will take integrals in
tensor products $B\ptpk$ or $B\itpk$, see \refS{Smom}, and such
tensor products are typically not reflexive,
even if $B$ is reflexive, see for example \refT{THtp}, \refE{EPD1i}
and, more generally,
\cite[Section 4.2]{Ryan}.
\end{remark}

\begin{remark}\label{Rpettis}
Suppose that $X$ is Dunford integrable. 
Then $\gx X$ is Dunford integrable  
for every bounded random variable $\gx$, \ie,
for every  $\gx\in L^\infty(\P)$.
Moreover, the adjoint of the bounded linear map $\dunx:B\q\to L^1(\P)$ given by
\eqref{tx} 
is the map $\dunx\q:L^\infty(\P)\to B\qx$ given by 
\begin{equation}\label{txq}
\Txq \gx= \E(\gx X).
\end{equation}

By definition, $X$ is Pettis integrable  if $\E(X\etta_E)\in B$,
\ie, if $\dunx\q(\etta_E)\in B$, for every measurable set $E$; 
since the simple functions are dense in $L^\infty(\P)$, 
this is
equivalent to $\dunx\q(\gx)=\E(\gx X)\in B$ for every $\gx\in
L^\infty(\P)$. Hence, $X$ is 
Pettis integrable if and only if $\dunx\q: L^\infty(\P)\to B$.
\end{remark}

\begin{remark}
  It follows that if $X$ is Dunford integrable, then $X$ is Pettis
  integrable if and only if $\dunx:B\q\to L^1(\P)$ is weak$^*$-weak continuous.
(If $X$ is bounded, then $\dunx$ is always \emph{sequentially}
 weak$^*$-weak continuous by dominated convergence; this is not enough,
 as shown by Examples \refand{Eloo}{Egoi}.)
\end{remark}

\begin{remark} \label{Rpettisofp}
  It follows easily from \eqref{dunford} that if $X$ is Dunford integrable
  and $\gx\in L^\infty(\P)$, then $\E(\gx X)=\E (\gx_1 X)$ where
  $\gx_1\=\E(\gx\mid \cF_X)$,
where $\cF_X$ is the sub-$\gs$-field of $\cF$ generated by all
$\innprod{\xx,X}$, $\xx\in B\q$. 
Hence, $X$ is Pettis integrable if and only if $\E(\gx_1X)\in B$ for every
$\cF_X$-\meas{} $\gx_1$.
Each such $\gx_1$ is a Borel function of a
countable family $(\innprod{\xx_i,X})_i$. 
It follows that the question whether $X$ is Pettis integrable or not depends
only on the distribution of $X$ 
(or equivalently, the joint distribution of $\innprod{\xx,X}$, $\xx\in B\q$)
and not on the underlying probability space $\ofp$.
\end{remark}

\begin{remark}
  \label{RBB}
If $X$ is a $B$-valued random variable and $B$ is a closed subspace of
another Banach space
$B_1$, then $X$ can also be seen as a $B_1$-valued random variable.
It is easily verified (using the Hahn--Banach theorem) that $X$ is Bochner,
Dunford or Pettis integrable as a $B$-valued \rv{} if and only if it is so
as a $B_1$-valued \rv, and the expectations $\E X$ in $B$ and $B_1$ coincide
in all cases.
\end{remark}

By the definition (and discussion) above, 
$X$ is Dunford integrable if and only if $\xx\mapsto \xx(X)$ is a
bounded linear operator $B\q\to L^1(\P)$.
If $X$ is Pettis integrable, then furthermore this operator is weakly
compact, \ie, it maps the unit ball into a relatively weakly compact subset
of $L^1(\P)$, see \cite{Talagrand:Pettis},
\cite{Huff} or \cite[Proposition 3.7]{Ryan}. 
A subset of $L^1(\P)$ is relatively weakly compact if and only if it
is uniformly integrable, see 
\cite[Corollary IV.8.11 and Theorem V.6.1]{Dunford-Schwartz}, 
where we recall  that a family $\set{\xi_\ga}$ of \rv{s} is \ui{} if 
\begin{equation}\label{ui}
  \begin{cases}
\sup_\ga\E|\xi_\ga|<\infty\quad \text{and}\\	
\sup_\ga\E\bigpar{\ettae|\xi_\ga|}\to0 \quad\text{as } \P(E)\to0,  
  \end{cases}
\end{equation}
see \eg{} \cite[Section 5.4]{Gut} or \cite[Lemma 4.10]{Kallenberg}.
This yields the following necessary condition.

\begin{theorem}
  \label{Tui}
If\/ $X$ is Pettis integrable, then
the family
\set{\xx(X):\xx\in B\q,\,\norm{\xx}\le1}
of (real-valued) \rv{s}
is uniformly integrable.
\nopf
\end{theorem}

The converse does not hold, see Examples \ref{Eloo} and \ref{Egoi}, but 
\citet{Huff} has shown that it holds if $X$ is \wassep.
(In particular, the converse holds when $B$ is separable.)

\begin{theorem}[\citet{Huff}] \label{THuff}
If \set{\xx(X):\xx\in B\q,\,\norm{\xx}\le1} 
is uniformly integrable and
$X$ is \wassep, then $X$ is Pettis integrable.
\nopf
\end{theorem}

\begin{remark}
Actually, \citet{Huff} uses a condition that he calls \emph{separable-like};
the definition given in \cite{Huff} is somewhat
stronger than \wassepy, but 
it seems likely that he really intended what we call \wassepy,
and the proof in \cite{Huff} uses only \wassepy.
See also \citet{Stefansson} where \wassep{} is called \emph{determined by a
  separable subspace} and said to be the same as Huff's separable-like.
(Stef\'ansson \cite[Theorem 2.8]{Stefansson} has 
extended \refT{THuff} by weakening the
condition of \wassepy, replacing separable by weakly compactly generated, 
but we shall not use his results, which seem more
difficult to apply in our situation.)
\end{remark}

\begin{corollary}\label{CHuff}
  If $X$ is \wmeas{} and \wassep, and further 
$\EE\norm{X}<\infty$, then $X$ is Pettis integrable.
\end{corollary}
\begin{proof}
There exists a measurable real-valued
$Z$ with $\norm X\le Z$ and $\E Z<\infty$.
  If $\norm\xx\le1$, then $|\xx(X)|\le\norm{X}\le Z$, and thus 
\set{\xx(X):\xx\in B\q,\,\norm{\xx}\le1} is uniformly integrable.
\end{proof}

\begin{remark}\label{RPpq}
  Let $1<p\le\infty$ and assume that 
$\xx(X)\in L^p(\P)$ for every $\xx\in B\q$, \ie, that the map $\dunx:B\q\to 
L^1(\P)$ above maps $B\q$ into $L^p(\P)$. Then $\dunx:B\q\to L^p(\P)$ is
bounded by the closed graph theorem and thus the adjoint $\dunx\q$,
defined in \refR{Rpettis} as an operator $L^\infty(\P)\to B\qx$, extends to
$L^q(P)$, where $q\in[1,\infty)$ is the conjugate exponent given by
  $1/p+1/q=1$.
Furthermore, \eqref{txq} holds for every $\gx\in L^q(\P)$; note that
$\xx(\gx X)=\gx\xx(X)\in L^1(\P)$ by \Holder's inequality
so $\gx X$ is Dunford integrable.

If furthermore $X$ is Pettis integrable, then $\dunx\q:L^\infty(\P)\to B$
by \refR{Rpettis},
and by continuity this extends to  $\dunx\q:L^q(\P)\to B$.

Note that if $C$ is the norm of $T_X:B\q\to L^p(\P)$, then for every 
$\xx\in B\q$ with $\norm\xx\le1$, $\E|\xx(X)|^p=\E|T_X(\xx)|^p\le C^p$,
which implies that \set{\xx(X):\norm\xx\le1} is \ui{}
\cite[Theorem 5.4.2]{Gut}. 
Hence, by \refT{THuff}, if furthermore $X$ is \wassep, in particular if $B$
is separable, then $X$ is Pettis integrable.
\end{remark}

\begin{example}\label{Ewiener} 
The standard Brownian motion 
(the Wiener process)
$W(t)$ is a \rv{} with values in $\coi$. 
$W$ is easily seen to be \Bmeas{} (\cf{} \refC{CCKcC})
and $\E\norm{W}<\infty$; thus $W$ is Bochner integrable (and thus Pettis
integrable). 
The operator $T_X:\coi\q=M\oi\to L^1(\P)$ is given by
\begin{equation}\label{rom}
  T_X(\mu)\=\innprod{\mu,W}=\intoi W(t)\dd\mu(t).
\end{equation}
$T_X$ obviously maps $\coi\q\to L^p(\P)$ for any $p<\infty$,
since $\intoi W\dd\mu$ is Gaussian.
(In fact, $\E\norm{W}^p<\infty$ when $p<\infty$.)
The adjoint map $\dunx\q:L^q(\P)\to B$, which by \refR{RPpq} is defined for
every $q>1$, maps a \rv{} $\xi\in L^q(\P)$ to the function in $\coi$ given by
\begin{equation}\label{spqr}
  \dunx\q\xi(t)
=\innprod{\gd_t,\dunx\q\xi}
=\innprod{T_X\gd_t,\xi}
=\innprod{W(t),\xi}
=\E\bigpar{\xi W(t)}.
\end{equation}
Hence, \eqref{txq} says that
\begin{equation}
  \E(\xi W)(t)=\E(\xi W(t)).
\end{equation}
  In particular, $\E W=0$, as is obvious by symmetry.
\end{example}

We also state another sufficient condition for Pettis integrability that only 
applies in special Banach spaces.  
\begin{theorem}[{\citetq{Theorem II.3.7}{DiestelUhl}}] \label{T-c0}
  Suppose that $X$ is Dunford integrable and \assep, and that $B$ does not
  contain a subspace isomorphic to $c_0$. Then $X$ is Pettis integrable.
\nopf
\end{theorem}

\refE{Ec0} below shows, using \refR{RBB},
that the condition on $B$ in \refT{T-c0} also is necessary. 
(Hence, in some sense, \refE{Ec0} is the canonical example of a Dunford
integrable variable 
that is not Pettis integrable, at least in separable spaces where there  is
no measurability problem.)

\begin{example}\label{Ec0}
  Let $N$ be a positive integer-valued random variable, and consider the
  random variable $X\=a_Ne_N$ in $c_0$, where $e_n\in c_0$ is the 
  $n$:th vector in the standard basis and $a_n$ are some real numbers.
Let $p_n\=\P(N=n)$.
It is easily seen that then
\begin{romenumerate}
\item $X$ is Bochner integrable if and only if $\sumn |p_na_n|<\infty$.

\item 
$X$ is Dunford  integrable if and only if $\sup|p_na_n|<\infty$.

\item 
$X$ is Pettis  integrable if and only if $p_na_n\to0$ as \ntoo.
\end{romenumerate}
The integral $\E X$ equals $(p_na_n)_{n=1}^\infty$ in all cases where it is
defined. (Thus, when $X$ is Dunford integrable, 
$\E X\in\ell^\infty= c_0\qx$ but in general $\E X$ does not belong to
$c_0$.)  
\end{example}

\begin{example}\label{El2}
Let  $X\=a_Ne_N$ as in \refE{Ec0}, but now regarded as an element of
$B=\ell^2$. 
It is easily seen that then
\begin{romenumerate}
\item $X$ is Bochner integrable if and only if $\sumn |p_na_n|<\infty$.

\item 
$X$ is Pettis or Dunford  integrable if and only if $\sumn |p_na_n|^2<\infty$.
\end{romenumerate}
(There is no difference between Dunford and Pettis since $B$ is reflexive.) 

The integral $\E X$ equals $(p_na_n)_{n=1}^\infty$ in all cases where it is
defined. 
\end{example}

\begin{example}\label{El1} 
Let again $X\=a_Ne_N$ as in \refE{Ec0}, but now regarded as an element of
$B=\ell^1$. 
In this case, 
\begin{romenumerate}
\item $X$ is Bochner integrable $\iff$
$X$ is Pettis integrable $\iff$
$X$ is Dunford  integrable $\iff$ 
$\sumn |p_na_n|<\infty$.
\end{romenumerate}
In fact, let $\xx=(\sign(a_n))_1^\infty\in\ell^\infty=(\ell^1)\q$.
Then $\innprod{\xx,X}=|a_N|$, so if $X$ is Dunford integrable, then 
$\E|a_N|<\infty$, which implies Bochner integrability.
The integral $\E X$ equals $(p_na_n)_{n=1}^\infty$.
\end{example}

\begin{example}\label{Ewuc} 
Examples \ref{Ec0}--\ref{El1} are examples of the following general fact:
  Let $N$ be a positive integer-valued random variable, 
with $\P(N=n)=p_n$, let $(x_n)\seq$ be a sequence in a Banach space $B$, and 
let $X\=x_N$.
It is then easy to see the following characterizations, see \eg{}
\cite[Proposition 3.12 and Appendix B]{Ryan}
and \cite[Chapter IV and p.~44]{Diestel}.
\begin{romenumerate}
\item $X$ is Bochner integrable if and only if 
$\sumn p_nx_n$ converges  absolutely. 

\item 
$X$ is Pettis  integrable if and only if
$\sumn p_nx_n$ converges unconditionally. 

\item 
$X$ is Dunford  integrable if and only if the series
$\sumn p_nx_n$ is weakly unconditionally Cauchy.
\end{romenumerate}
(This example is perhaps clearer if we do not restrict ourselves to
probability measures, and regard $\sumn x_n$ as the integral of the function
$n\mapsto x_n$ defined on $\bbN$ equipped with counting measure.
The sum converges as a Bochner integral, Pettis integral or Dunford integral
if and only if it is absolutely summable, unconditionally summable or weakly
unconditionally Cauchy, respectively.)
\end{example}

\begin{example}\label{Edoi}
  Let $B=\doi$, 
and let, as in \refE{Eetta},
$X$ be the random element of $\doi$ given by $X=\etta_{[U,1]}$, 
where $U\sim \U(0,1)$.
Then $X$ is not \assep{} and thus not \Bmeas{};
thus $\E X$ does not exist as a Bochner integral.
($X$ is also not Borel measurable,  at least typically, see \refE{Eetta}.)

On the other hand, it follows from \refT{TDw} that $X$ is weakly measurable,
and since $X$ is bounded, it is Dunford integrable. It is easily verified,
using \refC{C1} and Fubini's theorem, that if $\xx\in D\q$, then
\begin{equation*}
  \E \innprod{\xx,X}=\innprod{\xx,\idt},
\end{equation*}
where $\idt$ denotes the identity function $t\mapsto t$. Hence the Dunford
integral $\E X=\idt\in\doi$. It is similarly seen that if $E$ is any event,
then the Dunford integral $\E(X\etta_E)$ is the function 
$t\mapsto \P(\set{U\le t} \cap E)$, which is continuous and thus belongs to
$D$. Hence, $X$ is Pettis integrable.

By \refR{RBB}, $X$ is also Pettis integrable as a \rv{} in $L^\infty\oi$,
since the hyperplane \set{f\in\doi:f(1-)=f(1)} in $\doi$ also can be seen as
a subspace of $L^\infty\oi$. 
(Cf.~\cite[Example 4-2-4a)]{Talagrand:Pettis}.)
\end{example}

To find examples of bounded Dunford integrable random variables that are not
Pettis integrable is more difficult and technical. 
Note that by \refT{THuff}, such random variables cannot be \wassep.
We give one example from
\citet{FremlinT}, omitting the (quite complicated) details. See \refE{Egoi} for
another example.

\begin{example}[\citet{FremlinT}]\label{Eloo}
  Let $\gO=\setoi^\infty$ with the infinite product measure
  $\mu=\bigpar{\frac12\gd_0+\frac12\gd_1}^\infty$ (this is the Haar measure
  if we  regard $\gO$ as the compact group $\bbZ_2^\infty$),  
and let $X:\gO\to\ell^\infty$ be   the inclusion.
Then $X$ is a \rv{} with values in $\ell^\infty$ such that the coordinates
$X_n$ are \iid{} $\Be(1/2)$.
It is shown in \cite{FremlinT} and \cite[Chapter 13]{Talagrand:Pettis} 
(by slightly different arguments)
that $X$ is \emph{not} weakly measurable on $\gO$ (with the
product $\gs$-field = the Borel $\gs$-field), but that the measure $\mu$ can
be extended to a larger $\gs$-field 
making $X$ weakly measurable.
More precisely, it is easily seen that
$(\ell^\infty)\q=\ell^1\oplus c_0^\perp$, and if 
$\xx\in\ell^1\subset(\ell^\infty)\q$ is 
given by $(a_n)_1^\infty\in\ell^1$, 
then $\xx(X)=\sum_1^\infty
a_n X_n$ where $X_n$ as said above are \iid{} $\Be(1/2)$; clearly $\xx(X)$
is measurable in this case. The extension of $\mu$ constructed in
\cite{FremlinT} and \cite{Talagrand:Pettis} is such that $\xx(X)$ is \as{}
constant if $\xx\in c_0^\perp$. Hence $\xx(X)$ is measurable in this case
too, and by linearity for every $\xx\in (\ell^\infty)\q$, so $X$ is \wmeas.
Moreover, the extension is such that
in the particular case that $\xx\in c_0^\perp$ is a multiplicative
linear functional, $\xx(X)=1$ a.s.

Consequently, using this extension of $\mu$, 
$X$ is bounded and \wmeas, and thus Dunford
integrable. However, 
if $X$ had a Pettis integral $y=\E X\in\ell^\infty$, then 
$y_n=\E X_n=\frac12$ for each $n$, since $X_n\sim\Be(\frac12)$, and thus
$y=(\frac12,\frac12,\dots)$. However, if $\xx$ is a multiplicative linear
functional in $c_0^\perp$, 
then $\xx(y)=\E \xx(X)=1$, a contradiction.
Consequently, 
$X$ is not Pettis integrable.

By \refT{THuff}, $X$ is not \wassep.
\end{example}

To summarize, 
we have defined three types of integrals. The Bochner integral is 
the most convenient, when it exists, but the requirement of Bochner
measurablility is too strong for many applications in non-separable spaces.
The Pettis integral is more general, and will be our main tool in such
cases; it also requires only a weaker integrability condition. 
The Dunford integral is even more general, but in general it is
an element the bidual $B\qx$ instead of $B$, which makes it less useful.

Note also that all three integrals are linear and behave as expected under
bounded linear operators:
If $T:B\to B_1$ and $X$ is integrable in one of these senses, then
$TX\in B_1$ is integrable in the same sense, and $\E (TX)= T (\E X)$
($\E (TX)= T\qx(\E X)$ for Dunford integrals).

\begin{remark}  \label{RL1}
It can be shown that the space of Bochner integrable $B$-valued random
variables equals $L^1(\P)\ptensor B$
\cite[Example 2.19]{Ryan},
while $L^1(\P)\itensor B$ is the completion of the space of \Bmeas{} Pettis
integrable $B$-valued \rv{s}
\cite[Proposition 3.13]{Ryan}. 
(If $B$ is separable, this is by \refT{Twm}
just the
completion of the space of Pettis integrable \rv{s}; typically  the
latter space is not complete, so  it is necessary to take the
completion.) 
\end{remark}

\begin{remark}
There are several further definitions of integrals of Banach space valued
functions, 
see for example
\cite{Birkhoff}, \cite{Hildebrandt}, \cite{FremlinMendoza},
although only
some of these definitions work on a general probability space
as required here (for example the Birkhoff integral that lies between the
Bochner and Pettis integrals).
These integrals
too could be used to define moments as in the next section, 
but we do not know of any properties
of them that make them more useful for our purposes than the three integrals
defined above, so we do not consider them. 
\end{remark}

\section{Moments}\label{Smom} 

If $B$ is a Banach space and  
$X$ is a  $B$-valued \rv, defined on a
probability space $(\gO,\cF,\P)$, 
and further $k$ is a positive integer, 
we define the 
\emph{projective \kth{} moment} 
of $X$ as the expectation $\E X\tpk = \int X\tpk \dd\P$
whenever this expectation (integral) exists in the projective tensor product
$B\ptpk$. 
The expectation can here be taken in any of the three senses defined in
\refS{Sintegrals}; hence we talk about the moment existing in 
Bochner sense, Pettis sense or Dunford sense. Note that if the \kth{}
moment exists
in Bochner or Pettis sense, then it is an element of $B\ptpk$,
but if it exists in Dunford sense, 
then it is an element of $(B\ptpk)\qx$ and may be outside $B\ptpk$.

Similarly, we can regard $X\tpk$ as an element of the injective \tp{}
$B\itpk$ and take the
expectation in that space.
We thus define the
\emph{injective \kth{} moment} 
of $X$ as the expectation $\E X\tpk = \int X\tpk \dd\P$
whenever this expectation (integral) exists in the injective tensor product
$B\itpk$. Again, this can exist in 
Bochner sense, Pettis sense or Dunford sense; in the first two cases it is
an element of $B\itpk$,
but in the third case it is an element of $(B\itpk)\qx$.

\begin{example}
The first moment (projective or injective; there is no difference when
$k=1$)  is the
expectation $\E X$,
which is an element of $B$ when it exists in Bochner or Pettis sense, and an
element of $B\qx$ when it exists in Dunford sense.  
The examples in \refS{Sintegrals} show some cases.
\end{example}

We may for clarity or emphasis
denote  $X\tpk$ by
$X\ptpk$ when we regard it as an element of  $B\ptpk$ and  by
$X\itpk$ when we regard it as an element of  $B\itpk$.
In particular, we distinguish between the projective and injective moments
by writing them as 
$\E X\ptpk\in B\ptpk$ and $\E X\itpk\in B\itpk$. They are related by the
following simple result.

\begin{theorem}\label{TPI}
If the projective \kth{} moment $\E X\ptpk$ exists in one of the senses
above, then the injective 
\kth{} moment 
exists too, in the same sense, and is given by
\begin{equation}\label{momitpk}
  \E X\itpk = \iota(\E X\ptpk)\in B\itpk
\end{equation}
in the Bochner or Pettis case and 
\begin{equation}\label{momitpk**}
  \E X\itpk = \iota\qx(\E X\ptpk)\in (B\itpk)\qx
\end{equation}
in the Dunford case.  
\end{theorem}
\begin{proof}
The identity map on $B\tpk$ extends to the
continuous linear map $\iota: B\ptpk\to B\itpk$.  
\end{proof}

We next consider the problem of deciding when these moments exist, in
the different senses.
There are six different cases to consider. We shall see (\refT{TPIB} and the
examples in \refS{Sex}) that the conditions for existence differ for five of
them, for both measurability and integrability reasons. This multiplicity of
cases may be bewildering, but in many applications there is no problem. If
$B$ is 
separable (or $X$ is \assep) there is no problem with measurability
(\refT{Twm}) and if further $\E\norm X^k<\infty$, then both moments exist in
the strongest sense (\refT{TPIB}), and thus in all senses.

We begin by  considering conditions for
the existence of moments in the strongest sense, \ie, as Bochner integrals.

\begin{lemma}\label{Lbk}
  If $X$ is \Bmeas{} in $B$, then $X\tpk$ is \Bmeas{} in $B\ptpk$, for every
  $k\ge1$. In particular, then $\ga(X,\dots,X)$ is measurable for any
  bounded $k$-linear form $\ga\in L(B^k;\bbR)$.
\end{lemma}
\begin{proof}
  The (non-linear) mapping $x\mapsto x\tpk$ is continuous $B\to B\ptpk$.
The final claim follows by \refT{Tproj*}.
\end{proof}

\begin{remark}\label{Rweaktp}
There is no general corresponding result for \wmeas{} $X$
in the projective \tp{} $B\ptpk$, see
\refE{EPD3}. If $B$ is separable, or if $X$ is \assep, there is no
problem since then $X$ is \Bmeas{} by \refT{Twm}, but \wassepy{} is not
enough by \refE{EPD3}. Nevertheless, $\doi$ is an example of a non-separable
space where $X\tpk$ is \wmeas{} in the projective \tp{}
for every \wmeas{} $X$, see \refC{CDM}.
\end{remark}

\begin{lemma}\label{Lbs}
The following are equivalent, for any $k\ge1$:
\begin{romenumerate}
\item \label{lbsp}
 $X\tpk$ is \assep{} in $B\ptpk$.
\item \label{lbsi}
 $X\tpk$ is \assep{} in $B\itpk$.
\item \label{lbs1}
$X$ is \assep{} in $B$.
\end{romenumerate}
\end{lemma}

\begin{proof} 
\ref{lbs1}$\implies$\ref{lbsp} as in the proof of \refL{Lbk} and
\ref{lbsp}$\implies$\ref{lbsi} since 
$\iota:B\ptpk\to B\itpk$ is continuous.
Hence it remains to prove
\ref{lbsi}$\implies$\ref{lbs1}.

Let $A$ be a separable subspace of $B\itpk$ such that $X\tpk\in A$
  a.s. 
Then there exists a countable family of elementary tensors 
$F=\set{e_{i1}\tensor\dotsm\tensor e_{ik}}$ such that $A$ is included in the
closed linear span of $F$. Let $B_1\subseteq B$ be the closed linear span
of $\set{e_{ij}}_{ij}$. Then $B_1$ is separable, and if $\xx\perp B_1$, then 
$\xx(e_{ij})=0$ for all $i$ and $j$, and thus
$\xx\tensor\dotsm\tensor\xx\perp F$; by linearity and continuity, this
extends to the closed linear span of $F$, and thus
$\xx\tensor\dotsm\tensor\xx\perp A$. 

Hence, if $x\tpk\in A$, then 
$\innprod{\xx,x}^k=\innprod{(\xx)\tpk,x\tpk}=0$ and thus $\xx\perp x$ 
for every $\xx\perp B_1$, which implies that $x\in B_1$. Consequently,
$X\in B_1$ a.s.
\end{proof}

\begin{remark}
  \refL{Lbs} includes a partial converse to \refL{Lbk}, considering only the
  \assep{} condition. There is no complete converse to \refL{Lbk} since
  $X\tpx2$ may be Bochner measurable even if $X$ is not; this happens even
in the one-dimensional case $B=\bbR$, as shown by the trivial example when
$X$ is a 
  non-measurable function such that $X=\pm1$ everywhere; then $X\tpx2=X^2=1$
  is measurable.
We will thus usually assume that $X$ is at least weakly measurable.
\end{remark}

It is now easy to characterise when the moments exist as  Bochner
integrals.

\begin{theorem}\label{TPIB}
Suppose that $X$ is \wmeas.
Then the following are equivalent.
\begin{romenumerate}
 \item \label{tpib1}    
The projective moment $\E X^{\ptensor k}$ exists in Bochner sense.
\item \label{tpib2}    
The injective moment $\E X^{\itensor k}$ exists in Bochner sense.
\item 
$\E\norm{X}^k<\infty$ and $X$ is \assep.
\end{romenumerate}
\end{theorem}

\begin{proof}
By \refT{TBochner} and \refL{Lbs}, \ref{tpib1} and \ref{tpib2} both imply
that $X$ is \assep: Hence it suffices to consider the case when $X$ is
\assep, so by \refT{Twm} 
$X$ is \Bmeas.
By \refL{Lbk}, $X\tpk$ is \Bmeas, and the result follows by
\refT{TBochner}, since
$\norm{X\tpk}_{B\ptpk}=\norm{X\tpk}_{B\itpk}=\norm{X}^k$.
\end{proof}

We turn to Dunford integrals.
We first give a simple result on the existence of the weak \kth{} moment
\eqref{momweak}. Cf.\ \refR{RPpq}.

\begin{lemma}\label{Lweak} 
Suppose that $X$ is weakly measurable.
Then the following are equivalent.
\begin{romenumerate}
\item \label{lweak2}
The weak \kth{} moment $\E\bigpar{\xx_1(X)\dotsm\xx_k(X)}$ exists for every
\\
$\xx_1,\dots,\xx_k\in B\q$.   
\item \label{lweak3}  
$\E|\xx(X)|^k<\infty$ for every $\xx\in B\q$.
\item \label{lweak4}  
$\sup\set{\E|\xx(X)|^k:\norm \xx\le1}<\infty$.
\item \label{lweakt}
$\dunx:\xx\mapsto \xx(X)$ is a bounded operator $B\q\to L^k(\P)$. 
\end{romenumerate}

In this case, $\E\bigabs{\xx_1(X)\dotsm\xx_k(X)}$ is bounded for
$\xx_1,\dots,\xx_k$ in the unit ball of $B\q$.
\end{lemma}

\begin{proof}
\ref{lweak2}$\iff$\ref{lweak3}:
If \ref{lweak2} holds, then \ref{lweak3} follows by choosing
$\xx_1=\dots=\xx_k=\xx$. The converse follows by \Holder's inequality.

\ref{lweak4}$\iff$\ref{lweakt}: By definition.

\ref{lweak3}$\iff$\ref{lweakt}:
If \ref{lweak3}  holds, then $\dunx:B\q\to L^k(\P)$
is bounded by the closed graph theorem.
The converse is trivial.

The final claim follows by \ref{lweak4} and \Holder's inequality.
\end{proof}

In all cases we know, the existence of the weak \kth{} moment is equivalent
to the existence of the injective \kth{} moment in Dunford sense. We
have, however, failed to prove this in full generality and suspect that there
are counterexamples. We thus give a theorem with some technical sufficient
conditions, and
 leave it as an open problem whether the theorem holds more generally.
(This is, at least for bounded $X$, equivalent to whether
$X$ \wmeas{} implies $X\tpk$ \wmeas{} in $B\itpk$; cf.\ \refR{Rweaktp}
which shows that this does not hold for the projective \tp.)

We first state a lemma.
\begin{lemma}
\label{Lxxmeas}
Suppose that $X$ is weakly measurable and \assep. 
Then $\innprod{\xx,X}$ is
jointly measurable on $B\q\times \gO$, where $B\q$ is given the
Borel $\gs$-field for the \weakx{} topology.
\end{lemma}
\begin{proof}
$X$ is \Bmeas{} by \refT{Twm} and thus there is a
sequence $X_n$ of measurable simple \rv{s} in $B$ such that $X_n\to X$ a.s.,
see \refR{Rmeas}. Then each $\innprod{\xx ,X_n}$ is 
jointly measurable on $B\q\times \gO$, 
and thus $\innprod{\xx,X}$ is
jointly measurable.
\end{proof}

\begin{theorem}\label{TID} 
 \begin{romenumerate}[-18pt]
\item \label{tidi} 
If the injective \kth{} moment $\E X\itpk$ exists in Dunford sense, then the
weak \kth{} moment $\E\bigpar{\xx_1(X)\dotsm\xx_k(X)}$ exists for
  every $\xx_1,\dots,\xx_k\in B\q$, and 
\begin{equation}
  \label{tid}
\E\bigpar{\xx_1(X)\dotsm\xx_k(X)}=\innprod{\E X^{\itensor k},\mx\xx\tensor}.
\end{equation}
Furthermore,
$\sup\set{\E|\xx(X)|^k:\norm \xx\le1}<\infty$. 

\item\label{tidd}
Suppose that $X$ is weakly measurable,
and that one of the following additional condition holds.
\begin{bienumerate}
\item \label{tid1} 
$B$ is separable.
\item \label{tid2} 
$X$ is \assep.
\item \label{tid3}  
Every integral $k$-linear form  $B^k\to\bbR$ is nuclear. 
\end{bienumerate}
Then the  injective moment $\E X^{\itensor k}$ exists in Dunford sense if
and only if the weak \kth{} moment exists, \ie, if and only if
$\E|\xx(X)|^k<\infty$ for every $\xx\in B\q$.
  \end{romenumerate}
\end{theorem}

\begin{proof}
\pfitemref{tidi}
Directly from the definition of the Dunford integral, since $\mx\xx\tensor$
is a continuous linear functional on $B\itpk$.
The final claim follows by \refL{Lweak}. 

\pfitemref{tidd} 
By \ref{tidi} and \refL{Lweak}, it remains to show that if the weak \kth{}
moment exists, then
$\chi(X\tpk)$ is an integrable \rv{}, and in particular measurable, for every
$\chi\in(B\itpk)^*$. By \refT{Tinj*}, 
$\chi$ is represented by a signed
measure $\mu\in M(K^k)$, where $K$ is the closed unit
ball of $B\q$, and
\begin{equation}\label{thk}
  \chi(X\tpk)=\int_{K^k}\xx_1(X)\dotsm\xx_k(X)\dd\mu(\xx_1,\dots,\xx_k).
\end{equation}

\ref{tid1} obviously is a special case of \ref{tid2}.
If \ref{tid2} holds, then
 $\innprod{\xx,X}$ is
jointly measurable  on $B\q\times \gO$ by \refL{Lxxmeas}.
Hence,
$\xx_1(X)\dotsm\xx_k(X)$ is jointly measurable on $K^k\times \gO$, and
thus we can take expectations in \eqref{thk} and apply Fubini's theorem,
yielding 
\begin{equation}\label{thx}
\E \chi(X\tpk)=\int_{K^k}\E\bigpar{\xx_1(X)\dotsm\xx_k(X)}
\dd\mu(\xx_1,\dots,\xx_k);
\end{equation}
note that $\E\bigabs{\xx_1(X)\dotsm\xx_k(X)}$ is bounded on $K^k$ by
\refL{Lweak}, so the double integral is absolutely convergent.

If \ref{tid3} holds, then the integral form $\chi$ is nuclear and thus, by
\eqref{proj2*}, 
\begin{equation}
  \label{thz}
\chi=\sumn \gl_n \xx_{1n}\tensor\dotsm\tensor\xx_{kn}
\end{equation}
where $\gl_n\ge0$,
$\sumn\gl_n<\infty$ and each
$\xx_{in}\in K$. Consequently, using \eqref{*tp},
\begin{equation}\label{thy}
  \innprod{\chi,X\tpk}
=
\sumn \gl_n \xx_{1n}(X)\dotsm\xx_{kn}(X) ,
\end{equation}
which is integrable by \refL{Lweak}.
Taking expectations in \eqref{thy}, we see that \eqref{thx} holds in this
case too, now for the finite discrete measure
\begin{equation}
  \label{thw}
\mu\=\sumn \gl_n \gd_{(\xx_{1n},\dotsm,\xx_{kn})}.
\end{equation}
\end{proof}

\begin{remark}
We do not know any characterization of the Banach spaces $B$ such that every
integral $k$-linear form is nuclear. For $k=2$, this can be translated to
operators $B\to B\q$; a sufficient condition then is that $B\q$ has the
Radon--Nikod\'ym property and the \ap, see \cite[Theorem 5.34]{Ryan}.
\end{remark}

\begin{corollary}
  \label{Citpk=d}
  Suppose that $X$ and\/ $Y$ are \wmeas{} $B$-valued \rv{s}, 
and that one of the following conditions holds:
\begin{alphenumerate}
\item  \label{cit1} 
$B$ is separable.
\item 
$X$ and $Y$ are \assep.
\item 
Every integral $k$-linear form  $B^k\to\bbR$ is nuclear. 
\end{alphenumerate}
Then \eqref{weak=} holds
if and only 
$\E X^{\itensor k} = \E Y^{\itensor k}$,
with the injective \kth{} moments existing in Dunford sense.
(In other words, the injective moment is determined by
the weak moment.)
\end{corollary}

\begin{proof}
If $\E X\itpk=\E Y\itpk$, then the weak moments are equal by \eqref{tid}.

Conversely, suppose that the weak moments exist and are equal.
By \refT{TID}, the injective moments 
$\E X\itpk$ and $\E Y\itpk$ exist in Dunford sense.
Moreover, the proof of \refT{TID} shows that for any
$\chi\in(B\itpk)\q$, \eqref{thx} holds for some signed measure $\mu$,
which shows that if the weak moments are equal, then 
$$
\innprod{\chi,\E X\itpk}
=\E\innprod{\chi, X\itpk}
=\E\innprod{\chi, Y\itpk}
=\innprod{\chi,\E Y\itpk}
$$ 
for every $\chi$ and thus 
$\E X\itpk=\E Y\itpk$ in $(B\itpk)\qx$.
\end{proof}

Recall that $B\itpk$ can be regarded as a subspace of $L((B^*)^k\lxx \bbR)$,
the bounded $k$-linear forms on $B^*$. This leads to the following
interpretation of the injective \kth{} moment when it exists in Pettis (or
Bochner) sense, 
and thus is an element of $B\itpk$, 
which again shows that the
injective \kthm{} is essentially the same as the weak moment defined by
\eqref{momweak}. 

\begin{theorem}\label{Tjepp}
If $X$ is a $B$-valued \rv{} such that the injective \kth{} moment
$\E X\itpk$  exists in Pettis sense, then
$\E X\itpk\in B\itpk$ is the $k$-linear form on $B\q$
\begin{equation*}
(\xx_1,\dots,\xx_k)\mapsto \E \bigpar{\xx_1(X)\dotsm\xx_k(X)}.  
\end{equation*}
\end{theorem}
\begin{proof} 
When  $\E X\itpk\in B\itpk$ is regarded as a $k$-linear form,
its value 
at $(\xx_1,\dots,\xx_k)\in (B^*)^k$ equals
$\innprod{\E X^{\itensor k},\mx\xx\tensor}$, and the result follows by
\eqref{tid}. 
\end{proof}

This yields a simpler version of \refC{Citpk=d}, assuming that the moments
exist in Pettis sense.

\begin{corollary}
  \label{Citpk=p}
  Suppose that $X$ and\/ $Y$ are  $B$-valued \rv{s} 
such that 
the injective \kth{} moments $\E X^{\itensor k}$ and $\E Y^{\itensor k}$  
exist in Pettis sense. 
Then \eqref{weak=} holds
if and only 
$\E X^{\itensor k} = \E Y^{\itensor k}$.
(In other words, the injective moment is determined by
the weak moment.) 
\nopf
\end{corollary}

For the projective \kth{} moment in Dunford sense, there is a general similar
equivalence, now using arbitrary bounded $k$-linear forms on $B$.
We have no simple necessary and sufficient condition for the existence,
but we give a sufficient condition which is necessary in at least some cases
(\refE{EPD1}), but not in others (\refE{EPD2}).

\begin{theorem} \label{TPD}
The following are equivalent.
\begin{romenumerate}
\item \label{tpd1}  
The projective moment $\E X^{\ptensor k}$ exists in Dunford sense.
\item \label{tpd2}  
The  moment $\E{\ga(X,\dots,X)}$ exists for
  every bounded $k$-linear form $\ga:B^k\to\bbR$.
\end{romenumerate}
In this case, 
\begin{equation}
  \label{tpd}
\E{\ga(X,\dots,X)}=\innprod{\E X^{\ptensor k},\ga}
\end{equation}
for every $\ga\in L(B^k;\bbR)$. 

Moreover, if $\EE\norm{X}^k<\infty$
and $\ga(X,\dots,X)$ is \meas{}
for every $\ga\in L(B^k;\bbR)$,
then \ref{tpd1} and \ref{tpd2} hold.
\end{theorem}

\begin{proof}
The equivalence of \ref{tpd1} and \ref{tpd2} is an immediate consequence of
the definition of the Dunford integral and \refT{Tproj*}, and so is
 \eqref{tpd}.

Furthermore,  
$|\ga(X,\dots,X)|\le \norm\ga\norm{X}^k$.
Hence, if $\EE\norm{ X}^k<\infty$, then 
\ref{tpd2} holds provided $\ga(X,\dots,X)$  is measurable.
\end{proof}

\begin{remark}\label{Rsymmetric}
Although we use multilinear forms $\ga$ in \eqref{multi=} and \refT{TPD}, 
we are only   interested in the 
values $\ga(x,\dots,x)$ on the diagonal. 
This restriction to the diagonal defines a function $\tilde\ga(x):B\to\bbR$,
which is a quadratic form for $k=2$, a cubic form for $k=3$, etc.,
and \eqref{multi=} can be expressed as $\E\tilde\ga(X)=\E\tilde\ga(Y)$ for
all such forms $\tilde\ga$.

Note also that it suffices to consider \emph{symmetric} multilinear forms
$\ga$, since we always may replace $\ga$ by its symmetrization. 
\end{remark}

\begin{corollary}\label{Cptpk=}
  Suppose that $X$ and\/ $Y$ are $B$-valued \rv{s}.
Then
\eqref{multi=} holds, with finite and well-defined expectations for every
$\ga$, 
if and only 
if the projective \kth{} moments $\E X^{\ptensor k}$ and $\E Y^{\ptensor k}$ 
exist in Dunford sense and
$\E X^{\ptensor k} = \E Y^{\ptensor k}$.
\end{corollary}
\begin{proof}
  An immediate consequence of Theorem \ref{TPD}, together with
  \refT{Tproj*}.
\end{proof}

The problem whether \eqref{multi=} and \eqref{weak=} are equivalent is thus
reduced to the problem whether
$\E X^{\itensor k} = \E Y^{\itensor k}$
implies
$\E X^{\ptensor k} = \E Y^{\ptensor k}$, at least if we
assume that the projective moments exist in Dunford sense, and that one of
the additional assumptions in \refC{Citpk=d} or \ref{Citpk=p} holds. 
By \eqref{momitpk**}, it then is sufficient 
that $\iota\qx$ is injective. However, in applications we prefer not to
use the bidual
(recall that tensor products typically are not reflexive,
even if $B$ is reflexive, see \refR{Rdual});
we thus prefer to
use moments in Pettis or Bochner sense.
For these moments, the question whether
$\E X\ptpk=\E Y\ptpk \iff \E X\itpk=\E Y\itpk$
(for arbitrary $X$ and $Y$ in a given Banach space $B$) is 
by \eqref{momitpk}
almost equivalent
to whether 
$\iota: B\ptpk\to B\itpk$  is injective; this will be studied in
\refS{Sapprox}. 

\begin{remark}
We write ``almost'', because $\E X\ptpk$ is a symmetric tensor, so we are
really only interested in whether $\iota$ is injective on the subspace of
symmetric tensors in $B\ptpk$ (\ie, on the symmetric tensor product).
We conjecture that $\iota$ is injective on this subspace if and only it is
injective on the full tensor product $B\ptpk$, but as far as we know this
question has not been investigated  and we leave it as an open problem.
\end{remark}

We turn to considering conditions for
the existence of moments in Pettis sense.
We only give a result for the injective moments, 
corresponding to \refT{TID},
since we do not know any
corresponding result for projective moments.

\begin{theorem} \label{TiPettis}
  \begin{romenumerate}[-18pt]
  \item \label{tipettis1}
If the injective \kth{} moment $\E X\itpk$ 
exists in Pettis sense,
then \set{|\xx(X)|^k:\xx\in B\q,\,\norm{\xx}\le1} 
is uniformly integrable.
  \item \label{tipettis2}
Suppose that $X$ is weakly measurable,
and that one of the following additional condition holds.
\begin{bienumerate}
\item \label{tip1} 
$B$ is separable.
\item \label{tip2}  
$X$ is \assep.
\item \label{tip3}  
Every integral $k$-linear form  $B^k\to\bbR$ is nuclear,
and $X\tpk$ is \wassep{} in $B\itpk$.
\end{bienumerate}
Then the  injective moment $\E X^{\itensor k}$ exists in Pettis sense if
and only if 
\set{|\xx(X)|^k:\xx\in B\q,\,\norm{\xx}\le1} 
is uniformly integrable.
  \end{romenumerate}
\end{theorem}

\begin{proof}
\pfitemref{tipettis1}
By \refT{Tui}, since 
$\innprod{(\xx)\tpk,X\tpk}=\xx(X)^k$ and
$(\xx)\tpk\in (B\itpk)\q$ with norm
$\norm{\xx}_{B\q}^k\le1$ when $\norm{\xx}_{B\q}\le1$.

\pfitemref{tipettis2} 
We shall modify the proof of \refT{TID}.
Let $\chi\in(B\itpk)^*$ with $\norm{\chi}\le1$.
As in the proof of \refT{TID}, by \refT{Tinj*} there exists a signed
measure $\mu\in M(K^k)$, where $K$ is the closed unit
ball of $B\q$, such that \eqref{thk} holds; further $\norm\mu=\norm\chi\le1$.
Taking absolute values, and replacing $\mu$ by $|\mu|$, we obtain
\begin{equation}\label{thkp}
  \bigabs{\chi(X\tpk)}\le\int_{K^k}\bigabs{\xx_1(X)\dotsm\xx_k(X)}
\dd\mu(\xx_1,\dots,\xx_k).
\end{equation}

If \ref{tip1}  or \ref{tip2} 
holds, then
 $\innprod{\xx,X}$ is
jointly measurable  on $B\q\times \gO$ by \refL{Lxxmeas};
thus we can take expectations in \eqref{thkp} and apply Fubini's theorem,
yielding 
\begin{equation}\label{thxp}
\E\bigabs{\chi(X\tpk)}\le\int_{K^k}\E\bigabs{\xx_1(X)\dotsm\xx_k(X)}
\dd\mu(\xx_1,\dots,\xx_k).
\end{equation}

If \ref{tip3} holds, then the integral form $\chi$ is nuclear and thus
\eqref{thz} holds, with 
$\gl_n\ge0$,
$\sumn\gl_n<\infty$ and each
$\xx_{in}\in K$. 
We do not know whether the nuclear norm of $\chi$ always equals the integral
norm $\norm\chi$ when \ref{tip3} holds, but, at least,
the open mapping theorem implies that there exists a constant $C$ (possibly
depending on $B$) such that we can choose a representation \eqref{thz} with
$\sum_n\gl_n\le C$. Then \eqref{thy} holds, and taking absolute values and
expectations we see that \eqref{thxp} holds in this case too,
for the measure \eqref{thw}, which satisfies $\norm\mu=\sum_n\gl_n\le C$.

Consequently, in both cases \eqref{thxp} holds, with $\norm\mu\le C$
(where $C=1$ in cases \ref{tip1}--\ref{tip2}).
By the arithmetic-geometric inequality, \eqref{thxp} implies
\begin{equation}\label{thxp+}
  \begin{split}
  \E\bigabs{\chi(X\tpk)}
&\le\int_{K^k}\E\Bigpar{\frac1k\sumik\abs{\xx_i(X)}^k}
\dd\mu(\xx_1,\dots,\xx_k)
\\
&=\frac1k\sumik \int_{K^k}\E\abs{\xx_i(X)}^k
\dd\mu(\xx_1,\dots,\xx_k)
\\
&\le C\sup\bigset{\E|\xx(X)|^k:{\xx}\in K}, 
  \end{split}
\end{equation}
and applying
\eqref{thxp+} to $\ettae X$, for an arbitrary event $E\in\cF$, we obtain
\begin{equation}\label{thxpe}
  \E\bigpar{\ettae|\chi(X\tpk)|}
=
  \E\bigabs{\chi(\ettae X\tpk)}
\le C\sup_{\xx\in K}\E\bigpar{\ettae|\xx(X)|^k}. 
\end{equation}
This holds for any $\chi\in(B\itpk)\q$ with $\norm\chi\le1$.
Hence, 
it follows from \eqref{thxp+}--\eqref{thxpe} and the condition \eqref{ui} for
uniform integrability  
that if \set{|\xx(X)|^k:\xx\in K} is \ui, 
then also
$\set{\innprod{\chi,X\tpk}:\chi\in(B\itpk)\q, \,\norm{\chi}\le1}$
is \ui.
Furthermore, 
$X\itpk$ is \wassep, by \ref{tip1}--\ref{tip2} and \refL{Lbs} or assumption
in  \ref{tip3}.
The result now follows by 
  \refT{THuff} applied to $X\itpk$, 
\end{proof}

\begin{remark}
Changing the norm in $B$ to an equivalent one will not change the \tp{s}
$B\ptpk$ and $B\itpk$ (except for a change of norms), and the existence (in
any of the three senses above) and values of the moments
$\E X\ptpk$ and $\E X\itpk$ will not be affected.
\end{remark}

\begin{remark}\label{RT}
The moments behave as expected under linear transformations.
If $X$ is a $B$-valued \rv{} and $T:B\to B_1$ is a bounded linear map into
another Banach space $B_1$, then $TX$ is a $B_1$-valued \rv{}, 
which is [Borel, weakly, Bochner] \meas{} if $X$ is. 
If the moment $\E X\ptpk$ or $\E X\itpk$ exists in any of the three senses
above,
then $\E (TX)\ptpk$ or $\E (TX)\itpk$ exists in the same sense; moreover,
for moments in Bochner or Pettis sense 
$\E (TX)\ptpk=T\ptpk(\E X\ptpk)$, and for moments in Dunford sense
$\E (TX)\ptpk=(T\ptpk)\qx(\E X\ptpk)$, and similarly for injective moments.
\end{remark}

\begin{remark}
  \label{RBBtp} 
If $X$ is a $B$-valued random variable and $B$ is a closed subspace of
another Banach space
$B_1$, then $X$ can also be seen as a $B_1$-valued random variable.
For the injective tensor product, then $B\itpk$ is a closed subspace of
$B_1\itpk$, see \refR{Rinjinj}, and thus by \refR{RBB}, the injective moment
$\E X\itpk$ exists in $B\itpk$  (or $(B\itpk)\qx$)
in any of the three senses
if and only if it exists in $B_1\itpk$ (or $(B_1\itpk)\qx$)
in the same sense; moreover,
then the value of the moment in the two spaces coincide.

For the projective moments, the situation is more complicated since
\refR{Rinjinj} does not hold for the projective \tp.
If we consider moments in Bochner sense
(and assume that $X$ is \wmeas), 
then by \refT{TPIB}
$\E X\ptpk$ exists in $B\ptpk$ if and only it exists in $B_1\ptpk$.
However, for projective moments in
Pettis  sense, we can in general only say that if the moment $\E X\ptpk$ 
exists in $B\ptpk$, then it exists in $B_1\ptpk$, and the values are
the same (by \refR{RT} applied to the inclusion map); we shall see in
\refE{EPD2} that the converse does not hold.
(And similarly for Dunford sense, where we have to consider the biduals.)
This shows that when considering projective moments of a Banach space valued
\rv, we may have to be careful to specify which Banach space we are using.

If $B$ is a complemented subspace of $B_1$ there is no problem: then there
is a bounded projection $P:B_1\to B$, and it follows from \refR{RT} that
also the projective
moments exist for $X$ as a $B$-valued \rv{} 
if and only if they exist  for $X$ as a $B_1$-valued \rv{}.
\end{remark}

\begin{remark}
We may define moments also in other tensor products (not considered in the
  present paper) in the same way; one example is to take the Hilbertian
  tensor product in \refE{EHilbert}
when $B$ is a Hilbert space. 
When the projective \kth{} moment exists,
these moments too are given by
  mapping the projective \kth{} moment
$\E X\ptpk\in B\ptpk$ to the chosen tensor product as in
  \eqref{momitpk} or   \eqref{momitpk**}. 
This is one reason to take the projective
  \kth{} moment as the standard \kth{} moment, when it exists.
\end{remark}

\begin{remark}
 It is possible to define mixed moments of random variables $X_1,\dots,X_k$
with  values in possibly different Banach spaces $B_1,\dots,B_k$ in the same
way, by taking the expectation of $\mx X\tensor$ in $\mx B\ptensor$ or $\mx
B\itensor$. 
Analoguous results hold, but are left to the reader.
In most cases, we can  consider
$X=(X_1,\dots,X_k)$ in the direct sum $B\=\mx B\oplus$ and take its moments
(provided they exist); they contain the mixed
moments as components, and we are reduced to the case 
treated above
of moments of
a single variable.
For example, with $k=2$,
$(B_1\oplus B_2)\ptpx2$ is the direct sum of 
$B_1\ptensor B_1$, $B_1\ptensor B_2$, $B_2\ptensor B_1$ and $B_2\ptensor
B_2$, and the components of $\E X\tpx2$ in these subspaces are
$\E X_1\tpx2$, $\E (X_1\tensor X_2)$, $\E (X_2\tensor X_1)$
and $\E X_2\tpx2$, where the two mixed moments are the same, using the  
natural isomorphism  $B_1\ptensor B_2\cong B_2\ptensor B_1$.
\end{remark}

\section{Examples}\label{Sex}

We give here some simple (counter)examples to illustrate the results above.
Further examples, more important for applications, are given later.

We let $N$ be a positive integer-valued random variable, with $p_n=\P(N=n)$,
and let $U\sim\U(0,1)$; we may suppose that
$U$ is the identity function defined on $\oi$ with Lebesgue measure.
We use standard notations from \refS{Snotations}.

\begin{example}\label{EPD1}
Let $B=H$ be a separable Hilbert space and let $k=2$.

Let $\ga\in L(B^2;\bbR)$ be the inner product in $H$. Then
$\ga(X,X)=\innprod{X,X}=\norm{X}^2$.
Hence, if the projective moment $\E X\ptpx2$ exists in Dunford sense, then,
by \refT{TPD}, $\E\norm{X}^2<\infty$.
Furthermore, weak and Bochner measurability are equivalent by \refT{Twm},
and it follows, using also \refL{Lbk}, that the projective second moment 
$\E X\ptpx2$ exists in Dunford sense if and only if it exists in Bochner
sense,
and consequently if and only if it exists in Pettis sense.

In this case, the projective second moment thus exists in any sense if and
only if $X\ptpx2$ is measurable and $\E\norm{X}^2<\infty$. In particular, the
sufficient condition in \refT{TPD} is also necessary in this case.
\end{example}

\begin{example}\label{EPD1i}
Specialize \refE{EPD1} to  $B=\ell^2$ and let $X=a_Ne_N$ for some
sequence $(a_n)_1^\infty$.
We have seen that:
\begin{romenumerate}
\item The projective second moment $\E X\ptpx2$ exists,
in any of the three senses, if and only if
$\E\norm{X}^2=\E|a_N|^2=\sumn p_na_n^2<\infty$. 
\end{romenumerate}
The projective \tp{} $\ell^2\ptensor\ell^2$ can be seen as the space 
$\trx(\ell^2)$ of
trace class  operators on $\ell^2$
(see \refT{THtp}),
and $\E X\ptpx2$ is the
diagonal operator $\sumn p_na_n^2 e_n\tensor e_n$.

The injective \tp{} $\ell^2\itensor\ell^2$ can, similarly, be seen as the
space $\cx(\ell^2)$ of compact operators in $\ell^2$.
A diagonal operator $\sum_nb_n e_n\tensor e_n$ has norm $\sup|b_n|$, and the
subspace of diagonal operators in $\cx(\ell^2)$ is isomorphic to $c_0$.
Since $X\tensor X=a_N^2 e_N\tensor e_N$ belongs to this subspace, it follows
from \refE{Ec0} that  
\begin{romenumerateq}
\item 
$\E X\itpx2$  exists in Dunford  sense $\iff$  $\sup p_na_n^2<\infty$.

\item 
$\E X\itpx2$  exists in Pettis   sense $\iff p_na_n^2\to0$ as \ntoo.

\item $\E X\itpx2$ exists in Bochner sense $\iff$ $\sumn p_na_n^2<\infty$.
\end{romenumerateq}
$\E X\itpx2$  is, when it exists, the diagonal operator 
$\sum p_na_n^2e_n\tensor e_n$, just as the projective second moment.
As usual, in the Dunford case,
$\E X\itpx2\in \cx(\ell^2)\qx=B(\ell^2)=L(\ell^2\lxx\ell^2)$. In fact, the
diagonal operator $\sum p_na_n^2e_n\tensor e_n$ is compact, \ie{}
$\E X\itpx2\in \cx(\ell^2)$, if and only if $p_na_n^2\to0$ as \ntoo.
\end{example}

\begin{example}\label{EPD2}
  Let $B=L^1\oi$. Then $B\ptensor B= L^1(\oi^2)$, 
see \eg{} \cite[Theorem 46.2 and Exercise 46.5]{Treves}.

Let $X\=a_Nr_N$, where $(a_n)_1^\infty$ is some sequence of real numbers and
$r_n\in L^1\oi$ are the Rademacher functions.
Then $X\tensor X=a_N^2r_N\tensor r_N$. 

By Khintchin's inequality \cite[Theorem II.1]{Blei}
(which applies as well to $r_n\tensor r_n\in L^1(\oi^2)$, since these
functions too can be seen as a sequence of independent symmetric $\pm1$
\rv{s}), the $L^1$-norm and $L^2$-norm are equivalent on the closed
linear span $R_2$ of \set{r_n\tensor r_n} in $L^1(\oi^2)$. 
Since $X\tensor X\in R_2$, the expectation $\E(X\tensor X)$ in 
$B\ptensor B = L^1(\oi^2)$, in any of the three senses,
can just as well be computed in $L^2(\oi^2)$. However, the functions
$r_n\tensor r_n$ form an orthonormal sequence in $L^2$, 
and there is thus an
isomorphism between $R_2$ and $\ell^2$, given by
$r_n\tensor r_n\mapsto e_n$. Hence it follows
from \refE{El2} that the projective moment 
$\E X\ptpx2$ exists in Bochner sense if and only if 
$\E\norm{X}^2=\sum p_n a_n^2<\infty$,
while
$\E X\ptpx2$ exists in Dunford or Pettis sense if and only if $\sum p_n^2
a_n^4<\infty$.

For the injective moments, we use \refR{RBBtp}. 
$X$ lies in the closed linear span $R_1$ of \set{r_n} in $L^1\oi$,
which by Khintchin's inequality 
is isomorphic to $\ell^2$; thus $\E X\itensor X$ may be calculated
in $R_1\itensor R_1\cong \ell^2\itensor\ell^2$. This brings us back to
\refE{EPD1i}, and thus (ii)--(iv) in \refE{EPD1i} hold in the present case too. 

Note that if we consider $X$ as an $R_1$-valued \rv, then  
\refE{EPD1i}(i) shows that
the projective
moment
$\E X\ptpx2$ exists in Pettis (or Dunford) sense if and only if 
$\sum_n p_na_n^2<\infty$. Hence, choosing $p_n$ and $a_n$ such that
$p_na_n^2=1/n$, we see that,
although $R_1$ is a closed subspace of $L^1\oi$, 
$\E X\ptpx2$ exists in Pettis (or Dunford)
sense if we regard $X$ as a $L^1\oi$-valued \rv,
but not if we regard $X$ as an $R_1$-valued \rv, \cf{} \refR{RBBtp}.
\end{example}

\begin{example}\label{EPD3}
  Let $B=\ell^2\oi$ and let $X\=a(U)e_U$ where $a:\oi\to\ooi$ is 
some given function.
Then $\innprod{X,f}=0$ \as{} for any $f\in \ell^2\oi$, since $f$ has countable
support, and thus $X$ is \wmeas{}
and \wassep. Trivially, $\E X=0$ in Pettis sense.
However, a subspace $A$ such that $X\in A$ \as{} has to contain $e_t$ for
\aex{} $t\in\oi$, and is thus not separable. Consequently, $X$ is not
\assep, and therefore not Bochner integrable.

Furthermore, every integral bilinear form on a Hilbert space is nuclear
(\refT{THnuc} below), 
and thus
by \refT{TID}\ref{tidd}\ref{tid3}, $\E X\itpx2$ exists in Dunford sense; by
\eqref{tid}, $\E X\itpx2=0$. 
(By \refT{TH2ii}, $\E X\itpx2=0$ also in Pettis sense.)

Let $\ga\in L(B^2;\bbR)$ be the inner product in $B=\ell^2\oi$.
Then $\ga(X,X)=\innprod{X,X}=\norm{X}^2=a(U)^2$.
If we assume that $a$ is a non-measurable function (for Lebesgue measure on
$\oi$), then $a(U)$ is non-measurable. Hence, $\ga(X,X)$ is non-measurable,
and by \refT{TPD}, the projective moment $\E X\ptpx2$ does not exist in
Dunford sense (and thus not in the other, stronger, senses).

In particular, we see that although $X$ is \wmeas{} in $B$, $X\tensor X$ is
not \wmeas{} in $B\ptensor B$, since $\innprod{\ga,X\tensor X}=\ga(X,X)$ is
not measurable. (We use here \refT{Tproj*}.)

Furthermore, $X\tensor X$ is
not \wassep{} in $B\ptensor B$, by the following argument. 
Suppose that $A$ is a separable subspace of
$B\ptensor B$. Then, as in the proof of \refL{Lbs}, 
there exists a countable family of elementary tensors 
$F=\set{e_{i1}\tensor e_{i2}}$ such that $A$ is included in the
closed linear span of $F$. Let $R\=\bigcup_{i,j}\supp(e_{ij})$; then
$R\subseteq\oi$ is countable and if $y\in A$, then $\supp(y)\subseteq R\times
R$.
Define $\gb\in(\ell^2\ptensor\ell^2)\q$ by $\gb(f,g)\=\sum_{t\notin R}
f(t)g(t)$. Then $\gb(e_{i1},e_{i2})=0$ for all $i$, and thus $\gb\perp F$
and $\gb\perp A$,
but $\gb(X,X)=\innprod{X,X}=a(U)^2\neq0$ \as.
\end{example}

\begin{example}  \label{EPDc0} 
Let, as in \refE{Ec0}, $B=c_0$ and $X=a_Ne_N$.
Thus $X\tpk=a_N^ke_N\tpk$.

Let $\gDo_k\subset c_0\tpk$ be the subspace of finite linear combinations of 
tensors $e_N\tpk$; we claim that the closure of $\gDo_k$ is the same in both
the projective and injective tensor products $c_0\ptpk$ and $c_0\itpk$ 
and that it equals
 $\gD_k\=\set{\sumn b_n e_n\tpk:(b_n)\nnn\in c_0}$. 
(For the injective \tp, this follows immediately from \refT{Tc0tp}, but the
projective case is less obvious.)
Since $\gD_k$
obviously is isomorphic to $c_0$, it then follows from \refE{Ec0} that:
\begin{romenumerate}
\item 
$\E X\ptpk$  exists in Dunford  sense $\iff$  
$\E X\itpk$  exists in Dunford  sense $\iff$  
$\sup p_na_n^k<\infty$.

\item 
$\E X\ptpk$  exists in Pettis   sense $\iff$ 
$\E X\itpk$  exists in Pettis   sense $\iff$ 
$p_na_n^k\to0$ as \ntoo.

\item 
$\E X\ptpk$ exists in Bochner sense $\iff$ 
$\E X\itpk$ exists in Bochner sense $\iff$ 
$\sumn p_na_n^k<\infty$.
\end{romenumerate}
Hence there is no difference between projective and injective moments in
this example.
(Cf.\ \refE{EPD1i}, where we see that $\gD_2$ is the closure of
$\gDo_2$ also in $\ell^2\itensor\ell^2$, but not in $\ell^2\ptensor\ell^2$.)

To verify the claim, consider for simplicity first $k=2$, and let $u=\sumnm
b_ne_n\tensor e_n\in\gDo_2$. 
Recalling \eqref{normproj} and \eqref{norminj}, and taking $\xx_1=\xx_2=e_n$
in the latter, we have immediately
\begin{equation}\label{cz}
  \norm{u}_\pi
\ge\norm{u}_\eps\ge\max_n|b_n|.
\end{equation}
For $\zeta=(\zeta_1,\dots,\gz_M)$ with
$\gz_i=\pm1$, take $v_\gz\=\sumnm \gz_nb_n e_n \in c_0$
and 
$w_\gz\=\sumnm \gz_n e_n \in c_0$.
Then, see \eqref{normproj} and
\eqref{pik}, 
\begin{equation}
\norm{v_\gz\tensor w_\gz}_\pi=\norm{v_\gz}_{c_0}\norm{w_\gz}_{c_0}
=\max_n|b_n|.  
\end{equation}
Taking the average of $v_\gz\tensor w_\gz$ over the $2^M$ possible choices
of $\gz$, we obtain $u$, and thus
\begin{equation}\label{cz1}
\norm{u}_\pi\le\max_n|b_n|.  
\end{equation}
Consequently, we have equalities in \eqref{cz}, for any $u\in \gDo$, which
shows that the closure in either $c_0\ptensor c_0$  or $c_0\itensor c_0$ is
isomorphic to $c_0$ and equals $\gD_2$.

The argument extends to arbitrary $k\ge2$ by letting the possible
values of $\gz_n$
be the \kth{} roots of unity and considering
$v_\gz\tensor w_\gz\tpx{k-1}$. (These vectors are complex, but we can separate
them into real and imaginary parts, possibly introducing a constant factor
$2^k$ in the norm estimate \eqref{cz1}.)
\end{example}

\begin{example}\label{Eg} 
We say that a \rv{} $X$ in a Banach space $B$ is \emph{weakly Gaussian}
if $\xx(X)$ is Gaussian with mean 0 (for convenience) for any $\xx\in B\q$.
To exclude cases such as \refE{EPD3} where $\xx(X)=0$ \as{} (and thus is
Gaussian) but $X$ does not look very Gaussian, we say that $X$ is 
\emph{Gaussian} if it is weakly Gaussian and \assep.
(By \refT{Twm}, this is equivalent to weakly Gaussian and \Bmeas.) 

If $X$ is (weakly) Gaussian, then $\xx(X)$ is Gaussian and thus has finite
moments of all orders, for any $\xx\in B\q$. Thus
\refT{TID}\ref{tidd}\ref{tid2}
shows that every injective moment $\E X\itpk$ exists in Dunford sense.

Moreover, if $X$ is Gaussian,
by \cite[Lemma 3.1 and Corollary 3.2]{LedouxT} (applied to a suitable
separable subspace $B_1\subseteq B$ with $X\in B_1$ \as),
$\norm{X}^k<\infty$ for every $k$. 
Hence, \refT{TPIB} shows that the
projective and injective moments $\E X\ptpk$ and $\E X\itpk$ exist in
Bochner sense for every $k\ge1$.

The odd moments vanish by symmetry.

The even injective moments can be expressed in terms of the second moment 
$\Sigma\=\E X\itpx2$
as follows. Consider first $k=4$. Then, by \refT{Tjepp} and Wick's theorem
\cite[Theorem 1.28]{SJIII},
$\E X\itpx4\in B\itpx4$ is determined by
\newcommand\fyll
 {\phantom{\quad=\E\bigpar{\xx_1(X)\xx_2(X)}\E\bigpar{\xx_3(X)\xx_4(X)}}\;}
\begin{equation*}
  \begin{split}
&  \innprod{X\itpx4,\xx_1\tensor\xx_2\tensor\xx_3\tensor\xx_4}
=
\E\bigpar{\xx_1(X)\xx_2(X)\xx_3(X)\xx_4(X)}
\\& \quad
=
\E\bigpar{\xx_1(X)\xx_2(X)}\E\bigpar{\xx_3(X)\xx_4(X)}
+
\E\bigpar{\xx_1(X)\xx_3(X)}\E\bigpar{\xx_2(X)\xx_4(X)}
\\&\fyll
+
\E\bigpar{\xx_1(X)\xx_4(X)}\E\bigpar{\xx_2(X)\xx_3(X)}	
\\&\quad
=\innprod{\gS,\xx_1\tensor\xx_2}\innprod{\gS,\xx_3\tensor\xx_4}
+
\innprod{\gS,\xx_1\tensor\xx_3}\innprod{\gS,\xx_2\tensor\xx_4}
\\&\fyll
+
\innprod{\gS,\xx_1\tensor\xx_4}\innprod{\gS,\xx_2\tensor\xx_3}.
  \end{split}
\end{equation*}
This can be written as 
\begin{equation}\label{eper}
  \E X\itpx4 =\Sigma\tensor\Sigma 
+ \pi_{(23)}\bigpar{\Sigma\tensor\Sigma}
+ \pi_{(24)}\bigpar{\Sigma\tensor\Sigma},
\end{equation}
where $\pi_{\gs}$ denotes permuting the coordinates by the premutation
$\gs$.
There are 24 permutations of $\set{1,2,3,4}$, but 8 of these leave
$\gS\tensor \gS$ invariant, and we may write \eqref{eper} more symmetrically
as 
$  \E X\itpx4 =3\symm\bigpar{\Sigma\tensor\Sigma}$, where  $\symm$ means the
symmetrization of the tensor by taking the average over all permutations of
the coordinates.

More generally,
for any even $k=2\ell$, by the same argument,
\begin{equation}\label{gaussinj}
\E X\itpx{2\ell} 
=\frac{(2\ell)!}{2^\ell\ell!}\symm\lrpar{(\E X\itpx2)\itpx\ell}
\in B\itpx{2\ell},
\end{equation}
which generalizes the standard formula 
$\E \xi^{2\ell}=\frac{(2\ell)!}{2^\ell\ell!}(\Var\xi)^{\ell}$
for a real-valued centred Gaussian variable $\xi$.

We conjecture that the corresponding formula for projective moments holds
too. If $B$ has the \ap, this follows by \eqref{gaussinj} and \refT{Tiotak}
below, but we leave the general case as an open problem.
\end{example}

\section{The approximation property}\label{Sapprox}

Let $B_1,B_2$ be  Banach spaces.
Recall that a finite rank operator $F:B_1\to B_2$ is a continuous
linear operator whose
range has finite dimension; equivalently, it is a
linear operator that can be written as a finite sum
$F(x)=\sumin \xx_i(x) y_i$ for some $\xx_i\in B_1^*$ and $y_i\in B_2$.

We say that a linear operator $T:B_1\to B_2$ is 
\emph{\uapprox}
if for every $\eps>0$ there exists
a finite rank operator $F:B_1\to B_2$ such that $\norm{T-F}<\eps$.
Similarly, we say that a linear operator $T:B_1\to B_2$ is 
\emph{\capprox}
if for every compact set $K\subset B_1$ and every
$\eps>0$ there exists
a finite rank operator $F:B_1\to B_2$ such that
$\sup\set{\norm{Tx-Fx}:x\in  K}<\eps$. 

If $B$ is a Banach space, then the following properties are equivalent; see
\eg{} \cite[Section 1.e]{LTzI} and \cite[Chapter 4]{Ryan} for proofs.
The Banach space $B$ is said to have the \emph{approximation property} when
these properties hold.

\begin{romenumerate}[-10pt]

\item \label{apapp}
The identity operator $I:B\to B$ is \capprox.

\item 
For every Banach space $B_1$, every bounded operator $T:B\to B_1$ is \capprox.

\item 
For every Banach space $B_1$, every bounded operator $T:B_1\to B$ is \capprox.

\item \label{apcomp}
For every Banach space $B_1$, every compact operator $T:B_1\to B$ is \uapprox.

\item \label{aptr}
For every pair of sequences $x_n\in B$ and $\xx_n\in B^*$,
$n\ge1$, such that
$\sumn \norm{x_n}\norm{\xx_n}<\infty$ and 
$\sumn \xx_n(x)x_n=0$ for all $x\in B$, we have $\sumn \xx_n(x_n)=0$.

\item \label{apbb*}
For every pair of sequences $x_n\in B$ and $\xx_n\in B^*$, $n\ge1$,  such that 
$\sumn \norm{x_n}\norm{\xx_n}<\infty$ and 
$\sumn \xx_n(x)x_n=0$ for all $x\in B$, we have $\sumn x_n\tensor\xx_n=0$ in
  $B\ptensor B^*$.

\item \label{apbb1}
For every Banach space $B_1$ and every pair of sequences $x_n\in B$ and
  $y_n\in B_1$, $n\ge1$, 
  such that 
$\sumn \norm{x_n}\norm{y_n}<\infty$ and 
$\sumn \xx(x_n)y_n=0$ for all $\xx\in B^*$, we have $\sumn x_n\tensor y_n=0$ in
  $B\ptensor B_1$.
\end{romenumerate}

\begin{remark}
  The property dual to \ref{apcomp}, \viz{} that 
every compact operator $T:B\to B_1$ is \uapprox, for every Banach space
$B_1$, is \emph{not} equivalent to the other properties; in fact, it is
equivalent to the \ap{} of $B^*$.

Moreover, it is known that
  if $B^*$ has the \ap, then $B$ has the \ap, but the converse does not
  hold. (See \refE{Edual} for a concrete example. There are also
  counterexamples that are separable with a separable dual \cite{LTzI}.)
\end{remark}

We can reformulate conditions \ref{apbb*} and \ref{apbb1} as follows,
recalling the canonical injection $\iota$ in \eqref{iota}.
(This is implicit in the references above, 
and explicit in \eg{} \cite{Szankowski}.) 

\begin{theorem}\label{Tiota2}
Let $B$ be a Banach space.
If $B$ has the approximation property, then the
  canonical mapping $\iota: B\ptensor B_1\to B\itensor B_1$ is injective for
  every Banach space $B_1$.

Conversely, if the
  canonical mapping $\iota: B\ptensor B^*\to B\itensor B^*$ is injective,
  then $B$ has the approximation property.
\end{theorem}

\begin{proof}
  Let $u\in B\ptensor B_1$; then $u=\sumn x_n\tensor y_n$ for some $x_n\in B$
  and $y_n\in B_1$ with $\sumn \norm{x_n}\norm{y_n}<\infty$.
We can regard $B\itensor B_1$ as space of bilinear forms on $B^*\times
B_1^*$, and then, for any $\xx\in B^*$ and $\yy\in B_1\q$
\begin{equation*}
\iota(u)(\xx,\yy)= \sumn \xx(x_n)\yy(y_n)
= \yy\Bigpar{\sumn \xx(x_n)y_n}.
\end{equation*}
Hence, $\iota(u)=0$ if and only if 
$\sumn \xx(x_n)y_n=0$ for every $\xx\in B\q$. 

Consequently, \ref{apbb1} says precisely that, for any $B_1$,
if $u\in B\ptensor B_1$ and $\iota(u)=0$, then $u=0$, \ie{} that
$\iota:B\ptensor B_1\to B\itensor B_1$ is injective.

Furthermore, \ref{apbb*} is the special case $B_1=B\q$, and thus says that
$\iota:B\ptensor B\q\to B\itensor B\q$ is injective.
\end{proof}

This can be extended to tensor products of several spaces. We state only the
case of tensor powers of a single space, which is the case we need.

\begin{theorem}\label{Tiotak}
If a Banach space $B$ has the approximation property, then the
  canonical mapping $\iota: B\ptpk \to B\itpk$ is injective.  
\end{theorem}

\begin{proof}
  We use induction in $k$. The case $k=1$ is trivial, and $k=2$ is a
  consequence of \refT{Tiota2}. 
For $k\ge3$ we write $\iota:B\ptpk\to B\itpk$ as the composition
  \begin{equation*}
B\ptpk = B\ptensor B\ptpx{(k-1)}	
\to B\itensor B\ptpx{(k-1)}	
\to B\itensor B\itpx{(k-1)}	
=B\itpk
  \end{equation*}
where the first map is injective by \refT{Tiota2} and the second is
injective by induction and \refL{Linj}.
\end{proof}

\begin{remark}
 The \ap{} for $B$ is \emph{not} equivalent to 
$\iota:B\ptensor  B\to B\itensor B$ being injective.
In fact, a counterexample by Pisier \cite[Theorem 10.6]{Pisier}
shows that there exists an infinite-dimensional Banach space $B$
such that
$B\ptensor  B= B\itensor B$ (with equivalent norms); moreover, this space $B$
lacks the \ap.
\end{remark}

The study of the \ap{} was initiated by
\citet{Grothendieck-resume,Grothendieck-memoir} who found most of the results
above but did not know whether any Banach spaces without the \ap{}
exist. The first counterexample was found by \citet{Enflo}, who constructed
a separable, reflexive Banach space $B$ without the \ap. 
A modification of the counterexample given by \citet{Davie73,Davie75}, 
see also \citet[Theorem 2.d.6]{LTzI}, shows that $B$ may be taken as a
subspace of $c_0$ or of $\ell^p$, for any $2<p<\infty$.
Another counterexample was found by \citet{Szankowski}, who showed that the
space $B(H)$ of bounded operators in an infinite-dimensional Hilbert space
does not have the \ap. 

On the other hand, it is easy to see that any Banach space with a (Schauder)
basis has the \ap;
this includes all classical examples of
separable Banach spaces. 
(In fact, Enflo's counterexample \cite{Enflo} was also the first known
separable Banach space without a basis.)
There are also many non-separable Banach spaces with the \ap.
The list of Banach spaces with the \ap{} includes, for example,
$\ell^p$ ($1\le p\le\infty$), $c_0$, $L^p(\mu)$ ($1\le p\le\infty$, $\mu$
any measure), $C(K)$ ($K$ a compact set).

\begin{example}\label{Edual}
  The tensor products $\ell^2\ptensor\ell^2$ and $\ell^2\itensor\ell^2$ 
have bases and thus have the \ap, see \eg{} \cite[Proposition 4.25 and
  Exercise 4.5]{Ryan}. As said in \refE{EHilbert}, see also \refT{THtp},
these spaces can be identified with the spaces of trace class operators (=
nuclear operators) and compact operators  in $\ell^2$, respectively;
moreover, $\ell^2\ptensor\ell^2\cong (\ell^2\itensor\ell^2)\q$.
However, $(\ell^2\ptensor\ell^2)\q$ can be identified with the space
$B(\ell^2)$ of bounded operators in $\ell^2$, which as just said does not
have the \ap{} \cite{Szankowski}.
\end{example}

We can now prove Theorems \refand{Tapprox}{Tmotex}.

\begin{proof}[Proof of \refT{Tapprox}]
By \refT{TPIB}, 
the projective and injective moments
$\E X\ptpk$, $\E Y\ptpk$, $\E X\itpk$, $\E Y\itpk$ exist in Bochner sense.
If \eqref{weak=}  holds, then $\E X\itpk=\E Y\itpk$ by \refC{Citpk=p}. 
By \eqref{momitpk}, this can be written $\iota(\E X\ptpk)=\iota(\E Y\ptpk)$, 
and since \refT{Tiotak} shows that $\iota$ is injective, 
we have $\E X\ptpk=\E Y\ptpk$.
By \refC{Cptpk=}, this yields \eqref{multi=}.

The converse is trivial.
\end{proof}

\begin{remark}
More generally, 
  \refT{Tiotak} implies that for any Banach space $B$ with the \ap, if
$\E X\ptpk$ and $\E Y\ptpk$ (and thus also $\E X\itpk$ and $\E Y\itpk$) 
exist in Pettis sense
and
$\E X\itpk=\E Y\itpk$, then $\E X\ptpk=\E Y\ptpk$. 
We do not know whether this remains valid if we only assume that the moments
exist in Dunford sense. 
(\refT{TCKD2} is a positive result in a special case.)
\end{remark}

\begin{proof}[Proof of \refT{Tmotex}]
  Let $B_0$ be a Banach space without the \ap{} and let $B\=B_0\oplus B_0^*$,
with the norm $\norm{(x,\xx)}_B\=\max\bigpar{\norm x, \norm{\xx}}$.
By choosing $B_0$ to be separable and reflexive, we obtain 
$B$ separable and reflexive too.
We shall show that there exist bounded \Bmeas{}
\rv{s} $X$ and $Y$ in $B$ such that,
for $k=2$,
\eqref{weak=} holds but not \eqref{multi=}.

Since $B_0$ does not have the \ap, by \ref{aptr} above, 
there exist sequences $x_n\in B_0$ and $\xx_n\in B_0^*$ with 
such that
$\sumn \norm{x_n}\norm{\xx_n}<\infty$ and 
$\sumn \xx_n(x)x_n=0$ for all $x\in B$, but $\sumn \xx_n(x_n)=0$.
Let
$a_n\=\norm{x_n}\,\norm{\xx_n}$, so $0<\sumn a_n<\infty$. We may eliminate
all $(x_n,\xx_n)$ with $a_n=0$, and we may thus assume that $a_n>0$ for each
$n$. Define $y_n\=x_n/\norm{x_n}\in B_0$, 
$\yy_n\=\xx_n/\norm{\xx_n}\in B_0^*$
and $p_n\=a_n/\summ a_m$. Thus $y_n$ and $\yy_n$ are unit vectors and $\sumn
p_n=1$. Furthermore, the properties of $x_n$ and $\xx_n$ translate to
\begin{align}
  \sumn p_n \yy_n(z)y_n &= 0 \qquad\text{for every $z\in B_0$},
\label{eq0}
\\
  \sumn p_n \yy_n(y_n) &\neq 0.
\label{neq00}
\end{align}
Let $N$ be a random positive integer with the distribution $\P(N=n)=p_n$,
and let $X$ and $Y$ be the $B$-valued \rv{s}
\begin{align*}
  X&\=(y_N,\yy_N),
\\
  Y&\=(y_N,-\yy_N).
\end{align*}
Note that $\norm X =\norm Y = 1$ a.s.
Then, if $\ga$ is any bounded bilinear form $B\times B\to\bbR$, then
\begin{multline*}
\ga(X,X)
=
\ga\bigpar{(y_N,0),(y_N,0)}
+\ga\bigpar{(y_N,0),(0,\yy_N)}
+\ga\bigpar{(0,\yy_N),(y_N,0)}
\\
+\ga\bigpar{(0,\yy_N),(0,\yy_N)}
\end{multline*}
and similarly for $Y$. Hence,
\begin{equation}\label{sof}
\E\ga(X,X)-\E\ga(Y,Y)=
2\E\ga\bigpar{(y_N,0),(0,\yy_N)}
+2\E\ga\bigpar{(0,\yy_N),(y_N,0)}.
\end{equation}
In particular, letting $\ga$ be the bounded bilinear form 
$\ga\bigpar{(x,\xx),(y,\yy)}\=\xx(y)$, we have by \eqref{neq00}
\begin{equation*}
  \E \ga(X,X)-\E\ga(Y,Y) = 2\E\yy_N(y_N)=2\sumn p_n \yy_n(y_n)\neq0.
\end{equation*}
Hence, \eqref{multi=} does not hold.

On the other hand, if $\xx_1,\xx_2\in B\q$, then
$\xx_i(x,\xx)=\zz_i(x)+\zzz_i(\xx)$ for some $\zz_i\in B_0\q$ and $\zzz_i\in
B_0\qx$, $i=1,2$. 
Hence,
\begin{equation}\label{jul}
  \begin{split}
\E\bigpar{\xx_1(y_N,0)\xx_2(0,\yy_N)}	
=
\E\bigpar{\zz_1(y_N)\zzz_2(\yy_N)}	
=
\zzz_2\Bigpar{\E\bigpar{\zz_1(y_N)\yy_N}}.
  \end{split}
\end{equation}
However, the continuous linear functional 
$$\yy\=\E\bigpar{\zz_1(y_N)\yy_N}= \sumn p_n\zz_1(y_n)\yy_n\in B_0\q
$$ 
satisfies, for every $x\in B_0$,
\begin{equation*}
  \yy(x)= \sumn p_n\zz_1(y_n)\yy_n(x)
= \zz_1\Bigpar{\sumn p_n\yy_n(x) y_n}=0
\end{equation*}
by \eqref{eq0}. Thus $\yy=0$ and \eqref{jul} yields
\begin{equation*}
\E\bigpar{\xx_1(y_N,0)\xx_2(0,\yy_N)}	
=
\zzz_2(\yy)=0.
\end{equation*}
Interchanging $\xx_1$ and $\xx_2$, we see that also 
$\E\bigpar{\xx_1(0,\yy_N)\xx_2(y_N,0)}=0$. Thus \eqref{sof} with 
$\ga(x_1,x_2)\=\xx_1(x_1)\xx_2(x_2)$ yields
\begin{equation*}
  \E\bigpar{\xx_1(X)\xx_2(X)} -  \E\bigpar{\xx_1(Y)\xx_2(Y)}=0,
\end{equation*}
which shows that \eqref{weak=} holds.
\end{proof}

The counterexample in this proof has also other unpleasant consequences.

\begin{example}
  \label{Exq}
Let the separable Banach space $B$ and the \rv{s} $X$ and $Y$ be as in the
proof of \refT{Tmotex}.
By Corollaries \refand{Cptpk=}{Citpk=p}, 
$\E (X\ptensor X)\neq\E(Y\ptensor Y)$ in $B\ptpx2$
but
$\E (X\itensor X)=\E(Y\itensor Y)$ in $B\itpx2$,
where all moments exist (in Bochner sense) by \refT{TPIB}.

We can embed $B$ as a closed subspace of a Banach space $B_1$ with
the  \ap, for example by the Banach--Mazur theorem which says that every
separable Banach space can be embedded as a closed subspace  of $\coi$.
Let $i:B\to B_1$ denote the embedding.

We may regard $X$ and $Y$ also as $B_1$-valued \rv{s}. 
Then the injective second moment $\E(X\itensor X)$ in $B\itensor B$
is mapped by $i\itensor i$ to 
the injective second moment $\E(X\itensor X)$ in $B_1\itensor B_1$,
and it follows that 
$\E(X\itensor X)=\E(Y\itensor Y)$ also in $B_1\itensor B_1$, \ie, $X$ and
$Y$ have the same injective second moments in $B_1$.

It follows from \refT{Tiotak} (as in the proof of \refT{Tapprox}) that
$\E(X\ptensor X)=\E(Y\ptensor Y)$ in $B_1\ptensor B_1$, \ie, $X$ and
$Y$ have the same projective second moments in $B_1$ although they have
different projective second moments in $B\tensor B$.
This shows that if a Banach space valued random variable takes values in a
subspace of the Banach space, we may have to be careful with in which space we
calculate the projective moments. 
(Note that \refL{Linj} and \refR{Rinjinj} show that
there is no 
such problem for injective moments.)

Since $i\ptensor i$ maps the second projective moments in $B\ptensor B$ to
the second projective
moments in $B_1\ptensor B_1$, we have shown that
$
 (i\ptensor i)\E(X\ptensor X) =  (i\ptensor i)\E(Y\ptensor Y)
$, 
although $\E(X\ptensor X) \neq\E(Y\ptensor Y)$. Consequently, $i\ptensor
i:B\ptensor B\to B_1\ptensor B_1$ is not injective, see \refR{Rprojinj}.
\end{example}

\begin{remark}\label{Rxq}
A simplified version of the counterexample in \refE{Exq}, without mentioning
moments, is the following: let $B$ be a subspace of a
Banach space $B_1$ such that $B_1$ has the \ap{} but $B$ has not.
Let $i:B\to B_1$ be the inclusion. 

Then $\iota:B\ptensor B\q\to B\itensor B\q$ is not injective by
\refT{Tiota2}.
Thus the composition
$(i\itensor I)\iota:B\ptensor B\q\to B\itensor B\q\to B_1\itensor B\q$ is
not injective, but this equals
the composition $\iota(i\ptensor I):B\ptensor B\q\to B_1\ptensor B\q\to
B_1\itensor B\q$. On the other hand, since $B_1$ has the \ap;
$\iota:B_1\ptensor B\q\to B_1\itensor B\q$ is injective by \refT{Tiota2}; hence
$i\ptensor I:B\ptensor B\q\to B_1\ptensor B\q$ is not injective.  
\end{remark}

\section{Hilbert spaces}\label{SH} 

Consider the case $B=H$, a Hilbert space.
We shall give some special results for  second moments.
(We do not know whether the results extend to moments of order
$k\ge3$ or not, and leave this
as open problems.)
We begin with some well-known results.

\begin{theorem}
  \label{THap}
A Hilbert space has the \ap.
\end{theorem}

\begin{proof}
 Property \ref{apapp} in \refS{Sapprox} is easily verified using
suitable orthogonal projections.  
\end{proof}

Next, we note that $H\q=H$; hence the 
correspondence \eqref{keiller} yields an
isometric embedding of $H\itensor H$ into $B(H)=L(H;H)$, the space of
bounded linear operators on $H$. We identify a tensor in $H\itensor H$ and
the corresponding operator without further comment; hence we regard
$H\itensor H$ as a subspace of $B(H)$. Moreover, an elementary tensor
$x\tensor y$ corresponds to an operator of rank 1, and every operator of
rank 1 is given by an elementary tensor; hence the tensors in $H\tensor H$,
which are finite  sums of elementary tensors, are exactly the operators on
$B$ of finite rank. The injective tensor product $H\itensor H$ is thus the
closure in $B(H)$ of the  set of finite rank operators, which shows 
(see \ref{apcomp} in \refS{Sapprox})
that
$H\itensor H = \ch$,
the space of compact operators $H\to H$.

The natural map $H\ptensor H\to \cN(H,H)$ onto the nuclear forms is a 
bijection;
moreover, the nuclear and integral bilinear forms on $H$ coincide.
Equivalently, the integral and nuclear operators $H\to H$ coincide,
and furthermore, the set of
them equals the set of \emph{trace class operators}, which we denote by
$\trh$. 
We can thus identify all these spaces of bilinear forms or
operators with $H\ptensor H$. 
(See \eg{} \cite[Chapter 48 and Proposition 49.6]{Treves}.
See further \eg{} \cite[Exercise IX.2.20]{Conway} or \cite[Chapter 31]{Lax}.)

With these identifications, the map $\iota:H\ptensor H\to H\itensor H$ is
just the inclusion map $\trh\to \ch\subseteq\bh$. 
 (By \refT{Tiota2} and $H\q=H$, the \ap{} is equivalent to the fact that
$\iota:H\ptensor H\to H\itensor H$ is injective, which we thus also see
explicitly.)

We summarize these descriptions of the tensor products.

\begin{theorem}
  \label{THtp}
Under the identification \eqref{keiller} of tensors and operators $H\q=H\to
H$, we have $H\ptensor H = \trh$, the space of trace class operators on
$H$ (which equals the space of nuclear operators),
and $H\itensor H = \ch$, the space of compact operators.
\nopf
\end{theorem}

\begin{theorem}
  \label{THnuc}
Every integral bilinear form on $H$ is nuclear. The space of these forms can
be identified with $H\ptensor H$.
\nopf
\end{theorem}

The general formulas $(B_1\ptensor B_2)\q=L(B_1,B_2;\bbR)$
and $(B_1\itensor B_2)\q=\cI(B_1,B_2)$, see Theorems \ref{Tproj*}--\ref{Tinj*}, 
can be translated to operators as 
$\trh\q=B(H)$ and $\ch\q=\trh$, 
 where the dualities are given by the trace form
$\innprod{T,S}=\Tr(TS\q)$. 

In particular, $(H\itensor H)\qx=\ch\qx=B(H)$.
Thus, if the second injective moment exists in Bochner or Pettis
sense, it is by the correspondence
\eqref{keiller} given by a compact operator in $\ch$, and
if it exists in Dunford sense it is, again by \eqref{keiller}, given by an
operator in $B(H)$. 
Similarly, 
if the second projective moment exists in Bochner or Pettis
sense, it is given by a trace class operator in $\trh$, and the 
second injective moment equals the same operator.
In all these cases, the following theorem shows that the second moment is a
positive 
operator. (In particular, it is self-adjoint.)

\begin{theorem}
  \label{TH>0} 
If the second moment $\E X\ptpx2$ exists in Bochner or Pettis sense,
or $\E X\itpx2$ exists in any sense, 
then,
regarding the moment as an operator in $B(H)$, it is a positive operator.
\end{theorem}
\begin{proof}
  In all cases, it follows that the injective
  moment $\E  X\itpx2$ exists in Dunford sense. Thus, for $x\in H$, by
  \eqref{tid},
  \begin{equation*}
	\innprod{\E X\itpx2 x,x} = 
	\innprod{\E X\itpx2, x\tensor x} = 
\E \innprod{ X,x}^2 \ge0.
\qedhere
  \end{equation*}
\end{proof}

\begin{remark}
  The only case remaining is the second projective moment $\E X\ptpx2$ in
  Dunford sense, which belongs to $\trh\qx=B(H)\q$. We shall see in
  \refE{EHmeas} (under a set theory hypothesis) that this moment is not
  always given by an operator on $H$.

However, if $H$ is separable, or more generally if $X$ is \assep, 
and $\E X\ptpx2$ exists in Dunford sense, then it exists in Pettis sense too
and thus $\E X\ptpx2\in \trh$.
This follows from \refT{T-c0}, since $\trh$ does not contain a subspace
isomorphic to $c_0$
(\eg{} by \cite[Theorem V.10]{Diestel}, since $H\ptensor H$ is a
separable dual space when $H$ is separable; we omit the details); 
moreover, we shall prove a more general result
by a different method in \refT{TH2}, which shows that in essentially all
cases (again depending on a set theory hypothesis), the Dunford and Pettis
senses coincide and thus
$\E X\ptpx2\in \trh$.
\end{remark}

We next characterize when the injective second moment exists, in the three
different senses; we begin with two lemmas. 

\begin{lemma}
  \label{LHwassep}
Every \wmeas{} $H$-valued \rv{} is \wassep.
\end{lemma}

\refE{EPD3} shows that $X$ is not necessarily \assep.

\begin{proof}
  Let $\set{e_s}_{s\in S}$ be an ON basis in the Hilbert space $H$,
and let $\xi_s\=|\innprod{e_s,X}|^2$.
Since $X$ is \wmeas, each $\xi_s$ is a non-negative \rv.

Let $\cC$ be the collection of countable subsets $A\subseteq S$. For
$A\in\cC$, let
\begin{equation*}
  \xi_A\=\sum_{s\in A}\xi_s
=\sum_{s\in A}|\innprod{e_s,X}|^2
\le \norm{X}^2 <\infty
\end{equation*}
Thus every $\xi_A$ is a finite non-negative \rv, and $A_1\subseteq A_2
\implies \xi_{A_1}\le\xi_{A_2}$.

Let $m_A\=\E\arctan\xi_A$, and
\begin{equation}
  \label{mxw}
\mxw\=\sup\set{m_A:A\in\cC}.
\end{equation}
Note that $\mxw\le\pi/2<\infty$.

There exist $A_n\in\cC$ with $m_{A_n}>\mxw-1/n$, so taking
$\axw\=\bigcup_{n=1}^\infty A_n$ we have $m_{\axw}\ge\mxw$ and thus
the $\sup$ in \eqref{mxw} is attained.
Moreover, for any $s\notin\axw$, $\axw\cup\set s\in\cC$ and thus
$m_{\axw\cup\set s} \le \mxw= m_{\axw}$,
\ie,
\begin{equation}
  \label{q12}
\E\arctan\xi_{\axw\cup\set s}\le \E\arctan\xi_{\axw}.
\end{equation}
Since $\xi_{\axw\cup\set s} =\xi_{\axw}+\xi_s\ge \xi_{\axw}$,
\eqref{q12} implies 
$\arctan\xpar{\xi_{\axw}+\xi_s}=\arctan{\xi_{\axw}}$ \as, and since
$\xi_{\axw}<\infty$, thus $\xi_s=0$ a.s.
Consequently, if $s\notin\axw$, then
\begin{equation}
  \label{q13}
\innprod{e_s,X}=0\quad\text{a.s.}
\end{equation}

Let $M$ be the closed linear span of \set{e_s:s\in\axw}.
$M$ is a separable subspace of $H$. If $y\in M^\perp$, 
 then $y=\sum_{s\notin \axw}a_se_s$ with only a countable number of
 $a_s\neq0$; hence \eqref{q13} implies $\innprod{y,X}=0$ a.s.
\end{proof}

\begin{lemma}
  \label{LHwassep2}
If $X$ is a \wmeas{} $H$-valued \rv{}, then $X\tpx2$ is \wassep{} in 
$H\itpx2$.
\end{lemma}

\begin{proof}
By \refL{LHwassep}, there exists a separable subspace $M\subseteq H$ such
that $\xx(X)=0$ \as{} for every $\xx\perp M$.
Let $\tM$ be the closed subspace of $H\itensor H$ spanned by \set{x\tensor
  y:x,y\in M}.

Let $\ga\in (H\itensor H)\q$; then $\ga$ is a bilinear form on $H$ which by
\refT{THnuc} is  nuclear; thus there exist $\xx_n,\yy_n$ with
$\sumn\norm{\xx_n}\norm{\yy_n}<\infty$ and
\begin{equation}\label{sw}
  \ga(x,y)=\sumn \xx_n(x)\yy_n(y).
\end{equation}
Assume that $\ga\perp\tM$.

Let $P:H\to M$ be the orthogonal projection onto $M$ and $Q\=I-P$.
Decompose $\ga$ as $\ga_{PP}+\ga_{PQ}+\ga_{QP}+\ga_{QQ}$, where
$\ga_{PP}(x,y)\=\ga(Px,Py)$,
$\ga_{PQ}(x,y)\=\ga(Px,Qy)$,
$\ga_{QP}(x,y)\=\ga(Qx,Py)$,
$\ga_{QQ}(x,y)\=\ga(Qx,Qy)$.

For any $x,y\in H$, $\ga(Px,Py)=\innprod{\ga,Px\tensor Py}=0$, since
$x\tensor y\in \tM$. Hence $\ga_{PP}=0$.
Consider one of the other terms, for example $\ga_{PQ}$. By \eqref{sw},
\begin{equation}\label{sjw}
  \ga_{PQ}(X,X)=\sumn \innprod{\xx_n,PX}\innprod{\yy_n,QX}
=\sumn \innprod{P\xx_n,X}\innprod{Q\yy_n,X}.
\end{equation}
However, $Q\yy_n\perp M$ and thus, by the choice of $M$,
$\innprod{Q\yy_n,X}=0$ \as, for every $n$. Hence \eqref{sjw} yields
$\ga_{PQ}(X,X)=0$ a.s.

Similarly, $\ga_{PQ}(X,X)=0$ \as{} and $\ga_{PQ}(X,X)=0$ a.s., and thus
\begin{equation*}
\innprod{\ga,X\tensor X}=
  \ga(X,X)=\ga_{PQ}(X,X)+\ga_{QP}(X,X)+\ga_{QQ}(X,X)=0\quad\text{a.s.}
\end{equation*}
This holds for every $\ga\perp\tM$, and thus $X\tensor X$ is \wassep.
\end{proof}

\begin{theorem}
  \label{TH2i}
Suppose that $H$ is a Hilbert space and that $X$ is a \wmeas{} $H$-valued
\rv.
\begin{romenumerate}[-10pt]
\item \label{th2id}
The injective second moment $\E X\itpx2$ exists in Dunford sense if and only
if\/ $\E|\innprod{y,X}|^2<\infty$ for every $y\in H$.

\item \label{th2ip}
The injective second moment $\E X\itpx2$ exists in Pettis sense if and only
if the random variables
$|\innprod{y,X}|^2$, for $y\in H$ with ${\norm y\le1}$, are \ui.

\item \label{th2ib}
The injective second moment $\E X\itpx2$ exists in Bochner sense if and only
if\/ $X$ is \assep{} and\/ $\E\norm{X}^2<\infty$.
\end{romenumerate}
\end{theorem}

\begin{proof}
\pfitemref{th2id}
  By \refT{TID}\ref{tidd}\ref{tid3}, using \refT{THnuc}.

\pfitemref{th2ip}
By \refT{THnuc}, \refL{LHwassep2} and \refT{TiPettis}\ref{tipettis2}\ref{tip3}.

\pfitemref{th2ib}
A special case of \refT{TPIB}.
\end{proof}

In the Hilbert space case, the different types of integrability can also be
characterized by the value of the moment.

\begin{theorem}
  \label{TH2ii}
Suppose that $H$ is a Hilbert space and that $X$ is a \wmeas{} $H$-valued
\rv{} such that the injective second moment $\E X\itpx2$ exists in Dunford
sense. Regard $\E X\itpx2$ as a bounded operator in $B(H)$.
\begin{romenumerate}[-10pt]
\item \label{th2iip}
The injective second moment 
exists in Pettis sense if and only if\/
$\E X\itpx2 \in \ch$.

\item \label{th2iib}
The injective second moment exists in Bochner sense if and only
if\/ $X$ is \assep{} and\/  $\E X\itpx2\in \trh$. 
\end{romenumerate}
\end{theorem}

\begin{proof}
\pfitemref{th2iip}
If the moments exist in Pettis sense, then $\E X\itpx2\in H\itensor H=\ch$.

Conversely, suppose that $\E X\itpx2\in \ch$. Let $E$ be any event. 
Then
\begin{equation*}
  \E(\etta_E X)\itpx2 +   \E(\etta_{E\comp} X)\itpx2 
=
  \E(\etta_E X\itpx2) +   \E(\etta_{E\comp} X\itpx2)
=
  \E X\itpx2 .
\end{equation*}
Since $  \E(\etta_E X)\itpx2\ge0$ and $\E(\etta_{E\comp} X)\itpx2 \ge0$ (in
operator sense) by \refT{TH>0}, it follows that
  \begin{equation}
	\label{sw2}
0 \le   \E(\etta_E X\itpx2) \le \E X\itpx2.
  \end{equation}
It is easily verified that if $S,T\in B(H)$ with $0\le S\le T$ and $T$
compact, then $S$ too is compact. 
(For example because $S\qq= VT\qq$ for some $V\in B(H)$, and thus
$S=S\qq(S\qq)\q=VTV\q$.) 
Hence, \eqref{sw2} implies that $\E(\etta_E X\itpx2) \in \ch=H\itensor H$
for every event $E$, which means that $X\itpx2$ satisfies the definition of
Pettis integrability.

\pfitemref{th2iib}
If $\E X\itpx2$ exists in Bochner sense, then $\E X\ptpx2$ too exists in
Bochner sense by  \refT{TPIB}.
Hence, $\E X\itpx2 = \E X\ptpx2 \in H\ptensor H = \trh$.
Moreover, $X$ is \assep{} by \refT{TPIB}.

Conversely, suppose that $\E X\itpx2\in \trh$ and $X$ is \assep.
Let $M$
be  separable subspace of $H$ such that $X\in M$ \as{} and let
$\set{e_n}_n$ be a (countable) ON basis in $M$.
Then $\norm X^2=\sum_n \innprod{X,e_n}^2$ \as, and thus
\begin{equation*}
  \begin{split}
  \E\norm{X}^2
&=\sum_n\E\innprod{X,e_n}^2
=\sum_n\innprod{\E X\itpx2 e_n,e_n}
=\Tr\bigpar{\E X\itpx2} 
\\&\le \norm{\E X\itpx2}_{\trh}<\infty.	
  \end{split}
\end{equation*}
Thus $\E X\itpx2$ exists in Bochner sense by \refT{TPIB}.
\end{proof}

\begin{theorem}
  \label{TH2j} 
Suppose that $H$ is a Hilbert space and that $X$ is a \wmeas{} $H$-valued
\rv{} such that the injective second moment $\E X\itpx2$ exists in Dunford
sense. Regard $\E X\itpx2$ as a bounded operator in $B(H)$.
If\/  $\E X\itpx2\in \trh$ and $X$ is \assep, then
the projective second moment $\E X\itpx2$ exists in Bochner sense.
\end{theorem}
\begin{proof}
By  \refT{TH2ii}\ref{th2iib} and \refT{TPIB}.
\end{proof}

\begin{remark}\label{CH2j}
  If the projective second moment exists in Bochner or Pettis sense, it is
  an element of $H\ptensor H = \trh$. The second injective moment is the
  same, and is then thus a trace class operator.

\refT{TH2j} gives a converse when $X$ is \assep.
However, the converse does not hold in general; 
a \wmeas{} \rv{} in a non-separable Hilbert space
can have an
injective moment in Pettis sense that is a trace class operator, even
if the projective second moment does not exist; see \refE{EPD3}.
\end{remark}

For the projective second moment, all three senses
coincide for $H$-valued \rv{s}, provided $\dim(H)$ is not too large.

\begin{definition}
  A cardinal $\cardm$ is \emph{real-\meas} if there exists a set $S$ with
  cardinality $|S|=\cardm$ and a probability measure $\mu$ defined on the
$\gs$-field $2^S$
of \emph{all} subsets of $S$ that is \emph{diffuse}, \ie, such
  that $\mu\set s=0$ for every   $s\in S$. (Obviously, then such a measure
  $\mu$ exists for every set $S$ with $|S|=\cardm$.)

A cardinal $\cardm$ is \emph{\meas} if there exists such a measure that takes
only the values $0$ and $1$.
\end{definition}

\begin{remark}\label{Rcardinals}
If \meas{} cardinals exist, they have to be very large; larger than the
  first
 strongly inaccessible cardinal.
It is consistent with the standard ZFC axioms for set theory to assume that
there are no strongly inaccessible cardinals, and thus no \meas{} cardinals.
Whether it also is consistent to assume the existence of \meas{} cardinals
is not known. See \cite[Chapter IX.3--4]{KurMos}
and \cite{Kanamori}.

Real-\meas{} cardinals may be smaller. There exists a real-\meas{} cardinal
that is non-\meas{} $\iff$ $\cardc$ is real-\meas{} $\iff$ Lebesgue measure
on $\oi$ can be extended to all subsets of $\oi$, see \cite{Ulam} and
\cite[Section 16.2]{Talagrand:Pettis}.

If the Continuum Hypothesis holds, then $\cardc$ is not real-\meas{},
see \cite{BanachKur} and \cite{Ulam}, and
thus every real-\meas{} cardinal is \meas{} and thus extremely large.
Consequently, it is consistent to assume that there are no real-\meas{}
cardinals at all. In this case, the following theorem applies to all Hilbert
spaces without qualification.
(The separable case was given in \refE{EPD1}.)
\end{remark}

\begin{theorem}
  \label{TH2}
Let $H$ be a Hilbert space such that $\dim H$ is a non-real-measurable
cardinal.
If $X$ is an $H$-valued \rv, then the following are equivalent.
\begin{romenumerate}
\item \label{th2d}
$\E X\ptpx2$ exists in Dunford sense.
\item \label{th2p}
$\E X\ptpx2$ exists in Pettis sense.
\item \label{th2b}
$\E X\ptpx2$ exists in Bochner sense.
\item \label{th20}
$X$ is \assep, $X\ptpx2$ is \wmeas{} in $H\ptpx2$ and $\E\norm{X}^2<\infty$.
\end{romenumerate}
\end{theorem}

\begin{proof}
\ref{th2b}$\implies$\ref{th2p}$\implies$\ref{th2d} is trivial
and \ref{th20}$\implies$\ref{th2b} for any space by \refT{TBochner}, using 
\refL{Lbs} and \refT{Twm}. 
It remains to show that
\ref{th2d}$\implies$\ref{th20}.

Suppose that $\E X\ptpx2$ exists in Dunford sense and let $z\=\E X\ptpx2$.
By definition, $X\tpx2$ is weakly measurable in $H\ptpx2$.

Let $M$ be a closed subspace of $H$ and let $P_M:H\to M$ be the
orthogonal projection.
Then $\ga_M(x,y)\=\innprod{P_Mx,P_My}$ is a bounded bilinear form on $H$; 
$\ga_M$ can by \refT{Tproj*} be regarded as a continuous linear functional
on $H\ptensor H$ and
\begin{equation*}
  \innprod{\ga_M,z}
=
\E\innprod{\ga_M,X\tensor X}
= \E \ga_M(X,X) = \E\norm{P_MX}^2.
\end{equation*}
Hence, $\norm{P_MX}^2$ is measurable and its expectation is finite.
Define
\begin{equation*}
  \mu_M\= \E\norm{P_MX}^2<\infty.
\end{equation*}
In particular, $\E\norm{X}^2=\mu_H<\infty$.

If $M_1\subseteq M_2$, then 
$\norm{P_{M_1}X} \le \norm{P_{M_2}X}$ and thus $\mu_{M_1}\le \mu_{M_2}$. In
particular, $\mu_M\le \mu_H$ for every closed subspace $M\subseteq H$, so
\set{\mu_M:M\subseteq H} is bounded.

Furthermore, if $M_1\perp M_2$, then 
$\norm{P_{M_1\oplus M_2}X}^2 = \norm{P_{M_1}X}^2+\norm{P_{M_2}X}^2$
and thus
\begin{equation}
  \label{hmu+}
\mu_{M_1\oplus M_2}=\mu_{M_1}+\mu_{M_2}.
\end{equation}

Let $\cZ$ be the set of separable closed subspaces of $H$, and let
\begin{equation}\label{hb}
  \mux\=\sup\set{\mu_M:M\in\cZ}.
\end{equation}
Thus $\mux\le\mu_H<\infty$.
There exist $M_n\in\cZ$ such that $\mu_{M_n}>\mux-1/n$, $n\ge1$. Let $\Mx$
be the closed linear hull of $\bigcup_{1}^\infty M_n$. Then $\Mx$ is
separable so $\Mx\in\cZ$, and $M_n\subseteq \Mx$ so 
$\mu_{\Mx}\ge\mu_{M_n}>\mux-1/n$. Consequently, $\mu_{\Mx}=\mux$ and the
supremum in \eqref{hb} is attained.

By  \eqref{hmu+},
\begin{equation}\label{hc}
  \mu_H=\mu_{\Mx}+\mu_{\Mx^\perp}
=\mux+\mu_{\Mx^\perp}.
\end{equation}
Assume first $\mux=\mu_H$. Then \eqref{hc} yields
\begin{equation*}
0=\mu_{\Mx^\perp}=\E\norm{P_{\Mx^\perp}X}^2,
\end{equation*}
and thus $P_{\Mx^\perp}X=0$ \as; hence $X\in\Mx$ \as{} so $X$ is \assep.

It remains to show that if $\mux<\mu_H$, then $\dim H$ is real-measurable,
which contradicts our assumption.
In this case, choose an ON basis $\set{e_s}_{s\in S}$ in $\Mx^\perp$.
For each subset $A\subseteq S$, let $M_A$ be the closed subspace of $H$
spanned by $\set{e_s}_{s\in A}$ and let $\mu(A)\=\mu_{M_A}=\E\norm{P_{M_A}X}^2$.
Then $\mu$ is finitely additive by \eqref{hmu+}. Moreover, if $A_n\upto A$,
then $\norm{P_{M_{A_n}}X}\upto\norm{P_{M_A} X}$  and thus
$\mu(A_n)\upto\mu(A)$ by 
monotone convergence. Consequently, $\mu$ is a $\gs$-additive measure
defined on $(S,2^S)$. 
For any $s\in S$, $M_{\set s}=\bbR e_s$ is
one-dimensional and thus $\Mx\oplus M_{\set s}$ is separable, so
\begin{equation*}
  \mux+\mu_{M_{\set s}} =   \mu_{\Mx}+\mu_{M_{\set s}}
=   \mu_{\Mx\oplus M_{\set s}}\le \mux.
\end{equation*}
Hence, $\mu\set s\=\mu_{M_{\set s}}=0$, which shows that the measure $\mu$ is
diffuse. 
On the other hand, $M_S=\Mx^\perp$ and thus by \eqref{hc}
\begin{equation*}
  \mu(S)=\mu_{M_S}=\mu_{\Mx^\perp}=\mu_H-\mux>0.
\end{equation*}
Consequently, by normalizing $\mu$ we obtain a diffuse probability measure on
$(S,2^S)$, which shows that $|S|$ is a real-\meas{} cardinal.
Since $\dim(H)\ge|S|$, $\dim(H)$ is real-\meas{} too.
(In fact, $\dim H=|S|+\dim\Mx=|S|+\aleph_0=|S|$.)
\end{proof}

The assumption in \refT{TH2} that $\dim H$ is not real-measurable is necessary
as is seen by the following example. (Cf.\ the related \refE{EPD3}.)
Nevertheless, \refT{TH2} shows that the projective second moments
for practical applications only can be used
for \Bmeas{} random variables, \ie, in the \assep{} case. 

\begin{example}\label{EHmeas}
  Suppose that $\dim H$ is a real-measurable cardinal.
Let $\mu$ be a diffuse probability measure on $(S,2^S)$ for some set $S$ with
$|S|=\dim H$. Since then $H\cong \ell^2(S)$, we may assume that
$H=\ell^2(S)$.

Define $X:(S,2^S,\mu)\to H=\ell^2(S)$ by $X(s)=e_s$. Then $X$ is bounded,
and since \emph{every} function on $(S,2^S,\mu)$ is measurable, 
$X$ is \wmeas{} in $H$ and
$X\tensor X$
is \wmeas{} in $H\ptensor H$. 
Hence $\E X\ptpx2$ exists in Dunford sense.
($X$ and $X\tensor X$ are also Borel measurable,
for the same reason.)
Furthermore, if $\ga$ is the bilinear form on $H$ given by the inner
product,
$\innprod{\E X\ptpx2,\ga}=\E\ga(X,X)=\E \norm{X}^2=1$ and thus
$\E X\ptpx2\neq0$.

On the other hand, the argument in \refE{EPD3}
shows that $X$ is not \assep.
(In fact, if $M$ is any separable subspace of $H$, then $X\perp M$ a.s.)
By \refT{TPIB}, $\E X\ptpx2$ does not exist in Bochner sense.

Note further that $\innprod{\xx,X}=0$ \as, for every $\xx\in H\q=H$.
Thus 
the injective second moment $\E X\itpx2$ exists in Dunford (and Pettis) sense by
\refT{TH2i},
and \eqref{tid} shows that $\E X\itpx2=0$. 
However, using Dunford senses, $\E X\ptpx2\neq0$, as shown above,
and $\E X\itpx2=\iota\qx(\E X\ptpx2)$ by \eqref{momitpk**}.
Since $\iota: H\ptpx2\to H\itpx2$
is injective, 
this shows that $\E X\ptpx2\not\in H\ptpx2$; hence
$\E X\ptpx2$ does not exist in Pettis sense.
\end{example}

\begin{example} 
Consider Gaussian random variables in a Hilbert space $H$.
As shown in \refE{Eg}, if $X$ is Gaussian, then $\E X\ptpk$  and $\E X\itpk$
exist in Bochner sense. In particular, $\E X\ptpx2= \E X\itpx2$ is an
element of $H\ptensor H= \trh$, \ie, a trace class operator. By \refT{TH>0},
$\E X\ptpx2= \E X\itpx2$ is a positive trace class operator.

Conversely, if $\gS$ is any positive trace class operator in a Hilbert space
$H$, then there exists by the spectral theorem for compact self-adjoint
operators, 
see \eg{} \cite[Corollary II.5.4]{Conway}, 
an ON set $(e_n)\seq$ in $H$ such that 
\begin{equation}
  \gS=\sumn \gl_n e_n\tensor e_n,
\end{equation}
where $\gl_n\ge0$  and $\sumn\gl_n=\norm{\gS}_{\trx(H)}<\infty$.
Let $(\xi_n)\seq$ be \iid{} standard normal variables and define $X\=\sumn
\gl\qq \xi_n e_n$; it is easily seen that this sum converges \as, that $X$
is Gaussian, and that $\E X\ptpx2=\sumn \gl_n e_n\tensor e_n = \gS$.

Consequently, the second moment of a Gaussian random variable in a Hilbert
space is a positive trace class operator, and can be any such operator.
Moreover, the second moment determines the higher moments by
\eqref{gaussinj}, and thus the distribution by \refT{TU} below.

\end{example}

\section{$L^p(\mu)$}\label{SLp} 

Let $1\le p<\infty$ and let $B=L^p(\mu)=L^p(S,\cS,\mu)$, where $\mu$ is a
$\gs$-finite measure on a measurable space $(S,\cS)$.
Note first that $L^p(\mu)$ has the \ap, see \eg{} \cite[Example 4.5]{Ryan}.
(It suffices to consider the case when $\mu$ is a probability measure, and
then \ref{apapp} in \refS{Sapprox} is satisfied by using conditional
expectations on suitable finite sub-$\gs$-fields; we omit the details.)
Furthermore, $L^p(S,\cS,\mu)\q=L^q(S,\cS,\mu)$, where
$q\in(1,\infty]$ is the conjugate exponent given by
$p\qw+q\qw=1$
\cite[Corollary IV.8.1,5]{Dunford-Schwartz}.

If $X:(s,\go)\mapsto X(s,\go)\in\bbR$ is an 
$(\cS\times\cF)$-\meas{} function on $S\times\gO$, then
$X(s)=X(s,\cdot)$ is a (real-valued) \rv{} for every $s$; 
moreover, for each $\go\in\gO$, $X(\cdot,\go)$ is a function on $S$, and
if further
$\int_S|X(s,\go)|^p\dd\mu(s)<\infty$ a.s., then $X$ can be regarded as a
mapping $(\gO,\cF,\P)\to L^p(\mu)$,
\ie, as a \rv{} in  $B=L^p(\mu)$.
This \rv{} is \Bmeas, as stated in the following lemma;
see \cite[III.11.16--17]{Dunford-Schwartz}.
\begin{lemma}\label{LLp}
If $X:(s,\go)\mapsto X(s,\go)\in\bbR$ is an 
$(\cS\times\cF)$-\meas{} function on $S\times\gO$
and
$\int_S|X(s,\go)|^p\dd\mu(s)<\infty$ a.s., then $X$ can be regarded as a
\Bmeas{} mapping $(\gO,\cF,\P)\to L^p(\mu)$, \ie, a \Bmeas{}
$L^p(\mu)$-valued \rv;
conversely, every \Bmeas{} $L^p(\mu)$-valued \rv{} is (\as) represented in
this way by some $(\mu\times\P)$-\aex{} unique $(\cS\times\cF$)-\meas{} $X$
on $S\times\gO$.
\nopf
  \end{lemma}

\begin{theorem}
  \label{TLp}
Let $1\le p<\infty$, let $q$ be the conjugate exponent given by
$p\qw+q\qw=1$, and let $(S,\cS,\mu)$ be a $\gs$-finite measure space.
\begin{romenumerate}[-18pt]
\item \label{tlpx}
Suppose that $X:S\times\gO\to\bbR$ is $\ppar{\cS\times\cF}$-\meas{} and that
$\norm{X}_{L^p(\mu)}^p \=\int_S|X(s,\go)|^p\dd\mu(s)<\infty$ a.s.
Regard $X$ as an $L^p(\mu)$-valued \rv{} and suppose further that
$\E\norm{X}_{L^p(\mu)}^k<\infty$.
Then, $\E X\ptpk\in L^p(S,\mu)\ptpk$ 
and $\E X\itpk\in L^p(S,\mu)\itpk$ 
exist in Bochner sense.
Furthermore, $\E X\ptpk$  and $\E X\itpk$ 
are represented by the \aex{} finite function
\begin{equation}\label{Phik}
  \Phi_k(s_1,\dots,s_k)\=\E\bigpar{X(s_1,\go)\dotsm X(s_k,\go)}
\end{equation}
in the sense that if 
$g_1,\dots,g_k\in L^q(\mu)=(L^p(\mu))\q$, then 
\begin{multline}\label{emm}
\hskip1em
\innprod{\E X\ptpk,\mx g\tensor}
=\innprod{\E X\itpk,\mx g\tensor}
\\
=\int_{S^k}
 \Phi_k(s_1,\dots,s_k)g_1(s_1)\dotsm g_k(s_k)\dd\mu(s_1)\dotsm\dd\mu(s_k).
\end{multline}
\item \label{tlpxy}
Let $Y$ be another
$L^p(\mu)$-valued \rv{} represented by an \ppar{\cS\times\cF}-\meas{} function
$Y:S\times\gO\to\bbR$ such that $\E\norm{Y}_{L^p(\mu)}^k<\infty$,
and let
\begin{equation}\label{Psik}
  \Psi_k(s_1,\dots,s_k)\=\E\bigpar{Y(s_1,\go)\dotsm Y(s_k,\go)}.
\end{equation}
Then the following are equivalent
\begin{bienumerate}
\item \label{tlpxyp}
$\E X\ptpk=\E Y\ptpk$;
\item  \label{tlpxyi}
$\E X\itpk=\E Y\itpk $;
\item  \label{tlpxyw}
For any $g_1,\dots,g_k\in L^q(\P)$,
\begin{equation*}
\E\bigpar{\innprod{g_1,X}\dotsm \innprod{g_k,X}}
=
\E\bigpar{\innprod{g_1,Y}\dotsm \innprod{g_k,Y}}  ;
\end{equation*}
\item  \label{tlpxys}
$\Phi_k=\Psi_k$ \aex{} on  $S^k$.
\end{bienumerate}
\end{romenumerate}
\end{theorem}

\begin{proof}
\pfitemref{tlpx}
$X$ is \Bmeas{} by \refL{LLp}, and thus $\E X\ptpk$ and $\E X\itpk$ exist in
  Bochner sense by \refL{Lbk} and \refT{TBochner} (or by \refT{TPIB}).
Furthermore, by \eqref{tid} and Fubini's theorem, using \refL{LLp},
\begin{equation*}
  \begin{split}
\innprod{\E X\itpk&,\mx g\tensor}
=
\E\bigpar{\innprod{g_1,X}\dotsm \innprod{g_k,X}}
\\&
=\E\int_{S^k}X(s_1)\dotsm X(s_k)\,g_1(s_1)\dotsm
g_k(s_k)\dd\mu(s_1)\dotsm\dd\mu(s_k)
\\&
=\int_{S^k}\E\bigpar{X(s_1)\dotsm X(s_k)}g_1(s_1)\dotsm
g_k(s_k)\dd\mu(s_1)\dotsm\dd\mu(s_k),
  \end{split}
\end{equation*}
showing \eqref{emm}.

\pfitemref{tlpxy}
The moments exist in Bochner sense by \ref{tlpx}.
Then \ref{tlpxyp}$\iff$\ref{tlpxyi} by \refT{Tiotak}
and \eqref{momitpk},
\ref{tlpxyi}$\iff$\ref{tlpxyw} by \refC{Citpk=p}  
and \ref{tlpxyi}$\iff$\ref{tlpxys} by \eqref{emm}.
\end{proof}

\begin{example}
  \label{EL2}
In the special case $p=2$, $L^2(\mu)$ is a Hilbert space, and we
can also apply the results of \refS{SH}. 
(Since $X$ is \Bmeas{} by \refL{LLp}, and thus \assep, it suffices to
consider a separable subspace of $L^2(\mu)$.)
In particular,
\refT{TH2i} and \refT{TH2ii} give conditions for the existence of the
injective second moment in the different senses, while
\refT{TH2} shows that for the second projective moment, the different senses
coincide. 
\end{example}

\begin{example}
Another interesting special case is $p=1$.
We have $L^1(S,\mu)\ptpk=L^1(S^k,\mu^k)$,
see
\cite[Exercise 2.8]{Ryan}; thus $\E X\ptpk\in L^1(S^k,\mu^k)$ (when it
exists in Bochner or Pettis sense); clearly $\E X\ptpk$ equals the function
$\Phi_k$ in 
\eqref{Phik} when \refT{TLp} applies.

\refE{EPD2} shows that 
$\E X\ptpk$ may exist in Pettis  sense in $L^1(S,\mu)\ptpk=L^1(S^k,\mu^k)$
without existing in Bochner
sense. In this case 
$\E X\ptpk$ is still given by a function $\Phi_k$ in $L^1(S^k,\mu^k)$
but the pointwise formula \eqref{Phik} may fail; 
in \refE{EPD2}, 
we have $X(s,\go)=\pm a_N$ for all $s$ and $\go$, and thus
$|X(s_1,\go)X(s_2,\go)|=a_N^2$, so the expectation in \eqref{Phik}
exists (for any $s_1,s_2$) only when $\E a_N^2=\sum_n p_na_n^2<\infty$,
which in this case is 
the condition for $\E X\ptpx2$ to exist in Bochner sense, see \refE{EPD2};
we may choose $p_na_n^2=1/n$ to obtain our counterexample.
In this example, $\E X\ptpx2$ equals $\sum_n p_na_n^2 r_n\tensor r_n$, where
the sum converges in $L^2$, and thus in $L^1$; the sum also converges
pointwise \aex{} (since it is a sum of independent \rv{s} defined on $\oi^2$),
but (when  $p_na_n^2=1/n$, say) does not converge absolutely at any point.

We do not know any necessary and sufficient conditions for  
the existence of $\E X\ptpk$  in Pettis  sense in
$L^1(S,\mu)\ptpk=L^1(S^k,\mu^k)$, nor
for the existence of injective moments beyond \refT{TID}, where 
\ref{tidd}\ref{tid2} applies by \refL{LLp}.
(Even for $k=1$, the  existence of the mean $\E X$ in Pettis sense in
$L^1(\mu)$
seems
difficult to characterize exactly, since it is essentially equivalent to $X\in
L^1(\mu)\itensor L^1(\P)$, which has no simple description, see \refR{RL1}.)

Since $L^1(S^k,\mu^k)$ does not contain any subspace isomorphic to $c_0$ 
(\eg{} as a consequence of \cite[Theorem 6.31 and Corollary 6.21]{Ryan}),
it follows from \refT{T-c0} that if $X$ is \Bmeas{} (\eg{} by \refL{LLp}),
then $\E X\ptpk$ exists in Pettis sense as soon as it exists in Dunford sense.
\end{example}

\begin{example}
  Taking $S=\bbN$ with counting measure, we obtain $\ell^p$, $1\le  p<\infty$.
(See \eg{} Examples \ref{El2}, \refand{El1}{EPD1i}.)
In this case, $X=(X_n)_{n=1}^\infty$ and
$\Phi_k$ is by \eqref{Phik} the function
$\E\bigpar{X_{n_1}\dotsm X_{n_k}}$ on $\bbN^k$. 
Consequently, by \refT{TLp}, if $\E\norm{X}_{\ell^p}^k<\infty$ and
$\E\norm{Y}_{\ell^p}^k<\infty$, then $\E X\ptpk=\E Y\ptpk$ if and only if
all mixed \kth{} moments of $(X_n)_{n=1}^\infty$ 
and $(Y_n)_{n=1}^\infty$ coincide.

Note that moments may exist in Pettis sense also under weaker
assumptions, see \refE{El2}.

Note that the projective tensor product $\ell^p\ptensor\ell^p$ with
$2<p<\infty$
and the injective tensor product $\ell^p\itensor\ell^p$ with
$1<p<2$ are reflexive, see \cite[Corollary 4.24]{Ryan}. In these cases, 
at least, 
the second moment thus exists in Pettis sense as soon as it exists in Dunford
sense.
\end{example}

\section{$C(K)$} \label{SCK} 

In this section we study the case $B=C(K)$ where $K$ is a compact 
space. (By compact we mean compact Hausdorff.)
The perhaps most important example is $\coi$.

Note that several other Banach spaces are isomorphic to $C(K)$ for some 
compact $K$.
Hence the results in this section apply to these spaces too.

\begin{example}\label{Egb}
  Let $\cb(Z)$ be the space of bounded continuous functions on a completely
  regular topological space $Z$. Then $\cb(Z)= C(\gb Z)$, where $\gb Z$
  is the \SCc{} of $Z$, see \eg{} \cite[Section 3.6]{Engelking}, 
\cite[Chapter 6]{GillmanJerison}
and \cite[Section V.6 and Exercise VIII.2.7]{Conway}.
Note that $Z$ is a dense subspace of $\gb Z$, and that every bounded
continuous function on $Z$ has a unique continuous extension to $\gb Z$.

One important example is 
\begin{equation}
\ell^\infty = \cb(\bbN)=C(\gbn).  
\end{equation}
\end{example}

\begin{example}\label{ECdoi}
  We shall see in \refT{TcI} that $\doi=C(\hI)$ for a compact space $\hI$.
\end{example}

\begin{example}\label{EC*}
  If $B$ is any complex commutative $C^*$-algebra, then $B\cong C_\bbC(\Sigma)$,
the space of complex-valued continuous functions on the maximal ideal space
 $\Sigma$,
see \cite[Section VIII.2]{Conway}, and thus the subset
$B_\bbR\=\set{f:f=f\q}$ of 
hermitean (\ie, real)  elements is isomorphic to $C(\Sigma)$.
This includes Examples \ref{Egb}--\ref{ECdoi}.
\end{example}

\begin{example}\label{Eloccomp}
Let $L$ be a locally compact space and $L^*=L\cup\set\infty$ its one-point
compactification. 
Then $C_0(L)=\set{f\in C(L\q):f(\infty)=0}$ is a subspace of codimension 1 in
  $C(L\q)$. Hence, if $X$ is a \rv{} in $C_0(L)$, we can regard it as a
  \rv{} in $C(L\q)$.
Note that $C_0(L)$ is a complemented subspace. (Every subspace of
  finite codimension in a Banach space is complemented.)

In particular, $c_0=C_0(\bbN)$ is a (complemented) subspace of codimension 1
in $c=C(\bbNx)=C(\bbN\cup\set\infty)$.
(In fact, $c_0$ is also isomorphic to $c$, by the mapping $(a_n)_1^\infty
\mapsto (a_{n+1}+a_1)_1^\infty$, but it seems more convenient to use the
 inclusion.)
\end{example}

We begin by noting some well-known facts.
See \eg{} 
\cite[Example 4.2]{Ryan}, 
\cite[Section 3.2]{Ryan}
and
\cite[Theorem V.6.6]{Conway},
respectively,
 for  proofs.

\begin{theorem}\label{TCKap}
$C(K)$ has the approximation property, for any compact $K$.	
\nopf
\end{theorem}


\begin{theorem}\label{TCKitpk}
If\/ $K_1,\dots,K_k$ are compact spaces, then $C(K_1)\itensor\dotsm\itensor
C(K_k)
\allowbreak
=C(K_1\times\dots\times K_k)$ (isometrically) with the natural identification.
In particular, $C(K)\itpk=C(K^k)$.  
\nopf
\end{theorem}

\begin{theorem}\label{TCKsep}
 $ C(K)$ is separable if and only if $K$ is metrizable.
\nopf
\end{theorem}

\begin{corollary}\label{CCKtpk}
If\/ $K$ is a compact space, then $C(K)\ptpk$ can be regarded as a subspace
of $C(K)\itpk=C(K^k)$.
(As a vector space; typically, the norms
differ.) 
\end{corollary}

\begin{proof}
  By \refT{TCKap} and \refT{Tiotak}, $\iota:C(K)\ptpk\to C(K)\itpk$ is
  injective. Furthermore, $C(K)\itpk=C(K^k)$ by \refT{TCKitpk}.
\end{proof}

\begin{remark}\label{RCKtpk}
Except in trivial cases ($k=1$ or $K$ finite),
$\ck\ptpk$ is not a closed subspace of $\ck\itpk$, and thus the norms
are different and not even equivalent on $\ck\ptpk$.
This is implicitly shown for $c_0$
by \citet{Littlewood}, who showed (in our
terminology) the existence of a bounded bilinear form $\ga\in (c_0\ptpx2)\q$
which does not belong to $(c_0\itpx2)\q=c_0(\bbN^2)\q=\ell^1(\bbN^2)$;
the result can be transfered to $C(K)$ for any infinite compact $K$.
(See also the proof of \refT{Tdoik} for an argument from \cite{Varopoulos}
for \coi.)
\end{remark}

\begin{theorem}
  \label{TCK0}
Let $X$ be a $C(K)$-valued \rv, where $K$ is a compact space.
If\/ $\E X\itpk$ exists in Bochner or Pettis sense, then 
it is the function
in $C(K)\itpk=C(K^k)$ given by
\begin{equation}
  \label{eckt}
\E X\itpk(t_1,\dots,t_k)
=\E\bigpar{X(t_1)\dotsm X(t_k)}.
\end{equation}
\end{theorem}

\begin{proof}
$\E X\itpk\in C(K)\itpk=C(K^k)$ by \refT{TCKitpk}.
Furthermore, the point evaluations $\gd_t$, $t\in K$, are continuous linear
  functionals on $C(K)$, and thus
\begin{equation}\label{eckt1}
\E X\itpk(t_1,\dots,t_k)
= \innprod{\E X\itpk,\gd_{t_1}\tensor\dots\tensor\gd_{t_k}}
=\E\bigpar{X(t_1)\dotsm X(t_k)}.
\end{equation}
\end{proof}

Note that the function \eqref{eckt}
is the standard \kth{} moment function for a
stochastic process; in particular, for $k=2$ and $\E X=0$, $\E X\itpx2$ is
the covariance function.

\begin{example}\label{EwienerC}
Let $W$ be standard Brownian motion in $\coi$, see \refE{Ewiener}. 
Then all (projective and
injective) moments exist in Bochner sense, \eg{} by \refT{TCK1k} below.
All odd moments vanish by symmetry.
$\E W\itpx2$ is the covariance function $\E(W(s)W(t))=\Cov(W(s),W(t))=s\bmin
t$ regarded as a continuous function in $ C\xpar{\oi^2}$,
and $\E W\ptpx2$ is the same function regarded as an element of the subspace
$\coi\ptpx2$. Similarly,
$\E W\ptpx4=\E W\itpx4$ is the function in $C(\oi^4)$ given by, see \eg{}
\cite{SJIII},  
\begin{multline}\label{w4}
   \E\bigpar{W(t_1)W(t_2)W(t_3)W(t_4)}
\\
=(t_1\bmin t_2)(t_3\bmin t_4)+(t_1\bmin t_3)(t_2\bmin t_4)
+(t_1\bmin t_4)(t_2\bmin t_3). 
\end{multline}
\end{example}

\begin{remark}\label{Rmerc}
  If $\E X\itpk$ exists in Dunford sense, then \eqref{eckt} still defines a
bounded function on $K^k$, but the function is not necessarily  continuous
(\refE{Egoi}); 
moreover, this function by itself does not in general determine
  $\E X\itpk\in C(K^k)\qx$ uniquely, not even for $k=1$, 
and not even if it happens to be continuous. 
(Note that the point
  evaluations do not form a total set in $C(K^k)\q$.) 
Indeed, in \refE{ECK2--} we shall see a $K$ and a \rv{} $Z\in\ck$ such that
$\E Z$ exists in Dunford sense (but not Pettis sense) with $\E Z\neq0$, but
$\E(Z(t))=0$ for every $t\in K$. (This cannot happen when $K$ is metric, see
\refC{CCK11}; see further Remarks \refand{Rvenus}{RCKD2}.) 
We therefore prefer the Pettis or Bochner case for applications.
\end{remark}

Apart from the previously defined $\gs$-fields on $C(K)$, we let $\cC$ be
the $\gs$-field generated by the point evaluations $f\mapsto f(t)$, $t\in
K$.
Thus $X:\gO\to C(K)$ is $\cC$-measurable if and only if $X(t)$ is measurable
for every $t\in K$. Since $\gd_t\in C(K)\q$, it is immediate that every
\wmeas{} \rv{} $X$ is $\cC$-\meas; we shall see that the converse holds when
$K$ is metrizable but not in general (\refE{ECK2--+}).

Further, let $\cB(K)$ be the Borel $\gs$-field on $K$.

We also let $M(K)$ be the space of signed Borel measures on $K$, and $\cM$
be the $\gs$-field on $M(K)$ generated by the maps $\mu\mapsto
\innprod{\mu,f}\=\int_K f\dd\mu$, $f\in C(K)$. Note that every $\mu\in M(K)$
defines a continuous linear functional on $C(K)$, so there is a bounded
linear map $M(K)\to C(K)\q$. 
Moreover, by the Riesz representation theorem
(see \eg{} \cite[Theorem IV.6.3]{Dunford-Schwartz},  
\cite[Theorem 7.3.5]{Cohn},  
\cite[Theorem  III.5.7]{Conway}), 
this map is an isometric bijection of the subspace $\mr(K)$ of
regular measures onto $C(K)\q$.
(In many cases $\mr(K)=M(K)$, for example when $K$ is a metrizable compact
space, see \cite[Proposition 8.1.10]{Cohn}.)

\refT{TCKsep} implies that if $K$ is a compact metric space, there are no 
measurability problems; some results are given in the following
lemma.  
However, if $K$ is not metrizable, and thus $\ck$ is not
separable, the situation is more complicated; we have to be more
careful with measurability in statements, but even so, many results below do
not hold for arbitrary compact $K$, see the (counter)examples at the end of
the section. We thus state most of our results for the metrizable case only.
(See Sections \ref{Sc0} and \ref{SmeasD}--\ref{SDmom}, together with
Examples \ref{Eloccomp} and \ref{ECdoi}, for examples of
non-separable spaces $\ck$ where most results hold, although some new
arguments are required.)

\begin{remark}
  All spaces $\ck$ where $K$ is an uncountable compact metric space are
  isomorphic as Banach spaces \cite{Mil}, \cite{Pel} (although not
  isometrically); hence they are all the same as $\coi$ from an abstract
  point of view. We shall, however, not use this; we prefer to regard the
  spaces concretely.
\end{remark}

\begin{lemma}\label{LCKjoint}
  Let $K$ be a metrizable compact space.
  \begin{romenumerate}[-10pt]
  \item \label{lckjointk}
The mapping $(f,t)\to f(t)$ is
  jointly $(\cC\times\cB(K))$-measurable on $C(K)\times K$.

  \item \label{lckjointm} 
The mapping $(f,\mu)\to\int_K f\dd\mu$ is
  jointly $(\cC\times\cM)$-measurable on $C(K)\times M(K)$.

In particular, $(f,\xx)\to\innprod{\xx,f}$ is
  jointly $(\cC\times\cM)$-measurable on $C(K)\times \mr(K)=C(K)\times C(K)\q$.
 \end{romenumerate}
\end{lemma}

\begin{proof}
Choose a metric $d$ on $K$.  For each $n\ge1$ there is a finite covering of $K$
by open sets $U_{ni}$, $1\le i\le N_n$, of diameters $<1/n$. We may find a
partition of unity subordinate to $\set{U_{ni}}_i$, \ie, a set of functions
$f_{ni}\in C(K)$ such that $f_{ni}\ge0$, $\sum_i f_{ni}=1$ and $\set{t\in
  K:f_{ni}(t)>0}\subseteq U_{ni}$. (For example, let
$g_{ni}(t)\=d(t,U_{ni}\comp)$ and $f_{ni}\=g_{ni}/\sum_j g_{nj}$.)

Choose some $t_{ni}\in U_{ni}$. For any function $f\in C(K)$,
\begin{equation}\label{magn}
  \begin{split}
\sup_t\Bigabs{f(t)&-\sum_i f(t_{ni})f_{ni}(t)}
=\sup_t\Bigabs{\sum_i(f(t)- f(t_{ni}))f_{ni}(t)}
\\&
\le\sup_t\sum_i|f(t)- f(t_{ni})|f_{ni}(t)
\le \sup\set{\abs{f(t)-f(u)}:|t-u|<1/n}
\\&
\to0  	
  \end{split}
\raisetag\baselineskip
\end{equation}
as \ntoo. 
This shows first that
$f(t)=\limn\sum_i f(t_{ni})f_{ni}(t)$, which is $\cC\times \cB(K)$-\meas,
showing \ref{lckjointk}.

Moreover, \eqref{magn} implies that, for all $f\in C(K)$ and $\mu\in M(K)$,
\begin{equation*}
  \int_K f\dd\mu = \lim_{\ntoo} \sum_{i=1}^{N_i} f(t_{ni}) \int_K f_{ni}\dd\mu,
\end{equation*}
where 
the right hand side evidently is $(\cC\times\cM)$-\meas{}.

This proves the first claim in \ref{lckjointm}, and the final claim follows
by the Riesz representation theorem  discussed before the lemma.
\end{proof}

\begin{corollary}\label{CCKcC}
  If $K$ is a metrizable compact, then $\cC$ coincides with the $\gs$-field
  $\bw$ 
on $C(K)$
generated   by
the continuous linear functionals, and also with the Borel \gsf{} $\cB$. 
Hence, if $X:\gO\to C(K)$, then $X$ is \Bmeas{} $\iff$
$X$ is \wmeas{} $\iff$
$X$ is \cmeas.
\end{corollary}
\begin{proof}
  By \refL{LCKjoint}, each $f\mapsto\innprod{\xx,f}$ with $\xx\in C(K)\q$
is $\cC$-\meas. 
Conversely, each point evaluation $f(t)=\innprod{\gd_t,f}$ where $\gd_t\in
C(K)\q$. This proves that $\cC$ and $\bw$ coincide.

Furthermore, $\bw=\cB$ in any separable Banach space, see \refS{Smeas}, so
\refT{TCKsep} completes the proof.
\end{proof}

If $K$ is not metrizable, \ie, when $C(K)$ is not
separable, we cannot expect equivalence with Bochner measurability.
Moreover,
Examples \refand{ECK2--+}{EA} show that
then in general $\bw\neq\cC$; hence
we cannot expect equivalence between weak measurability and
$\cC$-measurability.
Moreover, \refE{ECK2--+} constructs a \rv{} that is \cmeas{} but not \wmeas.
(Recall that  we only consider complete \ps{s}.)
Nevertheless, we shall see in Theorems \refand{Tc0}{TDw} 
(using Examples \refand{Eloccomp}{ECdoi}) that 
there exist important cases of non-separable $C(K)$ such that
$\bw=\cC$ and thus the last
equivalence holds.

We can now complete the proof of \refT{TC} in the introduction.
\begin{proof}[Proof of \refT{TC}]
By Theorems \refand{TCKsep}{TCKap}, $C(K)$ is separable and 
has the \ap. 
Hence, \refT{Tapprox} applies and shows (together with its proof, or
\refC{Citpk=p}) that 
\eqref{multi=}$\iff$\eqref{weak=}$\iff \E X\itpk=\E Y\itpk$,
where the moments exist in Bochner sense by \refT{TPIB}.
Finally, 
$\E X\itpk=\E Y\itpk \iff$ \eqref{punkt=} by \eqref{eckt}.
\end{proof}

We continue with further results and next give a complete characterization
of the existence of injective moments in the different senses in the
metrizable case. 

\begin{theorem}\label{TCK1}
  Let $K$ be a metrizable compact space and suppose that $X$ is a $\cC$-\meas{}
  $C(K)$-valued \rv. 
Let $k\ge1$.
  \begin{romenumerate}[-10pt]
  \item \label{tck1d}
$\E X\itpk$ exists in Dunford sense $\iff$ 
the weak \kth{} moment  exists  $\iff$ 
$\sup_{t\in	K}\E|X(t)|^k<\infty$.
  \item \label{tck1p}
$\E X\itpk$ exists in Pettis sense $\iff$ 
the family
$\set{|X(t)|^k:t\in K}$ of \rv{s} is \ui.
  \item \label{tck1b}
$\E X\itpk$ exists in Bochner sense $\iff$ 
$\E\bigpar{\sup_{t\in	K}|X(t)|}^k<\infty$.
  \end{romenumerate}
\end{theorem}
\begin{proof}
By \refC{CCKcC}, $X$ is \wmeas.  

\pfitemref{tck1d}
Since $C(K)$ is separable by \refT{TCKsep}, \refT{TID}\ref{tidd}\ref{tid1}
shows that $\E X\itpk$ exists in Dunford sense if and only if 
the weak \kth{} moment exists.
Moreover, in this case
$\E|\xx(X)|^k\le C$ for some $C<\infty$ and all $\xx\in C(K)\q$ with
$\norm{\xx}\le1$. 
In particular, taking $\xx=\gd_t$, $\E|X(t)|^k\le C$.

Conversely, if $\E|X(t)|^k\le C$ for all $t\in K$, and $\mu\in
C(K)\q\subseteq M(K)$ with $\norm\mu=1$, then, by \Holder's inequality and
Fubini's theorem  
(using \refL{LCKjoint}\ref{lckjointk}),
\begin{equation}\label{rogat}
  \begin{split}
\E|\innprod{\mu,X}|^k
&=
  \E\lrabs{\int_K X(t)\dd\mu(t)}^k
\le \E \int_K|X(t)|^k\dd|\mu|(t)
=  \int_K\E|X(t)|^k\dd|\mu|(t)
\\&
\le C.	
  \end{split}
\raisetag{\baselineskip}
\end{equation}
Hence \refL{Lweak} shows that the weak \kth{} moment exists.

\pfitemref{tck1p}
Since $C(K)\itpk=C(K^k)$, its dual space is $C(K^k)\q=\mr(K^k)\subseteq M(K^k)$.
If $\mu\in M(K^k)$ with $\norm\mu\le1$ and $E$ is any event, then by the
arithmetic-geometric inequality and Fubini (using
\refL{LCKjoint}\ref{lckjointk})
\begin{equation}\label{krhmf}
  \begin{split}
\E\bigpar{\ettae|\innprod{\mu,X\tpk}|}
&=	\E\lrpar{\ettae\lrabs{\int_{K^k}X(t_1)\dotsm
	X(t_k)\dd\mu(t_1,\dots,t_k)}}
\\
&\le	\E\lrpar{\ettae\int_{K^k}\frac1k\bigpar{|X(t_1)|^k+\dots+|X(t_k)|^k}
\dd|\mu|(t_1,\dots,t_k)}
\\
&=\frac1k\int_{K^k}\sum_{i=1}^k\E\bigpar{\ettae|X(t_i)|^k}
 \dd|\mu|(t_1,\dots,t_k)
\\
&\le\sup_{t\in K}\E\bigpar{\ettae|X(t)|^k}.
  \end{split}
\raisetag{1.1\baselineskip}
\end{equation}
It follows from \eqref{krhmf}, using \eqref{ui}, 
that if $\set{|X(t)|^k:t\in K}$ is \ui, then so
  is the family $\set{\innprod{\mu,X\tpk}:\norm{\mu}\le1}$.
Since $C(K)\itpk$ is separable when $C(K)$ is, it follows by \refT{THuff}
that the moment $\E X\itpk$ exists in Pettis sense.

Conversely, if $\E X\itpk$ exists in Pettis sense,
then $\set{|X(t)|^k:t\in K}$ is \ui{} by \refT{TiPettis}\ref{tipettis1}, taking
$\xx=\gd_t$, $t\in K$.
 
\pfitemref{tck1b}
Immediate by  \refT{TPIB}, since $C(K)$ is separable. 
\end{proof}

  Already the case $k=1$ in \refT{TCK1} is non-trivial and gives the
  following characterisations of the existence of the expectation $\E X$ of a
$\ck$-valued \rv.
 \begin{corollary}
   \label{CCK11}
Let $K$ be a metrizable compact space and suppose that $X$ is a $\cC$-\meas{}
  $C(K)$-valued \rv.
  \begin{romenumerate}[-10pt]
  \item \label{tck11d}
$\E X$ exists in Dunford sense $\iff$ 
$\sup_{t\in	K}\E|X(t)|<\infty$.
  \item \label{tck11p}
$\E X$ exists in Pettis sense $\iff$ 
the family
$\set{X(t):t\in K}$ of \rv{s} is \ui.
  \item \label{tck11b}
$\E X$ exists in Bochner sense $\iff$ 
$\E\bigpar{\sup_{t\in	K}|X(t)|}<\infty$.
  \end{romenumerate}
In the Pettis and Bochner cases, 
$\E X$ is the continuous function 
$t\mapsto\E (X(t))$. 
Also when $\E X$ exists just in Dunford sense, it is  
given by this function (bounded by not necessarily continuous) in the sense
that for any $\mu\in \ck\q=M(K)$,
\begin{equation}
  \label{venus}
\innprod{\mu,\E X}=\int_K \E (X(t))\dd\mu(t).
\end{equation}
 \end{corollary}
 \begin{proof}
It remains only to verify \eqref{venus}. If $X$ is Pettis or Bochner 
integrable, \eqref{venus} follows  by \refT{TCK0}. In the more general
Dunford case we 
have by Fubini, using \refL{LCKjoint},
\begin{equation}\label{ven}
  \innprod{\mu,\E X}=
\E\innprod{\mu,X}=
\E\int_K X(t,\go)\dd\mu(t)=
\int_K \E (X(t,\go))\dd\mu(t).
\end{equation}
 \end{proof}

 \begin{remark}
   \label{Rvenus} 
For non-metrizable $K$, \eqref{venus} is not true in general for Dunford
integrable $X$, as seen in
\refE{ECK2--}; cf.\ \refR{Rmerc}.
(The reason that the proof above fails in this case must be that we cannot
use Fubini in \eqref{ven},  because $X(t,\go)$ is not jointly
\meas.) 
 \end{remark}

We know that if the moment exists in Pettis sense, it is given by the
function \eqref{eckt} which then has to be continuous.
The next theorem shows that
in the separable case (\ie, when $K$ is metrizable compact),
it is for the existence of an \emph{even} injective moment in Pettis sense also
sufficient that this function exists and is continuous. 
This is not true for odd moments, not even the first moment $\E X$, as is
seen in  \refE{ECK2-} 
below; nor does this hold in general for non-metric $K$
as is seen in \refE{ECK2--}.

\begin{theorem}\label{TCK2}
  Let $K$ be a metrizable compact space and suppose that $X$ is a $\cC$-\meas{}
  $C(K)$-valued \rv{} such that $\sup_{t\in	K}\E|X(t)|^2<\infty$.
Suppose that $k\ge2$ is even.
Then the following are equivalent.
  \begin{romenumerate}[-10pt]
  \item \label{tck2p}
$\E X\itpk$ exists in Pettis sense.
  \item \label{tck2k}
$(t_1,\dots,t_k)\mapsto \E\bigpar{X(t_1)\dotsm X(t_k)}$
is continuous on $K^k$.
  \item \label{tck21}
$t\mapsto \E{X(t)^k}$ is continuous on $K$.
  \end{romenumerate}
In this case, $\E X\itpk$ is the function in $C(K^k)$ given in \ref{tck2k}.
\end{theorem}

\begin{proof}
  \ref{tck2p}$\implies$\ref{tck2k}: By \refT{TCK0},
which also shows the final statement.

  \ref{tck2k}$\implies$\ref{tck21}: Trivial.

  \ref{tck21}$\implies$\ref{tck2p}:
If $t_n\to t$ in $K$, then $X(t_n)\to X(t)$ and thus
$X(t_n)^k\to X(t)^k$. Furthermore, by \ref{tck21}, and the fact that $k$ is
even, 
$ \E |X(t_n)|^k \to \E |X(t)|^k$. Hence
$\set{|X(t_n)|^k}$ is \ui{} and thus $|X(t_n)|^k \to|X(t)|^k$ in $L^1(\P)$,
see \eg{} \cite[Theorem  5.5.2]{Gut}.
Consequently, the map $t\mapsto |X(t)|^k$ is continuous $K\to L^1(\P)$, and
since $K$ is compact, $\set{|X(t)|^k:t\in K}$ is a compact subset of
$L^1(\P)$, and in particular weakly compact and thus \ui{}
\cite[Theorem IV.8.11]{Dunford-Schwartz}.
Thus \ref{tck2p} follows by \refT{TCK1}\ref{tck1p}.
\end{proof}

We turn to projective moments.
For the second moment, we can show that the conditions for the injective
moment in \refT{TCK1} also imply the existence of the projective second
moment.
This uses the following result by Grothendieck \cite{Grothendieck-resume}
see \eg{} \cite[Theorem 5.5]{Pisier} or \cite[Theorem V.2]{Blei}.
(These references also contain further related results.
In particular, Grothendieck's theorem is
essentially equivalent to  \emph{Grothendieck's inequality}, see
\cite{Blei,Pisier,LTzI}.)

Let $k_G$ denote \emph{Grothendieck's constant}. 
(It is known that 
$\xfrac\pi2\le k_G\le\pi/(2\log(1+\sqrt2))$, but for us the value is not
important. )
\begin{theorem}[Grothendieck] \label{TG}
If $K$ is a compact set and $\ga$ is a bounded bilinear form on $C(K)$,
then there 
exists a Borel probability measure $\mu$ on $K$ such that 
\begin{equation}\label{groth}
  |\ga(f,g)|\le \kkg\norm{\ga}\norm{f}_{L^2(K,\mu)}\norm{g}_{L^2(K,\mu)}.
\end{equation}
Hence, $\ga$ extends to a bounded bilinear form on $L^2(K,\mu)$ of
norm $\le \kkg\norm\ga$.
\nopf
\end{theorem}

\begin{remark}
The standard formulation is for a bilinear form $\ga$ on $C(K_1)\times C(K_2)$
  for two compact sets $K_1$ and $K_2$; then there exist probability
  measures $\mu_1$ and $\mu_2$ on $K_1$ and $K_2$ such that
\begin{equation}
  |\ga(f,g)|\le k_G\norm\ga \norm{f}_{L^2(K_1,\mu_1)}\norm{g}_{L^2(K_2,\mu_2)}.
\end{equation}
We are only interested in the special case $K_1=K_2$, and we may then
replace $\mu_1$ and $\mu_2$ by $\mu\=\frac12(\mu_1+\mu_2)$ to obtain
\eqref{groth}. 
(An inspection of the proof in \eg{} 
\cite{Grothendieck-resume} or
\cite{Blei}
shows that we may take $k_G$ in the version \eqref{groth} too provided $\ga$
is symmetric, but this is not enough in general; consider for example
$\ga(f,g)\=f(0)g(1)$ on $\coi$.
Recall that we only have to consider symmetric $\ga$ for our purposes, \cf{}
\refR{Rsymmetric}.) 
\end{remark}

This leads to the following improvement of \refT{TCK1} when $k=2$. 
This does not  extend to $k\ge3$ by \refE{Ec3}.

\begin{theorem}\label{TCKG}
  Let $K$ be a metrizable compact space and suppose that $X$ is a $\cC$-\meas{}
  $C(K)$-valued \rv. 
  \begin{romenumerate}[-10pt]
  \item \label{tckgd}
$\E X\ptpx2$ exists in Dunford sense 
$\iff$
$\E X\itpx2$ exists in Dunford sense 
$\iff$
the weak second moment  exists  $\iff$
$\sup_{t\in	K}\E|X(t)|^2<\infty$.
  \item \label{tckgp}
$\E X\ptpx2$ exists in Pettis sense $\iff$ 
$\E X\itpx2$ exists in Pettis sense $\iff$ 
the family
$\set{|X(t)|^2:t\in K}$ of \rv{s} is \ui.
  \item \label{tckgb}
$\E X\ptpx2$ exists in Bochner sense 
$\iff$ 
$\E X\itpx2$ exists in Bochner sense 
$\iff$ 
$\E\bigpar{\sup_{t\in	K}|X(t)|}^2<\infty$.
  \end{romenumerate}
\end{theorem}

\begin{proof} 
The forward implications are immediate, using
Theorems \ref{TPI}, \ref{TID}, \ref{TiPettis}, \ref{TPIB}
and \refL{Lweak}, see the proof of \refT{TCK1}.
It remains to show the converses.
Note that $X(t,\go)$ is jointly measurable by \refL{LCKjoint}.

\pfitemref{tckgd}
Suppose that $\E|X(t)|^2\le C$ for every $t\in K$.
Let $\ga$ be a bounded bilinear form on $C(K)$. By \refT{TG}, $\ga$ extends
to a bounded bilinear form on $L^2(K,\mu)$ for some probability measure
$\mu$ on $K$.
Since $C(K)\subseteq L^2(K,\mu)$, we can regard $X$ as an
$L^2(K,\mu)$-valued \rv. 
Since $X(t,\go)$ is jointly measurable,
it follows, see \refL{LLp}, that $X$ is
\Bmeas{} in $L^2(K,\mu)$.
Moreover, by Fubini,
\begin{equation}\label{abarn}
  \E\norm{X}_{L^2(\mu)}^2 
=\E\int_K|X(t,\go)|^2\dd\mu(t)
=\int_K\E|X(t)|^2\dd\mu(t)
\le C
<\infty.
\end{equation}
Hence, \refT{TPIB} shows that $\E X\ptpx2\in L^2(K,\mu)\ptpx2$ exists (in
Bochner sense). 
In particular, since $\ga$ extends to $L^2(K,\mu)$,
$\ga(X,X)$ is measurable by \refL{Lbk} and 
\begin{equation*}
\E\ga(X,X)
=\E\innprod{\ga, X\ptpx2}  
=\innprod{\ga,\E X\ptpx2}  
\end{equation*}
exists. Since $\ga$ is an arbitrary bounded bilinear form on $C(K)$,
\refT{TPD} now shows that $\E X\ptpx2$ exists in Dunford sense in
$\ck\ptpx2$ too.

\pfitemref{tckgp}
Assume that the family
$\set{|X(t)|^2:t\in K}$ is \ui.
Let $\ga$ be a bounded bilinear form on $C(K)$ with $\norm\ga\le1$
and let as above
$\mu$ be as in \refT{TG}. If $E\in\cF$ is any measurable set, then
by \eqref{groth} and Fubini,
\begin{equation*}
  \begin{split}
\E  |\ettae\cdot \ga(X,X)|
&\le \E\Bigpar{\ettae \kkg\int_K|X(t,\go)|^2\dd\mu(t) }
\\&
= \kkg\int_K \E\bigpar{\ettae |X(t,\go)|^2}\dd\mu(t) 	
\le  \kkg\sup_{t\in K} \E\bigpar{\ettae |X(t)|^2}
  \end{split}
\end{equation*}
It follows, using \eqref{ui}, that the family
$\bigset{\ga(X,X):\ga\in L\xpar{C(K)^2;\bbR},\,\norm\ga\le1}$ is \ui.
Moreover, $X\ptpx2$ is trivially \wassep{}  since $C(K)\ptpx2$ is separable.
Hence \refT{THuff} shows, using \refT{Tproj*}, that $\E X\ptpx2$ exists in
Pettis sense.

\pfitemref{tckgb}
This is (again) a special case of \refT{TPIB}, included here for completeness.
\end{proof}

As an immediate corollary, we can weaken the integrability condition in 
\refT{TC} in the case $k=2$ to
 $\set{|X(t)|^2:t\in K}$ and \set{|Y(t)|^2:t\in K} being \ui, since it is
enough for the proofs above of Theorems \refand{Tapprox}{TC}
that the moments $\E X\ptpx2$ and $\E Y\ptpx2$ exist in Pettis sense.
By the methods in the proof of \refT{TCK2}, we can weaken the condition
further; the following theorem shows that in this case, it suffices that the
moments exist in Dunford sense.

\begin{theorem}
  \label{TCKD2} 
  Let $K$ be a metrizable compact space and suppose that $X$ 
and\/ $Y$ are two \cmeas{}   $C(K)$-valued 
\rv{s} such that
 $\sup_{t\in K}|X(t)|^2<\infty$ 
and  $\sup_{t\in K}|Y(t)|^2<\infty$. 
Then, for $k=2$, the following are equivalent.
\begin{romenumerate}[-10pt]
\item \label{tckd2m}
\eqref{multi=}, \ie{},
$\E\ga(X,X)=\E\ga(Y,Y)$ for every  bounded bilinear form 
$\ga$ on $\ck$. 
\item \label{tckd2w}
 \eqref{weak=}, \ie,
$\E \bigpar{\xx_1(X)\xx_2(X)}=\E \bigpar{\xx_1(Y)\xx_2(Y)}$,
for any $\xx_1,\xx_2\in \ck^*$.
\item \label{tckd2t}
 \eqref{punkt=}, \ie,
 $ \E \bigpar{X(t_1) X(t_2)}
=
 \E \bigpar{Y(t_1)Y(t_2)}
$
for any $t_1,t_2\in K$.
\item \label{tckd2p}
$\E X\ptpx2=\E Y\ptpx2$ in $\ck\ptpx2$,
with the moments existing in Dunford sense.
\item \label{tckd2i}
$\E X\itpx2=\E Y\itpx2$ in $\ck\itpx2$,
with the moments existing in Dunford sense.
\end{romenumerate}
\end{theorem}

\begin{proof}
The implications \ref{tckd2m}$\implies$\ref{tckd2w}$\implies$\ref{tckd2t}
are  trivial, and  the equivalences
\ref{tckd2m}$\iff$\ref{tckd2p} and
\ref{tckd2w}$\iff$\ref{tckd2i} are 
Corollaries \ref{Cptpk=} and \ref{Citpk=d}\ref{cit1}. 
It remains to show that \ref{tckd2t}$\implies$\ref{tckd2m}.

Thus, let $\ga\in L(\ck^2;\bbR)= \bigpar{\ck\ptpx2}\q$
As in the proof of \refT{TCKG},
there exists by \refT{TG}
a \pmx{} $\mu$ on $K$ such that $\ga$ extends to
$L^2(K,\mu)$ and $\E X\ptpx2$ exists in $L^2(K,\mu)\ptpx2$ in Bochner sense;
similarly $\E Y\ptpx2$ exists in $L^2(K,\mu)\ptpx2$ in Bochner sense.

As a consequence,  the injective moments $\E X\itpx2$ and 
$\E Y\itpx2$ exist in $L^2(K,\mu)\itpx2$ in Bochner sense.
Let $g,h\in L^2(K,\mu)\q=L^2(K,\mu)$. Then, by Fubini's theorem, with
joint measurability by \refL{LCKjoint},
\begin{equation*}
  \begin{split}
\E\bigpar{\innprod{X,g}\innprod{X,h}}	
&=
\E\iint X(t,\go)g(t)X(u,\go)h(u)\dd\mu(t)\dd\mu(u)
\\&
=\iint \E\bigpar{X(t)X(u)}g(t)h(u)\dd\mu(t)\dd\mu(u)
  \end{split}
\end{equation*}
The same applies to $Y$. Consequently, if \ref{tckd2t} holds, then
$\E\bigpar{\innprod{X,g}\innprod{X,h}}=\E\bigpar{\innprod{Y,g}\innprod{Y,h}}$,
which is \eqref{weak=} for $L^2(K,\mu)$ (with $k=2$). Hence
\refC{Citpk=p} yields 
$\E X\itpx2=\E Y\itpx2$ in $L^2(K,\mu)\itpx2$.
Furthermore, $L^2(K,\mu)$ has the \ap{} by \refT{THap}, and thus
$\E X\ptpx2=\E Y\ptpx2$ in $L^2(K,\mu)\ptpx2$ by \refT{Tiotak}.
Consequently,
\begin{equation*}
  \E\ga(X,X)=\innprod{\ga,\E X\ptpx2}
=
\innprod{\ga,\E Y\ptpx2}
=\E\ga(Y,Y).
\end{equation*}
Since $\ga$ is an arbitrary bounded bilinear form on $\ck$, this completes
the proof. 
\end{proof}

\begin{remark}\label{RCKD2}
 \refT{TCKD2} shows that, when $K$ is metrizable, the moment
$\E X\ptpx2$ (or $\E X\itpx2$) is determined by the moment function
$\E(X(t)X(u))$ as soon as the moment exists in Dunford sense.
The corresponding result for $k=1$ is in \refC{CCK11}.

\refE{ECK2--} yields an example of a non-metrizable $K$ such that
\ref{tckd2t} does not imply \ref{tckd2w}.

We leave it as an open problem whether
\refT{TCKD2} extends to $k\ge3$. In particular,
for metrizable $K$ and $k\ge3$, 
if $\E X\itpk$  or $\E X\ptpk$ exists in Dunford sense,
does the moment function \eqref{eckt}
determine this moment uniquely?
\end{remark}

For projective moments of order $k\ge3$, we do not know any special results
for Pettis or Dunford integrability, but we have as always a simple result
for Bochner integrability. (Recall that this implies Pettis and Dunford
integrability, so we have a sufficient condition for them too.)

\begin{theorem}\label{TCK1k}
  Let $K$ be a metrizable compact space and suppose that $X$ is a $\cC$-\meas{}
  $C(K)$-valued \rv. 
Then
$\E X\ptpk$ exists in Bochner sense 
$\iff$ 
$\E X\itpk$ exists in Bochner sense 
$\iff$ 
$\E\bigpar{\sup_{t\in	K}|X(t)|}^k<\infty$.
\end{theorem}
\begin{proof}
  \refT{TPIB}.
\end{proof}

We end by a few counterexamples, partly taken or adapted from 
\cite{EdgarII} and 
\cite{Talagrand:Pettis} where further related examples are given. 
See also Examples \ref{Ec0}, \ref{Edoi} and \ref{Eloo}, which by Examples
\ref{Egb}--\ref{Eloccomp} can be seen as examples in some $\ck$.

\begin{example}\label{ECK2-}
Let $X$ be a $C(K)$-valued \rv{} such that $\E X$ exists in Dunford
sense but not in Pettis sense, and let $Y\=\eta X$, where $\eta=\pm1$ 
with the sign uniformly random and independent of $X$. 
Then $Y$ 
is Dunford integrable and $\E Y(t)=\E(\eta X(t))=0$ for every
$t\in K$, 
so $\E Y(\cdot)\in\ck$
but $Y$ is not Pettis integrable (since otherwise $X=\eta Y$ would be too).
We may for example take $X$ as in \refE{Ec0} ($X\in c_0\subset c=C(\bbNx)$,
so $K=\bbNx$ is metric) or as in \refE{Eloo} ($X$ is bounded).
(We cannot take both $K$ metric and $X$ bounded by \refT{TCK1}.)
\end{example}

\begin{example}[cf.\ {\cite[Section V.5]{Blei}}]  \label{Ec3} 
Let $\bbT\=[0,2\pi]$ (perhaps regarded as the unit circle)
and $\bbNx\=\bbN\cup\set\infty$ (the usual one-point compactification as in
\refE{Eloccomp}), and take $K\=\bbT\cup\bbNx$, where we regard $\bbT$ and
$\bbNx$ as disjoint. 
If $f\in\ck$, define $\hat f(n)\=\frac1{2\pi}\int_0^{2\pi}f(t)e^{-\ii nt}\dd
t$, \ie, the Fourier coefficients of $f|_{\bbT}$.

Define the trilinear form $\ga$ on $\ck$ by
\begin{equation}
  \label{ec3}
\ga(f,g,h)\=\Re\sumn \hat f(n)\hat g(n) h(n).
\end{equation}
By \Holder's inequality and Parseval's identity,
\begin{equation*}
\sumn\bigabs{ \hat f(n)\hat g(n) h(n)}
\le\norm{h} 
\biggpar{\int_{\bbT}|f(t)|^2\frac{\dd t}{2\pi}}\qq
\biggpar{\int_{\bbT}|g(t)|^2\frac{\dd t}{2\pi}}\qq
\le\norm{f}\norm{g}\norm{h};
\end{equation*}
thus the sum in \eqref{ec3} converges and $\ga$ is well-defined and bounded.

Let $N$ be an $\bbN$-valued \rv{} with $\P(N=n)=p_n$, let $a_n$ be some
positive numbers and define
the \rv{} $X\in\ck$ by
$X(t)\=\cos(Nt)$, $t\in\bbT$, and $X(n)\=a_n e_N(n)$, $n\in\bbNx$.
Then $\hat X(N)=\frac12$ and 
\begin{equation*}
  \ga(X,X,X)=\hat X(N)^2a_N=\tfrac14 a_N.
\end{equation*}
Hence,
\begin{equation}\label{e3}
\E |\ga(X,X,X)|=\tfrac14 \E a_N
=\tfrac14\sumn p_na_n.
\end{equation}
Furthermore,
$|X(t)|\le1$ for $t\in\bbT$ and $\E|X(n)|^3= p_na_n^3$ for $n\in\bbN$ (and
$0$ for $n=\infty$).

Choose, \eg, $p_n\=n^{-4/3}$ (for $n$ large) and $a_n\=n^{1/3}$.
Then $\E|X(n)|^3\to0$ as \ntoo, and thus \set{|X(x)|^3:x\in K}  is \ui.
Hence  $\E X\itpx3$ exists in Pettis sense by \refT{TCK1}.
On the other hand,
$\E|\ga(X,X,X)|=\infty$ by \eqref{e3}, and thus $\E X\ptpx3$ does not
exist (even in Dunford sense).
This shows that \refT{TCKG}\ref{tckgd}\ref{tckgp} do not hold for $k=3$.

The example may be modified for any given $k\ge3$ by taking
\begin{equation}
  \label{ec3k}
\ga(f_1,\dots,f_k)\=\Re\sumn \hat f_1(n)\hat f_2(n) f_3(n)\dotsm f_k(n),
\end{equation}
and $p_n=n^{-(k+1)/3}$.
\end{example}

\begin{example}[{cf.\ \cite[Example (2)]{Edgar} and 
 \cite[Example 5.5]{EdgarII}}] 
\label{Egoi}
  Let $\go_1$ be the first uncountable ordinal and let $K=[0,\go_1]$, the
  set of all ordinals $\le\go_1$ with the order topology. It is easily seen
  that $K$ is compact, 
Note that $K\setminus\set{\go_1}=[0,\go_1)$ is the (uncountable) set of all
  countable ordinals.

It is not difficult to see that every $f\in C\ogoi$ is constant on
$[\ga,\go_1]$ for some $\ga<\go_1$. Hence, the Baire $\gs$-field $\Ba$  on
$\ogoi$ is the $\gs$-field consisting of all subsets $A$ such that either
$A$ or its complement $A\comp$ is countable. 

Furthermore, every bounded increasing function
$f:\ogoi\to\bbR$ is  constant on
$[\ga,\go_1)$ for some $\ga<\go_1$. It follows (by considering the positive
and negative parts) that every 
regular signed 
Borel (or Baire)
measure on $\ogoi$ is supported on some
set $[0,\ga]\cup\set{\go_1}$; hence every regular signed
Borel measure has countable
support and is thus discrete. Consequently, $C\ogoi\q=\ell^1\ogoi$.

We define a probability measure
$\P$ on 
$(\ogoi,\Ba)$ by setting $\P(A)=0$ if $A$ is countable and $\P(A)=1$ if
$A\comp$ is countable. The mapping $X:\ga\mapsto X_\ga\=\etta_{[\ga,\go_1]}$
maps 
$(\gO,\Ba,\P)\to C\ogoi$. If $\gb<\go_1$, then $X_\ga(\gb)=\ett{\ga\le \gb}=0$
for \aex{} $\ga$, while $X_\ga(\go_1)=1$ for all $\ga$. In other words,
$X(\gb)=0$ \as{} if $\gb<\go_1$ but $X(\go_1)=1$.

We have seen that any $\xx\in C\ogoi\q$ is given by 
\begin{equation}\label{cogoi*}
\xx(f)=\sum_\ga \xi(\ga)f(\ga)   
\end{equation}
for some $\xi\in \ell^1\ogoi$. 
(Note that the sum really is countable.)
It follows that in this case,
\begin{equation}\label{cogoi2}
  \xx(X)=\xi(\goi)
\qquad\text{a.s.} 
\end{equation}
Consequently, $X$ is \wmeas, and since $X$ further
is bounded, the Dunford integral $\E X$ exists. We have,
by \eqref{cogoi2},
\begin{equation}
\innprod{\E  X,\xx}=\E\innprod{\xx,X}=\xi(\goi)  ,
\end{equation}
when $\xx$ is given by \eqref{cogoi*}. Consequently, $\E X$ is given by the
function $\etta_{\set{\goi}}$ that is 1 at $\goi$ and $0$ on
$[0,\goi)$. This function is not continuous, and thus $\E X\notin C\ogoi$,
  which shows that $X$ is not Pettis integrable.

By \refT{THuff}, $X$ is not \wassep. (This is also easily seen directly. If
$M$ is any separable subspace of $C\ogoi$, then there exists an $\ga<\goi$
such that every function in $M$ is constant on $[\ga,\goi]$. Hence, $X\notin
M$ a.s.)
\end{example}

Note that \refE{Egoi} gives a uniformly bounded continuous random function
$X(t)$ such that $\E X(t)$ is not continuous. This cannot happen on a metric
space $K$, since dominated convergence shows that $\E X(t)$ always is
sequentially continuous. (Indeed, if $X$ is a uniformly bounded \rv{} in
$C(K)$ with $K$ compact metric, then \refT{TCK1} applies and shows that all
injective moments exist in both Bochner and Pettis sense.)   

\begin{example}[\cite{FremlinT} and \cite{Talagrand:Pettis}]\label{Eloo2}
Recall that $\ell^\infty=C_b(\bbN)\cong C(\gb\bbN)$ where $\gb\bbN$ is the
\SCc{} of $\bbN$, see \refE{Egb}.
Furthermore, $\gbn$ can be regarded as the 
subset of the unit ball of $\looq$
consisting of all multiplicative
linear functionals on $\ell^\infty$. (This holds also as topological spaces,
with the  weak$^*$ topology on $\looq$.) A point $n\in\bbN$ then is
identified with the multiplicative linear functional $\gd_n:f\mapsto f(n)$.

 Let $X\qwe1$ and $X\qwe2$ be two independent copies of the random variable 
$X\in \ell^\infty=C(\gb\bbN)$ constructed in \refE{Eloo}, and let
 $Y=X\qwe1-X\qwe2$. 
We follow \citet{FremlinT} and \citetq{Theorem 4-2-5}{Talagrand:Pettis} to
show that $Y$ is Pettis integrable. (Recall from \refE{Eloo} that $X$ is not.)

The coordinates $Y_n$ are \iid{} \rv{s}, each with the distribution 
of $\eta_1-\eta_2$ with independent $\eta_1,\eta_2\sim \Be(1/2)$; hence $Y_n$
has the centred binomial distribution $\Bin(2,1/2)-1$. It follows that
$\sqrt2 Y_n$ is an orthonormal sequence in $L^2(\P)$. 
Consequently, by Bessel's inequality,  
if $\gx\in L^\infty(\P)\subset L^2(\P)$, then the sequence 
$y_\gx\=(\E(\gx Y_n))_1^\infty\in \ell^2\subset c_0$.

If $\xx=(a_n)_1^\infty\in \ell^1$, then by Fubini,
\begin{equation}\label{pettson}
  \innprod{\xx,y_\gx}
=\sumn a_n\E(\gx Y_n)
=\E\sumn a_n\gx Y_n
=\E(\gx\innprod{\xx,Y}).
\end{equation}
If $\xx\in c_0^\perp$, then $\xx(X\qwe1)$ and $\xx(X\qwe2)$ are 
by \refE{Eloo} equal to some
constant \as, and thus $\xx(Y)=\xx(X\qwe1)-\xx(X\qwe2)=0$ a.s.
Furthermore, $\innprod{\xx,y_\gx}=0$ since $y_\gx\in c_0$, and thus
\eqref{pettson} 
holds in this case too. Hence, \eqref{pettson} holds for all $\xx\in\looq$, 
and every $\gx\in L^\infty(\P)$, which shows that $Y$ is Pettis integrable
and $\E(\gx Y)=y_\gx$.

Consider now the injective second moment $\E Y\itpx2$.
Assume that this exists as a
Pettis integral; it then belongs to $C(\gbn)\itensor C(\gbn)=C((\gbn)^2)$.
Let us write $Q\=\E Y\itpx2\in C((\gbn)^2)$.
If $\xx\in\gbn\setminus\bbN$, then $\xx\in c_0^\perp$, and thus
$\xx(Y)=0$ \as; consequently,
\begin{equation}\label{loo2a}
Q(\xx,\xx)=  \innprod{\xx\tensor \xx,\E Y\itpx2}
=\E\innprod{\xx\tensor \xx, Y\tpx2}
=\E\innprod{\xx, Y}^2
=0.
\end{equation}
On the other hand, if $n\in\bbN$, then, similarly,
\begin{equation}\label{loo2b}
Q(n,n)=  \innprod{\gd_n\tensor \gd_n,\E Y\itpx2}
=\E\innprod{\gd_n\tensor \gd_n, Y\tpx2}
=\E\innprod{\gd_n, Y}^2
=\tfrac12,
\end{equation}
since $\innprod{\gd_n,Y}=Y_n\sim \Bin(2,\frac12)-1$.
However, $\bbN$ is dense in $\gbn$, so \eqref{loo2b} implies by continuity that
$Q(\xx,\xx)=\frac12$ for all $\xx\in\gbn$, which contradicts  \eqref{loo2a}. 
Consequently, the second moment
$\E Y\itpx2$ does \emph{not} exist as a Pettis integral.
(We do not know whether it exists as a Dunford integral.)
\end{example}

\begin{example}\label{ECK2--}
  Let $B\=\loo/c_0$. We identify $\loo$ and $C(\gbn)$, see \refE{Egb}; 
then $c_0=\set{f\in C(\gbn):f(x)=0 \text{ when }x\in\gbnx}$, and it follows
by the Tietze--Urysohn extension theorem  \cite[Theorem 2.1.8]{Engelking},
since $\bbN$ is open in $\gbn$,
that we can identify $B=\loo/c_0=C(\gbnx)$.
(This can also be seen from \refE{EC*}, since $\gbnx$ is the maximal ideal
space of the complex version of $\loo/c_0$.)

Let $\pi:\loo\to\loo/c_0$ be the quotient mapping.
(Thus $\pi$ is identified with the restriction mapping $C(\gbn)\to C(\gbnx)$.)
Let $X$ be the \wmeas{} $\loo$-valued \rv{} constructed in \refE{Eloo}, and
let $\hX=\pi(X)\in\loo/c_0=C(\gbnx)$.
Finally, let $Z\=1-\hX$. 
(Thus $0\le Z\le1$.)

Then $\hX$ and $Z$ are bounded and \wmeas{} $C(K)$-valued \rv, with $K=\gbnx$.
By \refE{Eloo}, if $\xx\in K=\gbnx$, then 
\begin{equation}
\label{jepp1}
\hX(\xx)=
X(\xx)=\innprod{\xx,X}=1 
\qquad\text{a.s.}
\end{equation}
Hence, 
\begin{equation}
\label{jeppz}
Z(\xx)=0
\qquad\text{a.s.}
\end{equation}
for every $\xx\in K$ and 
\begin{equation}
  \label{jepp}
\E\bigpar{Z(\xx_1)\dotsm Z(\xx_k)}
=0
\end{equation}
for any $\xx_1,\dots,\xx_k\in
K=\gbnx$, so the function defined in \eqref{eckt} exists and is continuous
on $K^k$ (in fact, constant 0).

However, let $\xx_n\in\looq$ be defined by
$\innprod{\xx_n,(a_i)_i}\=\frac1n\sum_{i=1}^n a_i$, 
and let $T:\loo\to\loo$ be the linear map $x\to(\xx_n(x))_n$. 
Choose any $\xx_0\in\gbnx$
and define $\txx\in\looq$ as
$T\q\xx_0$. \ie,
by 
\begin{equation*}
\innprod{\txx,x}
\=\innprod{\xx_0,Tx}
=\innprod{\xx_0,(\xx_n(x))_n},  
\qquad x\in\loo. 
\end{equation*}
If $x\in c_0$, then $\xx_n(x)\to0$ as \ntoo, and thus 
$Tx=(\xx_n(x))_n\in c_0$, so 
$\innprod{\txx,x}=\innprod{\xx_0,Tx}=0$; hence, $\txx\perp c_0$, and
$\txx\in (\loo/c_0)\q$. Furthermore, by the law of large numbers,
$\innprod{\xx_n,X}=\frac1n\sumin X_i\to\frac12$ a.s., and thus
\begin{equation}\label{jet}
\innprod{\txx,\hX}=\innprod{\txx,X}=\tfrac12 
\qquad\text{a.s.}
\end{equation}
and 
\begin{equation}\label{jetz}
\innprod{\txx,Z}
=\innprod{\txx,1-\hX}
=1-\innprod{\txx,\hX}=\tfrac12 
\qquad\text{a.s.}
\end{equation}

For $k=1$, $\E Z$ exists in Dunford sense, since $Z$ is bounded and \wmeas;
\eqref{jetz} implies that $\innprod{\txx,\E Z}=\E\innprod{\txx,Z}=\frac12$
and thus $\E Z\neq0$, although \eqref{jeppz} shows that 
$\innprod{\E Z,\gd_t}=\E(Z(t))=0$ for every $t\in K$.
In particular, it follows that $\E Z\in \ck\qx\setminus\ck$, and thus $Z$ is
not Pettis integrable.
(Cf.\ \refE{Eloo} which shows that $\E X$ does not exist in Pettis sense by
essentially the same argument.)
We see also that \eqref{venus} fails for $Z$ and $\mu=\txx$; hence, as said
in \refR{Rvenus}, $Z(t,\go)$ is not jointly \meas. 

For $k\ge2$, we do not know whether 
$\E Z\itpk$ or $\E Z\ptpk$ exists in Dunford sense, but they do not
exists in Pettis sense, by an extension of the argument for $\E Z$.
Indeed, if $\E Z\itpk$ exists in Pettis (or just Dunford) sense, then
by \eqref{jet}, with $\xx$ as above,
\begin{equation}
\innprod{(\txx)\tpk,  \E Z\itpk}=\E\innprod{\xx,Z}^k=\bigpar{\tfrac12}^k,
\end{equation}
so $\E Z\itpk\neq0$.
On the other hand, by \eqref{jepp},
\begin{equation}
  \label{jepp6}
\innprod{\gd_{t_1}\tensor\dotsm\tensor\gd_{t_k},\E Z\itpk}
=\E\bigpar{Z(t_1)\dotsm Z(t_k)}
=0
\end{equation}
for all $t_1,\dots,t_k\in K$.
Suppose that $\E Z\itpk$ exists in Pettis sense. Then, by
\refT{TCK0},
$\E Z\itpk$ is the function \eqref{eckt} on $K^k$, \ie, by \eqref{jepp6},
$\E Z\itpk=0$, 
a contradiction. Hence, $\E Z\itpk$ does not exist in Pettis sense.
This also shows that \refT{TCK2} does not hold for $K=\gbnx$.
\end{example}

\begin{example}\label{ECK2--+} 
Let $Z\in\loo/c_0$ and $\txx\in(\loo/c_0)\q$ be as in \refE{ECK2--}.
Let $\xi\sim\Be(1/2)$ be independent of $Z$, and define
$Z_1\=\xi Z\in \loo/c_0=C(\gbnx)$.

By \eqref{jeppz}, $Z_1(\xx)=0$ \as{} for every $\xx\in K=\gbnx$.
On the other hand, by \eqref{jetz}, 
\begin{equation}\label{erik}
\innprod{\txx,Z_1}=
\xi\innprod{\txx,Z}=
\tfrac12 \xi
\qquad\text{a.s.}  
\end{equation}

Define the measure $\mu$ on the \gsf{} $\cC$ as the distribution of $Z_1$, and
regard the \rv{} $Z_1$ as defined by the identity map $\ck\to\ck$ on the \ps{}
$(\ck,\cC\q,\mu)$, where $\cC\q$ is the completion of $\cC$.
(Recall that we want our \ps{} to be complete.)
This version of $Z_1$ is \cmeas. If $\gf(Z_1)$ is any measurable functional,
then $\gf$ is a $\cC\q$-measurable function on $\ck$, and thus $\gf$ is
$\mu$-\aex{} equal to a \cmeas{} function $\psi$ on $\ck$. By the definition
of $\cC$, this implies that $\psi(f)=\Psi\bigpar{f(t_1),f(t_2),\dots}$ for some
function $F$ on $\bbR^\infty$ and some sequence of points $t_i\in K$.
By \eqref{jeppz}, each $Z_1(t_i)$ is \as{} constant, and thus $\psi(Z_1)$ is
\as{} constant; since $\gf(Z_1)=\psi(Z_1)$ \as, also $\gf(Z_1)$ is \as{}
constant.
Since $\txx(Z_1)$ is not \as{} constant by \eqref{erik}, $\txx(Z_1)$ is not
measurable; hence $Z_1$ is \cmeas{} but not \wmeas.
It follows that $\cC\neq\bw$.
\end{example}

\begin{example}\label{EA}
For another (simpler) example with $\cC\neq\bw$, 
  let $K\=\setoi^{\cardc_1}$, where $\cardc_1=2^{\cardc}$ (or any cardinal
  number $>\cardc$), and let as in \refE{Eloo} $\mu$ be the product measure
$\mu\=\bigpar{\frac12\gd_0+\frac12\gd_1}^{\cardc_1}$.

Suppose that the linear functional $f\mapsto \chi(F)\=\int_K f\dd\mu$ is
\cmeas{}. Then there exist points $t_i\in K$, $i=1,2,\dots$, 
and a (\meas) function $\Phi:\bbR^\infty\to\bbR$ such that
\begin{equation}
\label{kc1}
\int_K f\dd\mu=\Phi(f(t_1),f(t_2),\dots),
\qquad f\in C(K).
\end{equation}

Each $t\in K$ is a function $\cardc_1\to\setoi$ which we denote by
$\ga\mapsto t(\ga)$.
Define an equivalence relation on $\cardc_1$ by $\gb\equiv\gam\iff
t_i(\gb)=t_i(\gam)$ for all $i$. The number of equivalence classes is
$2^{\aleph_0}=\cardc<\cardc_1$, and thus there exists $\gb,\gam\in\cardc_1$
with $\gb\neq\gam$ but $\gb\equiv\gam$.

Consider the normalized coordinate functions
$f_\ga(t)\=2t(\ga)-1:K\to\set{-1,1}$, $\ga\in\cardc_1$.
Since $\gb\equiv\gam$, $f_\gb(t_i)=f_\gam(t_i)$ for each $i$, and thus
$f_\gb(t_i)f_\gam(t_i)=1$. Consequently, if $g\=f_\ga f_\gam$ and $h\=1$,
then $g,h\in C(K)$ and $g(t_i)=h(t_i)$ for each $i$, so by \eqref{kc1}
\begin{equation*}
  \int_K g\dd\mu=\Phi(g(t_1),g(t_2),\dots)
=\Phi(h(t_1),h(t_2),\dots)
=\int_K h\dd\mu,
\end{equation*}
which is a contradiction since $\int g\dd\mu=0$ and $\int h\dd\mu=1$.

This contradiction shows that the continuous linear functional $\chi$ is not
\cmeas, and thus $\bw\supsetneq\cC$.
\end{example}

\section{$c_0(S)$}\label{Sc0} 

In this section we consider $B=c_0(S)$, for an arbitrary set $S$.
As discussed in \refE{Eloccomp}, we can regard $c_0(S)$ as a complemented
subspace of codimension 1 in
  $C(S\q)=C(S\cup\set\infty)$: 
$c_0(S)=\set{f\in C(S\q):f(\infty)=0}$.
(The results below could easily be formulated for $C(S\q)$ instead, but we
leave that to the reader.)

Note that $c_0(S)$ is separable (and $S\q$ metrizable) if and only if $S$ is
countable. (The discrete space $S$ is always metrizable, but that is not
enough.) 
The case when $S$ is countable is thus covered by the results (for $C(S\q)$)
in \refS{SCK}. We shall see that these results extend to arbitrary $S$
because of the special simple structure of $c_0(S)$.
This illustrates that some non-separable Banach spaces can be handled
without problems, and it is a background to Sections
\ref{SmeasD}--\ref{SDmom} where we (by technically more complicated
arguments) obtain similar results for $\doi$, which is more important for
applications. 

\begin{theorem}\label{Tc0tp}
  The injective \tp{} $\coss\itpk=\cossk$ (isometrically).
Moreover, $\coss$ has the \ap, and thus
$\coss\ptpk\subseteq\coss\itpk=\cossk$
(as a vector space).
\end{theorem}
\begin{proof}
  An easy consequence of Theorems \refand{TCKitpk}{TCKap}, applied to $\css$.
\end{proof}

The dual space $c_0(S)\q=\ell^1(S)$. Note that every element in $\ell^1(S)$
has countable support. Thus every $\xx\in c_0(S)\q$ depends only on countably
many coordinates.
This extends to multilinear forms as follows.
For a subset $A\subseteq S$, let $P_A$ be the projection in $c_0(S)$ defined
by 
\begin{equation}
  \label{pac}
P_A f(x)\=\ett{x\in A}f(x). 
\end{equation}

\begin{lemma}
  \label{LBH}
If $S$ is any set and 
$\ga$ is a
bounded $k$-linear form on $c_0(S)$, then there exist a countable subset
$A\subseteq S$ such that
\begin{equation}
  \label{lbh}
\ga(f_1,\dots,f_k)=\ga(P_{A}f_1,\dots,P_{A}f_k).
\end{equation}
\end{lemma}

\begin{proof}
Write, for convenience, 
\begin{equation}
  \label{nobh}
a(s_1,\dots,s_k)\=\ga(e_{s_1},\dots,e_{s_k}),  \qquad s_1,\dots,s_k\in S.
\end{equation}

Fix a finite set $F\subseteq S$ and let $X\in c_0(S)$ be random with $X(s)$,
$s\in F$, \iid{} with $\P(X(s)=+1)=\P(X(s)=-1)=\frac12$, while $X(s)=0$ for
$s\notin F$.

Let $X_1,\dots,X_k$ be independent copies of $X$.
Then
\begin{equation*}
  \ga(X_1,\dots,X_k)=\sum_{s_1,\dots,s_k\in F}a(s_1,\dots,s_k)X_1(s_1)\dotsm
  X_k(s_k) 
\end{equation*}
and thus, since different terms are orthogonal,
\begin{equation*}
\E| \ga(X_1,\dots,X_k)|^2=\sum_{s_1,\dots,s_k\in F}|a(s_1,\dots,s_k)|^2.
\end{equation*}
Hence,
\begin{equation*}
\sum_{s_1,\dots,s_k\in F}|a(s_1,\dots,s_k)|^2
\le \norm{\ga}^2.
\end{equation*}
Since this holds for every finite $F$,
\begin{equation}\label{bh}
\sum_{s_1,\dots,s_k\in S}|a(s_1,\dots,s_k)|^2
\le \norm{\ga}^2.
\end{equation}
In particular, only a countable number of 
$a(s_1,\dots,s_k)$ are non-zero.
Hence there exists a countable subset $A$ of $S$ such that
$a(s_1,\dots,s_k)=0$ unless $s_1,\dots,s_k\in A$.
Then \eqref{lbh} holds for every $f_1,\dots,f_k$ with finite supports, and
the general case follows by continuity.
\end{proof}

\begin{remark}
  \citet{BH} proved the stronger result
\begin{equation*}
\sum_{s_1,\dots,s_k\in S}|a(s_1,\dots,s_k)|^{2k/(k+1)}
<\infty,
\end{equation*}
where the case $k=1$ is just $c_0(S)\q=\ell^1(S)$ and
$k=2$ had been proved earlier by \citet{Littlewood};
see also \cite{Blei}. 
\end{remark}

The integral forms are, by definition, the elements of the dual of
$\coss\itpk$; by \refT{Tc0tp} this equals $\cossk\q=\ell^1(S^k)$.
Consequently, every integral $k$-linear form on $\coss$ is nuclear.

We let, as in \refS{SCK}, 
$\cC$ be the \gsf{} generated by the point evaluations.
Thus, a $c_0(S)$-valued \rv{}
 $X$ is \cmeas{} if and only if $X(s)$ is \meas{} for
every $s\in S$.

\begin{theorem}\label{Tc0}
The \gsf{s} $\cC$ and $\bw$ on $c_0(S)$  coincide,  for any $S$. 

Moreover, the following hold for any  $c_0(S)$-valued \rv{} $X$:
  \begin{romenumerate}[-10pt]
\item \label{tc0w}
$X$ is \wmeas{} if and only if $X$ is \cmeas.
\item \label{tc0n}
If $X$ is \cmeas, then there exists a countable subset $S_0\subseteq S$ such
that for every $s\notin S_0$, $X(s)=0$ \as{}
\item \label{tc0k}
If $X$ is \cmeas, then $X$ is \wassep.
Moreover, then $X\tpk$ is \wassep{} in $c_0(S)\ptpk$ and $c_0(S)\itpk$ for
every $k\ge1$.
  \end{romenumerate}
\end{theorem}

Note that \ref{tc0w} is proved for
any separable $C(K)$ in \refC{CCKcC}, but here $c_0(S)$ may be non-separable.

\begin{proof}
Since $c_0(S)\q=\ell^1(S)$, and every element of $\ell^1(S)$ 
has countable support, it follows 
that $\bw$ is generated by the point evaluations,
\ie,  that $\bw=\cC$.

\pfitemref{tc0w}
This is immediate from  $\bw=\cC$. 

\pfitemref{tc0n}
Suppose that $X$ is $\cC$-\meas, and let, for $\gd,\eps>0$, 
$$
S_{\gd\eps}\=\set{s\in S:\P(|X(s)|>\gd)>\eps)}.
$$ 
Suppose that $S_{\gd\eps}$ is
infinite for 
some $\gd,\eps>0$. Fix these $\gd$ and $\eps$, 
and let $s_i$, $i=1,2,\dots$, be an infinite sequence
of distinct elements of $S_{\gd\eps}$. 
Let $N\=\sum_i\ett{|X(s_i)|>\gd}$ be the
number of points $s_i$ where $|X|>\gd$. Since $X\in c_0(S)$, $N$ is a finite
\rv, and thus there exists $M<\infty$ such that $\P(N>M)< \eps/2$.
It follows that for every $s_i$, 
\begin{equation*}
  \begin{split}
  \E\bigpar{\ett{|X(s_i)|>\gd}\ett{N\le M}}
&=
\P\bigpar{|X(s_i)|>\gd ) \text{ and } N\le M}
\\&\ge
\P\bigpar{|X(s_i)|>\gd )}-\P(N> M)
>\eps/2.	
  \end{split}
\end{equation*}
Summing over all $i$ we obtain the contradiction
\begin{equation*}
M\ge  \E\bigpar{N\ett{N\le M}}=
  \sum_{i=1}^\infty\E\bigpar{\ett{|X(s_i)|>\gd}\ett{N\le M}}
\ge\sumi \eps/2
=\infty.
\end{equation*}
Consequently, each $S_{\gd\eps}$ is finite.
Let $S_0=\bigcup_{n=1}^\infty S_{n\qw,n\qw}$. Then $S_0$ is a countable
subset of $S$ and 
if  $s\notin S_0$, then
$X(s)=0$ \as{}

\pfitemref{tc0k}
Let $S_0$ be as in \ref{tc0n} and let
$B_1\=\set{f\in c_0(S):\supp(f)\subseteq S_0}$. Then $B_1$ is separable.
Moreover, if $\xx\in c_0(S)\q=\ell^1(S)$ with $\xx\perp B_1$, then 
$\xx=(a(s))\in\ell^1(S)$ with $a(s)=0$ for $s\in S_0$ and thus, since
$\set{s:a(s)\neq0}$ is countable,
\begin{equation*}
\xx(X)=\sum_{s\notin S_0} a(s)X(s) = 0 \qquad \text{a.s.}  
\end{equation*}
Thus $X$ is \wassep.

More generally, if $k\ge1$, then $B_1\tpk$ is a separable subspace of
$c_0(S)\ptpk$.
Suppose that $\ga\in (c_0(S)\ptpk)\q$ with $\ga\perp B_1\tpk$. 
By \refT{Tproj*},  $\ga$ is 
a bounded multilinear form
$c_0(S)^k\to\bbR$. Let $A$ be the countable subset given by \refL{LBH}.

Since $A $ is countable and $X(s)=0$ \as{} for $s\in A\setminus S_0$,
$P_A X=P_{A\cap S_0} X\in B_1$ a.s., and thus \as{}
\begin{equation*}
  \innprod{\ga,X\tpk}
=\ga(X,\dots,X)
=\ga(P_A X,\dots,P_AX)
=\innprod{\ga,(P_{A\cap S_0}X)\tpk}
=0.
\end{equation*}
This hold for every $\ga\perp B_1\tpk$, and thus $X\tpk$ is \wassep{} in
$c_0(S)\ptpk$. 
Since $\iota:c_0(S)\ptpk\to c_0(S)\itpk$ is continuous, 
$X\tpk$ is \wassep{} in $c_0(S)\itpk$ too. 
\end{proof}

However, $X$ is not always \assep{}; consider for example $X\=e_U\in
c_0\oi$ where $U\sim \U(0,1)$. (Cf.\ \refE{EPD3}, where we consider a
similar \rv{} in $\ell^2\oi$.)

\begin{theorem}\label{Tc0sep}
  If $X$ is a \rv{} in $\coss$, then $X$ is \assep{} if and only if there
  exists a countable subset $A\subseteq S$ such that $\supp(X)\subseteq A$ a.s.
\end{theorem}
\begin{proof}
If $A$ is countable, then 
\set{f\in\coss:\supp(f)\subseteq A} is a separable subspace of $\coss$; 
conversely,
every separable subspace of $\coss$
is included in some such subspace with $A$ countable.
\end{proof}

\begin{theorem}\label{Tc01}
Suppose that $X$ is a $\cC$-\meas{}
  $\coss$-valued \rv. 
Let $k\ge1$.
  \begin{romenumerate}[-10pt]
  \item \label{tc0d}
$\E X\itpk$ exists in Dunford sense $\iff$ 
the weak \kth{} moment  exists  $\iff$ 
$\sup_{s\in	S}\E|X(s)|^k<\infty$.
  \item \label{tc0p}
$\E X\itpk$ exists in Pettis sense $\iff$ 
the family
$\set{|X(s)|^k:s\in S}$ of \rv{s} is \ui{}
$\iff$ $\E|X(s)|^k\in\coss$.
  \item \label{tc0b}
$\E X\itpk$ exists in Bochner sense $\iff$ 
$\E\bigpar{\sup_{s\in S}|X(s)|}^k<\infty$ and there
exists a countable subset $A\subseteq S$ such that $\supp(X)\subseteq A$ a.s.
  \end{romenumerate}
If\/ $\E X\itpk$ exists in Bochner or Pettis sense, then 
it is the function
in $\coss\itpk=\cossk$ given by
\begin{equation}
  \label{ec0kt}
\E X\itpk(s_1,\dots,s_k)
=\E\bigpar{X(s_1)\dotsm X(s_k)}.
\end{equation}
\end{theorem}
\begin{proof}
  The proof of \refT{TCK1} holds with a few minor changes; we use
Theorems  \ref{TID}\ref{tidd}\ref{tid3}, \ref{Tc0}\ref{tc0k} and \ref{Tc0sep}, 
and note that 
\eqref{rogat} and \eqref{krhmf} hold without
  measurability problem since $\mu$ now is a discrete measure with
countable support.

Moreover, for \ref{tc0p}, if $s_n$ is any sequence of distinct elements in
$S$, then $|X(s_n)|^k\to0$ as \ntoo. 
Hence, if
$\set{|X(s)|^k:s\in S}$ of \rv{s} is \ui, then 
$\E|X(s_n)|^k\to0$, and it follows that $\E|X(s)|^k\in\coss$.
The converse is obvious.

Finally, \eqref{ec0kt} follows as in \refT{TCK0}.
\end{proof}

There is an obvious analogue of \refC{CCK11}, which we leave to the
reader. Note that \eqref{venus} holds for $\coss$, even when this space is
non-separable, because each $\mu$ in \eqref{ven} has a countable support.

For the second projective moment, we can again use Grothendieck's theorem,
and obtain the following version of \refT{TCKG}.

\begin{theorem}\label{Tc0G}
Suppose that $X$ is a $\cC$-\meas{}
  $\coss$-valued \rv. 
  \begin{romenumerate}[-10pt]
  \item \label{tc0gd}
$\E X\ptpx2$ exists in Dunford sense 
$\iff$
$\E X\itpx2$ exists in Dunford sense 
$\iff$
the weak second moment  exists  $\iff$
$\sup_{s\in	S}\E|X(s)|^2<\infty$.
  \item \label{tc0gp}
$\E X\ptpx2$ exists in Pettis sense $\iff$ 
$\E X\itpx2$ exists in Pettis sense $\iff$ 
$\E |X(s)|^2\in \coss$.
  \item \label{tc0gb}
$\E X\ptpx2$ exists in Bochner sense 
$\iff$ 
$\E X\itpx2$ exists in Bochner sense 
$\iff$ 
$\E\bigpar{\sup_{s\in S}|X(s)|}^2<\infty$
and there
exists a countable subset $A\subseteq S$ such that $\supp(X)\subseteq A$ a.s.
  \end{romenumerate}
\end{theorem}

\begin{proof}
  For any given bounded bilinear form $\ga$ on $\coss$, there exists by
  \refL{LBH} (or by \refT{TG}) a countable subset $S_0$ of $S$ such that 
$\ga(f,g)$ depends only on the restrictions of $f$ and $g$ to $S_0$.
Thus $\ga$ can be regarded as a bilinear form on $c_0(S_0)$, and can be
extended  to $C(S_0\q)$, and the proof of \refT{TCKG} applies, again using
Theorems \ref{Tc0}\ref{tc0k} and \ref{Tc0sep}.
\end{proof}

It is now easy to see that Theorems \ref{TC} and \ref{TCKD2}  hold for 
\cmeas{} \rv{s} in $\coss$; we leave the details to the reader.

The results above show that the space $c_0(S)$ behaves very well
even when $S$ is uncountable. However, the following example shows that
the moments may be degenerate. We note also (\refE{Eingaro}) that the norm
of a \wmeas{} \rv{} in $c_0(S)$ may fail to be \meas.

\begin{example}\label{Erika}
  Let $B=c_0\oi$ and $X=e_U$, $U\sim \U(0,1)$. (Cf.\ \refE{EPD3}.)

Let $\ga$ be a bounded $k$-linear form on $c_0\oi$, and let $A\subset\oi$ be 
a countable set as in
\refL{LBH}. Since $\P(P_AX\neq0)=\P(U\in A)=0$, it
follows from \eqref{lbh} that 
$$
\innprod{\ga,X\tpk}=\ga(X,\dots,X)=0
\qquad\text{a.s.}
$$
Consequently, every projective moment
$\E X\ptpk$ exists in Pettis sense, with 
$\E X\ptpk=0$, for every $k\ge1$. 
Hence also the injective moments exist in Pettis sense with $\E X\itpk=0$.
(No moment exists in Bochner sense, since $X$ is not \assep, see 
Theorems \ref{Tc0sep} and \ref{Tc01}\ref{tc0b}.)
\end{example}

\begin{example}\label{Eingaro} 
We modify \refE{Erika} by still taking $B=c_0\oi$ but now, as in \refE{EPD3}, 
$X=a(U)e_U$ where $a:\oi\to(0,1]$ is a non-measurable function.
Then $X$ is \wmeas, as in Examples \ref{EPD3} and \ref{Erika},
but $\norm X=a(U)$ is not \meas.
\end{example}

\section{$\doi$ as a Banach space}\label{SD}
Recall that
$D=\doi$ denotes the linear space of functions $\oi\to\bbR$ that are
right-continuous with left limits, see \eg{} \cite[Chapter 3]{Billingsley}.
In other words $f\in\doi$ if $\lim_{s\downto t}f(s)=f(t)$ for every $t\in[0,1)$,
and furthermore the left limit of $f$ at $t$, which we denote by
\begin{equation}\label{f-}
f(t-)\=\lim_{s\upto t}f(s),
\end{equation}
exists for every $t\in(0,1]$.
We further define
\begin{equation}
  \gD f(t)\=f(t)-f(t-),
\end{equation}
the jump at $t$. We may for completeness define $f(0-)\=f(0)$ and thus 
$\gD f(0)\=0$. 

It is easily seen that each function $f\in D$ is bounded. Hence  $\norm
f_D\=\sup_{t\in\oi}|f(t)|$ defines a norm on $D$; the resulting topology is
the \emph{uniform topology} on $D$, see \cite[Section 15]{Billingsley}. 
The
norm is complete, so $D$ is a Banach space. 
Note that $D$ is not separable,
since the uncountable set of functions $\etta_{[t,1]}$, $t\in\oi$, all have
distance 1 to each other; this leads to measurability problems when we
consider $D$-values random variables, as discussed in 
\refE{Eetta} and \cite[Section 15]{Billingsley}.  
Note also that the (separable) space $C=\coi$ of continuous functions on
$\oi$ is a closed subspace of $\doi$. 
The following relation between $\doi$ and $\coi$ was proved by 
\citet[Example 2]{Corson}.

\begin{theorem}[\cite{Corson}]
  \label{TgD}
$\gD$ is a bounded linear map of $\doi$ onto $\cooi$, with kernel $\coi$ and
  norm $\|\gD\|=2$. Furthermore, for any $f\in\doi$,
  \begin{equation}
	\label{tgd}
\norm{\gD f}_{\cooi} =2\inf\bigset{\norm{f+h}_{\doi}:h\in\coi}.
  \end{equation}
Hence, $\gD$ can be regarded as an isomorphism $D/C\to\cooi$ and
$\frac12\gD$ is an isometric isomorphism $D/C\to\cooi$.
\end{theorem}

\begin{proof}
It is well-known that $f\in D$ implies $\gD f\in\cooi$, but for completeness
we repeat the proof:
Given $\eps>0$, for every $x\in\oi$ we may find an open interval
$U_x=(x-\gd_x,x+\gd_x)$ such that $|f(y)-f(x)|<\eps/2$ for $y\in(x,x+\gd_x)$
and $|f(y)-f(x-)|<\eps/2$ for $y\in(x-\gd,x)$. 
(We consider only $y\in\oi$, and ignore $y\in U_x\setminus\oi$, if such
points exist.) 
Hence, $|\gD f(y)|\le\eps$ for
$y\in U_x\setminus \set x$. Since $\oi$ is compact, it can be covered by a
finite set of such intervals $U_{x_1},\dots,U_{x_n}$, and then 
$\set{y:\gD f(y)>\eps}\subseteq\set{x_1,\dots,x_n}$. Since $\eps>0$ is
arbitrary, this shows that $\gD f\in \cooi$.

Hence $\gD:\doi\to\cooi$. It is obvious that $\gD$ is linear and that
$f\in \ker(\gD)\iff f\in \coi$.  
%
Furthermore, $\norm{\gD f}_{c_0}\le 2\norm{f}_D$. The function
$f_0(x)\=\etta_{[1/2,1)}  - \etta_{[0,1/2)}$ in $D$  has 
$\norm{\gD f_0}_{c_0}=2= 2\norm{f_0}_D$, showing that equality can hold and
	  thus $\norm{\gD}=2$.

If $g\in\coooi$, let $\supp (g)=\set{x_1,\dots,x_n}$ with
$0<x_1<\dots<x_n\le1$. 
Let $f$ be the function in $D$ that is constant on $[0,x_1)$ and $[x_n,1]$,
linear on each $[x_i,x_{i+1})$ for $1\le i<n$, and satisfies
$f(x_i-)=-g(x_i)/2$, $f(x_i)=g(x_i)/2$, $1\le i\le n$.
Then $\gD f=g$ and $\norm {f}_D=\frac12\norm{g}_{\cooi}$.
(The case $g=0$ is trivial; take $f=0$.)

For an arbitrary $g\in\cooi$ and $\eps>0$, we can find $g_n\in\coooi$ such
that $g=\sumn g_n$ and $\sumn\norm{g_n}_{\cooi}\le \norm{g}_{\cooi}+\eps$.
Taking $f_n\in D$ as just constructed with $\gD f_n=g_n$ and
$\norm{f_n}_D=\frac12\norm{g_n}_{\cooi}$, we see that $f\=\sumn f_n\in D$
satisfies 
$\gD f=\sumn g_n=g$. Hence, $\gD$ is onto. Moreover, 
\begin{equation*}
\norm{f}_D\le\sumn\norm{f_n}_D=\frac12\sumn\norm{g_n}_{\cooi}
\le\frac12\norm{g}_{\cooi}+\frac{\eps}2.
\end{equation*}
Consequently,
\begin{equation*}
  \inf\bigset{\norm{f}_D:\gD f=g}\le \tfrac12\norm{g}_\cooi,
\end{equation*}
and since $\norm{\gD}=2$, we have equality:
\begin{equation*}
  \inf\bigset{\norm{f}_D:\gD f=g}= \tfrac12\norm{g}_\cooi,
\end{equation*}
which easily is seen to be equivalent to \eqref{tgd}.
\end{proof}

\begin{corollary}[\citet{Pestman}]
  \label{C1}
Every continuous linear functional $\chi\in \doi\q$ is given by  
\begin{equation}  \label{c1}
\chi(f)=\intoi f\dd\mu+\sum_{t\in\ooi}h(t)\gD f(t)
\end{equation}
for some unique $\mu\in M\oi$ and $h\in\ell^1\ooi$;
conversely, \eqref{c1} defines a continuous linear functional on $\doi$ for
every such $\mu$ and $h$. Furthermore,
\begin{equation}
  \label{c1a}
\tfrac12\norm{\chi}_{D\q}
\le\norm{\mu}_M+\norm{h}_{\ell^1}
\le2\norm{\chi}_{D\q}.
\end{equation}
\end{corollary}
Note that the formally uncountable
sum in \eqref{c1} really is the countable 
sum $\sum_{t\in\supp (h)}h(t)\gD f(t)$. 

\begin{proof}
  It is clear that for any $\mu\in M\oi$ and $h\in\ell^1\ooi$, \eqref{c1}
  defines a linear functional $\chi$ with
  \begin{equation*}
|\chi(f)| 
\le \norm{\mu}_M\norm{f}_D + \norm{h}_{\ell^1\ooi}\norm{\gD f}_\cooi	
\le \bigpar{\norm{\mu}_M + 2\norm{h}_{\ell^1\ooi}}\norm{f}_D.
  \end{equation*}
Hence, $\chi\in D\q$ and
$\norm{\chi}_{D\q}\le\norm{\mu}_M+2\norm{h}_{\ell^1\ooi}$, showing the first
part of \eqref{c1a}.

Conversely, if $\chi\in D\q$, then the restriction of $\chi$ to $C$ is a
continuous linear functional on $C$, which by the Riesz representation
theorem is given by a real measure $\mu\in M\oi$ with 
$\norm{\mu}_M=\norm{\chi}_C\le\norm{\chi}_D$. As just said, 
$\chi_1(f)\=\intoi f\dd\mu$ defines a continuous linear functional on
$D$, with 
$\norm{\chi_1}_{D\q}\le\norm{\mu}_M\le\norm{\chi}_D$ and $\chi_1(f)=\chi(f)$
if $f\in C$.
Let $\chi_2\=\chi-\chi_1$. Then $\chi_2\in D\q$ and $\chi_2(f)=0$ if $f\in
C$. Hence $\chi_2$ can be regarded as an element of $(D/C)\q$.
By \refT{TgD}, 
\begin{equation*}
  (D/C)\q\cong \cooi\q\cong \ell^1\ooi,
\end{equation*}
using the isometric isomorphism $\frac12\gD$.
Hence, there exists $g\in\ell^1\ooi$ with
\begin{equation}  \label{c1b}
  \norm{g}_{\ell^1\ooi}=\norm{\chi_2}_{(D/C)\q}
=\norm{\chi_2}_{D\q}
\le \norm{\chi}_{D\q}+ \norm{\chi_1}_{D\q}
\le 2\norm{\chi}_{D\q}
\end{equation}
and
\begin{equation*}
  \chi_2(f)=\innprod{g,\tfrac12\gD f} = \frac12\sum_{t\in\ooi}g(t)\gD f(t).
\end{equation*}
The decomposition \eqref{c1} follows with $h=\frac12g$; furthermore, by
\eqref{c1b}, 
\begin{equation*}
  \norm{\mu}_M+\norm{h}_{\ell^1\ooi}\le\norm\chi_{D\q}+\norm\chi_{D\q}
=2\norm\chi_{D\q}.
\end{equation*}

Finally, to see uniqueness, suppose that $\mu\in M\oi$ and $h\in\ell^1\ooi$
are such that $\chi$ defined by \eqref{c1} equals 0. Then $0=\chi(f)=\intoi
f\dd\mu$ for every $f\in \coi$, and thus $\mu=0$. This implies further
$0=\chi(f)=\sum_t h(t)\gD f(t)$ for every $f\in\doi$, and thus by \refT{TgD}
$\sum_t h(t)g(t)=0$ for every $g\in\cooi$. Hence, $h=0$.
\end{proof}

\begin{remark}\label{Rc2}
  The decomposition \eqref{c1} can also be written
  \begin{equation}
	\label{c2}
\chi(f) = \intoi f(t)\dd\mu_1(t) + \intoi f(t-)\dd\mu_2(t),
  \end{equation}
where $\mu_2\=-\sum_{x\in\ooi}h(x)\gd_x$ and $\mu_1\=\mu-\mu_2$.
Conversely, every pair of measures $\mu_1,\mu_2\in M\oi$ defines a
continuous linear functional $\chi\in \doi\q$ by \eqref{c2}.
However, this
representation is not unique unless we impose further conditions (for
example that $\mu_2$ is discrete with $\mu_2\set0=0$, as in the construction
above); note that $\int f(t)\dd\mu(t)=\int f(t-)\dd\mu(t)$ for every $f\in D$
and every continuous measure $\mu$, since $f(t-)=f(t)$ except on the
countable set $\supp(\gD f)$.
\end{remark}

\begin{remark}
 $\coi$ is not a complemented subspace of $\doi$, \ie{} there does not
  exist a bounded linear projection $P:\doi\to\coi$, see
 \citet[Example 2]{Corson}.
Equivalently, there does not exist a right inverse of $\gD$, \ie,
a bounded linear map $T:\cooi\to \doi$
such that $\gD T g=g$ for every $g\in\cooi$. 
(The equivalence is standard, and follows because we can factor $I-P$
through $D/C\cong\cooi$ and thus define $T$ by $T\gD=I-P$, and conversely.)

To see this, suppose that such
a map $T$ exists. Then $g\mapsto T g(t)$ is a bounded linear functional on
$\cooi$ for every $t\in\oi$, and is thus given by some $h_t\in\ell^1\ooi$.
Let $N\=\bigcup_{t\in \bbQ\cap\oi} \supp (h_t)$. Then $N$ is countable, so there
exists $s\in\ooi\setminus N$. Consider $e_s(t)\=\ett{t=s}\in\cooi$. Since
$\supp\xpar{e_s}=\set s$ is disjoint from $\supp (h_t)$ for all rational $t$, it
follows that $T e_s(t)=\innprod{e_s,h_t}=0$ for all rational $t$. Since
$Te_s\in\doi$, this implies that $T e_s=0$, which contradicts $\gD T e_s=e_s$.
(See \citet[Example 2]{Corson} for a different proof.)

Nevertheless, \refC{C1} shows that $\coi\q=M\oi$ embeds as a complemented
subspace of $\doi\q$, and we have $\doi\q\cong\coi\q\oplus\cooi\q$.
The crucial fact is that each bounded
linear functional $\chi$ on $\coi$ extends in a
canonical (linear) way to a bounded linear functional on $\doi$, because
$\chi$ is given 
by a measure $\mu\in M\oi$ and we can define the extension by $\int f\dd\mu$
for any $f\in D$.
(This is similar to the decomposition $(\ell^\infty)\q=c_0\q\oplus c_0^\perp$
since $c_0\q=\ell^1$ embeds into $(\ell^\infty)\q$.)
In general, a closed subspace $B_0$ of a Banach space $B$ 
is said to be \emph{weakly complemented} if its 
annihilator is complemented, i.e. if there is a projection
$P:B^* \to B_0^\perp$; it is easy to see that this is equivalent to the
existence of a bounded linear map $i:B_0^*\to B\q$ such that $i\xx$ is an
extension of $\xx$ for every $\xx\in B_0\q$, and then 
$B\q=i(B_0\q)\oplus B_0^\perp\cong B_0\q\oplus B_0^\perp$.
Thus $\coi$ is a weakly complemented subspace of $\doi$.
(And $c_0$ is a  weakly complemented subspace of $\ell^\infty$; 
in fact,
$c_0$ is a  weakly complemented subspace of any Banach space $B\supset c_0$.)
\end{remark}

\section{$\doi$ as a Banach algebra}\label{S:algebra}

The product of two functions in $\doi$ is also a function in $\doi$, and
thus $\doi$ is a commutative Banach algebra. 
In order to use the general theory of complex Banach algebras (which is much
more satisfactory than the theory for real Banach algebras), we consider
in this section  $\doi$ as a complex space, consisting of all complex-valued
right continuous functions on $\oi$ with left limits. The results below will be
proved for this case, but it follows immediately that they hold for the real
case too, by considering the subset of real-valued functions.

The (complex) Banach algebra $\doi$ has an involution given by $f\mapsto
\bar f$, and obviously $\norm{f\bar f}=\norm{f}^2$, so $\doi$ is a
commutative $C\q$-algebra. 
As said in \refE{EC*},  every such algebra $A$ is isometric
to the space $C(K)$ of continuous functions on its maximal ideal space $K$,
see \eg{} \cite[Chapter VIII.2]{Conway}. The maximal ideal space can be
described as the set of all (non-zero) 
multiplicative linear functionals  
(\ie,  homomorphisms)
$h:A\to\bbC$ 
with the weak$\q$-topology 
(in this context known as the \emph{Gelfand topology}),
and the isometry $A\to C(K)$ 
maps $f\in A$ to the function $\hf:h\mapsto h(f)$ on $K$; moreover, $K$ is a
compact Hausdorff space. 

In the case of
$\doi$, the maximal ideal space has a simple description.

\begin{theorem}\label{ThI}
  The linear homomorphisms on $\doi$ are given by $f\mapsto f(t)$,
  $t\in\oi$, and $f\mapsto f(t-)$, $t\in\ooi$. The maximal ideal space $\hI$
  thus consists of two copies, $t$ and $t-$, of every point in $\ooi$,
  together with a single $0$.
\end{theorem}

\begin{proof}
This is a simple adaption of the standard argument for the Banach algebra of
continuous functions on a compact set.
  The mappings $f\mapsto f(t)$ and $f\mapsto f(t-)$ are non-zero
  homomorphisms, and they are obviously all distinct. 

Suppose that there
  exists another homomorphism $h:\doi\to\bbC$, and let $M\=\ker(h)$; thus
  $M$ is a maximal ideal in $\doi$. Since $h$ differs from all $f\mapsto
  f(t)$ and $f\mapsto f(t-)$, there exists for each $t\in\ooi$ two functions
  $f_{t+},f_{t-}\in M$ such that $f_{t+}(t)\neq0$ and $f_{t-}(t-)\neq0$. By
  taking a suitable linear combination of $f_{t+}$ and $f_{t-}$ we see that
  there exists $f_t\in M$ such that $f_t(t)\neq0$ and $f_t(t-)\neq0$; thus
  there exists an open set $U_t$ such that $t\in U_t$ and $|f_t|$ is bounded
  below in
  $U_t$. For $t=0$ we directly find $f_0$ with $f_0(0)\neq0$ and thus an
  open set $U_0$ with $0\in U_0$ and $f_0$ bounded below in $U_0$. 

The sets $U_t$ form an open cover of the compact set $\oi$, and thus there
exists a finite set $\set{t_1,\dots,t_n}$ such that $\bigcup_{i=1}^n
U_{t_i}\supseteq\oi$. The function 
$F\=\sumin |f_{t_i}|^2=\sumin f_{t_i}\bar f_{t_i}\in M$,
since $M$ is an 
ideal, 
and $\inf_{t\in\oi}F(t)>0$,  by the construction. Hence
$1/F\in\doi$, and 
$1=F\cdot(1/F)\in M$, which is a contradiction.

(Alternatively, one can use the description in \refC{C1} of the continuous
linear functionals and show that \eqref{c1} is multiplicative only in the
cases given in the theorem.)
\end{proof}

We give $\hI$ the 
Gelfand topology,
\ie, the topology generated by the
functions $\hf:\hI\to\bbC$; as said above $\hI$ is compact.
By \refT{ThI}, the points in $\hI$ are of two types, $t$ and $t-$; we call
the points $t$ \emph{ordinary}. (We may for symmetry define $t+=t$; then 
 $\hI=\set{t\pm:0<t\le1}\cup\set0$. We shall use the notation $t+$ for the
ordinary points in $\hI$ when we want to distinguish between $t+$ as an
element  of $\hI$ and $t$ as an element of \oi.)
We may thus regard $\oi$ as a subset of
$\hI$, consisting of the ordinary points. (But note that $\oi$ does not have
the subspace topology.)
We then have $\doi=C(\hI)$
as noted by \citetq{Example 5.7}{EdgarII} (and possibly known earlier);
we state this in detail:

\begin{theorem}\label{TcI}
  $\doi=C(\hI)$. More precisely, each function $f\in\doi$ extends 
to a unique continuous function on $\hI$, 
with $f(t-)$ given by  \eqref{f-}, 
and, conversely, if $f$ is a continuous
  function on $\hI$, then the restiction to the ordinary points $t\in\oi$
is a function in $\doi$.
\end{theorem}
\begin{proof}
  This is just a reformulation of the fact that the Gelfand transform
  $f\mapsto\hf$ is an isomorphism $\doi\to C(\hI)$, using the description of
  $\hI$ in \refT{ThI}.
\end{proof}

\begin{corollary}\label{CDap}
  $\doi$ has the \ap.
\end{corollary}
\begin{proof}
  In fact, $C(K)$ has the \ap{} for every compact  $K$, see
\refT{TCKap} or  \cite[Example 4.2]{Ryan}.
\end{proof}


The topological space $\hI$ 
is called the \emph{split interval} or \emph{two arrow space}. 
(Actually, this name is often used for a modification of $\hI$ obtained by
either adding  an extra point $0-$ or deleting $1$; both modifications are
symmetric with a natural involutive homeomorphism $t\pm\mapsto (1-t)\mp$.)

Note that there is a natural total order on $\hI$, with $x<y-<y$ when
$x,y\in\oi$ 
with $x<y$. (This is the lexicographic order, if we regard $\hI$ as a subset
of $\oi\times\set{-,+}$.) We define intervals in $\hI$ in the usual way, using
this order. Recall that any totally ordered set can be given a topology, the
\emph{order topology}, with a base consisting of all open intervals 
$(-\infty,a)$, $(a,\infty)$, and $(a,b)$, see \eg{}
\cite[Problems 1.7.4 and 3.12.3]{Engelking}.

Recall that a compact Hausdorff space is 
\emph{totally disconnected} 
or \emph{zero-dimensional} 
if it has a base consisting of open and closed sets,
and \emph{extremally disconnected} if the closure of every open set is open.
(These notions are used also for non-compact spaces, but then ``totally
disconnected''  is used in several, non-equivalent, meanings, coinciding for
compact spaces, see \eg{} \cite{Engelking}.)

\begin{theorem}\label{Torder}
The compact Hausdorff topology on $\hI$ 
equals the order topology. The space is totally disconnected but not
extremally disconnected.

Furthermore, 
$1$ is isolated (\ie, \set1 is open), 
each $t\in[0,1)$ has a \nbhb{} consisting of the intervals
  $[t,u)$, $u>t$, 
and each $t-$ has a \nbhb{} consisting of the intervals
 $[u,t-]$, $u<t$.
These \nbhb{s} consist of open and closed sets.
\end{theorem}
For $t-$, there is an alternative \nbhb{} consisting
of the intervals $(u,t-]$; these sets are open but not closed.
(Symmetrically, the intervals $[t,u-)$ form another open \nbhb{} at $t$.)

\begin{proof}
We first consider the order topology and note that the given collections of
sets are neighbourhood bases; 
this is
easily seen  since 
$[t,u)=(t-,u)$, 
$[u,t-]=(u-,t)$ 
and $\set1=(1-,\infty)$, with the interpretation $(0-,u)=(-\infty,u)$.

In particular, these intervals
form together a base for the order topology. If 
$J$ is any of these intervals, then $\etta_J$ is a function on $\hI$ whose
restriction $f$ to $\oi$ belongs to $\doi$, and it is easily verified that
$\hf=\etta_J$. Hence, by \refT{TcI}, $\etta_J\in C(\hI)$, and thus $J$ is
open and closed in $\hI$ with the Gelfand topology. 
This shows that the Gelfand topology is stronger than the order topology.
Since the Gelfand topology is compact and the order topology Hausdorff, 
the topologies  coincide.

We have seen that the given base consists of open and closed sets; hence the
space is totally disconnected. On the other hand, 
$U=\bigcup_{n=1}^\infty [1/(2n),1/(2n-1))$ is an open set whose closure
$\overline U = U\cup\set0$ is not open; hence $\hI$ is not extremally
  disconnected. 
\end{proof}

\begin{corollary}
  The compact space $\hI$ is 
separable and
first countable (each point has a countable \nbhb),
but not second countable ($\hI$ does not have a countable base).
\end{corollary}
\begin{proof}
$\hI$ is separable, since the rational numbers are dense. 

  We obtain countable \nbhb{s} by taking rational $u$ only in the intervals
  in \refT{Torder}.

On the other hand, if \set{U_\ga} is a base for the topology, then each set
$[t,1]$, which is open, 
contains some $U_{\ga(t)}$ with $t\in U_{\ga(t)}$. Then $t =\min
U_{\ga(t)}$, and thus the sets $U_{\ga(t)}$ are distinct. It follows that
every base contains (at least) $\cardc$ elements.
\end{proof}

\begin{corollary}\label{ChImetric}
  The compact space $\hI$ is not metrizable.  
\end{corollary}
\begin{proof}
A compact metrizable space is second countable \cite[Theorem 4.2.8]{Engelking}.

Alternatively, this follows by \refT{TCKsep} since  
$C(\hI)=\doi$ is not separable.
\end{proof}

\begin{remark}\label{Rsorgenfrey}
The topology on $\hI$ induces on the subspace $\oi$ the topology where
$\set1$ is open and each $t\in[0,1)$ has a \nbhb{} 
consisting of the intervals $[t,u)$, $u>t$.
This (or rather the corresponding topology on $\bbR$) is known as the
\emph{Sorgenfrey line}, and is a standard source of counterexamples in
topology, see 
\eg{} \cite{Sorgenfrey} and 
\cite[Examples 1.2.2,  2.3.12, 3.8.14, 5.1.31]{Engelking}.

For example, if $\tI$ denotes $\oi$ with this topology, so $\tI\subset\hI$,
then
$\tI$ and $\tI\times \tI$ are separable, but $\tI\times \tI$ contains the
closed subspace $\set{(t,1-t)}$ which is discrete and uncountable, and thus
not separable. 
(In particular, $\tI\times\tI$ cannot be metrizable, yielding  a third proof
of \refC{ChImetric}.)
Moreover, $\tI$ is paracompact and normal, but $\tI\times
\tI$ is neither
\cite{Sorgenfrey}.
 \end{remark}

\section{Measurability and random variables in $\doi$}\label{SmeasD} 

We equip $C\oi$, $D\oi$ and $\cooi$ with the $\gs$-fields generated by point
evaluations; we denote these by $\cC$ (as in \refS{SCK}), $\cD$ and $\cco$
(to distinguish it from $\cC$).
We further, as in \refS{SCK}, 
equip $M\oi$ with the $\gs$-field $\cM$ generated by the mappings
$\mu\to\int f\dd\mu$, $f\in\coi$.
In the present section, we shall always use these $\gs$-fields, even if we
do not always say so explicitly.

Note that $C\oi$ is a separable Banach space, and thus $\cC$ equals 
the Borel $\gs$-field $\cB(C)$ on $\coi$; moreover, it equals the
$\gs$-field $\bw(C)$ 
generated by the continuous linear functionals, see \refC{CCKcC}.

On the other hand, $\doi$ is \emph{not} separable, and $\cD$ is \emph{not}
equal to the Borel \gsf{} $\cB(D)$, see \refE{Eetta}. 
(The non-separability of $D$ causes several problems, and is the main source of
complications in the proofs below.)
However, we shall see that $\cD$
equals the $\gs$-field $\bw(D)$
generated by the continuous linear functionals.
It is also well-known that $\cD$ equals the Borel \gsf{} for the Skorohod
topology on \doi{}, see
\cite[Section 12]{Billingsley}. 
(This is a weaker topology which is separable metric but
not a vector space topology; it is the topology commonly used on $D$, but
it is not used in the present paper where we consider Banach spaces.)
Furthermore, $\cD$ also equals the \gsf{} generated by balls in $D$
\cite[Section 15]{Billingsley}. 

Also $\cooi$ is not separable, but we have seen in \refS{Sc0} that it
nevertheless has several nice properties.

We begin by proving some lemmas.

\begin{lemma}
  \label{LM1}
The mappings $(f,t)\mapsto f(t)$, 
$(f,t)\mapsto f(t-)$ and 
$(f,t)\mapsto \gD f(t)$ are measurable $D\times\oi\to\bbR$.
Furthermore, 
the map $\gD$ is measurable $\doi\to\cooi$.
\end{lemma}

\begin{proof}
  Since each $f\in D$ is right-continuous, $f(t)=\lim_\ntoo f(\ceil{nt}/n)$,
and $(f,t)\mapsto f(\ceil{nt}/n)$ is measurable for each $n$.
The measurability of $f(t-)$ is shown similarly, using $f((\ceil{nt}-1)/n)$
(for $t>0$), and $\gD f(t)=f(t)-f(t-)$.

In particular, $f\mapsto \gD f(t)$ is measurable for each fixed $t$, which
shows that $\gD:\doi\to\cooi$ is measurable.
\end{proof}

\begin{lemma}\label{LM3}
  There exists a sequence of measurable maps $x_k:\doi\to\oi$,
  $k=1,2,\dots$, such that if  
$f\in\doi$, then the non-zero values of $x_k(f)$, $k=1,2,\dots$ are the jump
  points of $f$, \ie, the points $x\in\oi$ with $|\gD f(x)|>0$, without
  repetition. 
\end{lemma}

\begin{proof}
If $f\in\doi$ and $I\subset\oi$ is an interval, define
\begin{equation}\label{oxberg}
  \var(f;I)\=
\sup\bigset{|f(x)-f(y)|:x,y\in I}.
\end{equation}
Note that it suffices to consider  $x,y$ in \eqref{oxberg} that are
rational, or equal to the right endpoint of $I$; this implies that
$f\mapsto\var(f;I)$ is measurable for each $I$.

Fix $\eps>0$ and $f\in D$ and consider for each $n$ the sequence
$x\nn_1,\dots,x\nn_{m(n)}$ of all dyadic rationals $k/2^n$ such that $1\le
k\le 2^n$ and $\var\bigpar{f;[(k-1)/2^n,k/2^n]}\ge\eps$; we assume that this
sequence is in increasing order and we extend it to an infinite sequence by
defining $x\nn_k=0$ for $k>m(n)$.
It is easily seen that, as \ntoo, $x\nn_k\to x_k$ for each $k$, where
$x_1,x_2,\dots$ are the points $x$ where $|\gD f(x)|\ge\eps$, taken in
increasing order and extended by 0's to an infinite sequence.
By construction, for each $k$, the maps $ f\mapsto x\nn_k$ are measurable,
and thus each $x_k$ is a measurable function of $f$. Repeating this
construction for $\eps=2^{-m}$, $m=0,1,2,\dots$, we may rearrange the
resulting sequences in a single sequence $(x_i)_1^\infty$, eliminating any
repetitions of non-zero values.
\end{proof}

\begin{remark}
In \refL{LM3}, we may further assume that $x_k(f)$ are arranged with 
$|\gD f(x_1)|\ge |\gD f(x_2)|\ge\dots$.  We will not use this, and leave
the proof (using \refL{LM1}) to the reader.
\end{remark}

\begin{lemma}
  \label{LM2}
The map
$(f,\mu)\mapsto \intoi f\dd\mu$ is measurable $D\oi\times M\oi\to\bbR$.
\end{lemma}

\begin{proof}
First note that the Riemann--Stieltjes sums
\begin{equation}
S_n(f,\mu)\= \sum_{i=0}^{n-1}  f(i/n)\mu[i/n,(i+1)/n)+f(1)\mu\set1
\end{equation}
are measurable, and that, using the notation in \eqref{oxberg},
\begin{equation}
  \label{eldris}
\lrabs{S_n(f,\mu)-\intoi f\dd\mu}
\le \norm\mu \max_i \var(f;[i/n,(i+1)/n]).
\end{equation}
Define $S^*(f;\mu)\=\limsup_\ntoo  S_n(f;\mu)$.
Then \eqref{eldris} implies
\begin{equation}
  \label{smagan}
  \begin{split}
\lrabs{S^*(f,\mu)-\intoi f\dd\mu}
&\le \norm\mu \limsup_\ntoo \max_i \var(f;[i/n,(i+1)/n])
\\&
=\norm\mu \max_{x\in\oi}|\gD f(x)|.	
  \end{split}
\end{equation}

Let $x_k(f)$ be as in \refL{LM3} and let $V_n$ be the (non-linear) 
\meas{} map $D\to D$ given by
\begin{equation}
  V_n f (x) \= \sumkn \gD f(x_k(f))\ett{x\ge x_k(f)}.
\end{equation}
Thus $\max_x|\gD (f-V_n f)(x)| = \max_{k>n}|\gD f(x_k(f))|\to0$
as \ntoo, 
since $\gD f\in\cooi$ by \refT{TgD},
and it follows from \eqref{smagan} that,
for any $f\in\doi$ and $\mu\in M\oi$,
\begin{equation}
  S^*(f-V_nf;\mu)-\intoi(f-V_nf)\dd\mu \to 0,
\qquad\text{as }\ntoo.
\end{equation}
Consequently, $\intoi f\dd\mu  $ is the limit as \ntoo{} of the measurable
functions 
\begin{equation*}
S^*(f-V_nf;\mu)+\intoi V_nf\dd\mu
= S^*(f-V_nf;\mu)+\sumkn \gD f(x_k(f))\mu[x_k(f),1],
\end{equation*}
where the last factor is measurable by \refL{LM1} since 
$\mu\mapsto G_\mu(t)\=\mu(t,1]$ is \meas{} $M\oi\to\doi$ and
$\mu[x,1]=G_\mu(x-)$ for $x>0$.
\end{proof}

\begin{theorem}\label{TDw}
Every continuous linear form $\chi\in\doi\q$ is \dmeas.
Hence
  $\cD=\bw(D)$, and a $D$-valued \rv{} $X$ is \dmeas{} if and only
	if it is \wmeas.
\end{theorem}

\begin{proof}
    By \refL{LM2}, 
$f\mapsto \intoi f\dd\mu$ is \dmeas{} for every $\mu\in
M\oi$.
It follows from \refC{C1} that
 $f\mapsto\innprod{\chi,f}$ 
is \dmeas{} for every $\chi\in \doi\q$. 

This implies $\bw(D)\subseteq\cD$.
Conversely, each point evaluation $f(t)=\innprod{\gd_t,f}$ where $\gd_t\in
\doi\q$. Hence $\cD=\bw(D)$.
\end{proof}

\begin{remark}
  \label{RDw}
We can extend the map in \refL{LM2} to $(f,\chi)\mapsto \innprod{f,\chi}$
for $f\in D$, $\chi\in D\q$; however, this map is \emph{not} jointly
\meas{} for $\cD$ and the \gsf{} $\cD\q$ on $D\q$ generated by
$\chi\mapsto\innprod{\chi,f}$, $f\in D$.
(The map is separably \meas{} by \refT{TDw}.)
In fact, let $a(t):\oi\to\bbR$ be a non-measurable function. Then the maps
$t\mapsto f_t\=\etta_{[t,1]}$ and $t\mapsto\chi_t\= g\mapsto a(t)\gD g(t)$
are \meas{} for $\cD$ and $\cD\q$, respectively; note that for any fixed $g$,
$t\mapsto\innprod{\chi_t,g}=a(t)\gD g(t)$ is \meas{} since it has countable
support. However, $t\mapsto\innprod{f_t,\chi_t}=a(t)$ is non-\meas.
\end{remark}

\begin{remark}\label{RDw2}
  We have $\doi=C(\hI)$, 
and it follows from \refL{LM1} that $\cD$ also equals the \gsf{} generated by
point evaluations on $\hI$.
We can also consider the map $(f,t)\mapsto f(t)$
  for $f\in\doi$ and $t\in\hI$. However, in contrast to \refL{LM1}, this map
  is \emph{not} jointly $\cD\times \cB(\hI)$-measurable, where $\cB(\hI)$
  is the Borel \gsf{} on $\hI$. To see this, let $A\subset\oi$ be a
  non-measurable set (with $0\in A$)
and let $\gf:I\to\hI$ be the function defined by
  $\gf(u)\= u+$ when $u\in A$ and
  $\gf(u)\= u-$ when $u\notin A$; then $\gf$ is increasing and thus Borel
  measurable. Furthermore, 
the function $u\mapsto f_u\=\ettau$ is a measurable map $\oi\to\doi=\chI$.
Hence, $u\mapsto(f_u,\gf(u))$ is \meas{} $\oi\to\chI\times\hI$.
However, the composition $u\mapsto f_u(\gf(u))=\etta_A(u)$ is
non-measurable,
showing that $(f,t)\mapsto f(t)$ is not measurable on $\chI\times\hI$.
\end{remark}

One of the main  purposes of this section is to prove the following theorem,
extending \refT{TDw} to multilinear forms.
(Note that the corresponding result for $\coi$ is immediate, since 
$\cC$ equals the Borel $\gs$-field on $\coi$ 
and $\coi$ is separable, which imply that
the product $\gs$-field $\cC^m$ on $(\coi)^m$
equals the Borel $\gs$-field.)

\begin{theorem}\label{TM}
  Every bounded multilinear form $T:(\doi)^m\to\bbR$, for any $m\ge1$, is
  \dmeas.
\end{theorem}

\begin{proof}
  We shall prove the more general result that for any $m,\ell\ge0$, any 
 bounded multilinear form $\ga:(\doi)^m\times(\cooi)^\ell\to\bbR$  is
 measurable. We do this by induction over $m$.

First, assume that $m=0$, so
$\ga:(\cooi)^\ell\to\bbR$.
Recall the projections $P_A$ in $c_0(S)$ defined in \eqref{pac} and
let $A$ be a countable subset as in \refL{LBH} 
(with $S=\ooi$ and $k=\ell$). 
If $A$ is finite, then $\ga(P_{A}f_1,\dots,P_{A}f_\ell)$ is a finite linear
combination of products of point evaluations, and thus measurable.

If $A$ is infinite, write $A=\set{x_1,x_2,\dots}$ and define
$P_n\=P_{\set{x_1,\dots,x_n}}$. For any $f\in\cooi$, $P_nf\to P_Af$ 
in $\cooi$
as
\ntoo, and thus
$\ga(P_{n}f_1,\dots,P_{n}f_\ell)
\to \ga(P_{A}f_1,\dots,P_{A}f_\ell)$. Each map $\ga(P_{n}f_1,\dots,P_{n}f_\ell)$
is measurable, and the result in the case $m=0$ follows by \eqref{lbh}.

Now suppose that $m\ge1$. Consider first the restriction $\ga_0$
of $\ga$ to $C\times D^{m-1}\times \cooi^\ell$. For fixed $f_2,\dots,f_m\in D$
and $g_1,\dots,g_\ell\in\cooi$, $\ga_0$ is a bounded linear functional on
$\coi$; thus $\ga_0$ can be regarded as a multilinear map
$T_0:D^{m-1}\times\cooi^\ell\to \coi^*=M\oi$. For a fixed $f_1\in\coi$,
the map 
$(f_2,\dots,f_m,g_1,\dots,g_\ell)\mapsto
\ga_0(f_1,f_2,\dots,f_m,g_1,\dots,g_\ell)$ is measurable by induction, which
by the definition of the $\gs$-field $\cM$ in $M\oi$ shows that $T_0$ is
measurable.

Define
\begin{equation}\label{wasa}
\ga_1(f_1,f_2,\dots,f_m,g_1,\dots,g_\ell)
\=\intoi f_1\dd T_0(f_2,\dots,f_m,g_1,\dots,g_\ell).
\end{equation}
Then
$\ga_1$ is a bounded multilinear form
$(\doi)^m\times(\cooi)^\ell\to\bbR$, and $\ga_1$ is measurable by  \refL{LM2}
and the measurability of $T_0$. Moreover, if $f_1\in\coi$, then
$\ga_1(f_1,f_2,\dots,f_m,g_1,\dots,g_\ell)
=\ga(f_1,f_2,\dots,f_m,g_1,\dots,g_\ell)$.

Define $\ga_2\=\ga-\ga_1$. Then $\ga_2=0$ on $C\times D^{m-1}\times\cooi^\ell$, so
$\ga_2$ can be regarded as a multilinear form on 
$(D/C)\times D^{m-1}\times\cooi^\ell$. By \refT{TgD}, $\gD:D/C\to\cooi$ is
an isomorphism, so there exists a bounded multilinear form
$\ga_3:\cooi\times D^{m-1}\times\cooi^\ell\to\bbR$ such that
\begin{equation}\label{vasa}
\ga_2(f_1,f_2,\dots,f_m,g_1,\dots,g_\ell)
=\ga_3(\gD f_1,f_2,\dots,f_m,g_1,\dots,g_\ell).
\end{equation}
Then $\ga_3$ is measurable by induction, and since $\gD:D\to\cooi$ is
measurable by \refL{LM1}, \eqref{vasa} shows that $\ga_2$ is measurable.

Thus, $\ga=\ga_1+\ga_2$ is measurable.
\end{proof}

\begin{corollary}
  \label{CDM}
Let $X$ be a \dmeas{} $D\oi$-valued \rv.
Then $X\tpk$ is \wmeas{} in $D\ptpk$ and in $D\itpk$ for every $k\ge1$.
\end{corollary}
\begin{proof}
  Immediate by Theorems \refand{Tproj*}{TM} and the continuous map
  $\iota:D\ptpk\to D\itpk$. 
\end{proof}

We have seen in \refE{Eetta} that a \dmeas{} $D$-valued \rv{} $X$ is not always
\assep. 
In fact, the following theorem describes the situation precisely: $X$ is
\assep{} if and only if the jumps only occur in a fixed countable set.
(Cf.\ \refT{Tc0sep} for $c_0(S)$.) 
\begin{lemma}
  \label{LDsep}
A subset $D_0\subseteq D$ is separable 
if and only if there exists a countable subset $N$ of $\oi$
such that $\supp(\gD f)\=\set{t:|\gD f(t)|\neq0}\subseteq N$ for each $f\in
D_0$. 
\end{lemma}

\begin{proof}
If $N\subset\oi$ is countable, then
\begin{equation}
  \label{d1}
D_1\=
\set{f\in\doi:\supp(\gD f)\subseteq N}
\end{equation}
is a separable subspace of $D$, for example because $\gD$ induces an
isomorphism $D_1/C\cong c_0(N)$, and both $C$ and $c_0(N)$ are separable.

Conversely, if $D_0\subseteq D$ is separable and $\set{f_n}_{n=1}^\infty$ is
a countable dense subset, then $N\=\bigcup_n \supp(\gD f_n)$ is countable and
$\supp(\gD f)\subseteq N$ for every $f\in D_0$.  
\end{proof}

\begin{theorem}  \label{TDassep}
Let $X$ be a \dmeas{} $D\oi$-valued \rv.
Then
$X$ is \assep{} if and only if there exists a countable subset $N$ of $\oi$
such that \as{} $\supp(\gD X)\=\set{t:|\gD X(t)|\neq0}\subseteq N$.
\end{theorem}
\begin{proof}
  Immediate from \refL{LDsep}.
\end{proof}

However, a  weak version always holds.

\begin{lemma}
  \label{LD1}
Let $X$ be a \dmeas{} $D\oi$-valued \rv.
Then there is a countable set $N\subset\oi$ such that for each
$t\in\oi\setminus N$,
$\gD X(t)=0$ a.s.
\end{lemma}

The exceptional null set may depend on $t$, in contrast to the condition in
\refT{TDassep}. 

\begin{proof}
  By \refT{Tc0}\ref{tc0n} applied to $\gD X\in \cooi$, which is
  $\cC_0$-measurable by \refL{LM1}.
\end{proof}

\begin{theorem}
  \label{TDwassep}
Let $X$ be a \dmeas{} $D\oi$-valued \rv.
Then $X$ is \wassep.
\end{theorem}

\begin{proof}
Let $N$ be as in \refL{LD1} and let
$D_1\subset D$ be the separable  subspace defined in \eqref{d1}.

Suppose that $\chi\in D\q$ with $\chi\perp D_1$, 
and represent $\chi$ as in \eqref{c1}.
Since $\chi\perp D_1$ and $C\subseteq D_1$, we see that if $f\in C$, then 
$0=\chi(f)=\intoi f\dd\mu$; thus $\mu=0$ and $\chi(f)=\sum_t h(t)\gD f(t)$
for all $f\in D$.
Moreover, if $t\in N$, then $\etta_{[t,1]}\in D_1$, and thus
$0=\chi(\etta_{[t,1]})=h(t)$. Hence, any $\chi\perp D_1$ has the form
\begin{equation}\label{ld2}
  \chi(f)=\sum_{t\in M}h(t)\gD f(t)
\end{equation}
for some countable set $M=\set{t:h(t)\neq0}$ with $M\cap N=\emptyset$.

For each $t\in M$, we have $t\notin N$ and thus $X(t)=0$ a.s., and thus
\as{} $\chi(X)=0$ by \eqref{ld2}.
\end{proof}

We extend this to $X\tpk$.

\begin{theorem}
  \label{TD3}
Let $X$ be a \dmeas{} $D\oi$-valued \rv.
Then, for every $k\ge1$,
$X\tpk$ is \wassep{} in the projective and injective tensor products 
$(\doi)\ptpk$ and $(\doi)\itpk$.
\end{theorem}

\begin{proof}
  It suffices to consider the projective tensor product since
$\iota:(\doi)\ptpk\to(\doi)\itpk$ is continuous.

Let, again, $N$ be as in \refL{LD1} and let
$D_1\subset D$ be the separable subspace \eqref{d1}.
Then $D_1\tpk$ is a separable subspace of $D\ptpk$. 
We claim that if
$\chi\in(D\ptpk)\q$ and $\chi\perp D_1\tpk$, then $\chi(X\tpk)=0$ a.s.,
which proves the lemma.
By \refT{Tproj*}, $\chi$ is a bounded $k$-linear form on $D$.
We will prove the claim by induction, using a more elaborate claim.

For $\ell,m\ge0$, let $\llm$ be the set of all 
$(\ell+m)$-linear forms
$\ga\in L(D^{\ell+m},\bbR)$ such that
\begin{align}
  \label{llm1}
\ga(f_1,\dots,f_{\ell+m})&=0
\quad\text{if } f_1,\dots, f_{\ell+m}\in D_1,
\intertext{and}
  \label{llm2}
\ga(f_1,\dots,f_{\ell+m})&=0
\quad\text{if } f_i \in C
\text{ for some $i\le\ell$}.
\end{align}

\begin{claim}
  If $\ell,m\ge0$, then $\ga(X,\dots,X)=0$ \as{} for every $\ga\in\llm$.
\end{claim}

The case $\ell=0$, $m=k$  yields the claim $\chi(X\tpk)=0$ \as{} above,
proving the lemma.
We prove the claim by induction on $m$.

\step1{$m=0$.}
If $m=0$, then \eqref{llm2} shows that $\ga$ can be seen as an $\ell$-linear
form on $D/C\cong \cooi$. Hence there exists an $\ell$-linear form
$\gb:\cooi^\ell\to\bbR$ such that
\begin{equation*}
  \ga(f_1,\dots,f_\ell)
=
\gb(\gD f_1,\dots,\gD f_\ell).
\end{equation*}
By \eqref{llm1}, $\gb(g_1,\dots,g_\ell)=0$ if $\supp g_i\subseteq N$ for
every $i$. Let $A$ be a countable subset of $\ooi$ given by \refL{LBH}
(applied to $\gb$). Then, as in the proof of \refT{Tc0}\ref{tc0k},
$P_A(\gD X)=P_{A\cap N}(\gD X)$ a.s., and thus a.s.
\begin{equation*}
  \ga(X,\dots,X)=\gb(\gD X,\dots,\gD X)
=\gb(P_A(\gD X),\dots,P_A(\gD X))=0.
\end{equation*}

\step2{$m\ge1$.}
Let  $k\=\ell+m\ge 1$.
If $k=1$, then the result is \refT{TDwassep} (and its proof). 
Thus assume $k\ge2$.

For fixed $f_1,\dots,f_{k-1}$, the map $f_k\mapsto\ga(f_1,\dots,f_{k})$ is a
bounded linear form on $D$. The restriction to $C$ gives an element of
$C\q=M\oi$, \ie{} a signed measure $\mu_{f_1,\dots,f_{k-1}}$ on $\oi$.
Define, similarly to \eqref{wasa},
\begin{equation*}
  \ga_1(f_1,\dots,f_k)\=
\intoi f_k\dd \mu_{f_1,\dots,f_{k-1}}.
\end{equation*}
Then $\ga_1$ is a bounded $k$-linear form on $D$, and thus so is
$\ga_2\=\ga-\ga_1$. 

If $f_1,\dots,f_{k-1}\in D_1$ and $f_k\in C\subseteq D_1$, 
then $\ga(f_1,\dots,f_k)=0$ by \eqref{llm1}; hence
$\mu_{f_1,\dots,f_{k-1}}=0$ and thus
$\ga_1(f_1,\dots,f_k)=0$ for any $f_k\in D$.
Similarly, by \eqref{llm2}, if $f_i\in C$ for some $i\le\ell$, 
then $\mu_{f_1,\dots,f_{k-1}}=0$ and 
$\ga_1(f_1,\dots,f_k)=0$.
Hence, \eqref{llm1} and \eqref{llm2} hold for $\ga_1$,
\ie, $\ga_1\in \llm$.
Consequently, $\ga_2=\ga-\ga_1\in\llm$ too.

Moreover, if $f_k\in C$, then the definition of $\ga_1$ yields
$ \ga_1(f_1,\dots,f_k)= \ga(f_1,\dots,f_k)$ and thus
\begin{equation*}
  \ga_2(f_1,\dots,f_k)=  \ga(f_1,\dots,f_k)-  \ga_1(f_1,\dots,f_k)=0.
\end{equation*}
Hence \eqref{llm2} holds for $i=k=\ell+m$ too, so (after relabelling)
$\ga_2\in L_{\ell+1,m-1}$, and by induction
$\ga_2(X,\dots,X)=0$ a.s.

Return to $\ga_1$.
For each fixed $f_k\in C$, 
$(f_1,\dots,f_{k-1})\mapsto\ga(f_1,\dots,f_{k-1},f_k)$ defines a
($k-1$)-linear form on $D$, which by \eqref{llm1}--\eqref{llm2} belongs to
$L_{\ell,m-1}$. By induction, thus $\ga(X,\dots,X,f_k)=0$ \as{} for each
fixed $f_k\in C$. By taking $f_k$ in a countable dense subset of $C$, it
follows that \as{}  $\ga(X,\dots,X,f)=0$  for every $f\in C$.
Thus, \as, $\mu_{X,\dots,X}=0$ and 
$\ga_1(X,\dots,X,f)=0$ for every $f\in D$; in particular,
$\ga_1(X,\dots,X,X)=0$.

We have shown that
$\ga(X,\dots,X)=
\ga_1(X,\dots,X)+
\ga_2(X,\dots,X)=
0$ \as{} when $\ga\in\llm$, which proves the claim.
\end{proof}

\section{Moments of $D$-valued \rv{s}}\label{SDmom}

After the preliminaries in the last sections, we can prove analogues of  
\refT{TCK1}--\ref{TCKD2} for $\doi$. (Note that although $\doi=C(\hI)$, we
cannot use these theorems in \refS{SCK} directly, since $\doi$ is
non-separable and thus $\hI$ is not metrizable.) 

We begin with  describing the space $\doi\itpk$ where the moments live.
By \refT{TCKitpk},
$\doi\itpk=\chI\itpk=C(\hI^k)$. We return from $\hI$ to \oi{} by taking
the restrictions of the functions in $C(\hI)^k$ to the (dense) subset 
$\oi^k\subset\hI^k$ (\ie, the normal points in $\hI^k$).

\begin{definition}\label{Ddoik}
  \doik{} is the Banach space of all functions $\oi^k\to\bbR$ that have a
  continuous extension to $\hI^k$; \doik{} is equipped with the supremum
  norm. Thus 
  $\doik$ is naturally isometric to $C(\hI^k)$. 
\end{definition}
This means that $f\in\doik$ if at each $(t_1,\dots,t_k)\in\oi^k$, 
$f$ has limits in the $2^k$ octants (with obvious modifications at the
boundary). More precisely, taking $k=2$ for notational convenience,
$f\in\doix2$ if and only if, for each $(s,t)\in\oi^2$, the limits
\begin{align*}
f(s+,t+)&\=\lim_{\substack{s'\to s,\; s'\ge s\\ t'\to t,\;t'\ge t}}f(s',t'), 
\\
f(s+,t-)&\=\lim_{\substack{s'\to s,\; s'\ge s\\ t'\to t,\;t'< t}}f(s',t'), 
\\
f(s-,t+)&\=\lim_{\substack{s'\to s,\; s'< s\\ t'\to t,\;t'\ge t}}f(s',t'), 
\\
f(s-,t-)&\=\lim_{\substack{s'\to s,\; s'< s\\ t'\to t,\;t'< t}}f(s',t')
\end{align*}
exist (as finite real numbers), except that we ignore all cases with an
argument $0-$. Note the slight asymmetry; we use $\ge$ but $<$. Note also
that necessarily $f(s+,t+)=f(s,t)$ when it exists.

\begin{theorem}\label{Tdoik}
If $k\ge2$, then 
  $\doi\ptpk\subsetneq \doi\itpk =\doik$.
(The  subspace $\doi\ptpk$ of $\doik$ is not closed, and the norms are
different and not equivalent on $\doi\ptpk$.)
\end{theorem}
\begin{proof}
  The equality $\doi\itpk=\doik$ follows by \refT{TCKitpk} applied to
  $\chI$, together with \refD{Ddoik}, and the inclusion $\doi\ptpk\subseteq
  \doi\itpk$ by \refC{CCKtpk}, or by  \refC{CDap} and \refT{Tiotak}.

The claim that $\doi\ptpk$ is not a closed subspace of $\doik$, is 
a special case of the general fact in \refR{RCKtpk}, applied to $\chI$,
but we give also a direct proof, using an argument from Varopoulos 
\cite{Varopoulos} (where $C(\bbT)\ptensor C(\bbT)$ is studied, and more
generally 
$C(K)\ptensor C(K)$  for an abelian
compact group $K$.)

By the
closed graph theorem the claim is
equivalent to the claim that that the norm on $\doi\ptpk$ 
is not equivalent to the norm inherited from $\doik$, \ie, the sup norm.

It suffices to consider the case $k=2$.   
If $F\in L^1(\oi^2)$, let $\hat F(m,n)\=\iint F(s,t)e^{2\pi\ii(ms+nt)}\dd
s\dd t$ 
  be its Fourier coefficients.
Since $\doi\subset L^2\oi$, Parseval's identity and \Holder's inequality
imply, together wih \eqref{proj}, that if $F\in\doi\ptpx2$, then 
$\sumnn|\hat F(n,n)|\le\norm{F}_{\doi\ptpx2}$.
However, there exist trigonometric polynomials $f(t)=\sum_n a_n e^{2\pi\ii nt}$
such that $\sup_t|f(t)|=1$ with $\sum_n|a_n|$  arbitrarily large. 
(For example Ces\'aro means of the Fourier series 
$\sum_2^\infty \sin{2\pi  \ii nt}/(n\log n)$, 
which represents a continuous function
\cite[Section V.1]{Zygmund}.) 
Taking
$F(s,t)\=f(s+t)\in C(\oi^2)$ we have $\norm{F}_{\doix2}=\sup|F(s,t)|=1$ and
$\norm{F}_{\doi\ptpk}\ge\sum_n|a_n|$ which is arbitrarily large.
\end{proof}

\begin{example}\label{EwienerD}
Let $W$ be standard Brownian motion, regarded as a random variable in $\doi$. 
Then
all (projective and injective) moments exist (in Bochner sense), and are the
same as for $W$ regarded as an random variable in $\coi$, see \refE{EwienerC}
and \refR{RBBtp}.
In particular, 
$\E W\itpx2$ is the covariance function $s\bmin t$ regarded as an element of 
$\doi\itpx2=D\xpar{\oi^2}$,
and $\E W\ptpx2$ is the same function regarded as an element of the subspace
$\doi\ptpx2$. Similarly, the fourth moment is given by the function \eqref{w4}.
\end{example}

\begin{example}
  The $\doi$-valued random variable $X=\ettaU$ in Examples \ref{Eetta} and
  \ref{Edoi} is
  not \assep, see \refE{Edoi}, and thus it has no moments in Bochner sense;
however, all moments exist in Pettis sense by \refT{TD1k} below.
The \kth{} moment is given by the function
\begin{equation*}
  \E\bigpar{X(t_1)\dotsm X(t_k)}
=\E\ett{U\le t_1,\dots,U\le t_k}
=\min\set{t_1,\dots,t_k}.
\end{equation*}
This function is continuous, and thus belongs to $C(\oi^k)\subset D(\oi^k)$.

Note that the second moment is the same as for Brownian motion
(\refE{EwienerD}), but not the fourth (or any other moment).
\end{example}

\begin{example}
  For a simple example with a discontinuous function in $D(\oi^2)$ as second
  moment, let $S_n=\sum_1^n\xi_i$ be a simple random walk, with $\xi_i$
  \iid{} and $\P(\xi_i=1)=\P(\xi_i=-1)=\frac12$. Let
  $X_n(t)\=S_{\floor{nt}}/\sqrt n$.
Then the second moment of $X_n\in\doi$ is the function in $D(\oi^2)$ given by
$\E(X_n(s)X_n(t))=\floor{n(t\bmin u)}/n$.
\end{example}

We turn to conditions for the existence of injective moments.

\begin{theorem}\label{TD1}
Suppose that $X$ is a $\cD$-\meas{}  $\doi$-valued \rv. 
Let $k\ge1$.
  \begin{romenumerate}[-10pt]
  \item \label{td1d}
$\E X\itpk$ exists in Dunford sense $\iff$ 
the weak \kth{} moment  exists  $\iff$ 
$\sup_{t\in	\oi}\E|X(t)|^k<\infty$.
  \item \label{td1p}
$\E X\itpk$ exists in Pettis sense $\iff$ 
the family
$\set{|X(t)|^k:t\in \oi}$ of \rv{s} is \ui.
  \item \label{td1b}
$\E X\itpk$ exists in Bochner sense $\iff$ 
$\E\bigpar{\sup_{t\in	\oi}|X(t)|}^k<\infty$
and there is a countable set $N\subset\oi$ such that $\supp(\gD X)\subseteq
N$ a.s.
  \end{romenumerate}

If\/ $\E X\itpk$ exists in Bochner or Pettis sense, then 
it is the function
in $\doi\itpk=\doik$ given by
\begin{equation}
  \label{td1}
\E X\itpk(t_1,\dots,t_k)
=\E\bigpar{X(t_1)\dotsm X(t_k)},
\qquad t_1,\dots,t_k\in\oi.
\end{equation}
\end{theorem}

\begin{proof}
We use $\doi=\chI$ and try to argue as in the proof of \refT{TCK1}; however,
several modifications are needed.

\pfitemref{td1d}
The forward implications are immediate as in \refT{TCK1}, but for the
remaining implication we argue somewhat differently.

Consider a linear form $\mu\in(\doi\itpk)\q=C(\hI^k)\q=\mr(\hI^k)$ with
$\norm\mu\le1$. Then $\mu$ can be regarded as an integral multilinear form
on $\doi^k$; by \refC{CDM}, $\innprod{\mu,X\tpk}$ is \meas.

A serious technical problem is that $X(t,\go)$ is in general \emph{not}
jointly \meas{} on $\hI\times\gO$, see \refR{RDw2}; hence we cannot use
Fubini's theorem as in \eqref{rogat} and \eqref{krhmf}.
We circumvent this as follows. 
Similarly to \eqref{krhmf},
\begin{equation}\label{pingst}
  \begin{split}
\bigabs{\innprod{\mu,X\tpk}}
&=	\lrabs{\int_{\hI^k}X(t_1)\dotsm	X(t_k)\dd\mu(t_1,\dots,t_k)}
\\
&\le\frac1k\int_{\hI^k}\sum_{i=1}^k|X(t_i)|^k
 \dd|\mu|(t_1,\dots,t_k)
\\
&=\int_{\hI}|X(t)|^k
 \dd\nu(t)
  \end{split}
\end{equation}
for some positive measure $\nu$ on $\hI$ with
$\norm{\nu}_{M(\hI)}=\norm{\mu}\le1$. 
We now regard $\nu$ as a continuous linear functional $\chi$ on $\chI=\doi$,
and represent it by \refC{C1} as
$f\mapsto \intoi f\dd\nu_1+\sum_{t\in\ooi}h(t)\gD f(t)$ with
\begin{equation}
\norm{\nu_1}_M+\norm{h}_{\ell^1}
\le2\norm{\nu}_{D\q}\le2.
\end{equation}
Thus, \eqref{pingst} yields
\begin{equation}
  \label{fhov}
|\innprod{\mu,X\tpk}|
\le
 \intoi |X(t)|^k\dd\nu_1(t)+\sum_{t\in\ooi}h(t)\gD |X|^k(t).
\end{equation}
Since $X(t,\go)$ is jointly \meas{} on $\oi\times\gO$ by \refL{LM1} and the
sum in \eqref{fhov} is countable, we can take the expectation in
\eqref{fhov} and use Fubini's theorem here, obtaining
\begin{equation}
  \label{ros}
\E\bigabs{\innprod{\mu,X\tpk}}
\le
 \intoi\E |X(t)|^k\dd|\nu_1|(t)
 +\sum_{t\in\ooi}|h(t)|\bigpar{\E |X(t)|^k+\E |X(t-)|^k}.
\end{equation}

Now suppose that $\sup_{t\in\oi}\E|X(t)|^k\le C$. Then Fatou's lemma applied
to a sequence 
$t_k\upto t$ yields also $\E|X(t-)|^k\le C$ for every $t>0$, and \eqref{ros}
implies 
\begin{equation}
  \label{rros}
\E\bigabs{\innprod{\mu,X\tpk}}
\le C\norm{\nu_1}_M+2C\norm{h}_{\ell^1}
\le4C.
\end{equation}
This shows that $\innprod{\mu,X\tpk}$ is integrable for every
$\mu\in(\doi\itpk)\q$, which shows that $X\itpk$ is Dunford integrable,
completing the proof of \ref{td1d}.

\pfitemref{td1p}
As in \refT{TCK1},  the forward implication follows
 by \refT{TiPettis}\ref{tipettis1}, taking
$\xx=\gd_t$, $t\in \oi$.

For the converse, apply \eqref{rros} to $\etta_E X$ , where $E\in\cF$ is an
event,  and obtain
\begin{equation}
\E\bigabs{\etta_E\innprod{\mu,X\tpk}}
\le4\sup_{t\in\oi}\E\bigpar{\etta_E|X(t)|^k},
\end{equation}
which by \eqref{ui} shows that if the family \set{|X(t)|^k} is \ui,
then the family
\set{\innprod{\mu,X\itpk}:\norm{\mu}\le1} is \ui.
Since \refT{TD3} shows that $X\itpk$ is \wassep, 
it follows by \refT{THuff}
that $\E X\itpk$ exists in Pettis sense.

\pfitemref{td1b}
Immediate by  Theorems \ref{TPIB} and \ref{TDassep}.

The final claim follows by \refT{Tdoik} and the  argument in the proof of
\refT{TCK0}, which yields \eqref{eckt1} in the present setting too.
\end{proof}

The case $k=1$ in \refT{TD1} gives the
  following characterisations of the existence of the expectation $\E X$ of a
 $\doi$-valued \rv{}.
  \begin{corollary}
	\label{CD11}
Let $X$ be a \dmeas{} \doi-valued \rv.
  \begin{romenumerate}[-10pt]
  \item \label{td11d}
$\E X$ exists in Dunford sense $\iff$ 
$\sup_{t\in	\oi}\E|X(t)|<\infty$.
  \item \label{td11p}
$\E X$ exists in Pettis sense $\iff$ 
the family
$\set{X(t):t\in \oi}$ of \rv{s} is \ui.
  \item \label{td11b}
$\E X$ exists in Bochner sense $\iff$ 
$\E\bigpar{\sup_{t\in	\oi}|X(t)|}<\infty$
and there is a countable set $N\subset\oi$ such that $\supp(\gD X)\subseteq
N$ a.s.
  \end{romenumerate}
In the Pettis and Bochner cases, 
$\E X\in\doi$ is the  function 
$t\mapsto\E X(t)$.
\nopf
  \end{corollary}

As in \refT{TCK2}, 
an \emph{even} injective moment exists in Pettis sense 
if and only if it exists in Dunford sense and it belongs to $\doi\itpk=\doik$.

\begin{theorem}\label{TD2}
Suppose that $X$ is a \dmeas{}
  $\doi$-valued \rv{} such that $\sup_{t\in	\oi}\E|X(t)|^2<\infty$.
Suppose that $k\ge2$ is even.
Then the following are equivalent.
  \begin{romenumerate}[-10pt]
  \item \label{td2p}
$\E X\itpk$ exists in Pettis sense.
  \item \label{td21}
The function $g(t)\= \E{X(t)^k}$ belongs to $\doi$, and
$g(t-)\= \E{X(t-)^k}$, $0<t\le1$.
  \end{romenumerate}
\end{theorem}

\begin{proof} 
  \ref{td2p}$\implies$\ref{td21}: 
If $t_n$ is a sequence in \oi{} and $t_n\downto t$, then $X(t_n)^k\to X(t)^k$,
while if $t_n\upto t$, then $X(t_n)^k\to X(t-)^k$.
By \refT{TD1}, $\set{|X(t)|^k:t\in \oi}$ is \ui; thus it follows that
$\E X(t_n)^k\to \E X(t)^k$ or $\E X(t_n)^k\to \E X(t-)^k$, respectively.

  \ref{td21}$\implies$\ref{td2p}:
If $t_n$ is any sequence in \oi, there exist a subsequence (still denoted
$t_n$ for convenience)
such that $t_n\to t$  for some  $t\in\oi$; we may furthermore select the
subsequence such that 
either $t_n\ge t$ or $t_n<t$ for all $n$.
In the first case, $X(t_n)\to X(t)$ and thus
$|X(t_n)|^k\to |X(t)|^k$; furthermore, by \ref{td21} and the fact that $k$ is
even, 
$ \E |X(t_n)|^k=g(t_n) \to g(t)=\E |X(t)|^k$. 
In the second case, similarly 
$|X(t_n)|^k\to |X(t-)|^k$ and
$ \E |X(t_n)|^k=g(t_n) \to g(t-)=\E |X(t-)|^k$. 
In both cases it follows that $|X(t_n)|^k$ converges  in $L^1(\P)$
(to  $|X(t)|^k$ or $|X(t-)|^k$), see \cite[Theorem  5.5.2]{Gut}.
Consequently,  $\set{|X(t)|^k:t\in \oi}$ is a relatively compact subset of
$L^1(P)$, and in particular relatively weakly compact and thus \ui{}
\cite[Theorem IV.8.11]{Dunford-Schwartz}.
Thus \ref{td2p} follows by \refT{TD1}\ref{td1p}.
\end{proof}

The extra condition $g(t-)\= \E{X(t-)^k}$ in \refT{TD2} cannot be omitted,
as seen by the following example.
\begin{example}
  Let $I_n\=[1-2^{-n},1-2^{-n-1})$ and let $X$ equal $2^{n/2}\etta_{I_n}$
	with probability $2^{-n}$, $n\ge1$.
Then $\E X(t)^2=\etta_{[1/2,1)}(t)\le1$ but $\set{X(t)^2}$ is not \ui; hence
it follows from \refT{TD1} that $\E X\itpx2$ exists in Dunford sense but not
in Pettis sense. 
Note that $g(t)\=\E X(t)^2=\etta_{[1/2,1)}(t)\in\doi$ 
but $\E{X(1-)^2}=0\neq g(1-)$.
\end{example}

For projective moments, we do not know any general necessary and sufficient
conditions for existence in Pettis or Dunford sense, but we have a simple
sufficient condition. 
\begin{theorem}\label{TD1k}
Let $X$ be a \dmeas{} \doi-valued \rv, and suppose that $\E
\norm{X}^k<\infty$. 
  \begin{romenumerate}[-10pt]
  \item \label{td1kp}
Then $\E X\ptpk$ exists in Pettis sense.
  \item \label{td1kb}
$\E X\ptpk$ exists in Bochner sense 
$\iff$ there is a countable set $N\subset\oi$ such that $\supp(\gD X)\subseteq
N$ a.s.
  \end{romenumerate}
\end{theorem}
\begin{proof}
  \pfitemref{td1kp}
Let $\ga\in L(D^k,\bbR)$ be a $k$-linear form. 
Then $\innprod{\ga,X\tpk}$ is \meas{} by \refC{CDM}, and
\begin{equation*}
\bigabs{  \innprod{\ga,X\tpk}}
\le\norm\ga \norm{X}^k.
\end{equation*}
It follows  that the family
$\bigset{\innprod{\ga,X\ptpk}:\norm\ga\le1}$ is \ui.
Moreover, $X\ptpk$ is  \wassep{} by \refT{TD3}.
Hence \refT{THuff} shows, using \refT{Tproj*}, that $\E X\ptpk$ exists in
Pettis sense. 

\pfitemref{td1kb}
This is another special case of \refT{TPIB}.
\end{proof}

For the second moment, we can as for $\ck$ use 
Grothendieck's theorem \refT{TG} 
to show that the conditions for the injective
moment in \refT{TD1} also imply the existence of the projective second
moment, thus improving \refT{TD1k} when $k=2$.
\refE{Ec3}
shows that this does not  extend to $k\ge3$.

\begin{theorem}\label{TDG}
Let $X$ be a \dmeas{} \doi-valued \rv.
  \begin{romenumerate}[-10pt]
  \item \label{tdgd}
$\E X\ptpx2$ exists in Dunford sense 
$\iff$
$\E X\itpx2$ exists in Dunford sense 
$\iff$
the weak second moment  exists  $\iff$
$\sup_{t\in	\oi}\E|X(t)|^2<\infty$.
  \item \label{tdgp}
$\E X\ptpx2$ exists in Pettis sense $\iff$ 
$\E X\itpx2$ exists in Pettis sense $\iff$ 
the family
$\set{|X(t)|^2:t\in \oi}$ of \rv{s} is \ui.
  \item \label{tdgb}
$\E X\ptpx2$ exists in Bochner sense 
$\iff$ 
$\E X\itpx2$ exists in Bochner sense 
$\iff$ 
$\E\bigpar{\sup_{t\in	\oi}|X(t)|}^2<\infty$
and there is a countable set $N\subset\oi$ such that $\supp(\gD X)\subseteq
N$ a.s.
  \end{romenumerate}
\end{theorem}

\begin{proof}
The forward implications 
follow directly, as in the proof of \refT{TCKG}, using also \refT{TDassep}.

\pfitemref{tdgd}
Let $\ga$ be a bounded bilinear form on $\doi=\chI$. By \refT{TG}, $\ga$ extends
to a bounded bilinear form on $L^2(\hI,\nu)$ for some probability measure
$\nu$ on $\hI$; more precisely, \eqref{groth} yields 
\begin{equation}
\bigabs{\innprod{\ga,X\tpx2}}
\le \kkg\norm{\ga}\int_{\hI}|X(t)|^2 \dd\nu(t).
\end{equation}
This is, apart from a constant, the same estimate as \eqref{pingst}
(proved for integral forms), and the same argument as in the proof of
\refT{TD1} yields, \cf{} \eqref{rros},
\begin{equation}  \label{cec}
\E\bigabs{\innprod{\ga,X\tpx2}}
\le\kkkkg\norm\ga \sup_{t\in\oi}\E|X(t)|^2.
\end{equation}
Furthermore, $\innprod{\ga,X\tpx2}$ is \meas{} by \refC{CDM}.
It follows that if 
$ \sup_{t\in\oi}\E|X(t)|^2<\infty$, then $\E X\ptpx2$ exists in Dunford sense.

\pfitemref{tdgp}
Assume that the family
$\set{|X(t)|^2:t\in K}$ is \ui.
By applying \eqref{cec} to $\etta_E X$ 
as in the proof of \refT{TD1}, 
we obtain from \eqref{ui}
that the family
$\bigset{\ga(X,X):\ga\in L\xpar{\doi^2;\bbR},\,\norm\ga\le1}$ is \ui.
Moreover, $X\ptpx2$ is  \wassep{} by \refT{TD3}.
Hence \refT{THuff} shows, using \refT{Tproj*}, that $\E X\ptpx2$ exists in
Pettis sense.

\pfitemref{tdgb}
This is again a special case of \refT{TPIB}.
\end{proof}

\begin{theorem}
  \label{TDD2} 
Let\/ $X$ and\/ $Y$ be \dmeas{} \doi-valued \rv{s} such that
 $\sup_{t\in \oi}|X(t)|^2<\infty$ 
and  $\sup_{t\in \oi}|Y(t)|^2<\infty$. 
Then the following are equivalent.
\begin{romenumerate}[-10pt]
\item \label{tdd2m}
$\E\ga(X,X)=\E\ga(Y,Y)$ for every  bounded bilinear form 
$\ga$ on $\doi$. 
\item \label{tdd2w}
$\E \bigpar{\xx_1(X)\xx_2(X)}=\E \bigpar{\xx_1(Y)\xx_2(Y)}$,
for any $\xx_1,\xx_2\in \doi^*$.
\item \label{tdd2t}
 $ \E \bigpar{X(t_1) X(t_2)}
=
 \E \bigpar{Y(t_1)Y(t_2)}
$
for any $t_1,t_2\in \oi$.
\item \label{tdd2p}
$\E X\ptpx2=\E Y\ptpx2$ in $\doi\ptpx2$,
with the moments existing in Dunford sense.
\item \label{tdd2i}
$\E X\itpx2=\E Y\itpx2$ in $\doi\itpx2$,
with the moments existing in Dunford sense.
\end{romenumerate}
\end{theorem}

\begin{proof} 
We argue as in the proof of \refT{TCKD2}, with some changes.
The implications \ref{tdd2m}$\implies$\ref{tdd2w}$\implies$\ref{tdd2t}
and \ref{tdd2p}$\implies$\ref{tdd2i}
are  trivial, the equivalence
\ref{tdd2m}$\iff$\ref{tdd2p} is
\refC{Cptpk=} and 
\ref{tdd2i}$\implies$\ref{tdd2w}  follows by \eqref{tid}.
It remains to show that \ref{tdd2t}$\implies$\ref{tdd2m}.

Thus, let $\ga\in L(\doi^2;\bbR)= L(\chI^2;\bbR)$. 
By \refT{TG}
there exists 
a \pmx{} $\nu$ on $\hI$ such that $\ga$ extends to
$L^2(\hI,\nu)$.
The main difference from \refT{TCKD2} is that we cannot assert that
$X(t,\go)$ is jointly 
\meas{} on $\hI\times\gO$, see \refR{RDw2}. It is, however, still possible
to regard $X$ as a map into $L^2(\hI,\nu)$. We prefer to state this slightly
differently, returning to \oi{} by the argument already used in the proof of
\refT{TD1}. 
We thus regard $\nu$ as a continuous linear functional $\chi$ on
$\chI=\doi$;
we now use the representation \eqref{c2} in \refR{Rc2} and write it as
\begin{equation}\label{d2}
\int_{\hI} f\dd\nu
=
\chi(f) = 
\intoi f(t)\dd\mu_1(t) + \intoi f(t-)\dd\mu_2(t),
\end{equation}
where $\mu_2$ is a discrete  measure supported on a countable set
$\set{t_n}\seq$. 
It is easily seen that both $\mu_1$ and $\mu_2$ are positive measures. 
(We may alternatively omit this verification and replace them by $|\mu_1|$
and $|\mu_2|$, possibly increasing $\nu$.) 

Let
$\gS\=\oi\cup\set{t_n-}\seq\subset\hI$. We define a \gsf{} 
$\cA$ on 
$\gS$ by $\cA\=\set{A\subseteq\gS:A\cap\oi\in\cB(\oi)}$, and let $\mu$ be
the measure $\mu_1+\mu_2'$ on $(\gS,\cA)$, where $\mu_2'$ is the measure on
$\set{t_n-}\seq$ given by $\mu_2'\set{t_n-}\=\mu_2\set{t_n}$.
Then \eqref{d2} can be written
\begin{equation}
\int_{\hI} f\dd\nu
=
\int_\gS f\dd\mu.
\end{equation}
Applying this to $|f|^2$, we see that 
$\norm{f}_{L^2(\hI,\nu)}=\norm{f}_{L^2(\gS,\mu)}$ for all $f\in\doi$, and
$\ga$ thus extends to a bounded bilinear form on ${L^2(\gS,\mu)}$.
Moreover, \refL{LM1} implies that $X(t,\go)$ is jointly \meas{} on
$\gS\times\gO$. Consequently, $X$ is a \Bmeas{} \rv{} in $L^2(\gS,\mu)$ by
\refL{LLp}.
Furthermore, the assumption that $\sup_t\E|X(t)|^2<\infty$ implies
by Fubini's theorem, as in \eqref{abarn}, that 
$\E\norm{X}_{L^2(\gS,\mu)}^2<\infty$, 
and thus 
$\E X\ptpx2$ exists in $L^2(\gS,\mu)\ptpx2$ in Bochner sense. The same
holds for  $\E Y\ptpx2$.

The proof is now completed as for \refT{TCKD2}, \emph{mutatis mutandis}.
\end{proof}

\begin{theorem}\label{TDX} 
Let\/ $X$ and\/ $Y$ be \dmeas{} \doi-valued \rv{s}, and suppose that 
either
  \begin{romenumerate}[-10pt]
  \item \label{tdxk}
$\E\norm{X}^k<\infty$ and $\E\norm{Y}^k<\infty$, or
  \item \label{tdx2}
$k=2$ and
$\sup\set{\E|X(t)|^2:t\in \oi}<\infty$, 
$\sup\set{\E|Y(t)|^2:t\in \oi}<\infty$.
 \end{romenumerate}
Then  
\eqref{multi=} is equivalent to \eqref{weak=}, and further to
\begin{equation}
  \E \bigpar{X(t_1)\dotsm X(t_k)}
=
 \E \bigpar{Y(t_1)\dotsm Y(t_k)},
\qquad t_1,\dots,t_k\in \oi.
\end{equation}
\end{theorem}

\begin{proof} 
\pfitemref{tdxk}
  By \refT{TD1k},
$\E X\ptpk$ and $\E Y\ptpk$ exist in Pettis sense.
Thus, $\E X\itpk$ and $\E Y\itpk$ too exist in Pettis sense,
and the result follows by \refT{Tiotak},
Corollaries \refand{Cptpk=}{Citpk=p} and \eqref{td1},
similarly as   in the proofs of Theorems
\refand{Tapprox}{TC}. 

\pfitemref{tdx2}
A simplified version of \refT{TDD2}.
\end{proof}

\begin{remark}
  As for \refT{TCKD2}, we do not know whether the condition \ref{tdxk} 
can be weakened for $k\ge3$.
\end{remark}

\section{Uniqueness}\label{Sunique} 

In the previous sections we have considered the \kth{} moment(s) for a fixed
$k$. In this section and the next, we consider the sequence of all moments.
In the present section we show that there are analogues of the classical
results for real-valued \rv{s} that the moments (under certain conditions)
determine the distribution. In \refS{Sconv} we consider convergence, where
the situation is more complicated and less satisfactory.

We suppose for simplicity that the Banach space $B$ is separable.
All \rv{s} are tacitly assumed to be (Borel) \meas, \cf{} \refT{Twm}.

We begin with two simple results on the existence of all moments.

\begin{theorem}
  \label{TXB}
Let $X$  be a $B$-valued \rv{}, where $B$ is a separable Banach space.
Then the following are equivalent.
\begin{romenumerate}[-10pt]
\item 
$\E\norm{X}^k<\infty$ for every $k\ge1$,
\item 
The projective moment $\E X\ptpk$ exists in Bochner sense for every $k\ge1$.

\item 
The injective moment $\E X\itpk$ exists in Bochner sense for every $k\ge1$.
\end{romenumerate}
\end{theorem}
\begin{proof}
  An immediate consequence of \refT{TPIB}.
\end{proof}

We do not know any general weaker
criterion for the existence of projective moments
in Pettis or Dunford sense.
(See \refT{TH2} for one case where no weaker criterion exists,
and \refT{TCKG} for a case when it does.) 
For injective moments we have the following.
\begin{theorem}
  \label{TXID}
Let $X$  be a $B$-valued \rv{}, where $B$ is a separable Banach space.
Then the following are equivalent.
\begin{romenumerate}[-10pt]
\item \label{txid=}
$\E|\innprod{\xx,X}|^k<\infty$ for every $k\ge1$,
\item \label{txidd}
The injective moment $\E X\itpk$ exists in Dunford sense for every $k\ge1$.
\item \label{txidp}
The injective moment $\E X\itpk$ exists in Pettis sense for every $k\ge1$.
\end{romenumerate}
\end{theorem}

\begin{proof}
\ref{txid=}$\iff$\ref{txidd}:  
By \refT{TID}\ref{tidd}\ref{tid1}.  

\ref{txidd}$\iff$\ref{txidp}:  
If \ref{txidd} holds, then
$\sup\set{\E|\innprod{\xx,X}|^k:\norm{\xx}\le1}<\infty$ for every $k\ge1$
by \refT{TID}\ref{tidi}.
Thus
$\sup\set{\E|\innprod{\xx,X}|^{k+1}:\norm{\xx}\le1}<\infty$,
which implies that
$\set{|\innprod{\xx,X}|^k:\norm{\xx}\le1}$ is \ui.
\refT{TiPettis}\ref{tipettis2}\ref{tip1} shows that  \ref{txidp} holds.
The converse is obvious.
\end{proof}

It is well-known that already on $\bbR$, there are \rv{s} with 
the same moments but different distributions, see \eg{} 
\cite[Section 4.10]{Gut}.
A well-known sufficient condition for the distribution of $X$ to
be uniquely determined by the moments
is the \emph{Carleman condition}
\cite{Carleman}
\begin{equation}
  \label{carleman}
\sumn \bigpar{\E|X|^{2n}}^{-1/2n}=\infty.
\end{equation}
Note that \eqref{carleman} is satisfied whenever $\E e^{t|X|}<\infty$
for some $t>0$.

\begin{theorem}\label{TU}
Let $B$ be a separable Banach space and let $X$ and $Y$ be two $B$-valued
\rv{s}.   
Suppose that $\E \norm{X}^k<\infty$ and $\E \norm{Y}^k<\infty$ for every
$k\ge1$ and that $\norm{X}$ satisfies the Carleman condition
\eqref{carleman}.
Then the following are equivalent.
\begin{romenumerate}
\item \label{tu=}
$X\eqd Y$.
\item \label{tup}
$\E X\ptpk=\E Y\ptpk$ for every $k\ge1$.
\item \label{tui}
$\E X\itpk=\E Y\itpk$ for every $k\ge1$.
\item \label{tuw}
$\innprod{\xx,X}\eqd\innprod{\xx,Y}$ for every $\xx\in B\q$.
\end{romenumerate}

\end{theorem}
\begin{proof}
Note that the moments exist in Bochner sense by \refT{TXB}.

\ref{tu=}$\implies$\ref{tup} is obvious and
  \ref{tup}$\implies$\ref{tui} follows by \refT{TPI}.

\ref{tui}$\implies$\ref{tuw}: 
If $\xx\in B\q$, then 
\begin{equation*}
\E\innprod{\xx,X}^k
=\innprod{\E X\itpk,(\xx)\tpk}
=\innprod{\E Y\itpk,(\xx)\tpk}
=\E\innprod{\xx,Y}^k  
\end{equation*}
for every
$k\ge1$ by \eqref{tid} and \ref{tui}.
Since
$\E|\innprod{\xx,X}|^k\le\norm{\xx}^k\E\norm{X}^k$, the Carleman condition
\eqref{carleman} holds for $\innprod{\xx,X}$ too, and it follows that 
\ref{tuw} holds.

\ref{tuw}$\implies$\ref{tu=}: 
This is well-known \cite{LedouxT}.
(Sketch of proof:
Any finite linear combination of elements of $B\q$ is another element of
$B\q$.
Hence \ref{tuw} implies, by the Cram\'er--Wold device, that \ref{tuw}
holds jointly for any finite number of functionals $\xx$.
A standard application of the monotone class theorem shows that 
$\P(X\in A)=\P(Y\in A)$ for every $A\in\bw=\cB$.)
\end{proof}

\begin{remark}
  The proof of \refT{TU} shows that it suffices that the Carleman condition
  holds for each $\innprod{\xx,X}$. Moreover, if we only consider injective
  moments, and the equivalences \ref{tu=}$\iff$\ref{tui}$\iff$\ref{tuw},
then the moment assumptions may be weakened to 
$\E|\innprod{\xx,X}|^k<\infty$ and $\E|\innprod{\xx,Y}|^k<\infty$ for every
$k\ge1$, with
 the injective moments existing in Pettis sense by \refT{TXID}.
\end{remark}

\begin{remark}
The assumption that $B$ is separable is essential; the equivalence
  \ref{tu=}$\iff$\ref{tuw} is not true in general for non-separable $B$. 
\refE{EPD3} gives an $X$ in $\ell^2\oi$ such that $\innprod{\xx,X}=0$ \as{}
for every $\xx\in\ell^2\oi\q=\ell^2\oi$, and thus 
$\innprod{\xx,X}\eqd\innprod{\xx,Y}$ with $Y=0$, although $X\neq0$ a.s. 
\end{remark}

\section{Convergence}\label{Sconv} 

As in the preceding section, we assume that the Banach space $B$ is separable,
and that all \rv{s} are  (Borel) \meas.

We consider a sequence $X_n$, $n\ge1$ of $B$-valued \rv{s}, and a potential
limit $X$. For definition and general properties of convergence in
distribution, denoted $X_n\dto X$, see \cite{Billingsley}.
In particular, recall that convergence in distribution can be described by a
metric, at least 
when $B$ is separable as here, see \cite[Theorem 6.8]{Billingsley}.
(The non-separable case is more complicated and related to the existence of
real-measurable cardinals, \cf{} \refR{Rcardinals},
see \cite[Appendix III]{Billingsley1}.)

As in the real-valued case, 
convergence in distribution implies convergence of the moments, provided
some suitable integrability condition holds uniformly.
We state one simple such result.

\begin{theorem}\label{Tem}
  Suppose that $X$ and $X_n$, $n\ge1$, are $B$-valued \rv{s}, where $B$ is a
  separable Banach space.
Suppose that $X_n\dto X$, and that $\sup_n\E\norm{X_n}^k<\infty$ for every
$k\ge1$. Then
$\E X_n\ptpk\to\E X\ptpk$ in $B\ptpk$
and $\E X_n\itpk\to\E X\itpk$ in $B\itpk$ as \ntoo, for every $k\ge1$.
\end{theorem}

\begin{proof}
By \refT{TXB}, all moments $\E X_n\ptpk$ and $\E X_n\itpk$ exist.
Furthermore, $\norm{X_n}\dto\norm{X}$, and thus
by Fatou's lemma (for convergence in distribution \cite[Theorem 5.5.8]{Gut}), 
$\E\norm{X}^k\le\liminf_\ntoo\E\norm{X_n}^k<\infty$. Hence  
all moments $\E X\ptpk$ and $\E X\itpk$ also exist.

By the Skorohod representation theorem \cite[Theorem 6.7]{Billingsley},
we may assume that $X_n\asto X$.
Then $X_n\ptpk\asto X\ptpk$, since the (non-linear) mapping $x\mapsto
x\tpk$ is continuous $B\to B\ptpk$.
Thus, $\norm{X_n\ptpk-X\ptpk}\asto0$.
Furthermore, 
\begin{equation}
  \begin{split}
\norm{X_n\ptpk-X\ptpk}
&\le
\norm{X_n\ptpk}+\norm{X\ptpk}
=
\norm{X_n}^k+\norm{X}^k,
  \end{split}
\end{equation}
and
since $\sup_n\E\norm{X_n}^{2k}<\infty$ and $\E\norm{X}^k<\infty$, it follows
that the family $\set{\norm{X_n\ptpk-X\ptpk}: n\ge1}$ is uniformly
integrable (for any fixed $k$), 
see \eg{} \cite[Theorems 5.4.2--4.6]{Gut};
hence
$\E\norm{X_n\ptpk-X\ptpk}
\to0$.
Consequently, 
\begin{equation*}
\norm{\E X_n\ptpk-\E X\ptpk}
=
\norm{\E(X_n\ptpk-X\ptpk)}
\le \E\norm{X_n\ptpk-X\ptpk}
\to0,
\end{equation*}
as \ntoo.
This proves $\E X_n\ptpk\to\E X\ptpk$, and $\E X_n\itpk\to\E X\itpk$ follows 
similarly, or by
\refT{TPI}. 
\end{proof}

If $B$ is finite-dimensional, then the converse to \refT{Tem} holds,
provided the moments determine the distribution of $X$ (for example by the
 \refT{TU}); this is the standard method of moments (in several variables).

In infinite dimensions, there is in general no converse.
We begin with a simple example showing that convergence of the injective
moments does not imply convergence in distribution.
Moreover, the example shows that  convergence of the injective moments
does not imply  convergence of the projective moments. 
(The converse implication is trivial by \refT{TPI}.)

\begin{example}
  \label{Egauss}
Regard $\bbR^n$ as a subspace of $\ell^2$ by the isometric embedding 
$(a_1,\dots,a_n)\mapsto(a_1,\dots,a_n,0,0,\dots)$.
Let $X_n\=n\qqw(\xi_1,\dots,\xi_n)\in \bbR^n\subset\ell^2$,
where $\xi_1,\xi_2,\dots{}\sim N(0,1)$ are \iid{} standard normal \rv{s}.

Note first that
$\norm{X_n}^2=\frac1n\sumin\xi_i^2\to1$ \as{} 
as \ntoo{} by the law
of large numbers. Thus $X_n$ does not tend to 0 in distribution.

Next, consider the injective moment $\E X_n\itpk$.
(It does not matter whether we regard $\E X_n\itpk$ as an element of
$(\bbR^n)\itpk$ or $(\ell^2)\itpk$, since
$(\bbR^n)\itpk$ is isometrically a subspace of $(\ell^2)\itpk$, 
see \refR{Rinjinj}.)

If $y=(\mxx a,)\in\bbR^n$, then 
$\innprod{y,X}=\sumin a_i n\qqw\xi_i$ is normal with mean 0 and variance
$\sumin a_i^2n\qw\E\xi_i^2=\norm{y}^2/n$, \ie{}
$\innprod{y,X}\sim N(0,\norm{y}^2/n)$.
Hence, for any $y_1,\dots,y_k\in\bbR^n$,
by \Holder's inequality,
\begin{equation*}
  \begin{split}
\abs{\innprod{\E X_n\itpk,\mx y\tensor}}
&
=\bigabs{\E\bigpar{\innprod{X,y_1}	\dotsm \innprod{X,y_k}}}
\\&
\le \prodik \xpar{\E|\innprod{X,y_i}|^k}^{1/k}
\le C_k\prodik \frac{\norm{y_i}}{\sqrt n}
  \end{split}
\end{equation*}
for some constant $C_k\=\E|\xi_1|^k$. Since $B\itpk$ can be seen
(isometrically) as a subspace of $L(B^k;\bbR)$, see \refS{Stensor}, it
follows that
\begin{equation}
  \norm{\E X_n\itpk} \le C_k n^{-k/2}.
\end{equation}
In particular, the injective moments $\E X_n\itpk\to0$ as \ntoo{} 
for every $k\ge1$.

Finally, the 
second moment $\E X_n\tpx2$ is given by the covariance matrix,
\begin{equation}
  \E X_n\tpx2=\Bigpar{\frac1n\E\xi_i\xi_j}_{i,j=1}^n
=\frac1n I,
\end{equation}
where $I$ is the identity matrix. Regarded as an operator, $I$ is the
identity operator in $\bbR^n$, which has trace norm
$\norm{I}_{\trx(\bbR^n)}=n$.
Since the projective tensor norm equals the trace norm (for Hilbert spaces),
see \refT{THtp}, we obtain
$\norm{\E X_n\ptpx2}=\norm{\frac1n I}_{\trx(\bbR^n)}=1$.
(Since there exists a projection $\ell^2\to\bbR^n$ of norm 1, it does not
matter whether we regard $\E X_n\ptpx2$ as an element of 
$(\bbR^n)\ptpx2$ or $(\ell^2)\ptpx2$.)
Hence the projective moments $\E X_n\ptpx2$ do not tend to 0.
\end{example}

In fact, this extends to every infinite-dimensional Banach space.

\begin{theorem}\label{T0i}
  Let $B$ be any infinite-dimensional Banach space.
Then there exists a sequence of \Bmeas{} \rv{s} $X_n$ in $B$ such that
 the injective moments $\E X_n\itpk\to0$ as \ntoo{} 
for every $k\ge1$ but $X_n$ does not tend to $0$;
in fact, $\norm{X_n}\pto1$.
\end{theorem}

\begin{proof}
Let $\eps_n\to0$ (for example, $\eps_n\=1/n$).
By  Dvoretzky's theorem, 
see \eg{} \cite{LedouxT} or \cite{Pisier-volume},
for every $n$, there is a $n$-dimensional subspace
$B_n$ of $B$ such that $B_n$ is isomorphic to $\bbR^n$ by an isomorphism
$T_n:\bbR^n\to B_n$ with $\norm{T_n},\norm{T_n\qw}\le1+\eps_n$.

Let $X'_n\in \bbR^n$ be as $X_n$ in \refE{Egauss}, and let
$X_n\=T_n X'_n\in B$.
\end{proof}

\begin{remark}\label{R0i}
  It is easy to see that we may replace $X_n$ in \refE{Egauss} or \refT{T0i}
  by $X_n/\norm{X_n}$, thus obtaining $\norm{X_n}=1$ \as{} and $\E
  X_n\itpk\to0$ for every $k\ge1$. (In \refE{Egauss}, this means taking
  $X_n$ uniformly distributed on the unit sphere of $\bbR^n$.)
\end{remark}

For projective moments, the situation is more complicated.
We next give another example (with $B=c_0$), showing that also convergence
of the projective 
moments does not imply convergence in distribution in general.
On the other hand, we then show that in a Hilbert space, it does
(assuming a Carleman condition).
Moreover, we shall show that in a Hilbert space,
even \emph{weak} convergence
(denoted $\wto$)
of the projective 
moments suffices to imply convergence in distribution. 

\begin{example}  \label{Ec0lim}
Let $B=c_0$ and let $X_n$ be the $B$-valued \rv{} given by
$\P(X_n=e_i)=1/n$, $i=1,\dots,n$. Then $\norm{X_n}=1$.

If $\ga$ is a $k$-linear form on $c_0$, then by \eqref{bh}, 
using the notation \eqref{nobh} and the \CSineq,
\begin{equation*}
  \begin{split}
|\innprod{\ga,\E X\ptpk}|
&=
|\E\innprod{\ga,X\ptpk}|
=\lrabs{\frac1n\sumin\ga(e_i,\dots,e_i)}
\le\Bigpar{\frac1n \sumin|a(i,\dots,i)|^2}\qq
\\&
\le
\Bigpar{\frac1n\sum_{s_1,\dots,s_k\in \bbN}|a(s_1,\dots,s_k)|^2}\qq
\le n\qqw\norm{\ga}	.
  \end{split}
\end{equation*}
Hence, using \refT{Tproj*},
\begin{equation*}
  \norm{X_n\ptpk}\le n\qqw\to0\qquad
\text{as \ntoo}.
\end{equation*}

Consequently, $\E X_n\ptpk\to0$ for every $k$, 
but
$X_n\not\pto0$.
\end{example}

\begin{theorem}
  \label{Tmag}
Suppose that $X$ and $X_n$, $n\ge1$, are $H$-valued \rv{s}, where $H$ is a
  separable Hilbert space.
Suppose further that $\norm{X}$ satisfies the Carleman condition
\eqref{carleman}.
If\/ 
$\E X_n\ptpk\wto \E X\ptpk$ in $B\ptpk$ for every $k\ge1$, with the moments
existing in Pettis sense, then $X_n\dto X$.
\end{theorem}

\begin{proof}
  For any fixed $y\in H$, and any $m\ge1$, 
  \begin{equation*}
\norm{x-y}^{2m}
=\innprod{x-y,x-y}^m	
=\sum_{k=0}^{2m}\ga_k(x,\dots,x),
  \end{equation*}
where $\ga_k$ is some bounded $k$-linear form (depending on $y$ and $m$).
Hence, if $\E X_n\ptpk\wto\E X\ptpk$ for every $k$, then
  \begin{equation}\label{olof}
	\begin{split}
\E\norm{X_n-y}^{2m}
&=\sum_{k=0}^{2m}\E\innprod{\ga_k,X_n\tpk}
=\sum_{k=0}^{2m}\innprod{\ga_k,\E X_n\tpk}
\\&
\to
\sum_{k=0}^{2m}\innprod{\ga_k,\E X\tpk}
=\E\norm{X-y}^{2m}	  .
	\end{split}
  \end{equation}
Since $\norm{X}$ satisfies the Carleman condition, it is straightforward to
show that $\norm {X-y}\le\norm X+\norm y$ does too. It thus follows from
\eqref{olof}, by the method of moments, that
\begin{equation}\label{hood}
\norm{X_n-y}^2\dto\norm{X-y}^2.  
\end{equation}
(We use here the fact that for positive random variables, the Carleman
condition can be relaxed to \eqref{carleman} for the square root, see
\cite[(4.10.2)]{Gut}. Alternatively, we can introduce random signs and apply
the method of moments to show that $\pm\norm{X_n-y}\dto\pm\norm{X-y}$,
where all odd moments vanish and thus converge trivially.)

The argument extends to any linear combination of $\norm{X_n-y_1}^2,\dots,
\norm{X_n-y_\ell}^2$ for any given $y_1,\dots,y_\ell\in H$; 
hence
\eqref{hood} holds with joint convergence for any finite set
$y_1,\dots,y_\ell\in H$.  

Consequently, if $A\subseteq B$ is a finite intersection of open balls
$B(y_i,r_i)\=\set{x:\norm{x-y_i}<r_i}$ such that 
$\P(X\in\partial B(y_i,r_i))=0$, then $\P(X_n\in A)\to\P(X\in A)$;
this implies $X_n\dto X$ by
\cite[Theorem 2.4]{Billingsley}.
\end{proof}

\begin{remark}
  The argument extends to the spaces $L^p(\mu)$ (assumed to be separable)
  provided $p$ is  an even integer (and $\norm{X}^{p/2}$ satisfies the
  Carleman condition). We do not know whether there are further Banach
  spaces such that
 $\E X_n\ptpk\to \E X\ptpk$ for every $k$ implies $X_n\dto  X$
(provided $X$ is bounded, say, for simplicity).

Specialising to $X=0$, we have the related problem:
For which Banach spaces $B$ does 
 $\E X_n\ptpk\to 0$ in $B\ptpk$ for every $k$ imply $X_n\pto0$?
\end{remark}

Next we  show that \emph{weak} convergence of the injective moments
is equivalent to
\emph{weak convergence in distribution}, 
meaning convergence in distribution of
$\innprod{\xx,X_n} $ for every $\xx\in B\q$.

\begin{theorem}\label{Temm}
  Suppose that $X$ and $X_n$, $n\ge1$, are $B$-valued \rv{s}, where $B$ is a
  separable Banach space.
Suppose further that 
\begin{equation}
  \label{temm}
\sup_{n\ge1}\E|\innprod{\xx,X_n}|^k<\infty
\end{equation}
for every $\xx\in B\q$ and every $k\ge1$,
and that every $\innprod{\xx,X}$, $\xx\in B\q$, satisfies the Carleman condition
\eqref{carleman}. 
Then the following are equivalent.
\begin{romenumerate}[-10pt]
\item \label{temma}
$\E X_n\itpk\wto \E X\itpk$ in $B\itpk$ for every $k\ge1$.
\item \label{temmb}
$\E\bigpar{\xx_1(X_n)\dotsm\xx_k(X_n)}\to\E\bigpar{\xx_1(X)\dotsm\xx_k(X)}$
for every $k\ge1$ and $\mxx\xx,\in B\q$.
(In other words, the weak moments converge pointwise.)
\item \label{temmc}
$\E\bigpar{\xx(X_n)^k}\to\E\bigpar{\xx(X)^k}$
for every $k\ge1$ and $\xx\in B\q$.
\item \label{temmd}
$\xx(X_n)\dto\xx(X)$
for every $\xx\in B\q$.
\item \label{temme}
$\xx(X_n)\dto\xx(X)$ jointly
for all $\xx\in B\q$.
\end{romenumerate}
\end{theorem}

Note that all (injective) moments exist in Pettis sense by 
\refT{TXID} together with \eqref{temm} and the Carleman condition for
$\innprod{\xx,X}$ (which implies that 
$\E|\innprod{\xx,X}|^k<\infty$ for every $k$).

\begin{proof}
  By assumption, for every $k\ge1$, 
$T:\xx\mapsto(\innprod{\xx,X_1},\innprod{\xx, X_2},\dots)$ 
maps $B\q$ into $\ell^\infty(L^k(\P))$.
By the closed graph theorem, $T$ is bounded, and thus
\begin{equation}
  \label{sax}
C_k\=
\sup\bigset{\E|\innprod{\xx,X_n}|^k: \norm{\xx}\le1,\; n\ge1}
<\infty.
\end{equation}

Note that, 
by definition, \ref{temma} holds if and only if
\begin{equation}\label{eemma}
  \innprod{\chi,\E X_n\itpk}
=
\E  \innprod{\chi, X_n\itpk}
\to
\E  \innprod{\chi, X\itpk}
=
  \innprod{\chi,\E X\itpk}
\end{equation}
for every $\chi\in(B\itpk)\q$ and every $k\ge1$.

\ref{temma}$\implies$\ref{temmb}:
Choosing $\chi=\mx\xx\tensor$ in \eqref{eemma}, \ref{temmb} follows by
\eqref{tid}. 

\ref{temmb}$\implies$\ref{temma}:
Let $\chi\in(B\itpk)\q$.
By \refT{Tinj*}, 
$\chi$ has a representation \eqref{tinj*}, 
and as shown in the proof of \refT{TID}, 
\eqref{thx} holds, together with the corresponding formula with $X$ replaced
by $X_n$. Thus,  \eqref{eemma} can be written
\begin{multline}\label{eeemma}
\int_{K^k}\E\bigpar{\xx_1(X_n)\dotsm\xx_k(X_n)}
\dd\mu(\xx_1,\dots,\xx_k)
\\
\to
\int_{K^k}\E\bigpar{\xx_1(X)\dotsm\xx_k(X)}
\dd\mu(\xx_1,\dots,\xx_k).
\end{multline}
The integrand converges pointwise by \ref{temmb}; furthermore,
by \Holder's inequality and \eqref{sax},
$\bigabs{\E\bigpar{\xx_1(X_n)\dotsm\xx_k(X_n)}}\le C_k$.
Consequently, \eqref{eeemma} holds by dominated convergence.

\ref{temmb}$\implies$\ref{temmc}:
A special case, obtained by taking $\xx_1=\dots=\xx_k=\xx$.

\ref{temmc}$\implies$\ref{temmb}:
A standard polarisation argument.
Given $\mxx\xx,$, use \ref{temmc} with $\xx\=t_1\xx_1+\dots+t_k\xx_k$, where
$\mxx t,$ are real numbers. Then both sides of \ref{temmc} are (homogeneous)
polynomials in $\mxx t,$, and since the left side converges to the right for
every $\mxx t,$, the coefficient of $\mx t{}$ converges too, which yields
\ref{temmb} (after dividing by $k!$).

\ref{temmc}$\implies$\ref{temmd}:
This is the usual method of moments for the real-valued \rv{s}
$\innprod{\xx,X_n}$, using  the Carleman condition.

\ref{temmd}$\implies$\ref{temmc}:
For any fixed $k$, the \rv{s} $\innprod{\xx,X_n}^k$ are \ui, by \eqref{temm}
with $k+1$. Hence \ref{temmd}$\implies$\ref{temmc}.

\ref{temmd}$\iff$\ref{temme}:
The joint convergence in \ref{temme} means, by definition, joint convergence
for any finite set $\mxx\xx,\in B\q$. Since any linear combination of
$\mxx\xx,$
is another element of
$B\q$, this follows from \ref{temmd} by the Cram\'er--Wold device.
(Cf.~the proof of \refT{TU}.)
\end{proof}

\begin{remark}\label{Remm} 
If all (injective) moments exist in Pettis sense and \refT{Temm}\ref{temma}
holds, then  
\ref{temmc} holds by \eqref{tid}, see the proof above, and thus \eqref{temm}
holds for every even $k$, and thus for every $k$, so \eqref{temm} is
  redundant in this case.
\end{remark}

\begin{example}
Take $B=\ell^1$ and  let $X_n$ be as in \refT{T0i} and \refR{R0i}. 
Then the injective moments $\E X_n\itpk\to0$, and \refT{Temm} shows that 
$\innprod{\xx,X_n}\dto0$ as \ntoo{}
for every $\xx\in(\ell^1)\q$;
equivalently,
$\innprod{\xx,X_n}\pto0$ as \ntoo{}.
However, $\norm{X_n}=1$ and thus $X_n\not\pto0$.

By Schur's theorem \cite[p.~85]{Diestel},
a sequence in $\ell^1$ converges weakly if and only
if it converges strongly (\ie, in norm). We see that this does not extend to
convergence in probability (or distribution) 
for sequences of random variables in $\ell^1$.
This also shows that Skorohod's representation theorem does not hold for
weak convergence in distribution: there is no way to couple the \rv{s} $X_n$
such that $X_n\wto 0$ a.s., since this would imply $X_n\to0$ \as{} by
Schur's theorem, and thus $X_n\pto0$.
\end{example}

If we know tightness by other means, {weak convergence in distribution} 
is equivalent to convergence in distribution.

\begin{corollary}\label{Cemm} 
  Suppose that $X$ and $X_n$, $n\ge1$, are $B$-valued \rv{s}, where $B$ is a
  separable Banach space.
Suppose further that all injective moments 
$\E X_n\itpk$ exist in Pettis sense, that
$\E X_n\itpk\wto \E X\itpk$ in $B\itpk$ for every $k\ge1$,
that the sequence $X_n$ is tight, 
and that every $\innprod{\xx,X}$, $\xx\in B\q$, satisfies the Carleman condition
\eqref{carleman}. 
Then $X_n\dto X$.
\end{corollary}

\begin{proof}
By \refT{Temm}   and \refR{Remm}, $\xx(X_n)\dto\xx(X)$ for every $\xx\in B\q$.
Since $X_n$ is tight, every subsequence has a subsubsequence that converges
in distribution to some \rv{} $Y$ in $B$ 
\cite{Billingsley}.
Then, along the subsubsequence, $\xx(X_n)\dto\xx(Y)$ for every $\xx\in B\q$,
and thus 
$\xx(Y)\eqd\xx(X)$, which implies $Y\eqd X$, see \refT{TU}.
Hence every subsequence has a subsubsequence converging (in distribution) to
$X$, which 
implies that the full sequence converges.
\end{proof}

In Hilbert spaces, we can use the second moment to deduce
tightness. 
We regard as usual the second moments as operators on $H$; recall that
they always are positive operators by \refT{TH>0}.
Recall also that if the second injective moment $\E X\itpx2$ exists (in any
sense, \eg{}  Dunford)
and is a trace class operator, then the projective moment  $\E X\ptpx2$ 
exists too, in Bochner (and thus Pettis) sense by \refT{TH2j}.
(And conversely, see \refC{CH2j}.) Moreover, the projective and injective
second moments are then given by the same operator, so it does not matter
which of them we use.

We identify $ H\ptensor H$ with the space $\trx(H)$ of nuclear (trace class)
operators on $H$, see \refS{SH}.
For operators $T,U$ in a Hilbert space $H$, we let as usual $T\le U$ mean
$\innprod{Tx,x}\le \innprod{Ux,x}$ for every $x\in H$; in particular,
$T\ge0$ ($T$ is positive) if $\innprod{Tx,x}\ge0$.

\begin{theorem}  \label{THtightT}
Let $H$ be a separable Hilbert space, and let \set{X_\ga:\ga\in A} be a
family of $H$-valued \rv{s}. 
Suppose that there is a nuclear operator $T\in H\ptensor H=\trx(H)$ such
that
$\E X_\ga\itpx2\le T$ for every $\ga\in A$, with the moment existing in
Dunford sense.
Then
\set{X_\ga:\ga\in A} is tight.
\end{theorem}

\begin{proof}
  By the spectral theorem for compact self-adjoint operators,
\eg{} \cite[Corollary II.5.4]{Conway}, 
there exists an ON basis $(e_n)$ in $H$ such that 
\begin{equation}\label{spectral}
  T=\sumn \gl_n e_n\tensor e_n,
\end{equation}
where $\gl_n$ is the eigenvalue corresponding to $e_n$, $\gl_n\ge0$ because
$T\ge\E X_\ga\itpx2\ge0$, and $\sumn\gl_n=\norm{T}_{\trx(H)}<\infty$.

Choose a positive sequence $a_n\to\infty$ such that $\sumn a_n\gl_n<\infty$,
and define
\begin{equation*}
\nnorm{x}^2\=\sumn a_n|\innprod{X,e_n}|^2\le\infty.  
\end{equation*}
Define $K_r\=\set{x\in H:\nnorm{x}\le r}$. It is well
known (and easy to see, \eg{} using \cite[Lemma 2.2]{LedouxT}) that each
$K_r$ is a compact subset of $H$. 
Moreover, for each $\ga$, using \eqref{tid} and \eqref{spectral},
\begin{equation*}
  \begin{split}
  \E\nnorm{X_\ga}^2
&=\sumn a_n\E|\innprod{X_\ga,e_n}|^2
=\sumn a_n\innprod{\E X_\ga\ptpx2 e_n,e_n}
\\&
\le \sumn a_n\innprod{T e_n,e_n}
=\sumn a_n\gl_n<\infty.	
  \end{split}
\end{equation*}
Hence, by Markov's inequality, $\P(X_\ga\notin K_r)\le \sumn
a_n\gl_n/r^2\to0$ as $r\to\infty$, uniformly in $\ga\in A$, which shows that
\set{X_\ga} is tight.
\end{proof}

\begin{theorem}  \label{THtight} 
Let $H$ be a separable Hilbert space, and let \set{X_\ga:\ga\in A} be a
family of $H$-valued \rv{s}. If the family  
\set{\E X_\ga\ptpx2:\ga\in A} of projective second moments is relatively
compact in $H\ptensor H=\trx(H)$, 
with the moments existing in Pettis sense,
then 
\set{X_\ga:\ga\in A} is tight.
\end{theorem}
\begin{proof}
Consider any sequence $(X_{\ga_n})_{n=1}^\infty$ with $\ga_n\in A$, $n\ge1$.
By the compactness assumption, there is a subsequence, which we simply
denote by $(X_n)$, such that $\E X_n\ptpx2\to V$ in $H\ptensor H=\trx(H)$
for some $V\in \trx(H)$. By taking a further subsequence, we may assume that
\begin{equation}
  \label{mr1}
\norm {\E X_n\ptpx2-V}_{\trx(H)} < 2^{-n}.
\end{equation}
Let $T_n\=\E X_n\ptpx2-V$; this is a symmetric nuclear operator and the
corresponding positive operator $|T_n|\=(T_n\q T_n)\qq$ satisfies
\begin{equation}\label{mr2}
  \norm{\,|T_n|\,}_{\trx(H)} =  \norm{T_n}_{\trx(H)} <2^{-n}.
\end{equation}
Define $T\=V+\sumn |T_n|\in \trx(H)$, where the sum converges in $\trx(H)$
by \eqref{mr2}.
Note that for any $x\in H$,
\begin{equation*}
  \innprod{\E X_n\ptpx2 x,x}
=   \innprod{T_n x,x} +  \innprod{V x,x}
\le  \innprod{|T_n| x,x} +  \innprod{V x,x}
\le \innprod{T x,x}.
\end{equation*}
Thus, $\E X_n\ptpx2 \le T$.
By \refT{THtightT},
the sequence $(X_n)$ is tight, and thus there is a subsequence that converges
in distribution.

We have shown that every sequence $(X_{\ga_n})$ has  a subsequence that
converges in distribution; this shows that \set{X_\ga} is tight.
\end{proof}

This leads to the following convergence criterion, combining the second
projective and arbitrary injective (or, equivalently, weak) moments.
Compare Theorems \refand{T0i}{Tmag}.

\begin{theorem}
  \label{THconv}
Suppose that $X$ and $X_n$, $n\ge1$, are $H$-valued \rv{s}, where $H$ is a
  separable Hilbert space.
If\/ 
$\E X_n\ptpx2\to \E X\ptpx2$ in $H\ptpx2=\trx(H)$ and
$\E X_n\itpk\wto \E X\itpk$ in $H\itpk$ for every $k\ge1$, with all moments
existing in Pettis sense, 
and furthermore every $\innprod{y,X}$, $y\in H$, satisfies the Carleman
condition \eqref{carleman},
then $X_n\dto X$.
\end{theorem}

\begin{proof}
  Since the sequence $\E X_n\ptpx2$ converges, it is relatively compact, and
  thus \refT{THtight} shows that the sequence $(X_n)$ is tight. The result
  follows from \refC{Cemm}.
\end{proof}

We do not know whether there are similar results for other Banach 
spaces. \refE{Ec0lim} shows that convergence of moments is not enough to
provide tightness in general.
Note that a commonly used sufficient condition for tightness of a family
$\set{X_\ga}$ in \coi, assuming $\set{X_\ga(0)}$ tight, is that 
$\E(X_\ga(s)-X_\ga(t))^4\le C |s-t|^\gb$ for some $C<\infty$,
$\gb>1$ (typically, $\gb=2$)
and all $s,t\in\oi$.
(See \cite[Theorem  12.3 and (12.51)]{Billingsley1} for a more
general result, 
and \cite[Theorem 15.6]{Billingsley1} for a similar result for $\doi$.)
By expanding the fourth power and using \eqref{eckt}, 
this can be seen as a continuity condition on the fourth
moments $\E X_\ga\itpx4\in C(\oi^4)$. 
This suggests that also for other spaces, it might be possible to find
tightness criteria using suitable subspaces of $B\itpk$ or $B\ptpk$.
We have, however, not explored this further.

\appendix

\section{The reproducing Hilbert space}  \label{Srepr}
In this appendix 
(partly based on \cite[Chapter 8]{LedouxT})
we  study a construction closely related to the 
injective second moment, and explore the connection.
We suppose that 
$X$ is \wmeas{} and furthermore that
$\xx(X)\in L^2(\P)$ for every $\xx\in B\q$;
this is by \refL{Lweak} equivalent to the existence of
the weak second moment $\E\bigpar{\xx_1(X)\xx_2(X)}$, $\xx_1,\xx_2\in B\q$.
Furthermore, this holds whenever the injective second moment $\E X\itpx2$
exists in Dunford 
sense, and the converse holds under weak conditions, for example when $B$ is
separable, see \refT{TID}.

By \refR{RPpq}, then $T_X:B\q\to L^2(\P)$ and $T_X\q:L^2(\P)\to B\qx$,
and thus the composition $T_X\q \Txx:B\q\to B\qx$.
This operator is characterised by
\begin{equation}
  \label{q1}
\innprod{T_X\q \Txx\xx,\yy}
=
\innprod{T_X\xx,T_X\yy}
=\E\bigpar{\xx(X)\yy(X)}.
\end{equation}
In other words, $T_X\q\Txx:B\q\to B\qx$ is the operator corresponding to the
weak moment, seen as a bilinear form $B\q\times B\q\to\bbR$.

By \eqref{tx}, \eqref{txq} and \refR{RPpq},
or by \eqref{dunford} and \eqref{q1}, 
\begin{equation}
  \Txq\Txx\xx=\E\bigpar{\xx(X)X},
\end{equation}
with the expectation existing in Dunford sense.

\begin{lemma}\label{Lr2}
  Suppose that $\xx(X)\in L^2(\P)$ for all $\xx\in B\q$.
Then the following are equivalent.
\begin{romenumerate}
\item\label{lr2p} 
$X$ is Pettis integrable, \ie, $\E X$ exists in Pettis sense.
\item \label{lr2b}
$\Txq:L^2(\P)\to B$.
\item\label{lr2bb}
 $\Txq\Txx:B\q\to B$.
\end{romenumerate}
\end{lemma}
\begin{proof}
  \ref{lr2p}$\iff$\ref{lr2b}:
By Remarks \refand{Rpettis}{RPpq}.

  \ref{lr2b}$\implies$\ref{lr2bb}: Trivial.

  \ref{lr2bb}$\implies$\ref{lr2b}:
If \ref{lr2bb} holds, then $T_X\q$ maps $\im(\Txx)$ into $B$, and thus
$T_X\q(\overline{\im(\Txx)})\subseteq B$. Since $\im(\Txx)^\perp=\ker(T_X\q)$,
it follows that $T_X\q:L^2(\P)\to B$.
\end{proof}

By \refR{RPpq}, the assertions \ref{lr2p}--\ref{lr2bb} in \refL{Lr2} hold
whenever $B$ is separable, and more generally when $X$ is \wassep.

Recall from \refS{Stensor} that $B\itensor B$ can be regarded as a subspace
of $L(B\q;B)$.
If the injective second moment $\E X\itpx2$ exists
in Pettis sense,
then
by \eqref{keiller}, \eqref{tid} and \eqref{q1},
$\E X\itpx2\in B\itensor B$ corresponds to the operator $T_X\q \Txx$.
(In particular, $\Txq\Txx\in L(B\q,B)$ so $X$ is Pettis integrable by
\refL{Lr2}.) 
Hence $T_X\q \Txx$ can  be seen as a form of the injective second moment.

\begin{remark}
Note that if, for example, $B$ is separable, then 
$T_X\q \Txx\in L(B\q,B)$ by \refL{Lr2}. However, it does not always
correspond to an element of $B\itensor B$. For example, 
if $B=c_0$ and $X$ is as in \refE{EPDc0}, then  $\E X\itpx2$ is the infinite
diagonal matrix 
$(p_na_n^2\gd_{mn})_{m,n=1}^\infty$, and the corresponding operator
$\Txq\Txx:c_0\q=\ell^1\to c_0$ is the multiplication operator
$(b_n)\nnn\mapsto (p_na_n^2b_n)\nnn$. 
Choosing $a_n$ such that $p_na_n^2=1$, 
$\E X\ptpx2$ is thus the (infinite) identity matrix and
$\Txq\Txx$ is the inclusion map
$\ell^1\to c_0$; 
in this case $\E X\itpx2$ exists in Dunford  sense but not in
Pettis sense by \refE{EPDc0}, and $\Txq\Txx\in L(c_0\q,c_0)$ but
$\Txq\Txx\notin c_0\itensor c_0$.
(Recall from \refT{Tc0tp} that $c_0\itensor c_0=c_0(\bbN^2)$, so the
identity matrix is not an element of $c_0\itensor c_0$.)   
\end{remark}

Assume in the remainder of this appendix also that the assertions of 
Lemma \ref{Lr2}
hold. (Recall that this is the case for example if $B$ is separable, or if 
$\E X\itpx2$ exists in Pettis sense.)
Thus
$T_X:B\q\to L^2(P)$,
$T_X\q:L^2(\P)\to B$ and
$T_X\q \Txx:B\q\to B$.
Furthermore,  $\im(\Txx)^\perp=\ker(T_X\q)$
and thus 
$\overline{\im(\Txx)}=\ker(T_X\q)^\perp$.

$T_X\q$ induces an bijection of  
$\ker(T_X\q)^\perp$ onto $\im(T_X\q)\subseteq B$.
Let $H_X$ be $\im(T_X\q)$ equipped with the inner product 
induced by this bijection, \ie,
\begin{equation}
  \innprod{x,y}_{H_X}\=\innprod{(T_X\q)\qw x, (T_X\q)\qw y}_{L^2(\P)}
\end{equation}
where $(T_X\q)\qw:H_X\to \ker(\Txq)^\perp\subseteq L^2(\P)$.
Thus $H_X$ is a Hilbert space isomorphic to $\ker(\Txq)^\perp$ 
and $T_X\q$ is a Hilbert space isomorphism
$\overline{\im(\Txx)}=\ker(T_X\q)^\perp\to H_X$.
$H_X$ is called the \emph{reproducing kernel Hilbert space} corresponding to
$X$ \cite{LedouxT}. 
Note that $H_X\subseteq B$ with a continuous inclusion,
since $T_X\q:\ker(T_X\q)^\perp\to B$ is continuous.
Furthermore, $\im(\Txx)$ is dense in
$\overline{\im(\Txx)}=\ker(T_X\q)^\perp$,
and thus $\im(\Txq\Txx)$ is dense in $H_X$.

The unit ball $K_X$ of $H_X$ is the image under $\Txq$ of the unit ball of
$L^2(\P)$. The latter unit ball is weakly compact, and since
$\Txq:L^2(\P)\to B$ is continuous, and therefore weakly continuous, $K_X$ is  
a weakly compact subset of $B$. In particular, $K_X$ is closed in $B$.

If $x\in H_X$ and $\xx\in B\q$, then $T_X\q \Txx\xx\in H_X$ and
\begin{equation}\label{q3}
  \innprod{T_X\q \Txx\xx,x}_{H_X}
=
  \innprod{\Txx\xx,(T_X\q)\qw x}
=
  \innprod{\xx,T_X\q(T_X\q)\qw x}
=   \innprod{\xx,x}.
\end{equation}
Hence, the operator $T_X\q \Txx:B\q\to H_X\subseteq B$ is the adjoint
of the inclusion $H_X\to B$.
Furthermore,
by \eqref{q3} and \eqref{q1}, for $\xx,\yy\in B\q$,
\begin{equation}\label{q4}
  \innprod{T_X\q \Txx\xx,T_X\q \Txx\yy}_{H_X}
=   \innprod{\xx,T_X\q \Txx\yy}
=\E\bigpar{\xx(X)\yy(X)}.
\end{equation}

The operator $\Txq\Txx$ and the Hilbert space
$H_X$ determine each other; more precisely, we
have the following.
\begin{theorem}\label{TXH}
If\/ $X$ and\/
$Y$ are $B$-valued \rv{s} such that $\Txx,\Tyx:B\q\to L^2(\P)$ and
$\Txq,\Tyq: L^2(\P)\to B$, then the following are equivalent:
\begin{romenumerate}
\item\label{txhh} $H_X= H_Y$
(as  vector spaces with  given inner products).
\item\label{txhb}
 The unit balls  $K_X$ and $K_Y$ 
are the same (as subsets of $B$).
\item\label{txhtt}
 $\Txq\Txx=\Tyq\Tyx$.
\item\label{txhw}
 $ \E\bigpar{\xx_1(X)\xx_2(X)} =  \E\bigpar{\xx_1(Y)\xx_2(Y)}$ for
  every $\xx_1,\xx_2\in B\q$.
\item \label{txhi}
$\E X\itpx2=\E Y\itpx2$, assuming that either both
moments exist in Pettis sense, or that $B$ is separable (in which case
the moments exist at least in Dunford sense).  
\end{romenumerate}
\end{theorem}
\begin{proof}
\ref{txhh}$\iff$\ref{txhb}: 
Each Hilbert space determines its unit ball.
Conversely, the unit ball determines the space  and its norm, and thus
the inner product by the polarisation identity 
$\innprod{x,y}=\frac14(\norm{x+y}^2-\norm{x-y}^2)$.

\ref{txhh}$\implies$\ref{txhtt}: Immediate from \eqref{q3}.

\ref{txhtt}$\implies$\ref{txhh}:
$\Txq\Txx$ determines both the set 
 $\im(\Txq\Txx)\subseteq H_X$ and, by \eqref{q4}, the inner product in $H_X$
restricted to 
this subspace; since  $\im(\Txq\Txx)$ is dense in $H_X$, and $H_X$ continuously
included in $B$, this determines $H_X$.

\ref{txhtt}$\iff$\ref{txhw}: By \eqref{q1}.

\ref{txhw}$\iff$\ref{txhi}: By \refC{Citpk=p} or \ref{Citpk=d}.
\end{proof}

As said above, $T_X\q \Txx$ can be seen as the second injective moment of
$X$; by \refT{TXH}, also the space $H_X$ can be seen as a manifestation of
the second injective moment.

The space $H_X$ is important for the law of iterated logarithm in Banach
spaces, since the unit ball $K_X$ turns out to be the natural limit set,
see Ledoux and Talagrand \cite[Chapter 8]{LedouxT} for a detailed
discussion.
In particular \cite[Theorem 8.5]{LedouxT}, 
if $B$ is separable and $S_n=\sumin X_i$ where $X_i$ are independent copies
of $X$, with $\E X=0$,
and further the sequence $S_n/\sqrt{2n\log\log n}$ is \as{}
relatively compact (which holds under rather general conditions, but not
always), then its set of limit points is \as{} $K_X$. 
See also \cite{Ale3,Ale4} for exceptional cases where $H_X$ still is
important.

\begin{example}
  Let $W\in\coi$ be standard Brownian motion, see Examples \ref{Ewiener} and
  \ref{EwienerC}. 
By \eqref{rom} and an integration by parts,
\begin{equation}
  T_X(\mu)=\intoi W(t)\dd\mu(t)=\intoi \mu[t,1]\dd W(t).
\end{equation}
Hence $\Im(\dunx)$ is the space of stochastic integrals $\intoi g(t)\dd W(t)$
where $g$ is a deterministic function of the type $g(t)=\mu[t,1]$ with
$\mu\in M\oi$, \ie, a function $g$ on $\oi$ of bounded variation.
Since $\norm{\int g\dd W}_{L^2(\P)}=\norm{g}_{L^2\oi}$ and the functions of
bounded variation are dense in $L^2\oi$, it follows that
\begin{equation}
  \overline{\Im(\dunx)}
=
\lrset{\intoi g\dd W:g\in L^2\oi}.
\end{equation}
Moreover, by \eqref{spqr},
\begin{equation}
  \dunx\q\Bigpar{\intoi g\dd W}(t)
=\lrinnprod{W(t),\intoi g(s)\dd W(s)}
=\int_0^t g(s)\dd s.
\end{equation}
Hence the reproducing Hilbert space $H_X$ is given by
\begin{equation}
  H_X=\lrset{\int_0^t g(s)\dd s:g\in L^2\oi};
\end{equation}
equivalently, this is the space of absolutely continuous functions $f$ on
$\oi$ with $f(0)=0$ and $f'\in L^2\oi$; the norm is $\norm{f'}_{L^2}$.
This is the usual \emph{Cameron--Martin space}, see \eg{} 
\cite[Example 8.19]{SJIII}.
(See \cite[Section VIII.4]{SJIII} for a generalization
to more general Gaussian processes.)
Note that the law of iterated logarithm for Brownian motion 
\cite[Theorem 27.18]{Kallenberg}
says that 
the cluster set of $Y_s(t)\=W(st)/\sqrt{2s\log\log s}\in\coi$ as $s\to\infty$
\as{} is the unit ball $K_X$ of $H_X$; this is another example of the
connection between
the law of iterated logarithm and the reproducing Hilbert space.
\end{example}

\begin{remark}
  The name \emph{reproducing Hilbert space} is in a more general context
  used for a Hilbert space $H$ of functions on some set $\cT$ such that each
  point evaluation $f\mapsto f(t)$ is a continuous linear functional on $H$,
  see \citet{Aronszajn} and \eg{} \cite[Appendix F]{SJIII}.
The definition implies that for each $t\in\cT$ there is an element $K_t\in
H$ such that $f(t)=\innprod{f,K_t}_H$ for all $f\in H$.
In particular, $K_s(t)=\innprod{K_s,K_t}_H$; the symmetric function
$K(s,t)\=K_s(t)=\innprod{K_s,K_t}$ on $\cT\times\cT$ is known as the
\emph{reproducing kernel}, see \cite[Theorem F.3]{SJIII} for some of its basic
properties. 

We can connect the space $H_X$ constructed above to this general setting by
taking $\cT=B\q$ and regarding 
elements of $B$ as functions on $B\q$ in the usual way, regarding $x\in B$
as the function $\xx\mapsto\innprod{x,\xx}$. The point evaluations are thus
the elements $\xx\in B\q$, which are continuous on $H_X\subseteq B$, and
\eqref{q3} shows that 
\begin{equation}
K_{\xx}= T_X\q \Txx\xx;  
\end{equation}
hence the reproducing kernel is, using \eqref{q4}, given by
\begin{equation}
  K(\xx,\yy)=\E\bigpar{\xx(X)\yy(X)}.
\end{equation}
In other words, $H_X$ is a reproducing Hilbert space of functions on $B\q$,
and the reproducing kernel equals the weak second moment of $X$ given by
\eqref{momweak}. 

We mention also that the construction above of $H_X$ is an instance of 
\cite[Theorem F.5]{SJIII},
taking (in the notation there)
$\cT=B\q$, $h_{\xx}=\xx(X)$, and $H=\overline{\Im(T_X)}$, 
the closed linear subspace of
$L^2(\P)$ spanned by \set{\xx(X):\xx\in B\q}.
(Then the operator $R$ defined there equals our $T_X\q$ by \eqref{txq} and
\eqref{dunford}.) 
\end{remark}

We finally mention the following result, adapted from \cite[Lemma 8.4]{LedouxT}.
For simplicity, we consider only separable $B$.
(Note that then $\Txq:L^2(\P)\to B$ by  \refR{RPpq}.) 
We do not know whether the result extends to non-separable spaces.
\begin{theorem}
  Suppose that $B$ is separable and that 
$\xx(X)\in L^2(\P)$ for every $\xx\in B\q$.
Then the following are equivalent.
\begin{romenumerate}
\item \label{hhk}
$K_X$ is a compact subset of $B$.
\item \label{hht}
$\Tx:B\q\to L^2(\P)$ is a compact operator.
\item \label{hhtq}
$\Txq:L^2(\P)\to B$ is a compact operator.
\item \label{hhtt}
$\Txq\Txx:B\q\to B$ is a compact operator.
\item \label{hhw*}
If $\xx_n\wxto0$ in $B\q$, then $\E\xx_n(X)^2\to0$.
In other words, $\Tx:B\q\to L^2(\P)$ is sequentially \weakx-norm continuous.
\item \label{hhui}
The family \set{\xx(X)^2:\norm\xx\le1} of \rv{s} is \ui.
\item \label{hhp}
The injective moment $\E X\itpx2$ exists in Pettis sense.
\end{romenumerate}
\end{theorem}

\begin{proof}
\ref{hhk}$\iff$\ref{hhtq}:
By the definition of compact operators (which says that $K_X$ is relatively
compact, \ie, $\overline{K_X}$ is compact) and 
the fact shown above that  $K_X$ is a closed subset of
$B$. 

\ref{hht}$\iff$\ref{hhtq}: Standard operator theory \cite[VI.3.4]{Conway}.

\ref{hhtq}$\implies$\ref{hhtt}: Immediate.

\ref{hhtt}$\implies$\ref{hhw*}:
Let $\xx_n\wxto0$ in $B\q$.
If $\gx\in L^2(\P)$, then $\Txq \gx\in B$, and thus
$\innprod{\Tx\xx_n,\gx}=\innprod{\xx_n,\Txq \gx}\to0$;
hence $\Tx\xx_n\wto0$ in $L^2(\P)$ and thus 
 $\Txq\Tx\xx_n\wto0$ in $B$.
Moreover,  
the sequence $\xx_n$ is bounded 
by the uniform boundedness principle
since it is \weakx{} convergent.
Since $\Txq\Txx$ is a compact operator,
it follows that  the sequence 
 $\Txq\Tx\xx_n$ is relatively compact in $B$,
and thus the weak convergence 
 $\Txq\Tx\xx_n\wto0$ 
implies norm convergence, \ie, 
 $\norm{\Txq\Tx\xx_n}\to0$.
Consequently, \eqref{q1} implies
$$
\E\xx_n(X)^2=\innprod{\Txq\Txx\xx_n,\xx_n}
\le\norm{\Txq\Txx\xx_n}_B\norm{\xx_n}_{B\q}
\to0.
$$

\ref{hhw*}$\implies$\ref{hht}: 
Let $K^*$ be
the closed unit ball of $B\q$ with the \weakx{} topology; then $K^*$ is
compact by Alaoglu's theorem 
\cite[Theorem V.3.2]{Conway}.
Moreover, since $B$ is
separable, $K^*$ is metrizable, and thus
sequential continuity on $K^*$ is equivalent to ordinary continuity.
Hence \ref{hhw*} implies 
that $\Tx$ is a continuous map from $K^*$ into $L^2(\P)$. Consequently
its image is 
compact, which means that $\Tx$ is a compact operator.

\ref{hht}$\implies$\ref{hhui}: By \ref{hht}, the family 
\set{\xx(X):\norm\xx\le1} is relatively compact in $L^2(\P)$,
which implies that
\set{\xx(X)^2:\norm\xx\le1}  is \ui.
(For example because $f\mapsto f^2$ is a continuous map 
$L^2(\P)\to L^1(\P)$, so the latter family is relatively compact in
$L^1(\P)$,
and thus \ui{} by \cite[Corollary IV.8.11]{Dunford-Schwartz}.) 

\ref{hhui}$\implies$\ref{hhw*}:
If $\xx_n\wxto0$, then $\xx_n(X)^2\to0$ \as, and thus uniform integrability
implies
$\E\xx_n(X)^2\to0$.

\ref{hhui}$\iff$\ref{hhp}:
By \refT{TiPettis}\ref{tipettis2}\ref{tip1}.
\end{proof}

\section{The Zolotarev distances}\label{Azolo}

The Zolotarev distance $\zeta_s(X,Y)$, where the parameter $s>0$ is a fixed real
number, is a measure of the distance (in some sense) between the
distributions  of two \rv{s} $X$ and $Y$ with values in the same Banach
space $B$. 
We give the definition of $\zeta_s$ in \refSS{SSzolo} below and explain the
connection to tensor products and moments, but first we recall some
preliminaries on Fr\'echet derivatives in \refSS{SSfrechet}.

Note that $\zeta_s(X,Y)$ depends only on the distributions $\cL(X)$ and
$\cL(Y)$; we may write $\zeta_s(\cL(X),\cL(Y))=\zeta_s(X,Y)$, and regard
$\zeta_s$ as a distance between probability distributions on
$B$, but it is often convenient to use the notation $\zeta_s(X,Y)$ with
\rv{s} $X$ and $Y$.

It is important to note that $\zeta_s(X,Y)$ may be infinite.
Hence $\zeta_s$ defines a metric only on suitable classes of probability
distributions on $B$, where we know \emph{a priori} that $\zeta_s(X,Y)<\infty$.

\subsection{Fr\'echet differentiablity}\label{SSfrechet}
We recall some well-known facts about derivatives of Banach space valued
functions, see \eg{} \cite{Cartan} for details.

Let $B$ and $\Bx$ be Banach spaces, and let $U$ be a non-empty open subset of
$B$. A 
function $f:U\to \Bx$ is \emph{(Fr\'echet) differentiable at a point $x\in U$}
if there exists a bounded linear operator $T:B\to \Bx$ such that
\begin{equation}\label{frechet}
  f(y)-f(x)-T(y-x)=o\bigpar{\norm{y-x}}
\end{equation}
as $y\to x$. The linear operator $T\in L(B;\Bx)$ then is uniquely determined;
it is called the \emph{derivative} of $f$ at $x$ and is denoted by $f'(x)$
or $Df(x)$.

The function $f:U\to \Bx$ is said to be \emph{differentiable}
if it is differentiable at every $x\in U$. In this case, the derivative 
$f'$ is a function $U\to L(B;\Bx)$.
If furthermore $f':U\to L(B;\Bx)$ is continuous, $f$ is said to be
\emph{continuously differentiable}.

Since $L(B;\Bx)$ is a Banach space, we may iterate: If the derivative
$f':U\to L(B;\Bx)$ is differentiable, its derivative $f''$ is called the
second derivative of $f$, and so on.
Note that $f''$ then is a function on $U$ with values in
$L(B;L(B;\Bx))=L(B^2;\Bx)$, so the second derivative $f''(x)$ at a point 
$x\in U$ is a bilinear map $B\times B\to \Bx$. More generally, the $k$:th
derivative $f\kkk(x)$, if it exists, is a $k$-linear map $B^k\to \Bx$. 
It can be shown \cite[Th\'eor\`eme 5.1.1]{Cartan} that this map is
symmetric.
Since $L(B^k;\Bx)\cong L(B\ptpk;\Bx)$ by an extension of \refT{Tproj*}, we can
also regard the $k$:th derivative $f\kkk(x)$
as a (symmetric) linear map $B\ptpk\to \Bx$.
We may take advantage of this by writing $f\kkk(x)(y,\dots,y)$
as $f\kkk(x)(y\tpk)$.

$C^0(U;\Bx)$ is defined to be the linear space of all continuous functions
$f:U\to \Bx$, 
and $C^1(U;\Bx)$ is the linear space of all continuously differentiable
functions $f:U\to \Bx$. More generally, $C^k(U;\Bx)$ is the linear space of
all $k$ times continuously differentiable functions $U\to \Bx$; this may be
defined formally by induction 
for $k\ge1$
as the space of all differentiable functions
$f:U\to \Bx$ such that $f'\in C^{k-1}(U;L(B;\Bx))$.

Note that \eqref{frechet} implies continuity at $x$; thus a differentiable
function is continuous. Hence, $C^0(U;\Bx) \supset C^1(U;\Bx) \supset
C^2(U;\Bx)\supset \dotsm$. 

For $0<\gamx\le1$, we define $\Lip_\gamx(U;\Bx)$ to be the linear space of
all functions $f:U\to \Bx$ such that
\begin{equation}\label{lipa}
 \norm{f}_{\Lip_\gamx}\=\sup_{x\neq y}\frac{\norm{f(x)-f(y)}}{\norm{x-y}^\gamx}
\end{equation}
is finite.
More generally, for $s>0$, we write $s=m+\gamx$ with $m\in\set{0,1,2\dots}$
and $0<\gamx\le1$ (thus $m=\ceil{s}-1$)
and define
\begin{equation}
\Lip_s(U;\Bx)\=\bigset{f\in C^m(U;\Bx):f\mm\in \Lip_\gamx\bigpar{U;L(B^m;\Bx)}}.
\end{equation}
with
\begin{equation}
  \label{lips}
\norm{f}_{\Lip_s}\=\norm{f\mm}_{\Lip_\gamx}.
\end{equation}

It follows from a Taylor expansion 
\cite[Th\'eor\`eme 5.6.1]{Cartan}
at 0 that if $f\in \Lip_s(B;\Bx)$, then
\begin{equation}\label{taylor}
  f(x)=\sum_{k=0}^m \frac1{k!}{f\kkk(0)(x\tpk)}
+O\bigpar{\norm{f}_{\Lip_s}\norm{x}^s},
\end{equation}
where the implicit constant is universal (it can be taken as 1).
(The term with $k=0$ in \eqref{taylor} is just the constant $f(0)$.)
In particular,
\begin{equation}\label{liz}
\norm{f(x)}=O\bigpar{1+\norm{x}^s}, \qquad x\in B.
\end{equation}

Note that $\norm{\;}_{\Lip_s}$ is only a seminorm.
In fact, we have the following.

\begin{lemma}
  \label{Llip}
Let $f:U\to \Bx$ where $U\subseteq B$ is connected, and let $s>0$.
Then the following are equivalent, with $m=\ceil s-1$,
\begin{romenumerate}
\item \label{lip0}
$\norm{f}_{\Lip_s}=0$.
\item \label{lip=}
$f\mm$ is constant.
\item \label{lipm+1}
$f\mmi(x)=0$ for all $x\in U$.
\item \label{lipsum}
$
	f(x)=\sum_{k=0}^m \gax_k(x,\dots,x),$
where $\gax_k\in L(B^k;\Bx)$ is a $k$-linear map.
\end{romenumerate}
\end{lemma}

\begin{proof}
  \ref{lip0}$\iff$\ref{lip=}:
It follows from \eqref{lipa} and \eqref{lips} that
\begin{equation*}
  \norm{f}_{\Lip_s}=0
\iff
  \norm{f\mm}_{\Lip_\gamx}=0
\iff
f\mm(x)=f\mm(y) \text{ for any $x,y\in U$}.
\end{equation*}

\ref{lip=}$\implies$\ref{lipm+1}:
Obvious by the definition of the derivative.

\ref{lipm+1}$\implies$\ref{lipsum}:
Suppose for simplicity that $U$ is convex.
(The general case follows easily, but we omit this.)
By a translation we may assume that $0\in U$, and then 
\ref{lipsum} follows from 
 Taylor's formula \cite[Th\'eor\`eme 5.6.1]{Cartan}.
(We have $\gax_k=f\kkk(0)/k!$.)

\ref{lipsum}$\implies$\ref{lip=}:
It is easily seen by induction that if each $\gax_k$ is symmetric, as we may
assume by symmetrization, then for $1\le j\le m$,
\begin{equation*}
  f^{(j)}(x)(y_1,\dots,y_j)
=\sum_{k=j}^m\frac{k!}{(k-j)!}\gax_k(x,\dots,x,y_1,\dots,y_j) 
\end{equation*}
(with $k-j$ arguments of $\gax_k$ equal to $x$); in particular,
$f^{(m)}(x)=m!\,\gax_m$, which does not depend on $x$.
\end{proof}

\subsection{Zolotarev distances}\label{SSzolo}
We now define the Zolotarev distance $\zeta_s$ as follows, for any $s>0$;
see Zolotarev \cite{Zolotarev76b,Zolotarev76,Zolotarev77,Zolotarev79} and
\eg{} \cite{NR,NeiningerS,Sulzbach} 
for further details. 

Let $B$ be a Banach space, 
and suppose that 
every bounded multilinear form $B^k\to \bbR$ is $\bw$-\meas{}
(\ie,  measurable for the product \gsf{} $\bw^k$),
for any $k\ge1$.
(For $k=1$, this holds by the definition of $\bw$.)
If $B$ is separable, this assumption always holds, since every bounded
multilinear form is continuous and thus Borel measurable on $B^k$; moreover,
when $B$ is separable the Borel \gsf{} on $B^k$ equals the product \gsf{}
$\cB^k$ and $\cB=\bw$.
The main example with $B$ non-separable is $\doi$, where 
$\bw=\cD$ by \refT{TDw} and every multilinear form is \dmeas{} by \refT{TM}.
(Another example is $c_0(S)$, 
where $\bw=\cC$  by \refT{Tc0} and every multilinear form is
\cmeas{} as a consequence of \refL{LBH}.) 

We let $\Lipw_s(B;\bbR)$ be the set of 
all $\bw$-\meas{} functions in
$\Lip_s(B;\bbR)$; this is evidently a subspace of $\Lip_s(B;\bbR)$.
If $B$ is separable, then  $\Lipw_s(B;\bbR)=\Lip(B;\bbR)$, since every
function in $\Lip_s(B;\bbR)$ is continuous and thus \Bormeas, and $\cB=\bw$. 

Let $X$ and $Y$ be two \wmeas{} $B$-valued random variables, and
suppose that $\EE\norm{X}^s, \EE\norm{Y}^s<\infty$. 
(We use upper integrals here, since the norms $\norm{X},\norm{Y}$ are not
necessarily \meas.)
Define 
\begin{equation}\label{marstrand}
  \zeta_s(X,Y)\=
\sup\bigset{|\E f(X)-\E f(Y)|:f\in \Lipw_s(B;\bbR) \text{ with } 
\norm{f}_{\Lip_s}\le1}.
\end{equation}
By assumption, $f$ is $\bw$-\meas, and thus $f(X)$ is \meas.
Moreover, by \eqref{liz}, $|f(X)|=O(1+\norm{X}^s)$ 
and by assumption $\EE\norm{X}^s<\infty$; hence 
$\E|f(X)|<\infty$.
Similarly, $\E|f(Y)|<\infty$, and thus
$|\E f(X)-\E f(Y)|$ is well-defined and finite for every $f\in\Lipw_s(B;\bbR)$.
Thus $\zeta_s(X,Y)$ is a well-defined number in $[0,\infty]$.
Moreover, 
$\E f(X)$ and $\E f(Y)$ depend only on the distributions
$\cL(X)$ and $\cL(Y)$, and thus $\zeta_s$ is really a distance between the 
distributions; we may write $\zeta_s(\cL(X),\cL(Y)) =\zeta_s(X,Y)$.

By \eqref{marstrand} and homogeneity, for any $f\in \Lipw_s(B;\bbR)$,
if $\zeta_s(X,Y)<\infty$,
\begin{equation}\label{marstrandx}
|\E f(X)-\E f(Y)| \le
\norm{f}_{\Lip_s}\,  \zeta_s(X,Y).
\end{equation}

It is clear that $\zeta_s(X,Y)\ge0$ with $\zeta_s(X,X)=0$, and that 
$\zeta_s$ is symmetric and that the triangle inequality holds.
Moreover, if $\xx\in B\q$, then $e^{\ii\xx}\in\Lip_s(B;\bbC)$ for any $s>0$
as is easily seen \cite{Zolotarev77}. Hence, by taking real and imaginary parts,
\eqref{marstrandx} implies that if $\zeta_s(X,Y)=0$, then 
\begin{equation}
  \label{chf}
\E e^{\ii\xx(X)}=\E e^{\ii\xx(Y)}
\end{equation}
and it is well-known that this implies that $\cL(X)=\cL(Y)$ on $\bw$.
(In fact, \eqref{chf} implies, by replacing $\xx$ by $t\xx$ with $t\in\bbR$,
that $\xx(X)$ and $\xx(Y)$ have the same \chf, and thus $\xx(X)\eqd\xx(Y)$
for every $\xx\in B\q$. This implies $\cL(X)=\cL(Y)$ on $\bw$ as seen in
\refT{TU} and its proof.)
Consequently, $\zeta_s$ is a metric on any set of probability distributions
on $(B,\bw)$ such that $\zeta_s$ is finite.

However, for $s>1$, $\zeta_s(X,Y)$ may be infinite.
The following lemma says exactly when $\zeta_s$ is finite.

\begin{lemma}\label{Lmac}
Suppose that $X$ and $Y$ are \wmeas{} $B$-valued \rv{s} such that  
$\EE\norm{X}^s, \EE\norm{Y}^s<\infty$.
Then the moments $\E X\ptpk$ and $\E Y\ptpk$ exist in
Dunford sense for $1\le k\le m=\ceil{s}-1$, and
the following are equivalent.
\begin{romenumerate}
\item \label{za}
$\zeta_s(X,Y)<\infty$.
\item \label{zb}
$\E \gax(X,\dots,X)=\E \gax(Y,\dots,Y)$
for every $k=1,\dots,m$ and every $\gax\in L(B^k;\bbR)$.

\item \label{zc}
$\E X\ptpk=\E Y\ptpk$ for $1\le k\le m$.
\end{romenumerate}
In particular, if $0<s\le1$, then \ref{zb} and \ref{zc} are vacuous, and
thus
$\zeta_s(X,Y)<\infty$ for all such $X$ and $Y$.
\end{lemma}
\begin{proof}
If $\ga\in L(B^k;\bbR)$ is any $k$-linear form, then $x\mapsto
\ga(x,\dots,x)$ is $\bw$-\meas{} by assumption and thus $\ga(X,\dots,X)$ is
\meas. Furthermore, if $1\le k\le m<s$, then $|\ga(X,\dots,X)|\le
\norm\ga\norm{X}^k=O(1+\norm{X}^s)$, and thus  
$\ga(X,\dots,X)$ is integrable.  \refT{TPD} shows that
$\E X\ptpk$ exists in
Dunford sense for $1\le k\le m$, and the same holds for
$\E Y\ptpk$.

\ref{za}$\implies$\ref{zb}:
Suppose that $\gax\in L(B^k;\bbR)$ with $k\le m$, and let
$f(x)\=\gax(x,\dots,x)$. By assumption, $f$ is $\bw$-\meas, and
by \refL{Llip} $f\in \Lip_s(B;\bbR)$ with $\normlips{f}=0$.
Consequently, $f\in\Lipw_s(B;\bbR)$, and if $\zeta_s(X,Y)<\infty$, then
\eqref{marstrandx} yields $\E f(X)=\E f(Y)$.

\ref{zb}$\implies$\ref{za}
Suppose that $f\in\Lipw_s(B;\bbR)$
with $\normlips{f}\le1$. By \eqref{taylor}, there exist
$k$-linear forms $\gax_k=f\kkk(0)/k!\in L(B^k;\bbR)$ and a function
$g:B\to\bbR$ with $|g(x)|\le\norm{x}^s$
such that
\begin{equation}\label{bit}
  f(x)=\sum_{k=0}^m \gax_k(x,\dots,x) + g(x).
\end{equation}
(Here, $\gax_0=f(0)$ is just a real constant.)
By assumption, $f(X)$ and all
$\gax_k(X,\dots,X)$ 
are \meas, and thus $g(X)$ is \meas, and similarly $g(Y)$ is \meas.
Hence we can use the decomposition \eqref{bit} and obtain 
using \ref{zb} (and the fact that $\gax_0$ is a constant),
with all terms finite by the assumptions $\EE \norm{X}^s,\EE \norm{Y}^s<\infty$,
\begin{equation*}
  \begin{split}
&\bigabs{\E f(X)-\E f(Y)	}
\\&\qquad
=
\lrabs{\sum_{k=0}^m \E\gax_k(X,\dots,X) + \E g(X)
-
\sum_{k=0}^m \E\gax_k(Y,\dots,Y) - \E g(Y)}
\\&\qquad
=\bigabs{\E g(X)-\E g(Y)}
\le \E|g(X)|+\E|g(Y)|
\le \EE\norm X^s + \EE\norm Y^s,
 \end{split}
\end{equation*}

Taking the supremum over all such $f$ yields
\begin{equation}
  \zeta_s(X,Y) \le \EE\norm X^s + \EE\norm Y^s<\infty.
\end{equation}

\ref{zb}$\iff$\ref{zc}:
This is \refC{Cptpk=}.
\end{proof}

\begin{example}
  \label{E123}
If $0<s\le1$, then \refL{Lmac} shows that 
$\zeta_s$ is a metric on the set of all
probability distributions on $(B,\bw)$ with a finite $s$:th moment of the norm
$\EE\norm{X}^s$.

If $1<s\le2$, and we still assume
$\EE\norm{X}^s, \EE\norm{Y}^s<\infty$, \refL{Lmac} shows that
$\zeta_s(X,Y)<\infty$ if and only if $\E X=\E Y$.
Hence $\zeta_s$ is a metric on the set of probability distributions with
$\EE\norm{X}^s$ finite and a given expectation. In this case it is natural to 
work in the set of probability distributions with expectation 0, which
easily is achived by subtracting the means from the variables, so this is no
serious restriction.

The next case $2<s\le3$ is substantially more complicated.
\refL{Lmac} shows that we need not only $\E X=\E Y$ but also 
$\E X\ptpx2=\E Y\ptpx2$. For real-valued \rv{s}, it is standard to obtain
this by 
considering the standardized variable $(\Var X)\qqw(X-\E X)$.
This extends to finite-dimensional spaces, where $\Var X$ is the covariance
matrix, 
see \eg{} \cite{NR},
but not to infinite-dimensional ones, and this is a serious problem
when using $\zeta_s$ with $s>2$ in Banach spaces. Nevertheless, it is at least
sometimes possible to modify the variables to achieve the desired exact
equality of the second moments, see \eg{} \cite{NeiningerS}, and results
like Theorems \ref{Tapprox} and \ref{TC} then are useful.

For $s>3$ we need not only equal first and second moments, but
also equal third moments
$\E X\ptpx3=\E Y\ptpx3$. This can in general not be achieved by any norming, not
even for real-valued \rv{s}.
If the variables are symmetric, however, all odd moments vanish. Thus, for
$s\le4$,
$\zeta_s$ is a metric of the set of all symmetric probability distributions
with a given second projective moment, and $\EE\norm{X}^s$ finite.
It seems likely that there might be applications with symmetric \rv{s} and
$3<s\le4$, but we do not know of any such cases, or of any other
applications of $\zeta_s$ with $s>3$.
\end{example}

The main use of the Zolotarev distances is to prove convergence in
distribution; the idea is that if $(X_n)\seq$ is a sequence of $B$-valued
\rv{s},  we can try to prove $X_n\dto X$ (for a suitable \rv{} $X$ in
$B$) by first proving 
\begin{equation}
  \label{zolo}
\zeta(X_n,X)\to0. 
\end{equation}
It turns out that \eqref{zolo} is by itself sufficient for convergence in
distributions in some 
Banach spaces, for example when $B$ is a Hilbert space \cite{GineL,SJ194},
but not always, for example not in \coi{} where extra conditions are needed
\cite{NeiningerS,Sulzbach}.

\begin{example}
One example where this approach has been particularly successful is
the contraction method used to prove convergence in distribution of a
sequence $X_n$ of random variables when
there is a recursive structure of the type
\begin{equation}\label{uwe}
  X_n \eqd \sum_{r=1}^K A_{r,n} X_{I_{r,n}}^{(r)}+b_n,
\end{equation}
where $(X_j^{(1)})_{j=0}^\infty,\dots,(X_j^{(K)})_{j=0}^\infty$ 
are \iid{} copies of 
$(X_n)_{n=0}^\infty$, and
$A_{r,n}$, $I_{r,n}$ and $b_n$  are given random variables independent
of all $X_j^{(r)}$.
(The idea is, roughly, to first find a good candidate $X$ for the limit by 
formally letting \ntoo{} in \eqref{uwe}, assuming that $A_{r,n}$, $I_{r,n}$
and $b_n$ converge in some suitable way. Then one uses \eqref{uwe} to obtain a
recursive estimate of the distance $\zeta_s(X_n,X)$, and use this to show
that $\zeta_s(X_n,X)\to0$.)
This method was introduced for real-valued random variables by
R\"osler \cite{Rosler1,Rosler2,Rosler99}, 
and has been extended to variables with
values in 
$\bbR^d$ \cite{Neininger00,NR},
$L^p\oi$ \cite{EickR},
Hilbert spaces \cite{SJ194},
\coi{} and \doi{} \cite{NeiningerS,Sulzbach};
see further \cite{NeiningerS,Sulzbach} and the further references given there.
These papers use not only the Zolotarev distance $\zeta_s$ considered here
but also some other probability metrics; nevertheless
the Zolotarev distances are essential in several of the applications.
Moreover, 
it is easily verified that for any real constant $t$,
\begin{equation}
\zeta_s(tX,tY)=|t|^s\zeta_s(X,Y)
\end{equation}
and, for any bounded linear operator $T\in L(B;B)$,
\begin{equation}
  \zeta_s(TX,TY)\le\|T\|^s\zeta_s(X,Y).
\end{equation}
In applications, these relations are typically applied with $|t|$ or $\norm
T$ small, and it is then advantageous to take $s$ large. In fact, in some
applications (see \eg{} \cite{NeiningerS})
one is forced to take $s>2$ in order to obtain the required
estimates, and then, as seen in \refL{Lmac} and \refE{E123}, it is essential
to have (or arrange) equalities of the second moments $\E X_n\ptpx2=\E
X\ptpx2$ in order for $\zeta_s(X_n,X)<\infty$ (which is necessary in order
to even start the recursion sketched above).
\end{example}

\newcommand\AAP{\emph{Adv. Appl. Probab.} }
\newcommand\JAP{\emph{J. Appl. Probab.} }
\newcommand\JAMS{\emph{J. \AMS} }
\newcommand\MAMS{\emph{Memoirs \AMS} }
\newcommand\PAMS{\emph{Proc. \AMS} }
\newcommand\TAMS{\emph{Trans. \AMS} }
\newcommand\AnnMS{\emph{Ann. Math. Statist.} }
\newcommand\AnnPr{\emph{Ann. Probab.} }
\newcommand\CPC{\emph{Combin. Probab. Comput.} }
\newcommand\JMAA{\emph{J. Math. Anal. Appl.} }
\newcommand\RSA{\emph{Random Struct. Alg.} }
\newcommand\ZW{\emph{Z. Wahrsch. Verw. Gebiete} }
\newcommand\DMTCS{\jour{Discr. Math. Theor. Comput. Sci.} }

\newcommand\AMS{Amer. Math. Soc.}
\newcommand\Springer{Springer-Verlag}
\newcommand\Wiley{Wiley}

\newcommand\vol{\textbf}
\newcommand\jour{\emph}
\newcommand\book{\emph}
\newcommand\inbook{\emph}
\def\no#1#2,{\unskip#2, no. #1,} 
\newcommand\toappear{\unskip, to appear}

\newcommand\urlsvante{\url{http://www.math.uu.se/~svante/papers/}}
\newcommand\arxiv[1]{\url{arXiv:#1.}}
\newcommand\arXiv{\arxiv}

\def\nobibitem#1\par{}


\end{document}